\documentclass[11pt,reqno]{amsart}
\usepackage{amsfonts}
\usepackage{}  
\usepackage{amsmath,amssymb,amsthm, comment,graphicx,color, cite}
\usepackage{tikz}
\usepackage{fancyhdr}
\usepackage{epsfig}
\usepackage{mathrsfs}
\usepackage{latexsym,hyperref}
\usepackage[margin=1.0in]{geometry}

\allowdisplaybreaks

\newtheorem{theorem}{Theorem}[section]
\newtheorem{lemma}{Lemma}[section]
\newtheorem{proposition}{Proposition}[section]
\newtheorem{corollary}{Corollary}[section]

\theoremstyle{definition}

\newtheorem{remark}{Remark}[section]

\numberwithin{equation}{section}

\newcommand{\dd}{\mathrm{d}}
\newcommand{\p}{\partial}

\allowdisplaybreaks

\begin{document}
\title[
Supersonic contact discontinuity in two-dimensional finitely long nozzle]
{Steady supersonic combustion flows with a contact discontinuity in two-dimensional finitely long nozzles}

\author[J. Gao]{Junlei Gao}
\address{School of Mathematical Sciences, Zhejiang Normal University, Jinhua, 321004, China}
\email{gaojunlei123math@zjnu.edu.cn}

\author[F. Huang]{Feimin Huang}
\address{Academy of Mathematics and Systems Science, Chinese Academy of Sciences, Beijing 100190, China}
\email{fhuang@amt.ac.cn}

\author[J. Kuang]{Jie Kuang}
\address{Innovation Academy for Precision Measurement Science and
Technology, Chinese Academy of Sciences, Wuhan 430071, China}
\email{jkuang@apm.ac.cn}

\author[D. Wang]{Dehua Wang}
\address{Department of Mathematics,
University of Pittsburgh,
Pittsburgh, PA 15260, USA.}
\email{dwang@math.pitt.edu}

\author[W. Xiang]{Wei Xiang}
\address{Department of Mathematics
City University of Hong Kong
Kowloon, Hong Kong, China}
\email{weixiang@cityu.edu.hk}


\date{\today}

\begin{abstract}
{In this paper, we are concerned with the two-dimensional steady supersonic combustion flows with a contact discontinuity moving through a nozzle of finite length.
Mathematically, it can be formulated as a free boundary value problem governed by the two-dimensional steady combustion Euler equations with a contact discontinuity as the free boundary.
The main mathematical difficulties are that the contact discontinuity is a characteristic free boundary and the equations for all states are coupled with each other due to the combustion process.
We first employ the Lagrangian coordinate transformation to fix the free boundary.
Then by introducing the flow slope and Bernoulli's function, we further reduce the fixed boundary value problem into an initial boundary value problem for a first-order inhomogeneous hyperbolic system coupled with several ordinary differential equations.
A new iteration scheme is developed near the background states by employing the intrinsic structure of the equation for the mass fraction of unburned gas.
We show that there is a fixed point for the iteration by deriving some novel $C^{1,\alpha}$-estimates of the solutions and applying the Schaulder fixed point theorem, and then the uniqueness of the fixed point is proved by a contraction argument.
Based on the inverse Lagrangian coordinate transformation, we show that the original free boundary problem admits a unique contact discontinuity solution provided that the incoming flows and the nozzle walls are small $C^{1,\alpha}$-perturbations of the background states.
On the other hand, a quasi-one-dimensional approximate system is often used to simplify the two-dimensional steady supersonic combustion model (cf. \cite{menshov}). The error between these two systems is estimated. Finally, given a piecewise $C^{1,\alpha}$-solution containing a contact discontinuity with piecewise constant states on the entrance of the nozzle, we can show that the solution is the piecewise constant states with a straight contact discontinuity.}
\end{abstract}

\maketitle


\section{Introduction}\setcounter{equation}{0}
In this paper, we are concerned with the steady supersonic combustion flow going through a two-dimensional finitely long nozzle with a contact discontinuity. As a typical elementary nonlinear wave phenomenon in gas dynamics, a contact discontinuity is often generated in nozzles. For example, a vehicle flying with a supersonic/ hypersonic speed will have a contact discontinuity in the scramjet nozzle (see \cite{courant-friedrichs, li-xu-yu-lv-cheng}). A rigorous mathematical analysis of the steady contact discontinuity in nozzles plays an important role in aerodynamics.
Mathematically, the study of the contact discontinuity for steady compressible flows can be classified into three categories: subsonic-subsonic flow, supersonic-supersonic flow and subsonic-supersonic flow (\emph{i.e.}, transonic flow).
Much mathematical progress has been made on these topics recently.
For the subsonic contact discontinuity in a nozzle, Bae \cite{bae} and Bae-Park \cite{bae-park1,bae-park2} established
the stability of straight subsonic contact discontinuity in two-dimensional almost straight infinitely long nozzles, in two-dimensional straight semi-infinitely long nozzles, and for axial symmetric flow in three dimensional axial symmetric nozzles. 
Chen-Huang-Wang-Xiang \cite{chen-huang-wang-xiang} obtained the existence and uniqueness of the two-dimensional steady contact discontinuity solution in general nozzles without any smallness assumptions. For the supersonic flow in nozzles, it will generally blow up if the flow is not constant and the nozzle are long enough (see \cite{huang-kuang-wang-xiang1}). So 
special structures of the nozzle, such as expanding nozzle (see \cite{chen-qu, lai,xu-yin2}), are needed to make sure that singularities can not be generated.
Recently, we \cite{huang-kuang-wang-xiang1} established the stability of the steady supersonic contact discontinuity governed by the full Euler equations in a two-dimensional finitely long nozzle.
For the transonic contact discontinuity in a nozzle, we \cite{huang-kuang-wang-xiang2} established the stability of such steady structure in a two-dimensional finitely long nozzle. {See also \cite{weng-zhang-1, weng-zhang-2, weng-zhang-3} for more recent progresses on subsonic and supersonic contact discontinuity in two or
three-dimensional axial symmetric finitely long nozzles.}
Meanwhile, great progress has also been made on other important problems involving steady contact discontinuities recently. One can refer \cite{wang-yu} for local nonlinear stability of three-dimensional steady contact discontinuity, \cite{chen-xin-zang} for subsonic Euler flows past an airfoil, \cite{cheng-du,cheng-du-wang,cheng-du-wang-2,cheng-du-xiang,cheng-du-zhang,du-wang,shi-tang-xie} for compressible subsonic jet flow, \cite{chensx1,chensx2, chen-hu-fang, fang-wang-yuan} for Mach reflections, \cite{chen-kukreja-yuan,pei-xiang,qu-xiang,wang-yuan} for supersonic/transonic Euler flows past a convex cornered wedge, and \cite{chen-kukreja,chen-zhang-zhu,xiang-zhang-zhao} for supersonic contact discontinuity over a wall.
Note that the effects of mass-additions and heat addition (rejection), as well as frictions in \cite{gao-yuan, gao-liu-yuan, yuan-zhao} on the stability of some steady flow structures 
are to some extent analogous to the combustion Euler flows considered in this paper.

\subsection{Mathematical problems and main results}
In the Euclidean coordinates, the steady inviscid compressible combustion Euler equations in two dimensions are
\begin{eqnarray}\label{eq:1.1}
\begin{cases}
\partial_x(\rho u)+\partial_y(\rho v)=0, \\
\partial_{x}(\rho u^{2}+p)+\partial_{y}(\rho uv)=0, \\
\partial_{x}(\rho uv)+\partial_{y}(\rho v^{2}+p)=0, \\
\partial_{x}\big(\rho u(E+\frac{p}{\rho})\big)+\partial_{y}\big(\rho v(E+\frac{p}{\rho})\big)=0,\\
\partial_{x}(\rho uY)+\partial_{y}(\rho vY)=-\rho\phi(T)Y,
\end{cases}
\end{eqnarray}
where
$(u, v)$, $\rho$, $p$ and $T$ are the velocity, density, pressure and temperature, respectively,
$\phi$ is the reaction rate which is a function of $T$,
$Y$ denotes the mass fraction of unburned gas with $Y\in[0,1]$, and $E$ is the total energy of the following form
\begin{eqnarray}\label{eq:1.2}
E=\frac{u^2+v^2}{2}+e+\mathfrak{q}_0Y.
\end{eqnarray}
Here $\mathfrak{q}_0>0$ is the special binding energy of the unburned gas and $e$ is the internal energy. The thermodynamic relation for $p$, $\rho$, $T$ and the entropy $S$ (cf. \cite{courant-friedrichs}) is
\begin{eqnarray}\label{eq:1.3}
de=\frac{p}{\rho^2}d\rho+ TdS.
\end{eqnarray}
For the ideal polytropic gas, we have the constitutive relations from \eqref{eq:1.3}: 
\begin{eqnarray}\label{eq:1.4}
p=A(S)\rho^{\gamma},\quad\  e=\textsf{c}_{\nu}T,\quad\ T=\frac{p}{\mathcal{R}\rho},\quad\ \gamma=1+\frac{\mathcal{R}}{\textsf{c}_{\nu}},\quad\ A(S)=(\gamma-1)\textrm{e}^{\frac{S-S_0}{\textsf{c}_{\nu}}},
\end{eqnarray}
where $S_0$, $\textsf{c}_{\nu}$, and $\mathcal{R}$ are positive constants.
The sonic speed of the ideal polytropic gas is
\begin{eqnarray}\label{eq:1.5}
c=\sqrt{\frac{\gamma p}{\rho}},
\end{eqnarray}
and the Mach number is given by
\begin{eqnarray}\label{eq:1.6}
M=\frac{\sqrt{u^2+v^2}}{c}.
\end{eqnarray}

We call the flow  supersonic if $M>1$, subsonic if $M<1$, and sonic if $M=1$. For supersonic flow, system \eqref{eq:1.1} is a hyperbolic system of balance law.
For subsonic flow, it is of hyperbolic-elliptic composited type. And for transonic flow, it is of hyperbolic-elliptic composited and mixed type.

Set $U\triangleq(u,v,p,\rho, Y)^{\top}$. For solutions $U\in C^1$, we can further derive from \eqref{eq:1.1} and \eqref{eq:1.4} that
\begin{eqnarray}\label{eq:1.7}
 \partial_{x}(\rho u A(S))+\partial_{y}(\rho v A(S))=(\gamma-1)\mathfrak{q}_0 \rho^{2-\gamma}\phi(T)Y,\quad\ \partial_{x}(\rho u B)+\partial_{y}(\rho v B)=\mathfrak{q}_0 \rho\phi(T)Y,
\end{eqnarray}
where $B$ is the Bernoulli's function defined by
\begin{eqnarray}\label{eq:1.8}
B\triangleq\frac{u^2+v^2}{2}+\frac{\gamma p}{(\gamma-1)\rho}.
\end{eqnarray}

\begin{figure}[ht]
\begin{center}
\begin{tikzpicture}[scale=0.9]

\draw [line width=0.03cm](-3.0,0.4)to[out=10, in=-165](3.2,0.4);
\draw [line width=0.03cm][red][dashed](-3.0,-0.55)to[out=10, in=-165](3.21,-0.7);
\draw [line width=0.03cm](-3.0,-1.7)to[out=-15, in=170](3.2,-1.9);

\draw [thin][->](1.0,-0.3)to[out=--10, in=-165](2.0,-0.3);
\draw [thin][->](1.0,-1.3)to[out=-18, in=-185](2.0,-1.3);

\draw [line width=0.02cm][black](-3.0,-0.45)--(-3.0,0.4);
\draw [line width=0.02cm][blue](-3.0,-1.7)--(-3.0,-0.45);
\draw [thin](3.2,-1.9) --(3.2,0.4);

\draw [thin][->] (-3.1,0) --(-3.7,0.3);
\draw [thin][->] (-3.1,-1.0) --(-3.7,-1.2);
\draw [thin][->] (1.5,0.4) --(2.2,0.8);
\draw [thin][->] (2.8,-0.3) --(3.7,0.1);
\draw [thin][->] (1.5,-1.9) --(2.4,-2.1);

\node at (-3.3,-0.55) {$O$};

\node at (-0.2,-0.3) {$\mathcal{N}_{+}$};
\node at (-0.2,-1.2) {$\mathcal{N}_{-}$};
\node at (-4.0,0.3) {$\Gamma^{+}_{\textrm{in}}$};
\node at (-4.0,-1.2) {$\Gamma^{-}_{\textrm{in}}$};
\node at (2.5,0.8) {$\Gamma_{+}$};
\node at (4.05,0.1) {$\Gamma_{\textrm{cd}}$};
\node at (2.8,-2.2) {$\Gamma_{-}$};
\end{tikzpicture}
\end{center}
\caption{Supersonic Combustion Contact Discontinuity in a Finitely Long Nozzle}\label{fig1.1}
\end{figure}
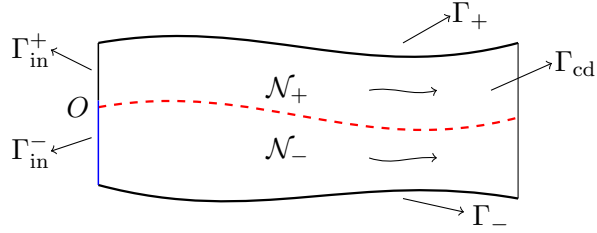

As shown in Fig. \ref{fig1.1}, we suppose that the steady supersonic combustion Euler flow is divided by a contact discontinuity into two layers and goes through a two-dimensional finitely long nozzle from left to right.
Let $g_{\pm}(x)$ be the given functions of $x\in[0,L]$ with $g_{+}>0$ and $g_{-}<0$. Denote by
\begin{eqnarray*}
\mathcal{N}\triangleq\big\{(x,y)\in \mathbb{R}^{2}|\ x\in(0,L),\ g_{-}(x)<y<g_{+}(x)\big\},
\end{eqnarray*}
the two-dimensional nozzle with its walls being described as
\begin{eqnarray*}
\Gamma_{+}\triangleq\big\{(x,y)\in \mathbb{R}^{2}|\ x\in(0,L),\ y=g_{+}(x)\big\},\quad \Gamma_{-}\triangleq\big\{(x,y)\in \mathbb{R}^{2}|\ x\in(0,L),\ y=g_{-}(x)\big\}.
\end{eqnarray*}

Suppose that the contact discontinuity in $\mathcal{N}$ is given by
\begin{eqnarray*}
\Gamma_{\textrm{cd}}\triangleq\big\{(x,y)\in \mathbb{R}^{2}|\  x\in(0,L),\ y=g_{\textrm{cd}}(x)\big\},
\end{eqnarray*}
where $g_{\textrm{cd}}(x)\in (g_{-}(x), g_{+}(x))$ for all $x\in[0,L]$ with $g_{\textrm{cd}}(0)=0$.
Then nozzle $\mathcal{N}$ is divided into two supersonic regions $\mathcal{N}_{\pm}$ as 
\begin{eqnarray*}
\mathcal{N}_{+}\triangleq\mathcal{N}\cap\big\{ g_{\rm cd}(x)<y<g_{+}(x)\big\},\quad\
\mathcal{N}_{-}\triangleq\mathcal{N}\cap\big\{g_{-}(x)<y<g_{\rm cd}(x)\big\}.
\end{eqnarray*}
The corresponding entrances of $\mathcal{N}_{\pm}$ are defined by
\begin{eqnarray*}
\Gamma^{+}_{\rm in}\triangleq\big\{(x,y)\in \mathbb{R}^{2}|\ x=0,\ 0<y<g_{+}(0) \big\},\quad\
\Gamma^{-}_{\rm in}\triangleq\big\{(x,y)\in \mathbb{R}^{2}|\ x=0,\ g_{-}(0)<y<0 \big\}.
\end{eqnarray*}

Denote by $U_{\pm}\triangleq(u_{\pm},v_{\pm}, p_{\pm},\rho_{\pm}, Y_{\pm})^{\top}$  the solutions to system \eqref{eq:1.1} in $\mathcal{N}_{\pm}$, respectively.
Then, along the walls $\Gamma_{\pm}$, $U_{\pm}$ satisfy
\begin{eqnarray}\label{eq:1.9}
(u_{\pm}, v_{\pm})\cdot \boldsymbol{\nu}_{\pm}=0, \qquad \mbox{on}\quad\ \Gamma_{\pm},
\end{eqnarray}
where $\boldsymbol{\nu}_{\pm}=\frac{(\mp g'_{\pm}, \pm1)}{\sqrt{1+g'^2_{\pm}}}$ are the unit outer normal vectors to the walls $\Gamma_{\pm}$.

On $\Gamma_{\textrm{cd}}$, it follows from the \emph{Rankine-Hugoniot} conditions that $U_{\pm}$ satisfy
\begin{eqnarray}\label{eq:1.10}
(u_{+}, v_{+})\cdot \boldsymbol{\nu}_{\textrm{cd}}=(u_{-}, v_{-})\cdot \boldsymbol{\nu}_{\textrm{cd}}=0,\qquad p_{+}=p_{-},
\quad \mbox{on} \quad \Gamma_{\textrm{cd}},
\end{eqnarray}
where $\boldsymbol{\nu}_{\textrm{cd}}=\frac{(-g'_{\textrm{cd}}(x),1)}{\sqrt{1+g'^2_{\textrm{cd}}(x)}}$ is {the} unit normal vector to $\Gamma_{\textrm{cd}}$.

At the entrance of the nozzle, we {impose} the conditions: 
\begin{eqnarray}\label{eq:1.11}
U_{\pm}=U^{\pm}_{\textrm{in}}(y)\triangleq\big(u^{\pm}_{\textrm{in}}, v^{\pm}_{\textrm{in}}, p^{\pm}_{\textrm{in}},\rho^{\pm}_{\textrm{in}}, Y^{\pm}_{\textrm{in}}\big)^{\top}(y),
\qquad \mbox{on}\quad \Gamma^{\pm}_{\textrm{in}},
\end{eqnarray}
where
\begin{eqnarray}\label{eq:1.12}
\sqrt{(u^{\pm}_{\textrm{in}})^2+(v^{\pm}_{\textrm{in}})^2}>c^{\pm}_{\textrm{in}}=\sqrt{\frac{\gamma p^{\pm}_{\textrm{in}}}{\rho^{\pm}_{\textrm{in}}}},\quad\ Y^{\pm}_{\textrm{in}}\in [0,1],  \qquad \mbox{on}\quad \Gamma^{\pm}_{\textrm{in}}.
\end{eqnarray}
Moreover, we also impose the following compatibility conditions for the upper and lower incoming supersonic flows at the corner points $O=(0,0)$ and $P_{\pm}=(0,g_{\pm}(0))$, respectively,
{\small \begin{eqnarray}\label{eq:1.13}
\begin{cases}
\Big(\frac{v^{+}_{\textrm{in}}}{u^{+}_{\textrm{in}}}\Big)(0)=\Big(\frac{v^{-}_{\textrm{in}}}{u^{-}_{\textrm{in}}}\Big)(0),\quad
p^{+}_{\textrm{in}}(0)=p^{-}_{\textrm{in}}(0), \\
\Big[\frac{(c^{+}_{\rm in})^2 v^{+}_{\rm in}}
{u^{+}_{\rm in}((u^{+}_{\rm in})^{2}-(c^{+}_{\rm in})^{2})}\big(\frac{v^{+}_{\rm in}}{u^{+}_{\rm in}}\big)'\Big](0)+\Big[\frac{(u^{+}_{\rm in})^2+(v^{+}_{\rm in})^2-(c^{+}_{\rm in})^2}
{\rho^{+}_{\rm in}(u^{+}_{\rm in})((u^{+}_{\rm in})^{2}-(c^{+}_{\rm in})^{2})}(p^{+}_{\rm in})'\Big](0)-\Big[\frac{(\gamma-1)\mathfrak{q}_{0}\phi(T^{+}_{\rm in})
v^{+}_{\rm in}}{(u^{+}_{\rm in})^2((u^{+}_{\rm in})^2-(c^{+}_{\rm in})^2)}Y^{+}_{\rm in}\Big](0) \\
\quad =\Big[\frac{(c^{-}_{\rm in})^2 v^{-}_{\rm in}}
{u^{-}_{\rm in}((u^{-}_{\rm in})^{2}-(c^{-}_{\rm in})^{2})}\big(\frac{v^{-}_{\rm in}}{u^{-}_{\rm in}}\big)'\Big](0)+\Big[\frac{(u^{-}_{\rm in})^2+(v^{-}_{\rm in})^2-(c^{-}_{\rm in})^2}
{\rho^{-}_{\rm in}(u^{-}_{\rm in})((u^{-}_{\rm in})^{2}-(c^{-}_{\rm in})^{2})}(p^{-}_{\rm in})'\Big](0)-\Big[\frac{(\gamma-1)\mathfrak{q}_{0}\phi(T^{-}_{\rm in})
v^{-}_{\rm in}}{(u^{-}_{\rm in})^2((u^{-}_{\rm in})^2-(c^{-}_{\rm in})^2)}Y^{-}_{\rm in}\Big](0), \\
\Big[\frac{(c^{+}_{\rm in})^2 (u^{+}_{\rm in})^3}
{u^{+}_{\rm in}((u^{+}_{\rm in})^{2}-(c^{+}_{\rm in})^{2})}\big(\frac{v^{+}_{\rm in}}{u^{+}_{\rm in}}\big)'\Big](0)+\Big[\frac{(c^{+}_{\rm in})^2 v^{+}_{\rm in}}
{u^{+}_{\rm in}((u^{+}_{\rm in})^{2}-(c^{+}_{\rm in})^{2})}(p^{+}_{\rm in})'\Big](0)-\Big[\frac{(\gamma-1)\mathfrak{q}_{0}\phi(T^{+}_{\rm in})
\rho^{+}_{\rm in}u^{+}_{\rm in}}{(u^{+}_{\rm in})^2-(c^{+}_{\rm in})^2}Y^{+}_{\rm in}\Big](0) \\
\quad =\Big[\frac{(c^{-}_{\rm in})^2 (u^{-}_{\rm in})^3}
{u^{-}_{\rm in}((u^{-}_{\rm in})^{2}-(c^{-}_{\rm in})^{2})}\big(\frac{v^{-}_{\rm in}}{u^{-}_{\rm in}}\big)'\Big](0)+\Big[\frac{(c^{-}_{\rm in})^2 v^{-}_{\rm in}}
{u^{-}_{\rm in}((u^{-}_{\rm in})^{2}-(c^{-}_{\rm in})^{2})}(p^{-}_{\rm in})'\Big](0)-\Big[\frac{(\gamma-1)\mathfrak{q}_{0}\phi(T^{-}_{\rm in})
\rho^{-}_{\rm in}u^{-}_{\rm in}}{(u^{-}_{\rm in})^2-(c^{-}_{\rm in})^2}Y^{-}_{\rm in}\Big](0),
\end{cases}
\end{eqnarray}
}
and
\begin{eqnarray}\label{eq:1.14}
\begin{cases}
g'_{\pm}(0)=\Big(\frac{v^{\pm}_{\textrm{in}}}{u^{\pm}_{\textrm{in}}}\Big)\big(g_{\pm}(0)\big),\\
g''_{\pm}(0)=-\Big[\frac{(c^{\pm}_{\rm in})^2 v^{\pm}_{\rm in}}
{u^{\pm}_{\rm in}((u^{\pm}_{\rm in})^{2}-(c^{\pm}_{\rm in})^{2})}\big(\frac{v^{\pm}_{\rm in}}{u^{\pm}_{\rm in}}\big)'\Big](g_{\pm}(0))-\Big[\frac{(u^{\pm}_{\rm in})^2+(v^{\pm}_{\rm in})^2-(c^{\pm}_{\rm in})^2}
{\rho^{\pm}_{\rm in}(u^{\pm}_{\rm in})^2((u^{\pm}_{\rm in})^{2}-(c^{\pm}_{\rm in})^{2})}(p^{\pm}_{\rm in})'\Big](g_{\pm}(0)) \\
\qquad\qquad+\Big[\frac{(\gamma-1)\mathfrak{q}_{0}\phi(T^{+}_{\rm in})
v^{+}_{\rm in}}{(u^{+}_{\rm in})^2((u^{+}_{\rm in})^2-(c^{+}_{\rm in})^2)}Y^{+}_{\rm in}\Big](g_{\pm}(0)).
\end{cases}
\end{eqnarray}

Now we can formulate the problem of a two-dimensional supersonic combustion contact discontinuity in nozzles into the following nonlinear free boundary value problem. 

\emph{$\mathbf{Problem\ 1.1.}$}\ For given datum $U^{\pm}_{\textrm{in}}$ and $g_{\pm}(x)$ satisfying \eqref{eq:1.12}, \eqref{eq:1.13} and \eqref{eq:1.14},
try to find  $U_{\pm}$ with a single contact discontinuity $\Gamma_{\textrm{cd}}$ in the nozzle
$\mathcal{N}$, which issues from the corner point $O$ at the entrance such that

\rm(i) $\mathcal{N}$ is divided by the contact discontinuity $\Gamma_{\textrm{cd}}$ into two supersonic domains $\mathcal{N}_{+}$ and $\mathcal{N}_{-}$;

\rm(ii) In the upper supersonic region $\mathcal{N}_{+}$, $U_{+}$ satisfies Euler equations \eqref{eq:1.1}, initial conditions \eqref{eq:1.11} at the entrance $\Gamma^{+}_{\textrm{in}}$ and boundary conditions \eqref{eq:1.9} on the upper nozzle wall $\Gamma_{+}$;

\rm(iii) In the lower supersonic region $\mathcal{N}_{-}$, $U_{-}$ satisfies Euler equations \eqref{eq:1.1}, initial conditions \eqref{eq:1.11} at the entrance $\Gamma^{-}_{\textrm{in}}$ and boundary conditions \eqref{eq:1.9} on the lower nozzle wall $\Gamma_{-}$;

\rm(iv) On $\Gamma_{\textrm{cd}}$, $U_{+}$ and $U_{-}$ satisfy conditions \eqref{eq:1.10}.

\smallskip
As a special solution to \emph{Problem\ 1.1}, we consider the case that a uniform supersonic incoming flow with a straight contact discontinuity at $y\equiv0$ goes through a two-dimensional finitely long straight nozzle with $g_{\pm}(x)\triangleq\underline{g}_{\pm}\equiv\pm1$ for $x\in[0,L]$ (see Fig.\ref{fig1.2} below), and
\begin{eqnarray}\label{eq:1.15}
U^{\pm}_{\textrm{in}}=\underline{U}^{\pm}_{\textrm{in}}\triangleq(\underline{u}^{\pm}, 0, \underline{p}^{\pm}, \underline{\rho}^{\pm},0)^{\top} \qquad \mbox{on}\qquad \underline{\Gamma}^{\pm}_{\textrm{in}},
\end{eqnarray}
where
\begin{eqnarray}\label{eq:1.16}
\underline{\rho}^{\pm}>0, \qquad \underline{u}^{\pm}>\underline{c}^{\pm}, \qquad \underline{p}^{+}=\underline{p}^{-} \qquad \mbox{for}\quad \underline{c}^{\pm}\triangleq\sqrt{\frac{\gamma\underline{p}^{\pm}}{\underline{\rho}^{\pm}}},
\end{eqnarray}
and 
\begin{eqnarray*}
\underline{\Gamma}^{+}_{\textrm{in}}\triangleq\big\{(x,y)\in\mathbb{R}^2|\ x=0,\ 0<y<1\big\},\quad\ \underline{\Gamma}^{-}_{\textrm{in}}\triangleq\big\{(x,y)\in\mathbb{R}^2|\ x=0,\ -1<y<0\big\}.
\end{eqnarray*}

\begin{figure}[ht]
\begin{center}
\begin{tikzpicture}[scale=1.0]
\draw [line width=0.04cm](-1.5,-1.5) --(3.2,-1.5);
\draw [line width=0.04cm](-1.5,0.4)--(3.2,0.4);
\draw [line width=0.04cm][red][dashed](-1.5,-0.45)--(3.2,-0.45);
\draw [thin](3.2,-1.5) --(3.2,0.4);

\draw [line width=0.02cm][black](-1.5,-0.45)--(-1.5,0.4);
\draw [line width=0.02cm][blue](-1.5,-1.5)--(-1.5,-0.45);

\draw [thin][->] (-1.6,0.0) --(-2.2,0.2);
\draw [thin][->] (-1.6,-0.9) --(-2.2,-1.0);

\draw [thin][->] (1.5,0.5) --(2.2,0.8);
\draw [thin][->] (2.8,-0.3) --(3.7,0.1);
\draw [thin][->] (1.5,-1.6) --(2.4,-1.9);

\draw [thin][->] (1.2,-0.1) --(2.2,-0.1);
\draw [thin][->] (1.2,-1.0) --(2.2,-1.0);

\node at (-1.8,-0.45) {$O$};

\node at (0.1,-0.1) {$\underline{\mathcal{N}}_{+}$};
\node at (0.1,-1.1) {$\underline{\mathcal{N}}_{-}$};
\node at (-2.5,0.3) {$\underline{\Gamma}^{+}_{\textrm{in}}$};
\node at (-2.5,-1.0) {$\underline{\Gamma}^{-}_{\textrm{in}}$};
\node at (2.6,0.8) {$\underline{\Gamma}_{+}$};
\node at (4.1,0.1) {$\underline{\Gamma}_{\textrm{cd}}$};
\node at (2.7,-1.9) {$\underline{\Gamma}_{-}$};
\end{tikzpicture}
\end{center}
\caption{Special Solution in a Finitely Long Straight Nozzle}\label{fig1.2}
\end{figure}
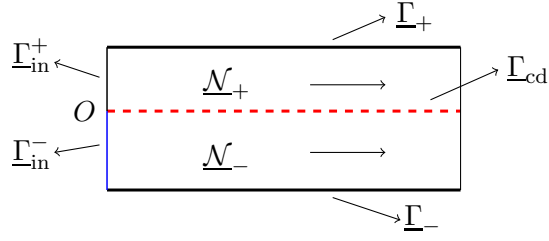
In this case, the two-dimensional finitely long straight nozzle can be described as
\begin{equation*}
\underline{\mathcal{N}}\triangleq\big\{(x,y)\in \mathbb{R}^{2}|\  0<x<L,\ -1<y<1 \big\},
\end{equation*}
with its boundaries being given by
\begin{eqnarray*}
\underline{\Gamma}_{+}\triangleq\big\{(x,y)\in \mathbb{R}^{2}|\ 0<x<L,\ y=1\big\},\quad\ \underline{\Gamma}_{-}\triangleq\big\{(x,y)\in \mathbb{R}^{2}|\ 0<x<L,\ y=-1\big\}.
\end{eqnarray*}
Let the straight contact discontinuity be
\begin{eqnarray}\label{eq:1.17}
\underline{\Gamma}_{\textrm{cd}}\triangleq\big\{(x,y)\in \mathbb{R}^{2}|\ 0<x<L,\ y=0 \big\},
\end{eqnarray}
which divides the nozzle $\underline{\mathcal{N}}$ into upper and lower supersonic regions as
\begin{equation*}
\underline{\mathcal{N}}_{+}\triangleq\underline{\mathcal{N}}\cap \{ 0<y<1 \}, \quad\  \underline{\mathcal{N}}_{-}\triangleq\underline{\mathcal{N}}\cap \{ -1<y<0 \}.
\end{equation*}
Finally, we define
\begin{eqnarray}\label{eq:1.18}
\underline{U}(x,y)\triangleq\left\{
\begin{array}{llll}
\underline{U}_{+}=\big(\underline{u}^{+}, 0,\underline{p}^{+},\underline{\rho}^{+}, 0\big)^{\top}, \qquad (x, y)\in \underline{\mathcal{N}}_{+},\\
\underline{U}_{-}=\big(\underline{u}^{-}, 0,\underline{p}^{-},\underline{\rho}^{-}, 0\big)^{\top}, \qquad (x, y)\in \underline{\mathcal{N}}_{-}.
\end{array}
\right.
\end{eqnarray}

Obviously, $\underline{U}_{\pm}$ satisfy equations \eqref{eq:1.1} and conditions \eqref{eq:1.9}-\eqref{eq:1.11}. 
So we call $(\underline{U},\underline{\Gamma}_{\textrm{cd}})$ a \emph{background solution} to \emph{Problem\ 1.1}. Obviously, we can extend $\underline{U}_{+}$ and $\underline{U}_{-}$ onto $(0,L)\times(-\infty,+\infty)$, respectively. Without loss of  the generality, we will still write $\underline{U}_{+}$ and $\underline{U}_{-}$ as the \emph{extended background solution}.
In this paper, we will consider a non-uniformly supersonic combustion  incoming flow with a contact discontinuity moving through a two-dimensional finitely long nozzle and show the stability of the \emph{background solution} $(\underline{U},\underline{\Gamma}_{\textrm{cd}})$ to give a positive answer to \emph{Problem 1.1}.
In order to complete it, let us introduce the standard H\"{o}lder space 
which we will work on throughout this paper.

For any function $\mathfrak{u}: \mathcal{N}\mapsto \mathbb{R}$, any integer $k\geq 0$, $\alpha\in (0,1]$, and any two points $\textbf{x}=(x,y)$ and $\mathbf{x}'=(x',y')$ in $\mathcal{N}$,
we define the standard semi-H\"{o}lder norm and H\"{o}lder norm as
\begin{eqnarray}\label{eq:1.19}
\begin{split}
&\| \mathfrak{u}\|_{k,0;\mathcal{N}}
 =\sum_{0\leq|{\color{black}\boldsymbol{\beta}}|\leq k}\sup_{\mathbf{x}\in \mathcal{N}}|D^{{\color{black}\boldsymbol{\beta}}}\mathfrak{u}(\mathbf{x})|,\quad
[\mathfrak{u}]_{k,\alpha;\mathcal{N}}
=\sum_{|{\color{black}\boldsymbol{\beta}}|=k}\sup_{{\mathbf{x}, \mathbf{x}'\in \mathcal{N} \atop \mathbf{x}\neq \mathbf{x}'}}\frac{|D^{{\color{black}\boldsymbol{\beta}}}\mathfrak{u}(\mathbf{x})-D^{{\color{black}\boldsymbol{\beta}}}\mathfrak{u}(\mathbf{x}')|}
 {|\mathbf{x}-\mathbf{x}'|^{\alpha}},
\end{split}
\end{eqnarray}
and
\begin{eqnarray}\label{eq:1.20}
\|\mathfrak{u}\|_{k,\alpha;\mathcal{N}}=\|\mathfrak{u}\|_{k,0;\mathcal{N}}+[\mathfrak{u}]_{k,\alpha,\mathcal{N}},
\end{eqnarray}
where ${\color{black}\boldsymbol{\beta}}=(\beta_{1}, \beta_{2})$ represents a multi-index for $\beta_{1}, \beta_2 \ge 0$ with $|{\color{black}\boldsymbol{\beta}}|=\beta_{1}+\beta_{2}$ and $D^{{\color{black}\boldsymbol{\beta}}}=\partial^{\beta_{1}}_{x}\partial^{\beta_{2}}_{y}$.
Then, we denote by $C^{k, \alpha}(\mathcal{N})$ the space of functions whose $\|\cdot\|_{k,\alpha;\mathcal{N}}$-norm is finite.  Furthermore, for the vector-value functions
$\boldsymbol{\mathfrak{u}}=(\mathfrak{u}_{1}, \mathfrak{u}_{2},\cdots, \mathfrak{u}_{n})$, define
\begin{eqnarray}\label{eq:1.21}
\|\boldsymbol{\mathfrak{u}}\|_{k,\alpha;\mathcal{N}}=\sum^{n}_{j=1}\|\mathfrak{u}_{j}\|_{k,\alpha;\mathcal{N}}.
\end{eqnarray}

Then, the first main result in this paper is on the existence and uniqueness of contact discontinuity solutions.
\begin{theorem}\label{thm:1.1}
Suppose that the reaction rate function $\phi(T)$ is $C^{1,1}$ with respect to $T$ for all $T>0$. Then, for any given $\alpha\in(0,1)$ and $L>0$, there exist positive constants $\epsilon_0$, $C_0$ and $C_{1}$ depending only on $\underline{U}$, $\alpha$ and $L$ such that if the incoming flow and boundaries satisfy
\begin{eqnarray}\label{eq:1.22}
\begin{split}
\sum_{k=\pm}\Big(\|U^{k}_{\rm in}-\underline{U}^{k}_{\rm{in}}\|_{1,\alpha; \Gamma^{k}_{\rm in}}+\|g_k-\underline{g}_{k}\|_{2,\alpha; \Gamma_{k}}\Big)<\epsilon,
\end{split}
\end{eqnarray}
for some $\epsilon\in (0,\epsilon_0)$, then \emph{Problem 1.1} admits a unique supersonic combustion contact discontinuity solution $(U_{\pm}, \Gamma_{\rm cd})$ with the following properties:

{\rm(1)}.\ Solution $U$ consists of two smooth supersonic flows $U_{-}\in C^{1,\alpha}(\mathcal{N}_{-})$ and $U_{+}\in C^{1,\alpha}(\mathcal{N}_{+})$
with a contact discontinuity $\Gamma_{\rm cd}$, and satisfies the following estimate
\begin{eqnarray}\label{eq:1.23}
\big\|U_{-}-\underline{U}^{-}_{\rm{in}}\big\|_{1,\alpha; \mathcal{N}_{-}}+\big\|U_{+}-\underline{U}^{+}_{\rm{in}}\big\|_{1,\alpha; \mathcal{N}_{+}}<C_{0}\epsilon.
\end{eqnarray}

{\rm(2)}.\ The contact discontinuity $y=g_{\rm cd}(x)$ is a $C^{2,\alpha}$-streamline satisfying $g_{\rm cd}(0)=0$ and
\begin{eqnarray}\label{eq:1.24}
\big\|g_{\rm cd}\big\|_{2,\alpha; [0,L)}<C_{1}\epsilon.
\end{eqnarray}
\end{theorem}

\begin{remark}\label{rem:1.1}
A typical example for $\phi(T)$ in Theorem \ref{thm:1.1} is of the Arrhenius type form (cf. \cite{chen-kuang-zhang, chen-wagner, williams, xiang-zhang-zhao}):
\begin{equation}\label{eq:1.25}
\phi(T)=T^{\vartheta}\textrm{e}^{-\frac{\mathcal{E}}{\mathcal{R}T}},
\end{equation}
where $\mathcal{E}$ is the action energy, and $\vartheta$ and $\mathcal{R}$ are positive constants.
\end{remark}

Based on Theorem \ref{thm:1.1}, let us consider the quasi-one-dimensional approximation in nozzle $\mathcal{N}$ as follows.
Denote by $\textrm{A}_{\pm}(x)\triangleq\pm[g_{\pm}(x)-g_{\rm{cd}}(x)]$ the distance between $\Gamma_{\rm cd}$ and $\Gamma_{+}$, and the distance between $\Gamma_{-}$ and $\Gamma_{\rm cd}$, respectively.
Then, we define the integral average of the states $V_{\pm}(x,y)=(u_{\pm}, p_{\pm}, \rho_{\pm}, Y_{\pm})^{\top}$ to \emph{Problem 1.1} with respect to $y$-direction as
\begin{eqnarray}\label{eq:1.26}
\bar{V}_{+}(x)=\frac{1}{\textrm{A}_{+}(x)}\int^{g_{+}(x)}_{g_{\rm cd}(x)}V_{+}(x,y)\dd y,\quad\ \bar{V}_{-}(x)=\frac{1}{\textrm{A}_{-}(x)}\int^{g_{\rm cd}(x)}_{g_{-}(x)}V_{-}(x,y)\dd y.
\end{eqnarray}

Formally, as shown in \cite{menshov}, the integral average of the two-dimensional steady supersonic combustion contact discontinuity flow $\bar{V}_{\pm}$ in regions $\mathcal{N}_{\pm}$
can be approximated by the following quasi-one-dimensional model:
\begin{eqnarray}\label{eq:1.27}
\begin{cases}
\partial_x\big(\rho_{\pm} u_{\pm} \textrm{A}_{\pm}(x)\big)=0, \\
\partial_{x}\big((\rho_{\pm} u^{2}_{\pm}+p_{\pm})\textrm{A}_{\pm}(x)\big)=\textrm{A}'_{\pm}(x)p_{\pm}, \\
\partial_{x}\Big(\big(\frac{1}{2}u^2_{\pm}+\frac{\gamma p_{\pm}}{(\gamma-1)\rho_{\pm}}\big)\rho_{\pm} u_{\pm}\textrm{A}_{\pm}(x)\Big)=\mathfrak{q}_0\textrm{A}_{\pm}(x)\rho_{\pm}\phi(T_{\pm})Y_{\pm},\\
\partial_{x}\big(\rho_{\pm} u_{\pm}Y_{\pm}\textrm{A}_{\pm}(x)\big)=-\textrm{A}_{\pm}(x)\rho_{\pm}\phi(T_{\pm})Y_{\pm}.
\end{cases}
\end{eqnarray}

So, a natural question is whether we can give a rigorous mathematical proof for the validation of the quasi-one-dimensional approximation. More precisely, we propose the following problem:

\smallskip
\emph{$\mathbf{Problem\ 1.2.}$}\ For given datum which is a small perturbation of a constant state, whether can we establish a rigorous mathematical validation
on the quasi-one-dimensional approximation
with a convergence rate with respect to the perturbation of the {data}?

%

\smallskip
A positive answer to \emph{Problem\ 1.2} is the following theorem.
\begin{theorem}\label{thm:1.2}
Suppose that $V^{\rm{A}}_{\pm}(x)=(u^{\rm{A}}_{\pm}, p^{\rm{A}}_{\pm}, \rho^{\rm{A}}_{\pm}, Y^{\rm{A}}_{\pm})^{\top}(x)$ are solutions of system \eqref{eq:1.27} with the initial data $\bar{V}^{\pm}_{\rm in}=(\bar{u}^{\pm}_{\rm in}, \bar{p}^{\pm}_{\rm in}, \bar{\rho}^{\pm}_{\rm in}, \bar{Y}^{\pm}_{\rm in})^{\top}$ defined by
\begin{eqnarray}\label{eq:1.28}
\bar{V}^{+}_{\rm in}=\frac{1}{\rm{A}_{+}(0)}\int^{g_{+}(0)}_{0}V^{+}_{\rm in}(y)\dd y,\qquad \bar{V}^{-}_{\rm in}=\frac{1}{\rm{A}_{-}(0)}\int^{0}_{g_{-}(0)}V^{-}_{\rm in}(y)\dd y.
\end{eqnarray}
Suppose that $\bar{V}_{\pm}(x)$ are the integral average of $V_{\pm}$ in $\mathcal{N}_{\pm}$ as defined in \eqref{eq:1.26}.
Then, there exist constants $\epsilon^{*}_0\in(0, \epsilon_{0})$ and $C_2>0$ depending only on $\underline{U}$, $L$ and $\alpha$ such that if the assumptions of Theorem \ref{thm:1.1} hold for $\epsilon\in(0,\epsilon^*_0)$, then
there holds the estimate that
\begin{eqnarray}\label{eq:1.29}
\big\|\bar{V}_{-}(x)-V^{\rm{A}}_{-}(x)\big\|_{1,\alpha; (0,L)}+\big\|\bar{V}_{+}(x)-V^{\rm{A}}_{+}(x)\big\|_{1,\alpha; (0,L)}<C_{2}\epsilon^2.
\end{eqnarray}
\end{theorem}

Finally, we will consider the global uniqueness of supersonic combustion contact discontinuity flow $(U_{\pm},\Gamma_{\rm{cd}})$ to the \emph{Problem\ 1.1} without assuming that $U_{\pm}$ is a perturbation of the \emph{background state} $\underline{U}$. More precisely, we will study the following global uniqueness problem.

\smallskip
\emph{$\mathbf{Problem\ 1.3.}$}\ {\color{black}If all the given data are constant states, does the $C^{1,\alpha}$-piecewise supersonic combustion contact discontinuity solution $(U,\Gamma_{\rm{cd}})$ to the
\emph{Problem\ 1.1} coincide with the piecewise constant states with a straight contact discontinuity 
in a given two-dimensional finitely long flat nozzle $\mathcal{N}$?}

\smallskip
We give an affirmed answer to the \emph{Problem\ 1.3} by the following theorem.
\begin{theorem}\label{thm-global uniqueness}
Assume that the reaction rate function $\phi(T)$ is $C^{1,1}$ with respect to $T>0$ and $\alpha\in(0,1)$.
Let $(U, \Gamma_{\rm cd})$ be a $C^{1,\alpha}$-smooth supersonic contact discontinuity solution to the \emph{Problem 1.1} in $\underline{\mathcal{N}}$,
consisting of two $C^{1,\alpha}$-smooth supersonic flows $U_{+}=(u_{+}, v_{+}, p_{+},\rho_{+}, Y_{+})^{\top}$ in $\underline{\mathcal{N}}\cap\{y>g_{\rm cd}(x)\}$
and $U_{-}=(u_{-}, v_{-}, p_{-},\rho_{-}, Y_{-})^{\top}$ in $\underline{\mathcal{N}}\cap\{y<g_{\rm cd}(x)\}$
with the contact discontinuity $\Gamma_{\rm cd}=\{(x,y)\in \mathbb{R}^{2}| x\in(0,L),\ y=g_{\rm{cd}}(x)\}$ in between,
where $g_{\rm cd}(x)\in C^{2,\alpha}([0,L))$,
\begin{eqnarray}\label{eq-uniqueness-1}
\begin{split}
u_{+}>0, \quad u_{+}\neq c_{+},\quad \rho_{+}>0, \quad  \mbox{\rm{in}} \quad \underline{\mathcal{N}}\cap\{y>g_{\rm{cd}}(x)\}, \\
u_{-}>0, \quad u_{-}\neq c_{-},\quad  \rho_{-}>0, \quad  \mbox{\rm{in}} \quad \underline{\mathcal{N}}\cap\{y<g_{\rm{cd}}(x)\},
\end{split}
\end{eqnarray}
and
\begin{eqnarray}\label{eq-uniqueness-2}
U_{\pm}= \underline{U}^{\pm}_{\rm{in}},\quad  \mbox{\rm{on}}\quad \underline{\Gamma}^{\pm}_{\rm{in}},\quad\  \mbox{\rm{and}}\quad v_{\pm}=0,\quad \mbox{\rm{on}}\quad \underline{\Gamma}_{\pm}.
\end{eqnarray}
Then,
\begin{eqnarray}\label{eq-uniqueness-3}
U_{+}= \underline{U}_{+},\ \  \mbox{\rm{in}}\ \ \underline{\mathcal{N}}\cap\{y>g_{\rm{cd}}(x)\}, \quad U_{-}= \underline{U}_{-},\ \  \mbox{\rm{in}}\ \ \underline{\mathcal{N}}\cap\{y<g_{\rm{cd}}(x)\}, \ \ g_{\rm{cd}}(x)=0,
\end{eqnarray}
that is
\begin{eqnarray}\label{eq-uniqueness-4}
(U, \Gamma_{\rm cd})\equiv(\underline{U}, \underline{\Gamma}_{\rm cd}),\quad\ \mbox{\rm{in}}\quad \underline{\mathcal{N}},
\end{eqnarray}
where $\underline{U}^{\pm}_{\rm{in}}$, $\underline{U}_{\pm}$ and $\underline{U}$ are defined by \eqref{eq:1.15} and \eqref{eq:1.18}.
\end{theorem}

\subsection{Main difficulties and outline of the proof}
Mathematically, 
the \emph{Problem 1.1} can be formulated as an initial-boundary value problem for two-dimensional steady combustion Euler flows with a contact discontinuity as the free boundary and with nozzle walls as the fixed boundaries.
There are two main difficulties. One is that the free boundary is characteristic which leads to the loss of the normal velocity along it. So both states on the upper and lower nozzle regions can not be determined independently. The other one is the strong coupled phenomena for each states 
due to the combustion reaction procedure which is different from our previous works in \cite{huang-kuang-wang-xiang2, huang-kuang-wang-xiang1} for steady compressible Euler flows. 

To overcome the difficulties, we first notice that a contact discontinuity is a streamline, so it is convenient to employ the Lagrangian coordinate transformation to transfer the free boundary value problem into a fixed boundary value problem. Then, by employing the flow slope $\omega$ (see \eqref{eq:2.23} and \eqref{eq:2.25}) and pressure $p$ rather than Riemann invariants, we can further reformulate the fixed boundary value problem into an initial-boundary value problem governed by a first order inhomogeneous hyperbolic system for $(\omega,p)$ coupled with ODEs for $(B, S, Y)$ (denoted by \emph{\rm{(IBVP)}} for the simplicity) with discontinuous coefficients.
In order to solve the \emph{\rm{(IBVP)}} governed by the hyperbolic PDEs-ODEs coupled system, we will develop a new iteration scheme based on an observation that the source terms for the equations of $(\omega, p, B, S, Y)$ are all separated with respect to $Y$ of square order. We will proceed it by the following four steps:

$\emph{Step 1}$.\ We first linearize the nonlinear problem \emph{\rm{(IBVP)}} near the background states $\underline{U}$, which is denoted as $\emph{\rm{(IBVP)}}^{*}$, and construct an iteration set $\mathcal{X}_{\sigma}$ (see \eqref{eq:3.1} and \eqref{eq:3.2}).

$\emph{Step 2}$.\ For given states $\delta \mathcal{U}$, we consider the solution $\delta \mathcal{U}^{*}=(\delta \omega^{*},\delta p^{*}, \delta B^{*},\delta S^{*}, \delta Y^{*})^{\top}$ to $\emph{\rm{(IBVP)}}^{*}$  and establish some \emph{a priori} estimates for it in H\"{o}lder spaces.
\begin{itemize}
  \item Specifically, noticing that the source term of the linearized equation of $\delta Y^{*}$ in $\eqref{eq:3.2}$ depends on the product of $\delta \mathcal{U}$
and $\delta Y$ so that $\delta Y^{*}$ can be dealt with and solved first. Then we can get $(\delta B^{*},\delta S^{*})$ in each domain $\tilde{\mathcal{N}}_{\pm}$ too.
  \item Next, for $(\delta \omega^{*},\delta p^{*})$, one needs to further develop the characteristic methods to solve the initial-boundary value problem of a first order $2\times 2$ hyperbolic inhomogeneous system with discontinuous coefficients in $\tilde{\mathcal{N}}_{+}$ and $\tilde{\mathcal{N}}_{-}$ at the same time.
It involves the interactions and reflections among the characteristic curves, contact discontinuity and the nozzle walls.

More precisely, in deriving the $C^{1,\alpha}$-estimates for $(\delta \omega^{*},\delta p^{*})$, one needs to divide {\color{black}the} $\tilde{\mathcal{N}}_{\pm}$ into finite sub-domains first. And then due to the coupled effect, we introduce new variables $z^{k,*}_{\pm}$ (see \eqref{eq:prop-3.2-2}) to derive $C^0$-estimate for $(\delta \omega^{*},\delta p^{*})$, and $\mathcal{Z}^{k,*}_{\pm}$ (see \eqref{eq:prop-3.2-12}) to derive $C^{0,\alpha}$-estimate for $(\nabla\delta \omega^{*},\nabla\delta p^{*})$ in each sub-domains.
Based on them, we can bond the sub-domains into a whole domain to get the global H\"{o}lder estimates for $(\nabla\delta \omega^{*},\nabla\delta p^{*})$.
\end{itemize}

$\emph{Step 3}$.\ 
{\color{black}By employing} the \emph{a priori} estimates obtained above, we can define a continuous map $\mathcal{T}$ (see \eqref{eq:4.1}) mapping from $\mathcal{X}_{\sigma}$ to $\mathcal{X}_{\sigma}$ for $\sigma>0$ sufficiently small. Then, it follows from the \emph{Schauder fixed point theorem} that there is a fixed point which is also unique by a contraction argument. Hence \emph{\rm{(IBVP)}} admits a unique $C^{1,\alpha}$-solution $(\omega, p, B, S, Y)$.

$\emph{Step 4}$.\ By applying the relations between $(\omega, p, B, S, Y)$ and $(u,v, p,\rho,Y)$ together with the inverse Lagrangian coordinate transformation, we finally show the existence and uniqueness of the solutions to \emph{Problem 1.1} which indicates that the steady supersonic combustion contact discontinuity in two dimensional nozzle is stable under small perturbations.

For the quasi-one-dimensional approximation problem,  
the key points are to derive the equations for the average integral of the states $V=(u, p,\rho,Y)^{\top}$ from equations \eqref{eq:1.1} and
to show the well-posednes of the corresponding quasi-one-dimensional equations in $C^{1,\alpha}$ space.
Then, by considering the equations of the errors between the integral average and the solutions to the quasi-one-dimensional model, we can establish a rigorous mathematical validation as mentioned in \cite{menshov} for the quasi-one-dimensional approximation with a square order convergence rate with respect to the $C^{1,\alpha}$-perturbation of the data.

For the problem of global uniqueness, the main difficulty is that there is no small assumption on the perturbation of the background state so that the linearization method cannot be applied.
To overcome the difficulty, {\color{black}we make use of supersonic 
property of the solution and the eigenvalues for genuinely nonlinear characteristic fields of the equations having opposite signs at background state} to reduce the problem into a new initial-boundary value problem for a first order semi-linear non-degenerate hyperbolic system coupled with some ordinary differential equations with zero datum and nonlinear source terms. We then can apply the characteristic method and some ODE theories to derive that the solutions of the new problem are zero in the whole nozzle which shows the global uniqueness of supersonic combustion contact discontinuity flow in a finite long nozzle.

\subsection{Organization of the paper} The remainder of the paper is organized as follows.
In Section 2, we introduce the Lagrangian coordinate transformation and reformulate \emph{Problem 1.1} into an initial-boundary value problem governed  by a first order inhomogeneous system coupled with ordinary differential equations, which is denoted by \emph{\rm{(IBVP)}} for simplicity.
In Section 3, we develop a new iteration scheme for \emph{\rm{(IBVP)}} near the background states, and then establish some novel \emph{a priori} estimates and some theories of ordinary differential equations in H\"{o}lder space. Based on them and by applying the Schauder fixed point theorem and a contraction argument, we show in Section 4 that the iteration scheme admits a unique fixed point, which is the unique solution to \emph{\rm{(IBVP)}}.
In Section 5, we are devoted to verifying the quasi-one-dimensional approximation in nozzle $\mathcal{N}$ with a supersonic combustion contact discontinuity. 
The global uniqueness is proven in section 6.
Finally, in the appendix, we give some properties and estimates on the characteristic curves that are defined in section 3 and some properties on eigenvalues which are used in Section 6.

\section{Mathematical Formulation in Lagrangian Coordinates}\setcounter{equation}{0}

In this section, we will consider \emph{Problem 1.1} in the Lagrangian {\color{black}coordinates}. To this end, we first introduce a Lagrange {\color{black}coordinate} transformation for supersonic flow in a nozzle and reformulate it as \emph{Problem 2.1} in the new {\color{black}coordinates}. Next, by introducing some new quantities, we can further reformulate \emph{Problem 2.1} into an initial-boundary value problem for the first order quasi-linear hyperbolic inhomogeneous system coupled with ordinary differential equations.

\subsection{Lagrangian coordinate transformation and \emph{Problem 2.1}}\setcounter{equation}{0}
Since the contact discontinuity $\Gamma_{\rm cd}$ is a free stream line, it is convenient to work in the Lagrangian {\color{black}coordinates}. Let $\big(U, \Gamma_{\rm cd}\big)$ be a supersonic combustion contact discontinuity solution of \emph{Problem 1.1}, define the mass flux as follows:
\begin{eqnarray}\label{eq:2.2}
\begin{split}
\textrm{m}_{-}\triangleq\int^{0}_{g_{-}(0)}(\rho^{-}_{\rm in} u^{-}_{\rm in})(\tau)\dd\tau,  \quad\ \textrm{m}_{+}\triangleq\int^{g_{+}(0)}_{0}(\rho^{+}_{\rm in} u^{+}_{\rm in})(\tau)\dd\tau.
\end{split}
\end{eqnarray}
By the equation of the conservation of mass in \eqref{eq:1.1}, we have for all $x\in[0,L]$
\begin{eqnarray}\label{eq:2.1}
\begin{split}
\int^{g_{\rm cd}(x)}_{g_{-}(x)}(\rho_{-} u_{-})(x,\tau)d\tau\equiv\textrm{m}_{-}, \quad\ \int^{g_{+}(x)}_{g_{\rm cd}(x)}(\rho_{+} u_{+})(x,\tau)d\tau\equiv\textrm{m}_{+}.
\end{split}
\end{eqnarray}

Define the Lagrange {\color{black}coordinate} transformation:
\begin{eqnarray}\label{eq:2.3}
\begin{split}
\mathcal{L}:\ \xi=x, \quad \eta=\int^{y}_{g_{-}(x)}(\rho u)(x,\tau)\dd\tau-\textrm{m}_{-}.
\end{split}
\end{eqnarray}
By direct computation, one has
\begin{eqnarray}\label{eq:2.4}
\begin{split}
\frac{\partial \xi}{\partial x}=1,\quad \frac{\partial \xi}{\partial y}=0,\quad \frac{\partial \eta}{\partial x}=-\rho v,\quad \frac{\partial \eta}{\partial y}=\rho u.
\end{split}
\end{eqnarray}

For any $(x,y)\in \Gamma_{\pm}$, it follows from \eqref{eq:2.1}-\eqref{eq:2.3} that
\begin{eqnarray*}
\eta\big|_{(x,y)\in\Gamma_{-}}=-\textrm{m}_{-},\quad \eta\big|_{(x,y)\in\Gamma_{+}}=\int^{g_{\textrm{cd}}(x)}_{g_{-}(x)}(\rho_{-} u_{-})(x,\tau)\dd\tau+\int^{g_{+}(x)}_{g_{\textrm{cd}}(x)}(\rho_{+} u_{+})(x,\tau)\dd\tau-\textrm{m}_{-}=\textrm{m}_{+}.
\end{eqnarray*}
Therefore, the nozzle $\mathcal{N}$ in $(\xi,\eta)$-{\color{black}coordinates} becomes 
\begin{eqnarray*}
\tilde{\mathcal{N}}\triangleq\big\{(\xi,\eta)\in\mathbb{R}^2|\ \xi\in(0,L),\ -\textrm{m}_{-}<\eta<\textrm{m}_{+}\big\},
\end{eqnarray*}
where its upper and lower boundaries are given by
\begin{eqnarray*}
\tilde{\Gamma}_{+}\triangleq\big\{(\xi,\eta)\in\mathbb{R}^2|\ \xi\in(0,L),\ \eta=\textrm{m}_{+}\big\},\quad\
\tilde{\Gamma}_{-}\triangleq\big\{(\xi,\eta)\in\mathbb{R}^2|\ \xi\in(0,L),\ \eta=-\textrm{m}_{-}\big\}.
\end{eqnarray*}

Moreover, since
$\eta\big|_{(x,y)\in\Gamma_{\rm cd}}=\int^{g_{\textrm{cd}}(x)}_{g_{-}(x)}(\rho u)(x,\tau)\dd\tau-\textrm{m}_{-}=0$,  the contact discontinuity $\Gamma_{\rm cd}$ in $(\xi,\eta)$-{\color{black}coordinates} becomes
\begin{eqnarray*}
\tilde{\Gamma}_{\rm{cd}}\triangleq\big\{(\xi,\eta)\in \mathbb{R}^{2}|\ \xi\in(0,L),\ \eta=0 \big\},
\end{eqnarray*}
which divides $\tilde{\mathcal{N}}$ into two supersonic layers
\begin{eqnarray*}
\tilde{\mathcal{N}}_{+}\triangleq \tilde{\mathcal{N}}\cap \big\{0<\eta<\rm{m}_{+}\big\},\quad\ \tilde{\mathcal{N}}_{-}\triangleq \tilde{\mathcal{N}}\cap \big\{-\rm{m}_{-}<\eta<0\big\}.
\end{eqnarray*}
The entrances for $\tilde{\mathcal{N}}_{\pm}$ are 
\begin{eqnarray*}
\tilde{\Gamma}^{+}_{\rm{in}}\triangleq  \big\{(\xi,\eta)\in \mathbb{R}^{2}|\ \xi=0,\ 0<\eta<\rm{m}_{+}\big\},
\quad\ \tilde{\Gamma}^{-}_{\rm{in}}\triangleq \big\{(\xi,\eta)\in \mathbb{R}^{2}|\ \xi=0,\ -\rm{m}_{-}<\eta<0\big\}.
\end{eqnarray*}
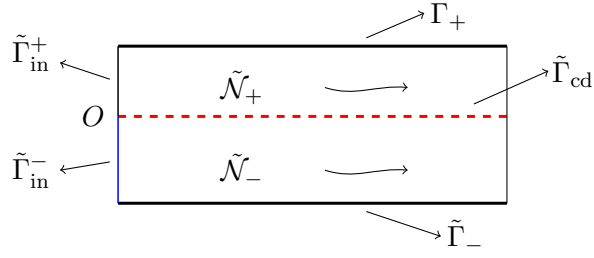
\begin{figure}[ht]
\begin{center}
\begin{tikzpicture}[scale=1.1]
\draw [thin][->] (-1.6,0)--(-2.2,0.2);
\draw [thin][->] (-1.6,-1.0) --(-2.2,-1.1);

\draw [line width=0.04cm](-1.5,-1.5) --(3.2,-1.5);
\draw [line width=0.04cm](-1.5,0.4)--(3.2,0.4);
\draw [line width=0.04cm][red][dashed](-1.5,-0.45)--(3.2,-0.45);
\draw [thin](3.2,-1.5) --(3.2,0.4);

\draw [thin][->](1.0,-0.1)to[out=-18, in=-185](2.0,-0.1);
\draw [thin][->](1.0,-1.1)to[out=-18, in=-185](2.0,-1.1);

\draw [line width=0.02cm][black](-1.5,-0.45)--(-1.5,0.4);
\draw [line width=0.02cm][blue](-1.5,-1.5)--(-1.5,-0.45);

\draw [thin][->] (1.5,0.5) --(2.2,0.8);
\draw [thin][->] (2.8,-0.3) --(3.7,0.1);
\draw [thin][->] (1.5,-1.6) --(2.4,-1.9);

\node at (-1.8,-0.45) {$O$};

\node at (0,-0.1) {$\tilde{\mathcal{N}}_{+}$};
\node at (0,-1.1) {$\tilde{\mathcal{N}}_{-}$};
\node at (-2.55,0.32) {$\tilde{\Gamma}^{+}_{\textrm{in}}$};
\node at (-2.55,-1.1) {$\tilde{\Gamma}^{-}_{\textrm{in}}$};
\node at (2.5,0.8) {$\tilde{\Gamma}_{+}$};
\node at (4.0,0.1) {$\tilde{\Gamma}_{\textrm{cd}}$};
\node at (2.7,-1.9) {$\tilde{\Gamma}_{-}$};
\end{tikzpicture}
\end{center}
\caption{Supersonic Combustion Contact Discontinuity in Lagrangian Coordinate}\label{fig2.1}
\end{figure}
%
In $(\xi,\eta)$-{\color{black}coordinates}, for notational simplicity, 
we still use the notation $U(\xi,\eta)=(u,v,p,\rho, Y)^{\top}(\xi,\eta)$ in $\tilde{\mathcal{N}}$ with $U(\xi,\eta)=U(x(\xi,\eta), y(\xi, \eta))$.
Similarly, let $U_{\pm}(\xi,\eta)=(u_{\pm},v_{\pm},p_{\pm},\rho_{\pm}, Y_{\pm})^{\top}(\xi,\eta)$ by $U_{\pm}(\xi,\eta)=U_{\pm}(x(\xi,\eta), y(\xi, \eta))$.
Then, by \eqref{eq:2.4}, equations \eqref{eq:1.1} in $(\xi,\eta)$-{\color{black}coordinates} become
\begin{eqnarray}\label{eq:2.5}
\begin{cases}
\partial_{\xi}\big(\frac{1}{\rho u}\big)-\partial_{\eta}\big(\frac{v}{u}\big)=0, \\
\partial_{\xi}\big(u+\frac{p}{\rho u}\big)-\partial_{\eta}\big(\frac{pv}{u}\big)=0, \\
\partial_{\xi}v+\partial_{\eta}p=0, \\
\partial_{\xi}B=\frac{\mathfrak{q}_0\phi(T)}{u}Y,\\
\partial_{\xi}Y=-\frac{\phi(T)}{u}Y,
\end{cases}
\end{eqnarray}
where \eqref{eq:1.4} and \eqref{eq:1.8} hold. 

Finally, it follows from conditions \eqref{eq:1.11}-\eqref{eq:1.12} that 
\begin{eqnarray}\label{eq:2.6}
U_{\pm}=U^{\pm}_{\textrm{in}}(\eta)\triangleq\big(u^{\pm}_{\textrm{in}}, v^{\pm}_{\textrm{in}}, p^{\pm}_{\textrm{in}},\rho^{\pm}_{\textrm{in}}, Y^{\pm}_{\textrm{in}}\big)^{\top}(\eta),
\qquad \mbox{on}\quad \tilde{\Gamma}^{\pm}_{\textrm{in}},
\end{eqnarray}
where $(u^{\pm}_{\rm{in}},v^{\pm}_{\rm{in}},p^{\pm}_{\rm{in}},\rho^{\pm}_{\rm{in}}, Y^{\pm}_{\rm in})(\eta)
=(u^{\pm}_{\rm{in}},v^{\pm}_{\rm{in}},p^{\pm}_{\rm{in}},\rho^{\pm}_{\rm{in}}, Y^{\pm}_{\rm in})(y(0,\eta))$ and 
\begin{eqnarray}\label{eq:2.7}
\sqrt{(u^{\pm}_{\textrm{in}})^2+(v^{\pm}_{\textrm{in}})^2}>c^{\pm}_{\textrm{in}}=\sqrt{\frac{\gamma p^{\pm}_{\textrm{in}}}{\rho^{\pm}_{\textrm{in}}}},\quad\ Y^{\pm}_{\textrm{in}}\in [0,1],  \qquad \mbox{on}\quad \tilde{\Gamma}^{\pm}_{\textrm{in}}.
\end{eqnarray}

On the nozzle walls $\tilde{\Gamma}_{\pm}$, it follows from \eqref{eq:1.9} that the flow satisfies
\begin{eqnarray}\label{eq:2.8}
v_{\pm}-g'_{\pm}(\xi)u_{\pm}=0,\qquad\ \mbox{on}\quad\ \tilde{\Gamma}_{\pm}.
\end{eqnarray}
On the contact discontinuity $\tilde{\Gamma}_{\rm cd}$, by \eqref{eq:1.10}, we have
\begin{eqnarray}\label{eq:2.9}
\frac{v_{+}}{u_{+}}=\frac{v_{-}}{u_{-}}=g'_{\rm{cd}}(\xi),\quad\  p_{+}=p_{-}, \quad\ \mbox{on}\quad\ \tilde{\Gamma}_{\rm cd}.
\end{eqnarray}

In addition, it follows from \eqref{eq:1.13} and \eqref{eq:1.14} that the incoming flows also satisfy the following compatibility conditions at the corner points $\tilde{O}=(0,0)$ and $\tilde{P}_{\pm}=(0,\pm\rm{m}_{\pm})$, respectively,
\begin{eqnarray}\label{eq:2.10}
\begin{cases}
\big(\frac{v^{+}_{\textrm{in}}}{u^{+}_{\textrm{in}}}\big)(0)=\big(\frac{v^{-}_{\textrm{in}}}{u^{-}_{\textrm{in}}}\big)(0),\quad p^{+}_{\rm in}(0)=p^{-}_{\rm in}(0),\\
\Big[\frac{\rho^{+}_{\rm in}(c^{+}_{\rm in})^2 v^{+}_{\rm in}}
{(u^{+}_{\rm in})^{2}-(c^{+}_{\rm in})^{2}}\big(\frac{v^{+}_{\textrm{in}}}{u^{+}_{\textrm{in}}}\big)'\Big](0)+\Big[\frac{(u^{+}_{\rm in})^2+(v^{+}_{\rm in})^2-(c^{+}_{\rm in})^2}
{u^{+}_{\rm in}((u^{+}_{\rm in})^{2}-(c^{+}_{\rm in})^{2})}(p^{+}_{\rm in})'\Big](0)-\Big[\frac{(\gamma-1)\mathfrak{q}_{0}\phi(T^{+}_{\rm in})
v^{+}_{\rm in}}{(u^{+}_{\rm in})^2((u^{+}_{\rm in})^2-(c^{+}_{\rm in})^2)}Y^{+}_{\rm in}\Big](0)\\
\quad=\Big[\frac{\rho^{-}_{\rm in}(c^{-}_{\rm in})^2 v^{-}_{\rm in}}
{(u^{-}_{\rm in})^{2}-(c^{-}_{\rm in})^{2}}\big(\frac{v^{-}_{\textrm{in}}}{u^{-}_{\textrm{in}}}\big)'\Big](0)+\frac{(u^{-}_{\rm in})^2+(v^{-}_{\rm in})^2-(c^{-}_{\rm in})^2}
{u^{-}_{\rm in}((u^{-}_{\rm in})^{2}-(c^{-}_{\rm in})^{2})}(p^{-}_{\rm in})'\Big](0)-\frac{(\gamma-1)\mathfrak{q}_{0}\phi(T^{-}_{\rm in})
v^{-}_{\rm in}}{(u^{-}_{\rm in})^2((u^{-}_{\rm in})^2-(c^{-}_{\rm in})^2)}Y^{-}_{\rm in}\Big](0),\\
\Big[\frac{(\rho^{+}_{\rm in})^2(c^{+}_{\rm in})^2 (u^{+}_{\rm in})^3}
{(u^{+}_{\rm in})^{2}-(c^{+}_{\rm in})^{2}}\big(\frac{v^{+}_{\textrm{in}}}{u^{+}_{\textrm{in}}}\big)'\Big](0)+\Big[\frac{\rho^{+}_{\rm in}(c^{+}_{\rm in})^2 v^{+}_{\rm in}}
{(u^{+}_{\rm in})^{2}-(c^{+}_{\rm in})^{2}}(p^{+}_{\rm in})'\Big](0)-\Big[\frac{(\gamma-1)\mathfrak{q}_{0}\phi(T^{+}_{\rm in})
\rho^{+}_{\rm in}u^{+}_{\rm in}}{(u^{+}_{\rm in})^2-(c^{+}_{\rm in})^2}Y^{+}_{\rm in}\Big](0)\\
\quad=\Big[\frac{(\rho^{-}_{\rm in})^2(c^{-}_{\rm in})^2 (u^{-}_{\rm in})^3}{(u^{-}_{\rm in})^{2}-(c^{-}_{\rm in})^{2}}\big(\frac{v^{-}_{\textrm{in}}}{u^{-}_{\textrm{in}}}\big)'\Big](0)+\Big[\frac{\rho^{-}_{\rm in}(c^{-}_{\rm in})^2 v^{-}_{\rm in}}
{(u^{-}_{\rm in})^{2}-(c^{-}_{\rm in})^{2}}(p^{-}_{\rm in})'\Big](0)-\Big[\frac{(\gamma-1)\mathfrak{q}_{0}\phi(T^{-}_{\rm in})
\rho^{-}_{\rm in}u^{-}_{\rm in}}{(u^{-}_{\rm in})^2-(c^{-}_{\rm in})^2}Y^{+}_{\rm in}\Big](0),
\end{cases}
\end{eqnarray}
and
\begin{eqnarray}\label{eq:2.11}
\begin{cases}
g'_{+}(0)=\big(\frac{v^{+}_{\textrm{in}}}{u^{+}_{\textrm{in}}}\big)(\textrm{m}_{+}),\quad g'_{-}(0)=\big(\frac{v^{-}_{\textrm{in}}}{u^{-}_{\textrm{in}}}\big)(-\textrm{m}_{-}), \\
g''_{+}(0)=-\Big[\frac{\rho^{+}_{\rm in}(c^{+}_{\rm in})^2 v^{+}_{\rm in}}
{(u^{+}_{\rm in})^{2}-(c^{+}_{\rm in})^{2}}\big(\frac{v^{+}_{\textrm{in}}}{u^{+}_{\textrm{in}}}\big)'\Big](\textrm{m}_{+})-\Big[\frac{(u^{+}_{\rm in})^2+(v^{+}_{\rm in})^2-(c^{+}_{\rm in})^2}
{u^{+}_{\rm in}((u^{+}_{\rm in})^{2}-(c^{+}_{\rm in})^{2})}(p^{+}_{\rm in})'\Big](\textrm{m}_{+})\\
\qquad\qquad+\Big[\frac{(\gamma-1)\mathfrak{q}_{0}\phi(T^{+}_{\rm in})
v^{+}_{\rm in}}{(u^{+}_{\rm in})^2((u^{+}_{\rm in})^2-(c^{+}_{\rm in})^2)}Y^{+}_{\rm in}\Big](\textrm{m}_{+}),\\
g''_{-}(0)=-\Big[\frac{\rho^{-}_{\rm in}(c^{-}_{\rm in})^2 v^{-}_{\rm in}}
{(u^{-}_{\rm in})^{2}-(c^{-}_{\rm in})^{2}}\big(\frac{v^{-}_{\textrm{in}}}{u^{-}_{\textrm{in}}}\big)'\Big](-\textrm{m}_{-})-\Big[\frac{(u^{-}_{\rm in})^2+(v^{-}_{\rm in})^2-(c^{-}_{\rm in})^2}
{u^{-}_{\rm in}((u^{-}_{\rm in})^{2}-(c^{-}_{\rm in})^{2})}(p^{-}_{\rm in})'\Big](-\textrm{m}_{-})\\
\qquad\qquad+\Big[\frac{(\gamma-1)\mathfrak{q}_{0}\phi(T^{-}_{\rm in})
v^{-}_{\rm in}}{(u^{-}_{\rm in})^2((u^{-}_{\rm in})^2-(c^{-}_{\rm in})^2)}Y^{-}_{\rm in}\Big](-\textrm{m}_{-}).
\end{cases}
\end{eqnarray}

Hence we can further reformulate \emph{Problem 1.1} in the Euler coordinate into the following problem in the Lagrangian coordinate.

\smallskip
\par \emph{$\mathbf{Problem\ 2.1.}$}\ For given initial datum $U^{\pm}_{\rm{in}}$ and boundary functions $g_{\pm}$ satisfying \eqref{eq:2.7}, \eqref{eq:2.10} and \eqref{eq:2.11}, to find a supersonic combustion solution $U$ with a contact discontinuity {\color{black}$\tilde{\Gamma}_{\mathrm{cd}}$} such that

$\rm{(\tilde{i})}$\ $\tilde{\mathcal{N}}$ is divided by $\tilde{\Gamma}_{\rm{cd}}$ into two supersonic regions $\tilde{\mathcal{N}}_{+}$ and $\tilde{\mathcal{N}}_{-}$;

$\rm{(\tilde{ii})}$\ In the upper supersonic region $\tilde{\mathcal{N}}_{+}$, $U_{+}$ satisfies the equations \eqref{eq:2.5}, the initial conditions \eqref{eq:2.6} on $\tilde{\Gamma}^{+}_{\rm in}$ and the boundary conditions \eqref{eq:2.8} on $\tilde{\Gamma}_{+}$;

$\rm{(\tilde{iii})}$\ In the lower supersonic region $\tilde{\mathcal{N}}_{-}$, $U_{-}$ satisfies the equations \eqref{eq:2.5}, the initial conditions \eqref{eq:2.6} on $\tilde{\Gamma}^{-}_{\rm in}$ and the boundary conditions \eqref{eq:2.8} on $\tilde{\Gamma}_{-}$;

$\rm{(\tilde{iv})}$\ On the contact discontinuity $\tilde{\Gamma}_{\rm{cd}}$, $U_{+}$ and $U_{-}$ satisfy \eqref{eq:2.9}.

\begin{theorem}\label{thm:2.1}
Assume that the reaction rate function $\phi(T)$ is $C^{1,1}$ with respect to $T$ for all $T>0$. Then, for any $\alpha\in(0,1)$, there exists a constant $\varepsilon_0>0$ depending only on $\underline{U}$, $\alpha$ and $L$ such that if
\begin{eqnarray}\label{eq:2.15}
\sum_{k=\pm}\Big(\big\|U^{k}_{\rm in}-\underline{U}_{k}\big\|_{1,\alpha; \tilde{\Gamma}^{k}_{\rm in}}+\big\|g_k-\underline{g}_{k}\big\|_{2,\alpha; \tilde{\Gamma}_{k}}\Big)<\varepsilon,
\end{eqnarray}
for some $\varepsilon\in(0,\varepsilon_0)$, \emph{Problem 2.1} admits a unique solution $U_{\pm}$ with the following estimate
\begin{eqnarray}\label{eq:2.16}
\big\|U_{-}-\underline{U}_{-}\big\|_{1,\alpha; \tilde{\mathcal{N}}_{-}}+\big\|U_{+}-\underline{U}_{+}\big\|_{1,\alpha; \tilde{\mathcal{N}}_{+}}<\tilde{C}_{0}\varepsilon,
\end{eqnarray}
where the constant $\tilde{C}_{0}>0$ depends only on $\underline{U}$, $\alpha$ and $L$.
\end{theorem}

\begin{corollary}\label{coro:2.1}
Once Theorem \ref{thm:2.1} is proved, Theorem \ref{thm:1.1} also holds.
\end{corollary}

\begin{proof}
Since
\begin{equation*}
\begin{split}
&\sum_{k=\pm}\Big(\big\|U^{k}_{\rm in}-\underline{U}_{k}\big\|_{1,\alpha; \tilde{\Gamma}^{k}_{\rm in}}+\big\|g_k-\underline{g}_{k}\big\|_{2,\alpha; \tilde{\Gamma}_{k}}\Big)\\
&\qquad\quad <\tilde{C}_{1}\sum_{k=\pm}\Big(\big\|U^{k}_{\rm in}-\underline{U}_{k}\big\|_{1,\alpha; \Gamma^{k}_{\rm in}}+\big\|g_k-\underline{g}_{k}\big\|_{2,\alpha; \Gamma_{k}}\Big)\\
&\qquad\quad <\tilde{C}_{1}\epsilon<\tilde{C}_{1}\epsilon_0,
\end{split}
\end{equation*}
if we choose $\epsilon_0>0$ small such that $\tilde{C}_{1}\epsilon_0<\varepsilon$, then Theorem \ref{thm:2.1} holds and
\begin{eqnarray*}
\frac{\partial(x, y)}{\partial(\xi,\eta)}=
\begin{vmatrix}
1& 0\\[3pt]
\frac{v}{u}& \frac{1}{\rho u}
\end{vmatrix}
=\frac{1}{\rho u}>0,
\end{eqnarray*}
when $\varepsilon>0$ is sufficiently small. So the map $\mathcal{L}$ is $C^{1,\alpha}$-homeomorphous. Then, the inverse Lagrange {\color{black}coordinate} transformation of $\mathcal{L}$ exists and can be defined as
\begin{eqnarray*}
\mathcal{L}^{-1} : \
\begin{cases}
x = \xi, \\
y = \int^{\eta}_{-\rm{m}_{-}}\Big(\frac{1}{\rho u}\Big)(x,\tau)\dd\tau+g_{-}(x).
\end{cases}
\end{eqnarray*}
Set $U(x,y)\triangleq (U\circ\mathcal{L}^{-1})(x, y)$. Then estimate \eqref{eq:1.23} in Theorem \ref{thm:1.1} follows from \eqref{eq:2.16}.
Furthermore, the contact discontinuity function $g_{\rm cd}$ is given by
\begin{eqnarray*}
g_{\rm cd}(x)=\int^{0}_{-\rm{m}_{-}}\Big(\frac{1}{\rho u}\Big)(x,\tau)\dd\tau+g_{-}(x), \  x\in [0, L).
\end{eqnarray*}
Obviously,
\begin{eqnarray*}
g'_{\rm cd}(x)=\frac{v}{u}(x,g_{\rm cd}(x)), \  x\in [0, L).
\end{eqnarray*}
So $g_{\rm cd}\in C^{2,\alpha}([0, L))$ and satisfies estimate \eqref{eq:1.24}.
\end{proof}
Hence, we only need to prove Theorem \ref{thm:2.1}.

\subsection{Reformulation of the \emph{Problem\ 2.1} in $\tilde{\mathcal{N}}_{-}\cup\tilde{\mathcal{N}}_{+}$}
In order to complete the proof of Theorem \ref{thm:2.1}, one needs further to make some reductions on the Euler equations \eqref{eq:2.5} as well as its boundary conditions
on the upper and lower supersonic regions $\tilde{\mathcal{N}}_{\pm}$, respectively.

Denote $\tilde{V}_{\pm}\triangleq\big(u_{\pm}, v_{\pm}, p_{\pm}\big)^{\top}$. Then, for $C^1$-solutions in $\tilde{\mathcal{N}}_{\pm}$,
we can rewrite the first three equations of system \eqref{eq:2.5} as the following first order nonlinear inhomogeneous hyperbolic system,
\begin{eqnarray}\label{eq:2.17}
\mathcal{A}_{1}\partial_{\xi}\tilde{V}_{\pm}+\mathcal{A}_{2}\partial_{\eta}\tilde{V}_{\pm}=\mathcal{F},
\end{eqnarray}
where
\begin{eqnarray*}
\mathcal{A}_{1}=\left(
\begin{matrix}
\rho_{\pm}u_{\pm} & 0 & 1 \\
0 & \rho_{\pm}u_{\pm} & 0\\
1 & 0 &\frac{u_{\pm}}{\rho_{\pm}c^2_{\pm}}
\end{matrix}
\right),\quad
\mathcal{A}_{2}=\left(
\begin{matrix}
0&0& -\rho_{\pm}v_{\pm} \\
0& 0 &\rho_{\pm}u_{\pm} \\
-\rho_{\pm}v_{\pm} & \rho_{\pm}u_{\pm} &0
\end{matrix}
\right),
\end{eqnarray*}
and
\begin{equation*}
\mathcal{F}=\left(
\begin{matrix}
0& \\
0& \\
\frac{(\gamma-1)\phi(T_{\pm})Y_{\pm}}{c^2_{\pm}} &
\end{matrix}
\right).
\end{equation*}

The eigenvalues of system \eqref{eq:2.17} are 
\begin{eqnarray}\label{eq:2.18}
\begin{split}
&\lambda^{0}_{\pm}=0, \quad \lambda^{\pm}_{+}=\frac{\rho_{+} c^{2}_{+} u_{+}}
{u^{2}_{+}-c^{2}_{+}}\bigg(\frac{v_{+}}{u_{+}}
\pm \sqrt{\frac{u^{2}_{+}+v^{2}_{+}}{c^2_{+}}-1}\bigg),\\
&\qquad \ \lambda^{\pm}_{-}=\frac{\rho_{-} c^{2}_{-} u_{-}}
{u^{2}_{-}-c^{2}_{-}}\bigg(\frac{v_{-}}{u_{-}}
\pm \sqrt{\frac{u^{2}_{-}+v^{2}_{-}}{c^2_{-}}-1}\bigg),
\end{split}
\end{eqnarray}
and the corresponding left eigenvectors are
\begin{equation*}
\boldsymbol{\ell}^{0}_{\pm}=\big(u_{\pm}, v_{\pm}, 0\big),\quad\ \boldsymbol{\ell}^{\pm}_{+}=\Big(\frac{\lambda^{\pm}_{+}}{\rho_{+}}+v_{+}, -u_{+}, -\lambda_{+}^{\pm}u_{+}\Big),
\quad\
\boldsymbol{\ell}^{\pm}_{-}=\Big(\frac{\lambda^{\pm}_{-}}{\rho_{-}}+v_{-}, -u_{-}, -\lambda_{-}^{\pm}u_{-}\Big).
\end{equation*}

Multiplying $\boldsymbol{\ell}_{+}^{0}$ on both sides of system \eqref{eq:2.17} in $\tilde{\mathcal{N}}_{+}$ to get
\begin{eqnarray}\label{eq:2.19}
u_{+}\partial_{\xi}u_{+}+v_{+}\partial_{\xi}v_{+}+\frac{1}{\rho_{+}}\partial_{\xi}p_{+}=0.
\end{eqnarray}
On the other hand, we derive from \eqref{eq:1.4}, \eqref{eq:1.8}, and the fourth equation in \eqref{eq:2.5} that 
\begin{eqnarray}\label{eq:2.20}
u_{+}\partial_{\xi}u_{+}+v_{+}\partial_{\xi}v_{+}+\frac{1}{\rho_{+}}\partial_{\xi}p_{+}
-\frac{A'(S_{+})}{\gamma-1}\rho^{\gamma-1}_{+}\partial_{\xi}S_{+}
=\frac{\mathfrak{q}_0\phi(T_{{+}})}{u_{{+}}}Y_{{+}}.
\end{eqnarray}

Plugging \eqref{eq:2.19} into \eqref{eq:2.20} and by \eqref{eq:1.4}, 
we obtain
\begin{eqnarray}\label{eq:2.21}
\partial_{\xi}S_{+}=-\frac{\gamma \mathcal{R}\mathfrak{q}_{0}\phi(T_{+})Y_{+}}{c^2_{+}u_{+}}, \qquad\mbox{in}\quad \tilde{\mathcal{N}}_{+}.
\end{eqnarray}

Next, multiplying $\boldsymbol{\ell}^{\pm}_{+}$ on both sides of system \eqref{eq:2.17} in $\tilde{\mathcal{N}}_{+}$ to deduce
\begin{eqnarray}\label{eq:2.22}
\begin{split}
&v_{+}\big(\partial_{\xi}u_{+}+\lambda^{\pm}_{+}\partial_{\eta}u_{+}\big)
-u_{+}\big(\partial_{\xi}v_{+}+\lambda^{\pm}_{+}\partial_{\eta}v_{+}\big)\\
&\qquad\qquad \mp\frac{\sqrt{u^{2}_{+}+v^{2}_{+}-c^{2}_{+}}}{\rho_{+}c_{+}}\big(\partial_{\xi}p_{+}+\lambda^{\pm}_{+}\partial_{\eta}p_{+}\big)
=-\frac{(\gamma-1)\mathfrak{q}_{0}\phi(T_{+})\lambda^{\pm}_{+}Y_{+}}{\rho_{+}c^2_{+}}.
\end{split}
\end{eqnarray}

Set
\begin{eqnarray}\label{eq:2.23}
\omega_{+}\triangleq\frac{v_{+}}{u_{+}}, \quad\ \Lambda_{+}\triangleq\frac{\sqrt{u^{2}_{+}+v^{2}_{+}-c^{2}_{+}}}{\rho_{+}c_{+}u^2_{+}}.
\end{eqnarray}

Then, equations \eqref{eq:2.22} in $\tilde{\mathcal{N}}_{+}$ can be rewritten as
\begin{eqnarray}\label{eq:2.24}
\begin{split}
\partial_{\xi}\omega_{+}+\lambda^{\pm}_{+}\partial_{\eta}\omega_{+}\pm\Lambda_{+}\big(\partial_{\xi}p_{+}+\lambda^{\pm}_{+}\partial_{\eta}p_{+}\big)
=\frac{(\gamma-1)\mathfrak{q}_{0}\phi(T_{+})\lambda^{\pm}_{+}Y_{+}}{\rho_{+}c^2_{+}u^2_{+}},\qquad\mbox{in}\quad \tilde{\mathcal{N}}_{+}.
\end{split}
\end{eqnarray}

Similarly, define
\begin{eqnarray}\label{eq:2.25}
\omega_{-}\triangleq\frac{v_{-}}{u_{-}}, \quad\ \Lambda_{-}\triangleq\frac{\sqrt{u^{2}_{-}+v^{2}_{-}-c^{2}_{-}}}{\rho_{-}c_{-}u^2_{-}}.
\end{eqnarray}

By a similar argument, 
$\omega_{-}$, $p_{-}$ and $S_{-}$ in $\tilde{\mathcal{N}}_{-}$ satisfy
\begin{eqnarray}\label{eq:2.26}
\begin{split}
\partial_{\xi}\omega_{-}+\lambda^{\pm}_{-}\partial_{\eta}\omega_{-}\pm\Lambda_{-}\big(\partial_{\xi}p_{-}+\lambda^{\pm}_{-}\partial_{\eta}p_{-}\big)
=\frac{(\gamma-1)\mathfrak{q}_{0}\phi(T_{-})\lambda^{\pm}_{-}Y_{-}}{\rho_{-}c^2_{-}u^2_{-}},\qquad\mbox{in}\quad \tilde{\mathcal{N}}_{-},
\end{split}
\end{eqnarray}
and
\begin{eqnarray}\label{eq:2.27}
\partial_{\xi}S_{-}=-\frac{\gamma \mathcal{R}\mathfrak{q}_{0}\phi(T_{-})Y_{-}}{c^2_{-}u_{-}}, \qquad\mbox{in}\quad \tilde{\mathcal{N}}_{-}.
\end{eqnarray}

Finally, we introduce the initial and boundary conditions for $(\omega_{\pm}, p_{\pm}, S_{\pm}, Y_{\pm})$. 
By \eqref{eq:2.6}, we know that
\begin{eqnarray}\label{eq:2.28}
(\omega_{\pm}, p_{\pm}, B_{\pm}, S_{\pm}, Y_{\pm})=(\omega^{\pm}_{\rm in}, p^{\pm}_{\rm in}, B^{\pm}_{\rm in}, S^{\pm}_{\rm in}, Y^{\pm}_{\rm in})(\eta),\qquad \mbox{on}\quad \tilde{\Gamma}^{\pm}_{\rm in},
\end{eqnarray}
where $B^{\pm}_{\rm in}$, $\omega^{\pm}_{\rm in}$ and $S^{\pm}_{\rm in}$ are given by
\begin{eqnarray}\label{eq:2.29}
\begin{split}
B^{\pm}_{\rm in}(\eta)\triangleq\frac{(u^{\pm}_{\rm in})^2+(v^{\pm}_{\rm in})^2}{2}
+\frac{\gamma p^{\pm}_{\rm in}}{(\gamma-1)\rho^{\pm}_{\rm in}},
\end{split}
\end{eqnarray}
and
\begin{eqnarray}\label{eq:2.30}
\omega^{\pm}_{\rm in}(\eta)\triangleq\big(\frac{v^{\pm}_{\rm in}}{u^{\pm}_{\rm in}}\big)(\eta),\quad\ S^{\pm}_{\rm in}(\eta)\triangleq A^{-1}\big(\frac{p^{\pm}_{\rm in}}{(\rho^{\pm}_{\rm in})^{\gamma}}\big)(\eta).
\end{eqnarray}

On the nozzle walls and contact discontinuity, it follows from \eqref{eq:2.8} and \eqref{eq:2.9} that 
\begin{eqnarray}\label{eq:2.31}
\omega_{+}=g'_{+}(\xi),\quad\mbox{on}\quad \tilde{\Gamma}_{+},\qquad \omega_{-}=g'_{-}(\xi),\quad\mbox{on}\quad \tilde{\Gamma}_{-},
\end{eqnarray}
and 
\begin{eqnarray}\label{eq:2.32}
\omega_{+}=\omega_{-}=g'_{\rm cd}(\xi),\qquad p_{+}=p_{-}, \qquad\mbox{on}\quad \tilde{\Gamma}_{\rm cd}.
\end{eqnarray}


Therefore, \emph{Problem 2.1} can be reformulated as the following initial-boundary value problem. 
\smallskip
\par \emph{$\mathbf{Problem\ 2.2.}$}\ For given functions $\big(\omega^{\pm}_{\rm in}, p^{\pm}_{\rm in}, B^{\pm}_{\rm in}, S^{\pm}_{\rm in}, Y^{\pm}_{\rm in}\big)(\eta)$ 
and 
$g_{\pm}$ which satisfy \eqref{eq:2.10} and \eqref{eq:2.11}, to find a piecewise smooth solution $(\omega_{\pm}, p_{\pm}, B_{\pm}, S_{\pm}, Y_{\pm})$ with a contact discontinuity $\tilde{\Gamma}_{\rm cd}$ to the following initial-boundary value problem:
\begin{eqnarray}\label{eq:2.33}
(\mathbf{IBVP})\
\begin{cases}
\partial_{\xi}\omega_{+}+\lambda^{\pm}_{+}\partial_{\eta}\omega_{+}\pm\Lambda_{+}\big(\partial_{\xi}p_{+}+\lambda^{\pm}_{+}\partial_{\eta}p_{+}\big)
=\frac{(\gamma-1)\mathfrak{q}_{0}\phi(T_{+})\lambda^{\pm}_{+}Y_{+}}{\rho_{+}c^2_{+}u^2_{+}}, &\quad  \mbox{in} \quad \tilde{\mathcal{N}}_{+}, \\
\partial_{\xi}B_{+}=\frac{\mathfrak{q}_0\phi(T_{+})Y_{+}}{u_{+}},\quad \partial_{\xi}S_{+}=-\frac{\gamma \mathcal{R}\mathfrak{q}_{0}\phi(T_{+})Y_{+}}{c^2_{+}u_{+}},\quad\ \partial_{\xi}Y_{+}=-\frac{\phi(T_{+})Y_{+}}{u_{+}}, &\quad \mbox{in}\quad \tilde{\mathcal{N}}_{+},\\
\partial_{\xi}\omega_{-}+\lambda^{\pm}_{-}\partial_{\eta}\omega_{-}\pm\Lambda_{-}\big(\partial_{\xi}p_{-}+\lambda^{\pm}_{-}\partial_{\eta}p_{-}\big)
=\frac{(\gamma-1)\mathfrak{q}_{0}\phi(T_{-})\lambda^{\pm}_{-}Y_{-}}{\rho_{-}c^2_{-}u^2_{-}}, &\quad  \mbox{in} \quad \tilde{\mathcal{N}}_{-},\\
\partial_{\xi}B_{-}=\frac{\mathfrak{q}_0\phi(T_{-})Y_{-}}{u_{-}},\quad\partial_{\xi}S_{-}=-\frac{\gamma \mathcal{R}\mathfrak{q}_{0}\phi(T_{-})Y_{-}}{c^2_{-}u_{-}}, \quad\ \partial_{\xi}Y_{-}=-\frac{\phi(T_{-})Y_{-}}{u_{-}}, &\quad  \mbox{in} \quad \tilde{\mathcal{N}}_{-},\\
(\omega_{+}, p_{+}, B_{+}, S_{+}, Y_{+})=(\omega^{+}_{\rm in}, p^{+}_{\rm in}, B^{+}_{\rm in}, S^{+}_{\rm in}, Y^{+}_{\rm in})(\eta),  &\quad  \mbox{on} \quad \tilde{\Gamma}^{+}_{\rm in},\\
(\omega_{-}, p_{-}, B_{-}, S_{-}, Y_{-})=(\omega^{-}_{\rm in}, p^{-}_{\rm in}, B^{-}_{\rm in}, S^{-}_{\rm in}, Y^{-}_{\rm in})(\eta),  &\quad  \mbox{on} \quad \tilde{\Gamma}^{-}_{\rm in},\\
\omega_{+}=g'_{+}(\xi), &\quad \mbox{on} \quad \tilde{\Gamma}_{+},\\
\omega_{+}=\omega_{-},\quad\ p_{+}=p_{-},&\quad \mbox{on} \quad \tilde{\Gamma}_{\rm cd},\\
\omega_{-}=g'_{-}(\xi), &\quad \mbox{on} \quad \tilde{\Gamma}_{-}.
\end{cases}
\end{eqnarray}

In the following, we denote by $\mathcal{U}_{\pm}\triangleq(\omega_{\pm}, p_{\pm}, B_{\pm}, S_{\pm}, Y_{\pm})^{\top}$ a solution of the initial-boundary value problem $(\mathbf{IBVP})$ with the initial datum $\mathcal{U}^{\pm}_{\rm in}=(\omega^{\pm}_{\rm in}, p^{\pm}_{\rm in}, B^{\pm}_{\rm in}, S^{\pm}_{\rm in}, Y^{\pm}_{\rm in})^{\top}$.
Then, we have the following result.
\begin{theorem}\label{thm:2.2}
Suppose that the reaction rate function $\phi(T)$ is $C^{1,1}$ with respect to $T$ for all $T>0$. Then, there exists a constant $\varepsilon_1>0$ depending only on $\underline{U}$, $\alpha$ and $L$ such that if
\begin{eqnarray}\label{eq:2.34}
\sum_{k=\pm}\Big(\|\mathcal{U}^{\pm}_{\rm in}-\underline{\mathcal{U}}_{\pm}\|_{1,\alpha; \tilde{\Gamma}^{k}_{\rm in}}+\big\|g_k-\underline{g}_{k}\big\|_{2,\alpha; \tilde{\Gamma}_{k}}\Big)
<\varepsilon,
\end{eqnarray}
for some $\varepsilon\in(0,\varepsilon_1)$, \emph{Problem 2.2} admits a unique solution $\mathcal{U}_{\pm}$ satisfying
\begin{eqnarray}\label{eq:2.35}
\big\|\mathcal{U}_{-}-\underline{\mathcal{U}}_{-}\big\|_{1,\alpha; \tilde{\mathcal{N}}_{-}}+\big\|\mathcal{U}_{+}-\underline{\mathcal{U}}_{+}\big\|_{1,\alpha; \tilde{\mathcal{N}}_{+}}<\tilde{C}_{2}\varepsilon,
\end{eqnarray}
for the constant $\tilde{C}_{2}>0$ depending only on $\underline{U}$, $\alpha$ and $L$. Here, $\underline{\mathcal{U}}_{\pm}\triangleq(0, \underline{p}^{\pm}, \underline{B}^{\pm}, \underline{S}^{\pm}, 0)^{\top}$ and
\begin{eqnarray}\label{eq:2.36}
\underline{B}^{\pm}\triangleq\frac{1}{2}(\underline{u}^{\pm})^{2}+\frac{\gamma A^{\frac{1}{\gamma}}(\underline{S}^{\pm})(\underline{p}^{\pm})^{1-\frac{1}{\gamma}}}{\gamma-1}.
\end{eqnarray}

\end{theorem}

\section{Linearized Problem}\label{Sec3}

In this section, we develop an iteration scheme by linearizing the ({\bf IBVP}) near the background state $\underline{\mathcal{U}}_{\pm}$,
and then show the well-posedness and some \emph{a priori} estimates of the solutions of the linearized problem, which will be used for the study of the nonlinear problem later.

\subsection{Iteration scheme for $(\mathbf{IBVP})$}
For $\sigma\in(0,1)$, we define the iteration set
\begin{eqnarray}\label{eq:3.1}
\mathcal{X}_{\sigma}\triangleq\big\{(\delta\mathcal{U}_{+}, \delta\mathcal{U}_{-})\in C^{1,\alpha}(\tilde{\mathcal{N}}_{+})\times C^{1,\alpha}(\tilde{\mathcal{N}}_{-}) |\ \|\delta\mathcal{U}_{+}\|_{1,\alpha;\tilde{\mathcal{N}}_{+}}+\|\delta\mathcal{U}_{-}\|_{1,\alpha;\tilde{\mathcal{N}}_{-}} \leq \sigma \big\}.
\end{eqnarray}
Then, for any given $\delta\mathcal{U}_{\pm}\in \mathcal{X}_{\sigma}$ with $\delta\mathcal{U}_{\pm}=\big(\delta\omega_{\pm}, \delta p_{\pm}, \delta B_{\pm}, \delta S_{\pm}, \delta Y_{\pm}\big)^{\top}$, let $\mathcal{U}_{\pm}\triangleq\underline{\mathcal{U}}_{\pm}+\delta\mathcal{U}_{\pm}$ and $\delta\mathcal{V}_{\pm}\triangleq\big(\delta\omega_{\pm}, \delta p_{\pm}, \delta B_{\pm}, \delta S_{\pm}\big)^{\top}$.
We update $\delta\mathcal{U}_{\pm}$ by $\delta\mathcal{U}^{*}_{\pm}$, which solves the following linearized problem:
\begin{eqnarray}\label{eq:3.2}
(\mathbf{IBVP})^{*}\
\begin{cases}
\partial_{\xi}\delta\omega^{*}_{+}+\lambda^{\pm}_{+}\partial_{\eta}\delta\omega^{*}_{+}\pm\Lambda_{+}\big(\partial_{\xi}\delta p^{*}_{+}+\lambda^{\pm}_{+}\partial_{\eta}\delta p^{*}_{+}\big)\\
\qquad\qquad\qquad\qquad\quad =a^{\pm}_{+}(\underline{\mathcal{U}}_{+})\delta Y^{*}_{+}+f^{\pm}_{+}(\delta\mathcal{V}_{+},\delta Y_{+}), &\quad  \mbox{in} \quad \tilde{\mathcal{N}}_{+}, \\
\partial_{\xi}\delta B^{*}_{+}=a_{B_{{+}}}(\underline{\mathcal{U}}_{+})\delta Y^{*}_{+}+f_{B_{+}}(\delta\mathcal{V}_{+},\delta Y_{+}), &\quad \mbox{in}\quad \tilde{\mathcal{N}}_{+},\\
\partial_{\xi}\delta S^{*}_{+}=a_{S_{+}}(\underline{\mathcal{U}}_{+})\delta Y^{*}_{+}+f_{S_{+}}(\delta\mathcal{V}_{+},\delta Y_{+}), &\quad \mbox{in}\quad \tilde{\mathcal{N}}_{+},\\
\partial_{\xi}\delta Y^{*}_{+}=a_{Y_{+}}(\underline{\mathcal{U}}_{+})\delta Y^{*}_{+}+f_{Y_{+}}(\delta\mathcal{V}_{+},\delta Y_{+}), &\quad \mbox{in}\quad \tilde{\mathcal{N}}_{+},\\
\partial_{\xi}\delta\omega^{*}_{-}+\lambda^{\pm}_{-}\partial_{\eta}\delta \omega^{*}_{-}\pm\Lambda_{-}\big(\partial_{\xi}\delta p^{*}_{-}+\lambda^{\pm}_{-}\partial_{\eta}\delta p^{*}_{-}\big)\\
\qquad\qquad\qquad\qquad\quad =a^{\pm}_{-}(\underline{\mathcal{U}}_{-})\delta Y^{*}_{-}+f^{\pm}_{-}(\delta\mathcal{V}_{-},\delta Y_{-}), &\quad  \mbox{in} \quad \tilde{\mathcal{N}}_{-},\\
\partial_{\xi}\delta B^{*}_{-}=a_{B_{-}}(\underline{\mathcal{U}}_{-})\delta Y^{*}_{-}+f_{B_{-}}(\delta\mathcal{V}_{-},\delta Y_{-}), &\quad  \mbox{in} \quad \tilde{\mathcal{N}}_{-},\\
\partial_{\xi}\delta S^{*}_{-}=a_{S_{-}}(\underline{\mathcal{U}}_{-})\delta Y^{*}_{-}+f_{S_{-}}(\delta\mathcal{V}_{-},\delta Y_{-}), &\quad  \mbox{in} \quad \tilde{\mathcal{N}}_{-},\\
\partial_{\xi}\delta Y^{*}_{-}=a_{Y_{-}}(\underline{\mathcal{U}}_{-})\delta Y^{*}_{-}+f_{Y_{-}}(\delta\mathcal{V}_{-},\delta Y_{-}), &\quad \mbox{in}\quad \tilde{\mathcal{N}}_{-},\\
\delta \mathcal{U}^{*}_{+}=\delta \mathcal{U}^{+}_{\rm in},  &\quad  \mbox{on} \quad \tilde{\Gamma}^{+}_{\rm in},\\
\delta \mathcal{U}^{*}_{-}=\delta \mathcal{U}^{-}_{\rm in},  &\quad  \mbox{on} \quad \tilde{\Gamma}^{-}_{\rm in},\\
\delta\omega^{*}_{+}=g'_{+}(\xi), &\quad \mbox{on} \quad \tilde{\Gamma}_{+},\\
\delta\omega^{*}_{+}=\delta\omega^{*}_{-},\quad \delta p^{*}_{+}=\delta p^{*}_{-},&\quad \mbox{on} \quad \tilde{\Gamma}_{\rm cd},\\
\delta\omega^{*}_{-}=g'_{-}(\xi), &\quad \mbox{on} \quad \tilde{\Gamma}_{-},
\end{cases}
\end{eqnarray}
where $a^{\pm}_{\pm}$, $a_{B_{\pm}}$, $a_{S_{\pm}}$, $a_{Y_{\pm}}$, $f^{\pm}_{\pm}$, $f_{B_{\pm}}$, $f_{S_{\pm}}$, $f_{Y_{\pm}}$ are defined by
\begin{eqnarray}\label{eq:3.3}
\begin{cases}
a^{\pm}_{+}(\underline{\mathcal{U}}_{+})\triangleq\frac{(\gamma-1)\mathfrak{q}_{0}\phi(\underline{T}^{+})
\underline{\lambda}^{\pm}_{+}}{\underline{\rho}^{+}(\underline{c}^{+})^{2}(\underline{u}^{+})^{2}},\quad
a^{\pm}_{-}(\underline{\mathcal{U}}_{-})\triangleq\frac{(\gamma-1)\mathfrak{q}_{0}\phi(\underline{T}^{-})
\underline{\lambda}^{\pm}_{-}}{\underline{\rho}^{-}(\underline{c}^-)^{2}(\underline{u}^-)^{2}}, \\
a_{B_{\pm}}(\underline{\mathcal{U}}_{\pm})\triangleq\frac{\mathfrak{q}_{0}\phi(\underline{T}^{\pm})}{\underline{u}^{\pm}},\quad
a_{S_{\pm}}(\underline{\mathcal{U}}_{\pm})\triangleq-\frac{\gamma\mathcal{R}\mathfrak{q}_{0}\phi(\underline{T}^{\pm})}{(\underline{c}^{\pm})^2\underline{u}^{\pm}},\quad
a_{Y_{\pm}}(\underline{\mathcal{U}}_{\pm})\triangleq-\frac{\phi(\underline{T}^{\pm})}{\underline{u}^{\pm}},
\end{cases}
\end{eqnarray}
and
\begin{eqnarray}\label{eq:3.4-1}
\begin{cases}
f^{\pm}_{+}(\delta\mathcal{V}_{+},\delta Y_{+})\triangleq(\gamma-1)\mathfrak{q}_{0}\cdot
\bigg(\frac{\phi(T_{+})\lambda^{\pm}_{+}}{\rho_{+}(c_{+})^2(u_{+})^2}
-\frac{\phi(\underline{T}^{+})\underline{\lambda}^{\pm}_{+}}{\underline{\rho}^{+}(\underline{c}^{+})^2(\underline{u}^{+})^2}\bigg)\cdot\delta Y_{+},\\[5pt]
f^{\pm}_{-}(\delta\mathcal{V}_{-},\delta Y_{-})\triangleq(\gamma-1)\mathfrak{q}_{0}\cdot
\bigg(\frac{\phi(T_{-})\lambda^{\pm}_{-}}{\rho_{-}(c_{-})^2(u_{-})^2}
-\frac{\phi(\underline{T}^{-})\underline{\lambda}^{\pm}_{-}}{\underline{\rho}^{-}(\underline{c}^{-})^2(\underline{u}^{-})^2}\bigg)\cdot\delta Y_{-},\\[4pt]
f_{B_{\pm}}(\delta\mathcal{V}_{\pm}, \delta Y_{\pm})
\triangleq\mathfrak{q}_{0}\cdot\bigg(\frac{\phi(T_{\pm})}{u_{\pm}}
-\frac{\phi(\underline{T}^{\pm})}{\underline{u}^{\pm}}\bigg)\cdot\delta Y_{\pm},\\
f_{S_{\pm}}(\delta\mathcal{V}_{\pm}, \delta Y_{\pm})\triangleq\gamma\mathcal{R}\mathfrak{q}_{0}\cdot\bigg(\frac{\phi(\underline{T}^{\pm})}{(\underline{c}^{\pm})^2\underline{u}^{\pm}}
-\frac{\phi(T_{\pm})}{(c_{\pm})^2u_{\pm}}\bigg)\cdot\delta Y_{\pm},\\
f_{Y_{\pm}}(\delta\mathcal{V}_{\pm},\delta Y_{\pm})\triangleq\bigg(\frac{\phi(\underline{T}^{\pm})}{\underline{u}^{\pm}}-\frac{\phi(T_{\pm})}{u_{\pm}}\bigg)\cdot \delta Y_{\pm}.
\end{cases}
\end{eqnarray}

For the linearized problem $(\mathbf{IBVP})^{*}$, we have the following theorem.
\begin{theorem}\label{thm:3.1}
Suppose that the reaction rate function $\phi(T)$ is $C^{1,1}$ with respect to $T$. Then, for any given states $(\delta\mathcal{U}_{+},\delta\mathcal{U}_{-})\in\mathcal{X}_{\sigma}$ with $\sigma>0$ sufficiently small,
the linearized initial-boundary value problem $(\mathbf{IBVP})^{*}$ admits a unique solution $(\delta \mathcal{U}^{*}_{+},\delta\mathcal{U}^{*}_{-})\in C^{1,\alpha}(\tilde{\mathcal{N}}_{+})\times C^{1,\alpha}(\tilde{\mathcal{N}}_{-})$ satisfying the estimate
\begin{eqnarray}\label{eq:3.5}
\begin{split}
\sum_{k=\pm}\|\delta \mathcal{U}^{*}_{k}\|_{1,\alpha; \tilde{\mathcal{N}}_{k}}
\leq \tilde{C}_{3}\sum_{k=\pm}\Big(\|\delta \mathcal{U}^{k}_{\rm in}\|_{1,\alpha; \tilde{\Gamma}^{k}_{\rm in}}+\|g_k-\underline{g}_{k}\|_{2,\alpha; \tilde{\Gamma}_{k}}
+\|\delta \mathcal{V}_{k}\|_{1,\alpha; \tilde{\mathcal{N}}_{k}}\|\delta Y_{k}\|_{1,\alpha; \tilde{\mathcal{N}}_{k}}\Big),
\end{split}
\end{eqnarray}
where the constant $\tilde{C}_{3}>0$ depends only on the $\underline{\mathcal{U}}_{\pm}$ and $\alpha$, $L$.
\end{theorem}

\subsection{\emph{A priori} estimates for solutions of $(\mathbf{IBVP})^{*}$}
Since the equations of $(\delta B^{*}_{\pm}, \delta S^{*}_{\pm}, \delta Y^{*}_{\pm})$ are independent on $\delta \omega^{*}_{\pm}$ and $\delta p^{*}_{\pm}$,
we first consider the solutions $(\delta B^{*}_{\pm}, \delta S^{*}_{\pm}, \delta Y^{*}_{\pm})$ to {\color{black}the} problem $(\mathbf{IBVP})^{*}$ by the following proposition.
\begin{proposition}\label{prop:3.1}
Assume the assumptions in Theorem \ref{thm:3.1} hold. For $(\delta\mathcal{U}_{+},\delta\mathcal{U}_{-})\in\mathcal{X}_{\sigma}$ with small $\sigma>0$, problem $(\mathbf{IBVP})^{*}$ admits a unique $C^{1,\alpha}$-solution $(\delta B^{*}_{\pm}, \delta S^{*}_{\pm}, \delta Y^{*}_{\pm})$ satisfying
\begin{eqnarray}\label{eq:prop-3.1-1}
\begin{split}
\|(\delta B^{*}_{k},\delta S^{*}_{k},\delta Y^{*}_{k})\|_{1,\alpha; \tilde{\mathcal{N}}_{k}}
&\leq \tilde{C}_{30}\Big(\|(\delta B^{k}_{\rm in}, \delta S^{k}_{\rm in},\delta Y^{k}_{\rm in})\|_{1,\alpha; \tilde{\Gamma}^{k}_{\rm in}}+\|\delta \mathcal{V}_{k}\|_{1,\alpha; \tilde{\mathcal{N}}_{k}}\|\delta Y_{k}\|_{1,\alpha; \tilde{\mathcal{N}}_{k}}\Big),
\end{split}
\end{eqnarray}
where $k=``\pm"$, and the constant $\tilde{C}_{30}>0$ depends only on $\underline{\mathcal{U}}_{\pm}$, $\alpha$ and $L$.
\end{proposition}
\begin{proof}
Since $\delta B^{*}_{k}, \delta S^{*}_{k}$ and $\delta Y^{*}_{k}$ in \eqref{eq:3.2} satisfy the initial value problem governed by linear ordinary differential equations, we can follow Lemma B.1 in \cite{gao-liu-yuan} to get the existence and uniqueness of $C^{1,\alpha}$-solution {\color{black}to it} with the following estimates
\begin{eqnarray*}
\begin{cases}
\|\delta Y^{*}_{k}\|_{1,\alpha; \tilde{\mathcal{N}}_{k}}
&\leq \mathcal{O}(1)\Big(\|\delta Y^{k}_{\rm in}\|_{1,\alpha; \tilde{\Gamma}^{k}_{\rm in}}+\|f_{Y_{k}}\|_{1,\alpha; \tilde{\mathcal{N}}_{k}}\Big),\\[5pt]
\|\delta B^{*}_{k}\|_{1,\alpha; \tilde{\mathcal{N}}_{k}} &\leq \mathcal{O}(1)\Big(\|\delta B^{k}_{\rm in}\|_{1,\alpha; \tilde{\Gamma}^{k}_{\rm in}}
+\|f_{B_{k}}\|_{1,\alpha; \tilde{\mathcal{N}}_{k}}+\|\delta Y^{*}_{k}\|_{1,\alpha; \tilde{\mathcal{N}}_{k}}\Big),\\[5pt]
\|\delta S^{*}_{k}\|_{1,\alpha; \tilde{\mathcal{N}}_{+}}&\leq \mathcal{O}(1)\Big(\|\delta S^{k}_{\rm in}\|_{1,\alpha; \tilde{\Gamma}^{k}_{\rm in}}
+\|f_{S_{k}}\|_{1,\alpha; \tilde{\mathcal{N}}_{k}}+\|\delta Y^{*}_{k}\|_{1,\alpha; \tilde{\mathcal{N}}_{k}}\Big),
\end{cases}
\end{eqnarray*}
where $k=``\pm"$, and the constants $\mathcal{O}(1)$ depend only on $\underline{\mathcal{U}}_{\pm}$ and $\alpha$, $L$.

Since $\phi\in C^{1,1}$, by \eqref{eq:3.4-1} and direct computation, we can get for $\sigma>0$ sufficiently small that
\begin{eqnarray*}
\|(f_{B_{k}}, f_{S_{k}},f_{Y_{k}})\|_{1,\alpha; \tilde{\mathcal{N}}_{k}}\leq \mathcal{O}(1)\|\delta \mathcal{V}_{k}\|_{1,\alpha; \tilde{\mathcal{N}}_{k}}\cdot\|\delta Y_{k}\|_{1,\alpha; \tilde{\mathcal{N}}_{k}},
\end{eqnarray*}
where the constant $\mathcal{O}(1)$ depends only on $\underline{\mathcal{U}}_{\pm}$ and $\alpha$, $L$. Combining all the four estimates above, we thus obtain estimate \eqref{eq:prop-3.1-1} and complete the proof of the proposition.
\end{proof}

Based on Proposition \ref{prop:3.1}, we will next consider the problem  $(\mathbf{IBVP})^{*}$ for $(\delta \omega^{*}_{\pm},\delta p^{*}_{\pm})$. To complete it, we will employ the characteristic methods by dividing the upper and lower regions into sub-domains (see Fig.\ref{fig3.4}) as follows.
Denote by $\eta=\Upsilon^{+,0}_{+}(\xi; \bar{\xi}^{2}_{+}, \rm{m}_{+})$ (or $\eta=\Upsilon^{-,0}_{-}(\xi; \bar{\xi}^{2}_{-}, \rm{m}_{-})$) the characteristic $\ell^{+,0}_{+}$ (or $\ell^{-,0}_{-}$) corresponding to $\lambda^{+}_{+}$ (or $\lambda^{-}_{-}$) starting from $\tilde{O}$ and intersecting the upper wall $\tilde{\Gamma}_{+}$ (or the lower wall $\tilde{\Gamma}_{-}$) at  point $(\bar{\xi}^{2}_{+}, \rm{m}_{+})$ (or $(\bar{\xi}^{2}_{-}, -\rm{m}_{-})$). Denote by $\eta=\Upsilon^{-,1}_{+}(\xi; \bar{\xi}^{2}_{\rm{cd}}, 0)$ (or $\eta=\Upsilon^{+,1}_{-}(\xi; \bar{\xi}^{2}_{\rm{cd}}, 0)$) the characteristic $\ell^{-,1}_{+}$ (or $\ell^{+,1}_{-}$) corresponding to $\lambda^{-}_{+}$ (or $\lambda^{+}_{-}$)
starting from $\tilde{P}_{+}$ (or $\tilde{P}_{-}$) and intersecting $\tilde{\Gamma}_{\rm{cd}}$ at point $(\bar{\xi}^{2}_{\rm{cd}}, 0)$. $\ell^{+,0}_{+}$ (or $\ell^{-,0}_{-}$) and $\ell^{-,1}_{+}$ (or $\ell^{+,1}_{-}$)
intersect at point $(\bar{\xi}^{1}_{+}, \bar{\eta}^{1}_{+})$ (or $(\bar{\xi}^{1}_{-}, \bar{\eta}^{1}_{-})$).
Let $\tilde{\mathcal{N}}^{\rm{I}}_{\pm}$ be the upper or lower triangle with $\tilde{\Gamma}^{\pm}_{\textrm{in}}$, $\ell^{\pm,0}_{\pm}$ and $\ell^{\mp,1}_{\pm}$ as its boundaries.
Let $\tilde{\mathcal{N}}^{\rm{II}}_{+}$ (or $\tilde{\mathcal{N}}^{\rm{II}}_{-}$) be the upper (or lower) triangle bounded by $\tilde{\Gamma}_{+}$ (or $\tilde{\Gamma}_{-}$ ), $\ell^{+,0}_{+}$ (or $\ell^{-,0}_{+}$) and $\ell^{-,1}_{+}$ (or $\ell^{+,1}_{-}$).
Let $\tilde{\mathcal{N}}^{\rm{III}}_{+}$ (or $\tilde{\mathcal{N}}^{\rm{III}}_{-}$) be the diamond bounded by $\ell^{+,0}_{+}$, $\ell^{-,1}_{+}$ and $\tilde{\Gamma}_{\rm{cd}}$ (or $\ell^{-,0}_{-}$, $\ell^{+,1}_{-}$ and $\tilde{\Gamma}_{\rm{cd}}$).
Let $\tilde{\mathcal{N}}^{\rm{IV}}_{+}$ (or $\tilde{\mathcal{N}}^{\rm{IV}}_{-}$) be the upper (or lower) diamond bounded by $\ell^{+,0}_{+}$ (or $\ell^{-,0}_{-}$), $\ell^{+,1}_{-}$ (or $\ell^{-,1}_{+}$), the characteristic issuing from $(\bar{\xi}^{2}_{\rm{cd}}, 0)$ and corresponding to $\lambda^{+}_{+}$ (or the characteristic issuing from $(\bar{\xi}^{2}_{\rm{cd}}, 0)$ and corresponding to $\lambda^{-}_{-}$), and the characteristic corresponding to $\lambda^{-}_{+}$  and starting from $(\bar{\xi}_{+}^{2},\rm{m}_{+})\in\tilde{\Gamma}_{+}$ (or the characteristic corresponding to $\lambda^{+}_{-}$ and starting from $(\bar{\xi}_{-}^{2},-\rm{m}_{-})\in\tilde{\Gamma}_{-}$).

\begin{figure}[ht]
\begin{center}
\begin{tikzpicture}[scale=1.2]

\draw [line width=0.04cm](-2.5,-2.0) --(3.5,-2.0);
\draw [line width=0.04cm](-2.5,0.8)--(3.5,0.8);
\draw [line width=0.04cm][red][dashed](-2.5,-0.5)--(3.5,-0.5);

\draw [thin](3.5,-2.0) --(3.5,0.8);
\draw [line width=0.02cm](-2.5,-2.0)--(-2.5,0.8);

\draw [line width=0.02cm](-2.5,-0.5)to[out=10, in=-120](-0.5,0.8);
\draw [line width=0.02cm][blue](-2.5,0.8)to[out=-20, in=120](0.9,-2.0);
\draw [line width=0.02cm](-2.5,-0.5)to[out=-20, in=120](-0.5,-2.0);
\draw [line width=0.02cm][blue](-2.5,-2.0)to[out=20, in=-120](0.7,0.8);

\draw [line width=0.02cm][red](-0.5,0.8)to[out=-20, in=130](2.5,-2.0);
\draw [line width=0.02cm](0.7,0.8)to[out=-20, in=120](2.9,-1.0);

\draw [line width=0.02cm][red](-0.5,-2.0)to[out=20, in=-120](2.1,0.8);
\draw [line width=0.02cm](0.9,-2.0) to[out=20, in=-120](2.9,0);

\node at (-2.0, 0.1) {$\tilde{\mathcal{N}}^{\rm{I}}_{+}$};
\node at (-1.1, 0.5) {$\tilde{\mathcal{N}}^{\rm{II}}_{+}$};
\node at (-1.0, -0.3) {$\tilde{\mathcal{N}}^{\rm{III}}_{+}$};
\node at (-0.3, 0.2) {$\tilde{\mathcal{N}}^{\rm{IV}}_{+}$};

\node at (-2.0, -1.2) {$\tilde{\mathcal{N}}^{\rm{I}}_{-}$};
\node at (-1.2, -1.7) {$\tilde{\mathcal{N}}^{\rm{II}}_{-}$};
\node at (-1.0, -0.9) {$\tilde{\mathcal{N}}^{\rm{III}}_{-}$};
\node at (-0.2, -1.3) {$\tilde{\mathcal{N}}^{\rm{IV}}_{-}$};

\node at (-2.5,0.8) {$\bullet$};
\node at (-2.5,-0.5) {$\bullet$};
\node at (-2.5,-2.0) {$\bullet$};
\node at (-0.5,0.8) {$\bullet$};
\node at (-1.05,0.105) {$\bullet$};
\node at (-0.22,-0.506) {$\bullet$};
\node at (-1.06,-1.29) {$\bullet$};
\node at (-0.5,-2.00) {$\bullet$};

\node at (3.9,0.8) {$\tilde{\Gamma}_{+}$};
\node at (3.9,-2.0) {$\tilde{\Gamma}_{-}$};
\node at (3.9,-0.5) {$\tilde{\Gamma}_{\textrm{cd}}$};
\node at (-2.8, 0.2) {$\tilde{\Gamma}^{+}_{\textrm{in}}$};
\node at (-2.8, -1.3) {$\tilde{\Gamma}^{-}_{\textrm{in}}$};
\node at (-2.8,-0.5) {$\tilde{O}$};
\node at (-2.8,0.8) {$\tilde{P}_+$};
\node at (-2.8,-2.0) {$\tilde{P}_-$};
\end{tikzpicture}
\end{center}
\caption{Linearized Problem $(\mathbf{IBVP})^{*}$ for $(\delta \omega^{*}_{\pm},\delta p^{*}_{\pm})$}\label{fig3.4}
\end{figure}
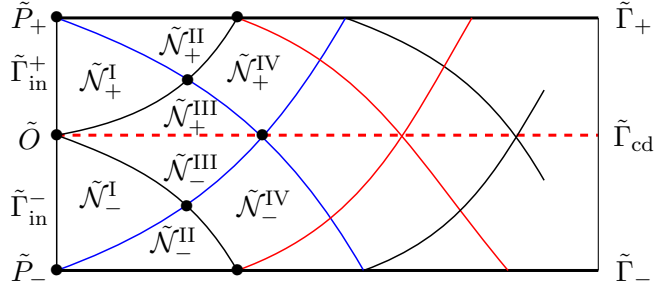

We first consider the problem $(\mathbf{IBVP})^{*}$ for $(\delta \omega^{*}_{\pm},\delta p^{*}_{\pm})$ in $\tilde{\mathcal{N}}^{\rm{I}}_{\pm}$. 
\begin{eqnarray}\label{eq:IBVP-I}
(\mathbf{IBVP})^{*}_{\rm{I}}\
\begin{cases}
\partial_{\xi}\delta\omega^{*}_{+}+\lambda^{\pm}_{+}\partial_{\eta}\delta\omega^{*}_{+}\pm\Lambda_{+}\big(\partial_{\xi}\delta p^{*}_{+}+\lambda^{\pm}_{+}\partial_{\eta}\delta p^{*}_{+}\big)\\
\qquad\qquad\qquad\qquad\qquad=a^{\pm}_{+}(\underline{\mathcal{U}}_{+})\delta Y^{*}_{+}+f^{\pm}_{+}(\delta\mathcal{V}_{+}, \delta Y_{+}), &\quad  \mbox{in} \quad \tilde{\mathcal{N}}^{\rm{I}}_{+}, \\
\partial_{\xi}\delta\omega^{*}_{-}+\lambda^{\pm}_{-}\partial_{\eta}\delta\omega^{*}_{-}\pm\Lambda_{-}\big(\partial_{\xi}\delta p^{*}_{-}+\lambda^{\pm}_{-}\partial_{\eta}\delta p^{*}_{-}\big)\\
\qquad\qquad\qquad\qquad\qquad=a^{\pm}_{-}(\underline{\mathcal{U}}_{-})\delta Y^{*}_{-}+f^{\pm}_{-}(\delta\mathcal{V}_{-}, \delta Y_{-}), &\quad  \mbox{in} \quad \tilde{\mathcal{N}}^{\rm{I}}_{-}, \\
(\delta\omega^{*}_{+},\delta p^{*}_{+})=(\delta\omega^{+}_{\rm in}, \delta p^{+}_{\rm in}), &\quad  \mbox{on} \quad \tilde{\Gamma}^{+}_{\rm in},\\
(\delta\omega^{*}_{-}, \delta p^{*}_{-})=(\delta\omega^{-}_{\rm in}, \delta p^{-}_{\rm in}), &\quad  \mbox{on} \quad \tilde{\Gamma}^{-}_{\rm in}.
\end{cases}
\end{eqnarray}

We have the following proposition for $(\mathbf{IBVP})^{*}_{\rm{I}}$.
\begin{proposition}\label{prop:3.2}
For $\sigma>0$ sufficiently small and given $(\delta\mathcal{U}_{+},\delta\mathcal{U}_{-})\in\mathcal{X}_{\sigma}$,
there exists a unique $C^{1,\alpha}$-solution $(\delta \omega^{*}_{\pm}, \delta p^{*}_{\pm})$ of the problem  $(\mathbf{IBVP})^{*}_{\rm{I}}$ with the estimate
\begin{align}\label{eq:prop-3.2-1}
\begin{split}
\|(\delta \omega^{*}_{k},\delta p^{*}_{k})\|_{1,\alpha; \tilde{\mathcal{N}}^{\rm{I}}_{k}}\leq  \tilde{C}_{31}\Big(\|(\delta \omega^{k}_{\rm in},\delta p^{k}_{\rm in})\|_{1,\alpha; \tilde{\Gamma}^{k}_{\rm in}}+\|\delta Y^{*}_{k}\|_{1,\alpha; \tilde{\mathcal{N}}^{\rm{I}}_{k}}+\sum_{j=\pm}\|f^{j}_{k}\|_{1,\alpha; \tilde{\mathcal{N}}^{\rm{I}}_{k}}\Big),
\end{split}
\end{align}
where $k=``\pm"$, and the constant $\tilde{C}_{31}>0$ depends only on $\underline{\mathcal{U}}_{\pm}$ and $\alpha$, $L$.
\end{proposition}

\begin{proof}
For its existence and uniqueness, we can follow the argument done in \cite{li-yu} by employing the \emph{Picard iteration scheme} to establish it. So we omit it for the shortness.
The remaining task is to derive the \emph{a priori} estimates for $(\delta \omega^{*}_{\pm},\delta p^{*}_{\pm})$. Without loss of the generality, we only consider the estimates for $(\delta \omega^{*}_{+},\delta p^{*}_{+})$ in $\tilde{\mathcal{N}}^{\rm{I}}_{+}$, since $(\delta \omega^{*}_{-},\delta p^{*}_{-})$ in $\tilde{\mathcal{N}}^{\rm{I}}_{-}$ can be dealt with in the same way.
We divide the proof into three steps.

\emph{Step1.\ $C^0$-estimates for $(\delta \omega^{*}_{+},\delta p^{*}_{+})$ in $\tilde{\mathcal{N}}^{\rm{I}}_{+}$.} Let us introduce
\begin{eqnarray}\label{eq:prop-3.2-2}
z^{+,*}_{+}=\delta \omega^{*}_{+}+\Lambda_{+}\delta p^{*}_{+}, \quad z^{-,*}_{+}=\delta \omega^{*}_{+}-\Lambda_{+}\delta p^{*}_{+}.
\end{eqnarray}
Then $(\mathbf{IBVP})^{*}_{\rm{I}}$ 
becomes the following initial value problem for $(z^{+,*}_{+},z^{-,*}_{+})$:
\begin{align}\label{eq:prop-3.2-3}
\begin{cases}
\partial_{\xi}z^{+,*}_{+}+\lambda^{+}_{+}\partial_{\eta}z^{+,*}_{+}-\frac{\p^{+}_{+}\Lambda_{+}}{2\Lambda_{+}}(z^{+,*}_{+}-z^{-,*}_{+})
=a^{+}_{+}(\underline{\mathcal{U}}_{+})\delta Y^{*}_{+}+f^{+}_{+}(\delta\mathcal{V}_{+}, \delta Y_{+}),&\quad  \mbox{in} \quad \tilde{\mathcal{N}}^{\rm{I}}_{+}, \\
\partial_{\xi}z^{-,*}_{+}+\lambda^{-}_{+}\partial_{\eta}z^{-,*}_{+}+\frac{\p^{-}_{+}\Lambda_{+}}{2\Lambda_{+}}(z^{+,*}_{+}-z^{-,*}_{+})=a^{-}_{+}(\underline{\mathcal{U}}_{+})\delta Y^{*}_{+}+f^{-}_{+}(\delta\mathcal{V}_{+}, \delta Y_{+}),&\quad  \mbox{in} \quad \tilde{\mathcal{N}}^{\rm{I}}_{+}, \\
(z^{+,*}_{+}, z^{-,*}_{+})=(z^{+}_{+,\rm in},z^{-}_{+,\rm in}),  &\quad  \mbox{on} \quad \tilde{\Gamma}^{+}_{\rm in},
\end{cases}
\end{align}
where $\p^{\pm}_{+}\triangleq\p_{\xi}+\lambda^{\pm}_{+}\p_{\eta}$,
$z^{\pm}_{+,\rm in}\triangleq\delta \omega^{+}_{\rm in}\pm\Lambda_{+, \rm in}\delta p^{+}_{\rm in}$, and $\Lambda_{+, \rm in}\triangleq\frac{\sqrt{u^2_{+,\rm in}+v^2_{+,\rm in}-c^2_{+,\rm in}}}{\rho_{+,\rm in}c_{+,\rm in}u^2_{+,\rm in}}$.

\begin{figure}[ht]
\begin{center}
\begin{tikzpicture}[scale=1.1]


\draw [line width=0.02cm](-2.5,-2.0)--(-2.5,0.8);
\draw [line width=0.04cm](-2.5,-2.0)to[out=20, in=-120](0.2,-0.1);
\draw [line width=0.02cm][blue](-2.5,-1.3)to[out=10, in=-120](-0.5,-0.4);
\draw [line width=0.02cm][red](-2.5,0.3)to[out=-5, in=130](-0.5,-0.4);
\draw [line width=0.04cm](-2.5,0.8)to[out=-10, in=140](0.2,-0.1);

\node at (-2.5,0.8) {$\bullet$};
\node at (-2.5,-2.0) {$\bullet$};
\node at (-2.5,-1.3) {$\bullet$};
\node at (-2.5,0.3) {$\bullet$};
\node at (0.2,-0.1) {$\bullet$};
\node at (-0.5,-0.4) {$\bullet$};

\node at (-1.8, -0.5) {$\tilde{\mathcal{N}}^{\rm{I}}_{+}$};
\node at (-2.8, -0.6) {$\tilde{\Gamma}^{+}_{\textrm{in}}$};
\node at (-2.8,0.8) {$\tilde{P}_+$};
\node at (-2.8,-2.0) {$\tilde{O}$};
\end{tikzpicture}
\end{center}
\caption{Estimates for $(\delta \omega^{*}_{+},\delta p^{*}_{+})$ in $\tilde{\mathcal{N}}^{\rm{I}}_{+}$}\label{fig3.5}
\end{figure}
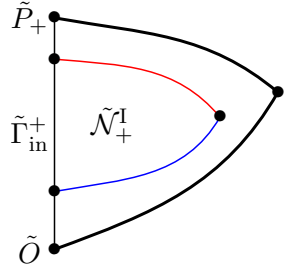

As shown in Fig. \ref{fig3.5}, for any point $(\bar{\xi},\bar{\eta})\in \tilde{\mathcal{N}}_{+}\cap\tilde{\mathcal{N}}^{\rm{I}}_{+}$, we define the characteristic curves for \eqref{eq:prop-3.2-3} passing through $(\bar{\xi},\bar{\eta})$ as
\begin{eqnarray}\label{eq:prop-3.2-4}
\begin{cases}
\frac{\mathrm{d} \Upsilon_{+}^{\pm}(\tau;\bar{\xi},\bar{\eta})}{\mathrm{d} \tau}=\lambda_{+}^{\pm}(\tau,\Upsilon_{+}^{\pm}(\tau;\bar{\xi},\bar{\eta})),\\
\Upsilon_{+}^{\pm}(\bar{\xi};\bar{\xi},\bar{\eta})=\bar{\eta}.
\end{cases}
\end{eqnarray}

Let $\gamma_{+}^{\pm}(\bar{\xi},\bar{\eta})\triangleq\Upsilon_{+}^{\pm}(0;\bar{\xi},\bar{\eta})$ be the $\eta$-component of the intersection point between the characteristics $\eta=\Upsilon_{+}^{\pm}(\tau;\bar{\xi},\bar{\eta})$ and the entrance $\tilde{\Gamma}^{+}_{\rm in}$.
Then, along $\eta=\Upsilon_{+}^{\pm}(\tau;\bar{\xi},\bar{\eta})$, we have
\begin{eqnarray}\label{eq:prop-3.2-5}
\begin{cases}
z^{+,*}_{+}(\bar{\xi},\bar{\eta})
=z^{+}_{+,\rm in}(\gamma^{+}_{+}(\bar{\xi},\bar{\eta}))
+\int_0^{\bar{\xi}}\frac{\p_{+}\Lambda_{+}}{2\Lambda_{+}}\big(z^{+,*}_{+}-z^{-,*}_{+}\big)(\tau,\Upsilon^{+}_{+}(\tau;\bar{\xi},\bar{\eta}))\mathrm{d} \tau\\
\qquad\qquad\qquad\qquad\qquad\qquad+\int_0^{\bar{\xi}}\big(a^{+}_{+}(\underline{\mathcal{U}}_{+})\delta Y^{*}_{+}
+f^{+}_{+}(\delta\mathcal{V}_{+}, \delta Y_{+})\big)(\tau,\Upsilon^{+}_{+}(\tau;\bar{\xi},\bar{\eta}))\mathrm{d} \tau,\\
z^{-,*}_{+}(\bar{\xi},\bar{\eta})
=z^{-}_{+,\rm in}(\gamma_{+}^{-}(\bar{\xi},\bar{\eta}))+\int_0^{\bar{\xi}}\frac{\p_{-}\Lambda_{+}}{2\Lambda_{+}}\big(z^{-,*}_{+}-z^{+,*}_{+}\big)
(\tau,\Upsilon^{-}_{+}(\tau;\bar{\xi},\bar{\eta}))\mathrm{d} \tau\\
\qquad\qquad\qquad\qquad\qquad\qquad+\int^{\bar{\xi}}_{0}\big(a^{-}_{+}(\underline{\mathcal{U}}_{+})\delta Y^{*}_{+}+f^{-}_{+}(\delta\mathcal{V}_{+}, \delta Y_{+})\big)(\tau,\Upsilon^{-}_{+}(\tau;\bar{\xi},\bar{\eta}))\mathrm{d} \tau.
\end{cases}
\end{eqnarray}

Let us introduce
\begin{eqnarray}\label{eq:prop-3.2-6}
R_{+}^{\rm{I}}(\kappa)=\big\{(\xi,\eta)| 0\leq\xi\leq\kappa,\ \Upsilon^{+,0}_{+}(\xi;\bar{\xi}^{2}_{+},\rm{m}_{+})\leq\eta\leq\Upsilon_{+}^{-,1}(\xi;\bar{\xi}^{2}_{\rm{cd}},0)\big\},
\ (0\leq\kappa\leq \bar{\xi}_{+}^{1}).
\end{eqnarray}

We define the norm
\begin{eqnarray}\label{eq:prop-3.2-7}
z_{\rm{I}^{+}}^{\pm,*}(\kappa)\triangleq\sup\limits_{(\xi,\eta)\in R_{+}^{\rm{I}}(\kappa)}|z_{+}^{\pm,*}(\xi,\eta)|.
\end{eqnarray}

Obviously, we have $\|z_{+}^{\pm,*}\|_{0,0; \tilde{\mathcal{N}}^{\rm{I}}_{+}}=z_{\rm{I}^{+}}^{\pm,*}(\bar{\xi}^{1}_{+}).$
So by \eqref{eq:prop-3.2-5} and \eqref{eq:prop-3.2-7}, we obtain   
\begin{align}\label{eq:prop-3.2-8}
\begin{split}
z_{\rm{I}^{+}}^{+,*}(\kappa)&\leq \|z_{+,\rm in}^{+}\|_{0,0; \tilde{\Gamma}^{+}_{\rm in}}+L\cdot\Big(|a^{+}_{+}(\underline{\mathcal{U}}_{+})|\cdot\|\delta Y^{*}_{+}\|_{0,0; \tilde{\mathcal{N}}^{I}_{+}}+\|f^{+}_{+}\|_{0,0; \tilde{\mathcal{N}}^{\rm{I}}_{+}}\Big)\\
&\quad\ +\Big\|\frac{\p_{+}\Lambda_{+}}{2\Lambda_{+}}\Big\|_{0,0; \tilde{\mathcal{N}}^{\rm{I}}_{+}}\cdot\int_0^{\kappa}\sum_{k=\pm}z_{\rm{I}^{+}}^{\pm,*}(\varsigma)\dd\varsigma\\
&\leq  C^{+}_{\rm{I}^{+}}\Big(\|z_{+,\rm in}^{+}\|_{0,0; \tilde{\Gamma}^{+}_{\rm in}}+\|\delta Y^{*}_{+}\|_{0,0; \tilde{\mathcal{N}}^{\rm{I}}_{+}}+\|f^{+}_{+}\|_{0,0; \tilde{\mathcal{N}}^{\rm{I}}_{+}}+\int_0^{\kappa}\sum_{k=\pm}z_{\rm{I}^{+}}^{\pm,*}(\varsigma)\dd\varsigma\Big),
\end{split}
\end{align}
and
\begin{align}\label{eq:prop-3.2-9}
\begin{split}
z_{\rm{I}^{+}}^{-,*}(\kappa)&\leq \|z_{+,\rm in}^{-}\|_{0,0; \tilde{\Gamma}^{+}_{\rm in}}+L\cdot\Big(|a^{-}_{+}(\underline{\mathcal{U}}_{+})|\cdot\|\delta Y^{*}_{+}\|_{0,0; \tilde{\mathcal{N}}^{\rm{I}}_{+}}+\|f^{-}_{+}\|_{0,0; \tilde{\mathcal{N}}^{\rm{I}}_{+}}\Big)\\
&\quad\ +\Big\|\frac{\p_{-}\Lambda_{+}}{2\Lambda_{+}}\Big\|_{0,0; \tilde{\mathcal{N}}^{\rm{I}}_{+}}\cdot\int_0^{\kappa}\sum_{k=\pm}z_{\rm{I}^{+}}^{\pm,*}(\varsigma)\dd\varsigma\\
&\leq  C^{-}_{\rm{I}^{+}}\Big(\|z_{+,\rm in}^{+}\|_{0,0; \tilde{\Gamma}^{+}_{\rm in}}+\|\delta Y^{*}_{+}\|_{0,0; \tilde{\mathcal{N}}^{\rm{I}}_{+}}+\|f^{-}_{+}\|_{0,0; \tilde{\mathcal{N}}^{\rm{I}}_{+}}+\int_0^{\kappa}\sum_{k=\pm}z_{\rm{I}^{+}}^{\pm,*}(\varsigma)\dd\varsigma\Big),
\end{split}
\end{align}
where the constants $C^{\pm}_{\rm{I}^{+}}>0$ depend only on $\underline{\mathcal{U}}_{+}$, $\alpha$ and $L$, and constant $\sigma>0$ is sufficiently small.
Summing \eqref{eq:prop-3.2-8} and \eqref{eq:prop-3.2-9}, and applying the Gronwall's inequality, we conclude
\begin{eqnarray*}
\begin{split}
&\sum_{k=\pm}z_{\rm{I}^{+}}^{k,*}(\kappa)\leq C_{\rm{I}^{+}}\Big( \sum_{k=\pm}\big(\|z_{+,\rm in}^{k}\|_{0,0; \tilde{\Gamma}^{+}_{\rm in}}+\|f^{k}_{+}\|_{0,0; \tilde{\mathcal{N}}^{\rm{I}}_{+}}\big)+\|\delta Y^{*}_{+}\|_{0,0; \tilde{\mathcal{N}}^{\rm{I}}_{+}}\Big).
\end{split}
\end{eqnarray*}
It implies by taking $\kappa=\bar{\xi}^{1}_{+}$ that
\begin{eqnarray}\label{eq:prop-3.2-10}
\begin{split}
\sum_{k=\pm}\|z_{+}^{k,*}\|_{0,0; \tilde{\mathcal{N}}^{\rm{I}}_{+}}\leq C_{\rm{I}^{+}}
\Big(\sum_{k=\pm}\big(\|z_{+,\rm in}^{k}\|_{0,0; \tilde{\Gamma}^{+}_{\rm in}}+\|f^{k}_{+}\|_{0,0; \tilde{\mathcal{N}}^{\rm{I}}_{+}}\big)
+\|\delta Y^{*}_{+}\|_{0,0; \tilde{\mathcal{N}}^{\rm{I}}_{+}}\Big),
\end{split}
\end{eqnarray}
where $C_{\rm{I}^{+}}=\max\{C^{+}_{\rm{I}^{+}},\,C^{-}_{\rm{I}^{+}}\}\Big(1+L\max\{C^{+}_{\rm{I}^{+}},\ C^{-}_{\rm{I}^{+}}\}\exp\big(\max\{C^{+}_{\rm{I}^{+}},\,C^{-}_{\rm{I}^{+}}\} L\big)\Big)$.

Finally, by using the relations
\begin{eqnarray}\label{eq:prop-3.2-2a}
\delta \omega^{*}_{+}=\frac{z^{-,*}_{+}+z^{+,*}_{+}}{2}, \quad \delta p^{*}_{+}=\frac{z^{+,*}_{+}-z^{-,*}_{+}}{2\Lambda_{+}},
\end{eqnarray}
together with {\color{black}the} estimate \eqref{eq:prop-3.2-10}, we can deduce that for $\sigma>0$ sufficiently small, there exists a constant $\tilde{C}_{\rm{I}^{+}}>0$ depending only on $\underline{\mathcal{U}}_{+}$, $\alpha$ and $L$ so that
\begin{eqnarray}\label{eq:prop-3.2-11}
\begin{split}
&\|(\delta \omega^{*}_{+},\delta p^{*}_{+})\|_{0,0; \tilde{\mathcal{N}}^{\rm{I}}_{+}}\leq \tilde{C}_{\rm{I}^{+}}\Big(\|(\delta \omega^{+}_{\rm in},\delta p^{+}_{\rm in})\|_{0,0; \tilde{\Gamma}^{+}_{\rm in}}+\|\delta Y^{*}_{+}\|_{0,0; \tilde{\mathcal{N}}^{\rm{I}}_{+}}+\sum_{k=\pm}\|f^{k}_{+}\|_{0,0; \tilde{\mathcal{N}}^{\rm{I}}_{+}}\Big).
\end{split}
\end{eqnarray}

\emph{Step 2.\ $C^0$-estimates for $(\nabla\delta \omega^{*}_{+}, \nabla\delta p^{*}_{+})$ in $\tilde{\mathcal{N}}^{\rm{I}}_{+}$.}
Set
\begin{eqnarray}\label{eq:prop-3.2-12}
\mathcal {Z}^{+,*}_{+}\triangleq\p_{\eta} \delta\omega^{*}_{+}+\Lambda_{+}\p_{\eta} \delta p^{*}_{+},\quad \mathcal {Z}^{-,*}_{+}\triangleq\p_{\eta} \delta\omega^{*}_{+}-\Lambda_{+}\p_{\eta} \delta p^{*}_{+}.
\end{eqnarray}
We can derive the equations for $(\mathcal {Z}^{+,*}_{+}, \mathcal {Z}^{-,*}_{+})$ by taking derivatives on $\eqref{eq:IBVP-I}_1$ along the $\eta$-direction.
\begin{eqnarray}\label{eq:prop-3.2-13}
\begin{cases}
\p_{\xi}\mathcal {Z}_{+}^{+,*}+\lambda^{+}_{+}\p_{\eta}\mathcal {Z}_{+}^{+,*}=\frac{\p_{+}\Lambda_{+}-2\Lambda_{+}\p_{\eta}\lambda^{+}_{+}}
{2\Lambda_{+}}\mathcal {Z}_{+}^{+,*}-\frac{\p_{-}\Lambda_{+}}{2\Lambda_{+}}\mathcal {Z}_{+}^{-,*}\\
\qquad\qquad\qquad+a^{+}_{+}(\underline{\mathcal{U}}_{+})\p_{\eta}\delta Y^{*}_{+}-\frac{[a^{+}_{+}(\underline{\mathcal{U}}_{+})-a^{-}_{+}(\underline{\mathcal{U}}_{+})]\p_{\eta} \Lambda_{+}}{2\Lambda_{+}}\delta Y^{*}_{+}\\
\qquad\qquad\qquad+\p_{\eta}f^{+}_{+}(\delta\mathcal{V}_{+}, \delta Y_{+})-\frac{\p_{\eta} \Lambda_{+}}{2\Lambda_{+}}\big[f^{+}_{+}(\delta\mathcal{V}_{+}, \delta Y_{+})-f^{-}_{+}(\delta\mathcal{V}_{+}, \delta Y_{+})\big],\\
\p_{\xi}\mathcal {Z}_{+}^{-,*}+\lambda^{-}_{+}\p_{\eta}\mathcal {Z}_{+}^{-,*}=\frac{\p_{-}\Lambda_{+}-2\Lambda_{+}\p_{\eta}\lambda^{-}_{+}}{2\Lambda_{+}}\mathcal {Z}_{+}^{-,*}-\frac{\p_{+}\Lambda_{+}}{2\Lambda_{+}}\mathcal {Z}_{+}^{+,*}\\
\qquad\qquad\qquad+a^{-}_{+}(\underline{\mathcal{U}}_{+})\p_{\eta}\delta Y^{*}_{+}-\frac{[a^{-}_{+}(\underline{\mathcal{U}}_{+})-a^{+}_{+}(\underline{\mathcal{U}}_{+})]\p_{\eta} \Lambda_{+}}{2\Lambda_{+}}\delta Y^{*}_{+}\\
\qquad\qquad\qquad+\p_{\eta}f^{-}_{+}(\delta\mathcal{V}_{+}, \delta Y_{+})-\frac{\p_{\eta} \Lambda_{+}}{2\Lambda_{+}}
\big[f^{-}_{+}(\delta\mathcal{V}_{+}, \delta Y_{+})
-f^{+}_{+}(\delta\mathcal{V}_{+}, \delta Y_{+})\big],
\end{cases}
\end{eqnarray}
with the initial conditions
\begin{eqnarray}\label{eq:prop-3.2-13a}
(\mathcal {Z}_{+}^{+,*}, \mathcal{Z}_{+}^{-,*})=(\mathcal {Z}_{+,\rm in}^{+},\mathcal {Z}_{+,\rm in}^{-}), \quad  \mbox{on} \quad \tilde{\Gamma}^{+}_{\rm in},
\end{eqnarray}
where $\mathcal {Z}_{+,\rm in}^{\pm}\triangleq(\delta\omega^{+}_{\rm in})'\pm\Lambda_{+,\rm in}(\delta p^{+}_{\rm in})'$. 
Then, integrating \eqref{eq:prop-3.2-13} along the characteristics defined by \eqref{eq:prop-3.2-4},  with the intial conditions \eqref{eq:prop-3.2-13a}, we get
\begin{align}\label{eq:prop-3.2-14}
\begin{cases}
\mathcal {Z}_{+}^{+,*}(\bar{\xi},\bar{\eta})
=\mathcal {Z}_{+,\rm in}^{+}(\gamma^{+}_{+}(\bar{\xi},\bar{\eta}))+\int_0^{\bar{\xi}}\big(\frac{\p_{+}\Lambda_{+}-2\Lambda_{+}\p_{\eta}\lambda^{+}_{+}}{2\Lambda_{+}}\mathcal {Z}_{+}^{+,*}-\frac{\p_{-}\Lambda_{+}}{2\Lambda_{+}}\mathcal {Z}_{+}^{-,*}\big)(\tau,\Upsilon^{+}_{+}(\tau;\bar{\xi},\bar{\eta}))\dd \tau\\
\qquad\qquad\qquad+\int^{\bar{\xi}}_{0}\big(a^{+}_{+}(\underline{\mathcal{U}}_{+})\p_{\eta}\delta Y^{*}_{+}-\frac{[a^{+}_{+}(\underline{\mathcal{U}}_{+})-a^{-}_{+}(\underline{\mathcal{U}}_{+})]\p_{\eta} \Lambda_{+}}{2\Lambda_{+}}\delta Y^{*}_{+}\big)(\tau,\Upsilon^{+}_{+}(\tau;\bar{\xi},\bar{\eta}))\dd \tau\\
\qquad\qquad\qquad+\int^{\bar{\xi}}_{0}\big(\p_{\eta}f^{+}_{+}(\delta\mathcal{V}_{+}, \delta Y_{+})-\frac{\p_{\eta} \Lambda_{+}}{2\Lambda_{+}}\big[f^{+}_{+}(\delta\mathcal{V}_{+}, \delta Y_{+})-f^{-}_{+}(\delta\mathcal{V}_{+}, \delta Y_{+})\big]\big)(\tau,\Upsilon^{+}_{+}(\tau;\bar{\xi},\bar{\eta}))\dd \tau,\\
\mathcal {Z}_{+}^{-,*}(\bar{\xi},\bar{\eta})
=\mathcal {Z}_{+,\rm in}^{-}(\gamma^{-}_{+}(\bar{\xi},\bar{\eta}))+\int_0^{\bar{\xi}}\big(\frac{\p_{-}\Lambda_{+}-2\Lambda_{+}\p_{\eta}\lambda^{-}_{+}}{2\Lambda_{+}}\mathcal {Z}_{+}^{-,*}-\frac{\p_{+}\Lambda_{+}}{2\Lambda_{+}}\mathcal {Z}_{+}^{+,*}\big)(\tau,\Upsilon^{+}_{+}(\tau;\bar{\xi},\bar{\eta}))\dd \tau\\
\qquad\qquad\qquad +\int_0^{\bar{\xi}}\big(a^{-}_{+}(\underline{\mathcal{U}}_{+})\p_{\eta}\delta Y^{*}_{+}-\frac{[a^{-}_{+}(\underline{\mathcal{U}}_{+})-a^{+}_{+}(\underline{\mathcal{U}}_{+})]\p_{\eta} \Lambda_{+}}{2\Lambda_{+}}\delta Y^{*}_{+} \big)(\tau,\Upsilon^{-}_{+}(\tau;\bar{\xi},\bar{\eta}))\dd \tau\\
\qquad\qquad\qquad +\int_0^{\bar{\xi}}\big(\p_{\eta}f^{-}_{+}(\delta\mathcal{V}_{+}, \delta Y_{+})-\frac{\p_{\eta} \Lambda_{+}}{2\Lambda_{+}}\big[f^{-}_{+}(\delta\mathcal{V}_{+}, \delta Y_{+})
-f^{+}_{+}(\delta\mathcal{V}_{+}, \delta Y_{+})\big]\big)(\tau,\Upsilon^{-}_{+}(\tau;\bar{\xi},\bar{\eta}))\dd \tau.
\end{cases}
\end{align}

Now we are going to estimate $\mathcal {Z}_{+}^{+,*}$ and $\mathcal {Z}_{+}^{+,*}$. Define 
\begin{eqnarray}\label{eq:prop-3.2-15}
\mathcal {Z}^{\pm,*}_{\rm{I}^{+}}(\kappa)\triangleq\sup\limits_{(\xi,\eta)\in R_{+}^{\rm{I}}(\kappa)}|\mathcal {Z}^{\pm,*}_{+}(\xi,\eta)|.
\end{eqnarray}
Obviously, we have
$\|\mathcal {Z}^{\pm,*}_{+}\|_{0,0; \tilde{\mathcal{N}}^{\rm{I}}_{+}}=\mathcal {Z}^{\pm,*}_{\rm{I}^{+}}(\bar{\xi}_+^1)$.
Then, mimic the argument in \emph{Step 1}, for $\sigma>0$ sufficiently small, we can deduce from \eqref{eq:prop-3.2-14} that
\begin{align*}
\begin{split}
\sum_{k=\pm}\mathcal {Z}_{\rm{I}^{+}}^{k,*}(\kappa)\leq \mathcal{O}(1)\Big(\sum_{k=\pm}\big(\|\mathcal {Z}_{+,\rm in}^{k}\|_{0,0; \tilde{\Gamma}^{+}_{\rm in}}+\|f^{k}_{+}\|_{1,0; \tilde{\mathcal{N}}^{\rm{I}}_+}\big)
+\|\delta Y^{*}_{+}\|_{1,0; \tilde{\mathcal{N}}^{\rm{I}}_{+}}+\int_0^{\kappa}\sum_{k=\pm}\mathcal {Z}_{\rm{I}^{+}}^{k,*}(\varsigma)\dd\varsigma\Big),
\end{split}
\end{align*}
where constant $\mathcal{O}(1)$ depends only on $\underline{\mathcal{U}}_{+}$ and $L$.

Hence it follows from the Gronwall's inequality that
\begin{eqnarray}\label{eq:prop-3.2-16}
\begin{split}
\sum_{k=\pm}\|\mathcal {Z}_{+}^{k,*}\|_{0,0; \tilde{\mathcal{N}}^{\rm{I}}_{+}}
\leq \mathcal{O}(1)\Big(\sum_{k=\pm}\big(\|\mathcal {Z}_{+,\rm in}^{k}\|_{0,0; \tilde{\Gamma}^{+}_{\rm in}}+\|f^{k}_{+}\|_{1,0; \tilde{\mathcal{N}}^{\rm{I}}_+}\big)
+\|\delta Y^{*}_{+}\|_{1,0; \tilde{\mathcal{N}}^{\rm{I}}_{+}}\Big).
\end{split}
\end{eqnarray}

On the other hand, by $\eqref{eq:IBVP-I}_{1}$ and \eqref{eq:prop-3.2-12}, we can get
\begin{eqnarray}\label{eq:prop-3.2-17}
\begin{cases}
\p_{\xi}\delta \omega^{*}_{+}=-\frac{\lambda_{+}^{+}}{2}\mathcal {Z}_{+}^{+,*}-\frac{\lambda_{+}^{-}}{2}\mathcal {Z}_{+}^{-,*}+\frac{a^{+}_{+}(\underline{\mathcal{U}}_{+})+a_{+}^{-}(\underline{\mathcal{U}}_{+})}{2}\delta Y^{*}_{+}+\frac{f^{+}_{+}(\delta\mathcal{U}_{+}, \underline{\mathcal{U}}_{+})+f^{-}_{+}(\delta\mathcal{U}_{+}, \underline{\mathcal{U}}_{+})}{2},\\
\p_{\xi} \delta p^{*}_{+}=-\frac{\lambda_{+}^{+}}{2\Lambda_{+}}\mathcal {Z}_{+}^{+,*}+\frac{\lambda_{+}^{-}}{2\Lambda_{+}}
\mathcal {Z}_{+}^{-,*}+\frac{a^{+}_{+}(\underline{\mathcal{U}}_{+})-a^{-}_{+}(\underline{\mathcal{U}}_{+})}{2\Lambda_{+}}\delta Y^{*}_{+}+\frac{f^{+}_{+}(\delta\mathcal{U}_{+}, \underline{\mathcal{U}}_{+})-f^{-}_{+}(\delta\mathcal{U}_{+}, \underline{\mathcal{U}}_{+})}{2\Lambda_{+}},
\end{cases}
\end{eqnarray}
and
\begin{eqnarray}\label{eq:prop-3.2-18}
\p_{\eta} \delta \omega^{*}_{+}=\frac{\mathcal {Z}_{+}^{+,*}+\mathcal {Z}_{+}^{-,*}}{2},\quad \p_{\eta} \delta p^{*}_{+}=\frac{\mathcal {Z}_{+}^{+,*}-\mathcal {Z}_{+}^{-,*}}{2\Lambda_{+}}.
\end{eqnarray}

Finally, combing \eqref{eq:prop-3.2-16}-\eqref{eq:prop-3.2-18} together, we arrive at
\begin{align}\label{eq:prop-3.2-20}
\begin{split}
&\|(\nabla\delta \omega^{*}_{+},\nabla\delta p^{*}_{+})\|_{0,0; \tilde{\mathcal{N}}^{\rm{I}}_{+}}\leq \tilde{C}^{*}_{\rm{I}^{+}}\Big(\|(\delta \omega^{+}_{\rm in},\delta p^{+}_{\rm in})\|_{1,0; \tilde{\Gamma}^{+}_{\rm in}}+\|\delta Y^{*}_{+}\|_{1,0; \tilde{\mathcal{N}}^{\rm{I}}_{+}}+\sum_{k=\pm}\|f^{k}_{+}\|_{1,0; \tilde{\mathcal{N}}^{\rm{I}}_{+}}\Big),
\end{split}
\end{align}
where the constant $\tilde{C}^{*}_{\rm{I}^{+}}>0$ depends only on $\underline{\mathcal{U}}_{+}$, $\alpha$ and $L$.

\emph{Step 3. $C^{\alpha}$-estimates for $(\nabla\delta \omega^{*}_{+},\nabla\delta p^{*}_{+})$ in $\tilde{\mathcal{N}}^{\rm{I}}_{+}$.}
Define
\begin{eqnarray*}
[\mathcal {Z}^{\pm,*}_{\rm{I}^{+}}](\kappa)\triangleq\sup\limits_{(\xi_1,\eta_1),(\xi_2,\eta_2)\in R_{+}^{\rm{I}}(\kappa)}\frac{|\mathcal {Z}^{\pm,*}_{+}(\xi_1,\eta_1)-\mathcal {Z}^{\pm,*}_{+}(\xi_2,\eta_2)|}{|(\xi_1,\eta_1)-(\xi_2,\eta_2)|^{\alpha}}.
\end{eqnarray*}
It is easy to see that
\begin{eqnarray}\label{eq:prop-3.2-22}
[\mathcal {Z}^{\pm,*}_{+}]_{0,\alpha,\tilde{\mathcal{N}}^{\rm{I}}_{+}}=[\mathcal {Z}^{\pm,*}_{\rm{I}^{+}}](\bar{\xi}^{1}_{+}).
\end{eqnarray}
Then, by the triangle inequality, we have
\begin{align}\label{eq:prop-3.2-21}
\begin{split}
[\mathcal {Z}^{\pm,*}_{\rm{I}^{+}}](\kappa)&\leq \sup\limits_{(\xi_1,\eta_1),(\xi_2,\eta_1)\in R_{+}^{\rm{I}}(\kappa)}\frac{|\mathcal {Z}^{\pm,*}_{+}(\xi_1,\eta_1)-\mathcal {Z}^{\pm,*}_{+}(\xi_2,\eta_1)|}{|\xi_1-\xi_2|^{\alpha}}\\
&\qquad +\sup\limits_{(\xi_2,\eta_1),(\xi_2,\eta_2)\in R_{+}^{\rm{I}}(\kappa)}\frac{|\mathcal {Z}^{\pm,*}_{+}(\xi_2,\eta_1)-\mathcal {Z}^{\pm,*}_{+}(\xi_2,\eta_2)|}{|\eta_1-\eta_2|^{\alpha}}\\
&\triangleq[\mathcal {Z}^{\pm,*}_{\rm{I}^{+}}]_{\xi}(\kappa)+[\mathcal {Z}^{\pm,*}_{\rm{I}^{+}}]_{\eta}(\kappa).
\end{split}
\end{align}

Therefore, in the following, we only need to estimate $[\mathcal {Z}^{\pm,*}_{\rm{I}^{+}}]_{\xi}(\kappa)$ and $[\mathcal {Z}^{\pm,*}_{\rm{I}^{+}}]_{\eta}(\kappa)$.
First, let us consider $[\mathcal {Z}^{\pm,*}_{\rm{I}^{+}}]_{\eta}(\kappa)$. Taking any two points $(\bar{\xi},\bar{\eta}_1),\ (\bar{\xi},\bar{\eta}_2)\in \tilde{\mathcal{N}}^{\rm{I}}_{+}$ with $\bar{\xi}$ being fixed and $\eta_1\neq \eta_2$, we apply \eqref{eq:prop-3.2-14} for $(\bar{\xi},\bar{\eta}_1)$ and $(\bar{\xi},\bar{\eta}_2)$ and subtract them to get
{
\begin{align}\label{eq:prop-3.2-23}
\begin{split}
&\big|\mathcal {Z}^{+,*}_{+}(\bar{\xi},\bar{\eta}_1)-\mathcal {Z}^{+,*}_{+}(\bar{\xi},\bar{\eta}_2)\big|\\
&\leq \big|\mathcal {Z}^{+}_{+,\rm in}(\gamma^{+}_{+}(\bar{\xi},\bar{\eta}_1))-\mathcal {Z}^{+}_{+,\rm in}(\gamma^{+}_{+}(\bar{\xi},\bar{\eta}_1))\big|\\
&+\int_0^{\bar{\xi}}\big|\big(\frac{\p_{+}\Lambda_{+}-2\Lambda_{+}\p_{\eta}\lambda^{+}_{+}}{2\Lambda_{+}}\mathcal {Z}_{+}^{+,*}\big)(\tau,\Upsilon^{+}_{+}(\tau;\bar{\xi},\bar{\eta}_1))-\big(\frac{\p_{+}\Lambda_{+}-2\Lambda_{+}\p_{\eta}\lambda^{+}_{+}}{2\Lambda_{+}}\mathcal {Z}_{+}^{+,*}\big)(\tau,\Upsilon^{+}_{+}(\tau;\bar{\xi},\bar{\eta}_2))\big|\dd\tau\\
&+\int_0^{\bar{\xi}}\big|\big(\frac{\p_{-}\Lambda_{+}}{2\Lambda_{+}}\mathcal {Z}_{+}^{-,*}\big)(\tau,\Upsilon^{+}_{+}(\tau;\bar{\xi},\bar{\eta}_1))-\big(\frac{\p_{-}\Lambda_{+}}{2\Lambda_{+}}\mathcal {Z}_{+}^{-,*}\big)(\tau,\Upsilon^{+}_{+}(\tau;\bar{\xi},\bar{\eta}_2))\big|\dd\tau\\
&+|a^{+}_{+}(\underline{\mathcal{U}}_{+})|\int_0^{\bar{\xi}}\big|\p_{\eta}\delta Y^{*}_{+}(\tau,\Upsilon^{+}_{+}(\tau;\bar{\xi},\bar{\eta}_1))
-\p_{\eta}\delta Y^{*}_{+}(\tau,\Upsilon^{+}_{+}(\tau;\bar{\xi},\bar{\eta}_2))\big|\dd \tau\\
&+\frac{|a^{+}_{+}(\underline{\mathcal{U}}_{+})|+|a^{-}_{+}(\underline{\mathcal{U}}_{+})|}{2}\int_0^{\bar{\xi}}
\big|\big(\frac{\p_{\eta}\Lambda_{+}}{\Lambda_{+}}Y_{+}^{*}\big)(\tau,\Upsilon^{+}_{+}(\tau;\bar{\xi},\bar{\eta}_1))
-\big(\frac{\p_{\eta}\Lambda_{+}}{\Lambda_{+}}Y_{+}^{*}\big)(\tau,\Upsilon^{+}_{+}(\tau;\bar{\xi},\bar{\eta}_2))\big|\dd\tau\\
&+\int_0^{\bar{\xi}}\big|\p_{\eta}f^{+}_{+}(\delta\mathcal{V}_{+}, \delta Y_{+})(\tau,\Upsilon^{+}_{+}(\tau;\bar{\xi},\bar{\eta}_1))
-\p_{\eta}f^{+}_{+}(\delta\mathcal{V}_{+}, \delta Y_{+})(\tau,\Upsilon^{+}_{+}(\tau;\bar{\xi},\bar{\eta}_2))\big|\dd \tau\\
&+\sum_{k=\pm}\int_0^{\bar{\xi}}\big|\big(\frac{\p_{\eta} \Lambda_{+}}{2\Lambda_{+}}f^{k}_{+}(\delta\mathcal{V}_{+}, \delta Y_{+})\big)(\tau,\Upsilon^{+}_{+}(\tau;\bar{\xi},\bar{\eta}_1))
-\big(\frac{\p_{\eta} \Lambda_{+}}{2\Lambda_{+}}f^{k}_{+}(\delta\mathcal{V}_{+}, \delta Y_{+})\big)(\tau,\Upsilon^{+}_{+}(\tau;\bar{\xi},\bar{\eta}_2))\big|\dd \tau\\
&\triangleq\sum_{k=1}^{7}\mathcal {J}^{+}_{{\rm I}^{+}_{k}}.
\end{split}
\end{align}
}
We now estimate the terms $\mathcal{J}^{+}_{\textrm{I}^{+}_{k}}\,( k=1, 2,\cdots,7)$ one by one. Denote by $\mathcal{O}(1)$ the constants depending only on $\underline{\mathcal{U}}_{+}$, $\alpha$ and $L$ in the following estimates. For $\mathcal {J}^{+}_{\rm{I}^{+}_{1}}$, by Lemma \ref{lem:A1}, we have
\begin{eqnarray*}
\begin{split}
\mathcal {J}^{+}_{\rm{I}^{+}_{1}}\leq [\mathcal {Z}^{+}_{+,\rm in}]_{0,\alpha; \tilde{\Gamma}^{+}_{\rm in}}|\gamma^{+}_{+}(\bar{\xi},\bar{\eta}_1)-\gamma^{+}_{+}(\bar{\xi},\bar{\eta}_2)|^{\alpha}
\leq \mathcal{O}(1)[\mathcal {Z}^{+}_{+,\rm in}]_{0,\alpha;\tilde{\Gamma}^{+}_{\rm in}}\cdot|\bar{\eta}_1-\bar{\eta}_2|^{\alpha}.
\end{split}
\end{eqnarray*}

For $\mathcal {J}^{+}_{\rm{I}^{+}_{2}}$, we choose $\sigma>0$ sufficiently small to have
\begin{eqnarray*}
\begin{split}
\mathcal {J}^{+}_{\rm{I}^{+}_{2}}&\leq\Big\|\frac{\p_{+}\Lambda_{+}-2\Lambda_{+}\p_{\eta}\lambda^{+}_{+}}{2\Lambda_{+}}\Big\|_{0,0; \tilde{\mathcal{N}}^{\rm{I}}_{+}}
\int_0^{\bar{\xi}}[\mathcal {Z}^{+,*}_{\rm{I}^{+}}]_{\eta}(\tau)|\Upsilon^{+}_{+}(\tau;\bar{\xi},\bar{\eta}_1)-\Upsilon^{+}_{+}(\tau;\bar{\xi},\bar{\eta}_2)|^{\alpha}\dd \tau\\
&\quad+\Big[\frac{\p_{+}\Lambda_{+}-2\Lambda_{+}\p_{\eta}\lambda^{+}_{+}}{2\Lambda_{+}}\Big]_{0,\alpha; \tilde{\mathcal{N}}^{\rm{I}}_{+}}\|\mathcal {Z}_{+}^{+,*}\|_{0,0; \tilde{\mathcal{N}}^{\rm{I}}_{+}}\int_0^{\bar{\xi}}|\Upsilon^{+}_{+}(\tau;\bar{\xi},\bar{\eta}_1)-\Upsilon^{+}_{+}(\tau;\bar{\xi},\bar{\eta}_2)|^{\alpha}\dd \tau\\
&\leq \mathcal {O}(1)\Big(\|\mathcal {Z}_{+}^{+,*}\|_{0,0; \tilde{\mathcal{N}}^{\rm{I}}_{+}}+\int_0^{\bar{\xi}}[\mathcal {Z}^{+,*}_{\rm{I}^{+}}]_{\eta}(\tau)\dd \tau\Big)\cdot|\bar{\eta}_1-\bar{\eta}_2|^{\alpha}.
\end{split}
\end{eqnarray*}

Similarly, for $\mathcal {J}^{+}_{\rm{I}^{+}_{3}}$, we have
\begin{eqnarray*}
\begin{split}
\mathcal {J}^{+}_{\rm{I}^{+}_{3}}\leq \mathcal {O}(1)\Big(\|\mathcal {Z}_{+}^{-,*}\|_{0,0; \tilde{\mathcal{N}}^{\rm{I}}_{+}}+\int_0^{\bar{\xi}}[\mathcal {Z}^{-,*}_{\rm{I}^{+}}]_{\eta}(\tau)\dd \tau\Big)\cdot|\bar{\eta}_1-\bar{\eta}_2|^{\alpha}.
\end{split}
\end{eqnarray*}

By Lemma \ref{lem:A1}, for $\mathcal {J}^{+}_{\rm{I}^{+}_{4}}$ and $\mathcal {J}^{+}_{\rm{I}^{+}_{5}}$, we have
\begin{eqnarray*}
\begin{split}
\mathcal {J}^{+}_{\rm{I}^{+}_{4}}&\leq|a^{+}_{+}(\underline{\mathcal{U}}_{+})|[\p_{\eta}\delta Y^{*}_{+}]_{0,\alpha; \tilde{\mathcal{N}}^{\rm{I}}_{+}}\int_0^{\bar{\xi}}|\Upsilon^{+}_{+}(\tau;\bar{\xi},\bar{\eta}_1)-\Upsilon^{+}_{+}(\tau;\bar{\xi},\bar{\eta}_2)|^{\alpha}\dd \tau\\
&\leq \mathcal {O}(1)[\nabla\delta Y^{*}_{+}]_{0,\alpha; \tilde{\mathcal{N}}^{\rm{I}}_{+}}\cdot|\bar{\eta}_1-\bar{\eta}_2|^{\alpha},
\end{split}
\end{eqnarray*}
and
\begin{eqnarray*}
\begin{split}
\mathcal {J}^{+}_{\rm{I}^{+}_{5}}&\leq\mathcal{O}(1) \Big\|\frac{\p_{\eta} \Lambda_{+}}{\Lambda_{+}}\Big\|_{0,0; \tilde{\mathcal{N}}^{\rm{I}}_{+}}\cdot\|\p_{\eta}\delta Y^{*}_{+}\|_{0,0; \tilde{\mathcal{N}}^{\rm{I}}_{+}}\int_0^{\bar{\xi}}|\Upsilon^{+}_{+}(\tau;\bar{\xi},\bar{\eta}_1)-\Upsilon^{+}_{+}(\tau;\bar{\xi},\bar{\eta}_2)|\dd \tau\\
&\quad+\mathcal{O}(1) \Big[\frac{\p_{\eta} \Lambda_{+}}{\Lambda_{+}}\Big]_{0,\alpha; \tilde{\mathcal{N}}^{\rm{I}}_{+}}\|\delta Y^{*}_{+}\|_{0,0; \tilde{\mathcal{N}}^{\rm{I}}_{+}}\int_0^{\bar{\xi}}|\Upsilon^{+}_{+}(\tau;\bar{\xi},\bar{\eta}_1)-\Upsilon^{+}_{+}(\tau;\bar{\xi},\bar{\eta}_2)|^{\alpha}\dd \tau\\
&\leq \mathcal {O}(1)\|\p_\eta\delta Y^{*}_{+}\|_{0,0; \tilde{\mathcal{N}}^{\rm{I}}_{+}}\cdot|\bar{\eta}_1-\bar{\eta}_2|
+\mathcal {O}(1)\|\delta Y^{*}_{+}\|_{0,0; \tilde{\mathcal{N}}^{\rm{I}}_{+}}\cdot|\bar{\eta}_1-\bar{\eta}_2|^{\alpha}\\
&\leq \mathcal {O}(1)\|\delta Y^{*}_{+}\|_{1,0; \tilde{\mathcal{N}}^{\rm{I}}_{+}}\cdot|\bar{\eta}_1-\bar{\eta}_2|^{\alpha}.
\end{split}
\end{eqnarray*}

By Lemma \ref{lem:A1}, we have for $\mathcal {J}^{+}_{\rm{I}^{+}_{6}}$ and $\mathcal {J}^{+}_{\rm{I}^{+}_{7}}$ that
\begin{eqnarray*}
\begin{split}
\mathcal {J}^{+}_{\rm{I}^{+}_{6}}&\leq\big[\p_{\eta}f^{+}_{+}\big]_{0,\alpha;\tilde{\mathcal{N}}^{\rm{I}}_{+}}
\int_0^{\bar{\xi}}|\Upsilon^{+}_{+}(\tau;\bar{\xi},\bar{\eta}_1)-\Upsilon^{+}_{+}(\tau;\bar{\xi},\bar{\eta}_2)|^{\alpha}\dd \tau\\
&\leq \mathcal{O}(1)\big[\nabla f^{+}_{+}\big]_{0,\alpha;\tilde{\mathcal{N}}^{\rm{I}}_{+}}\cdot|\bar{\eta}_1-\bar{\eta}_2|^{\alpha},
\end{split}
\end{eqnarray*}
and
\begin{eqnarray*}
\begin{split}
\mathcal {J}^{+}_{\rm{I}^{+}_{7}}&\leq\mathcal {O}(1)\Big(\sum_{k=\pm}\big(\|f^{k}_{+}\|_{0,0; \tilde{\mathcal{N}}^{\rm{I}}_{+}}+\|\p_{\eta}f^{k}_{+}\|_{0,0; \tilde{\mathcal{N}}^{\rm{I}}_{+}}\big)\Big)\cdot|\bar{\eta}_1-\bar{\eta}_2|^{\alpha}\\
&\leq \mathcal {O}(1)\Big(\sum_{k=\pm}\|f^{k}_{+}\|_{1,0; \tilde{\mathcal{N}}^{\rm{I}}_{+}}\Big)\cdot|\bar{\eta}_1-\bar{\eta}_2|^{\alpha},
\end{split}
\end{eqnarray*}
provided that $\sigma>0$ is sufficiently small.
%
%
%
Therefore, by letting $\bar{\xi}=\kappa$, it follows from \eqref{eq:prop-3.2-23} that 
\begin{align}\label{eq:prop-3.2-24}
[\mathcal {Z}^{+,*}_{\rm{I}^{+}}]_{\eta}(\kappa)\leq \mathcal{O}(1)\Big([\mathcal {Z}^{+}_{+,\rm in}]_{0,\alpha;\tilde{\Gamma}^{+}_{\rm in}}+\|\delta Y^{*}_{+}\|_{1,\alpha; \tilde{\mathcal{N}}^{\rm{I}}_{+}}+\sum_{k=\pm}\|f^{k}_{+}\|_{1,\alpha; \tilde{\mathcal{N}}^{I}_{+}}+\int_0^{\kappa}\sum_{k=\pm}[\mathcal {Z}^{k,*}_{\rm{I}^{+}}]_{\eta}(\tau)\dd\tau\Big).
\end{align}

In the same way, we can also get
\begin{align}\label{eq:prop-3.2-25}
\begin{split}
&[\mathcal {Z}^{-,*}_{\rm{I}^{+}}]_{\eta}(\kappa)\leq \mathcal {O}(1)\Big([\mathcal {Z}^{-}_{+,\rm in}]_{0,\alpha;\tilde{\Gamma}^{+}_{\rm in}}+\|\delta Y^{*}_{+}\|_{1,\alpha; \tilde{\mathcal{N}}^{\rm{I}}_{+}}
+\sum_{k=\pm}\|f^{k}_{+}\|_{1,\alpha; \tilde{\mathcal{N}}^{\rm{I}}_{+}}+\int_0^{\kappa}\sum_{k=\pm}[\mathcal {Z}^{k,*}_{\rm{I}^{+}}]_{\eta}(\tau)\dd\tau\Big).
\end{split}
\end{align}

Then, by summing \eqref{eq:prop-3.2-24} and \eqref{eq:prop-3.2-25}, and applying the Gronwall's inequality, we have
\begin{eqnarray}\label{eq:prop-3.2-26}
\begin{split}
&\sum_{k=\pm}[\mathcal {Z}^{k,*}_{\rm{I}^{+}}]_{\eta}(\kappa)\leq \mathcal{O}(1)\Big(\sum_{k=\pm}\big([\mathcal {Z}^{k}_{+,\rm in}]_{0,\alpha;\tilde{\Gamma}^{+}_{\rm in}}+\|f^{k}_{+}\|_{1,\alpha; \tilde{\mathcal{N}}^{\rm{I}}_{+}}\big)+\|\delta Y^{*}_{+}\|_{1,\alpha; \tilde{\mathcal{N}}^{\rm{I}}_{+}}\Big).
\end{split}
\end{eqnarray}

Next, we will estimate $[\mathcal {Z}^{\pm,*}_{\rm{I}^+}]_{\xi}(\kappa)$. Without loss of the generality, we choose any two points $(\bar{\xi}_1,\bar{\eta}),\ (\bar{\xi}_2,\bar{\eta})\in \tilde{\mathcal{N}}^{\rm{I}}_{+}$ with $\bar{\eta}$ being fixed and $\bar{\xi}_1> \bar{\xi}_2$.
It follows from \eqref{eq:prop-3.2-14} that
\begin{eqnarray}\label{eq:prop-3.2-27}
|\mathcal {Z}^{+,*}_{+}(\bar{\xi}_1,\bar{\eta})-\mathcal {Z}^{+,*}_{+}(\bar{\xi}_2,\bar{\eta})|\leq \tilde{\mathcal{J}}^{+}_{\rm{I}^{+}_{1}}+\tilde{\mathcal{J}}^{+}_{\rm{I}^{+}_{2}},
\end{eqnarray}
where 
\begin{eqnarray*}
\begin{split}
\tilde{\mathcal{J}}^{+}_{\rm{I}^{+}_{1}}&\triangleq\big|\mathcal {Z}^{+}_{+,\rm in}(\gamma^{+}_{+}(\bar{\xi}_1,\bar{\eta}))-\mathcal {Z}^{+}_{+,\rm in}(\gamma^{+}_{+}(\bar{\xi}_2,\bar{\eta}))\big|\\
&+\int_0^{\bar{\xi}_2}\big|\big(\frac{\p_{+}\Lambda_{+}-2\Lambda_{+}\p_{\eta}\lambda^{+}_{+}}{2\Lambda_{+}}\mathcal {Z}_{+}^{+,*}\big)(\tau,\Upsilon^{+}_{+}(\tau;\bar{\xi}_1,\bar{\eta}))-\big(\frac{\p_{+}\Lambda_{+}-2\Lambda_{+}\p_{\eta}\lambda^{+}_{+}}{2\Lambda_{+}}\mathcal {Z}_{+}^{+,*}\big)(\tau,\Upsilon^{+}_{+}(\tau;\bar{\xi}_2,\bar{\eta}))\big|\dd\tau\\
&+\int_0^{\bar{\xi}_2}\big|\big(\frac{\p_{-}\Lambda_{+}}{2\Lambda_{+}}\mathcal {Z}_{+}^{-,*}\big)(\tau,\Upsilon^{+}_{+}(\tau;\bar{\xi}_1,\bar{\eta}))-\big(\frac{\p_{-}\Lambda_{+}}{2\Lambda_{+}}\mathcal {Z}_{+}^{-,*}\big)(\tau,\Upsilon^{+}_{+}(\tau;\bar{\xi}_2,\bar{\eta}))\big|\dd\tau\\
&+|a^{+}_{+}(\underline{\mathcal{U}}_{+})|\int_0^{\bar{\xi}_2}\big|\p_{\eta}\delta Y^{*}_{+}(\tau,\Upsilon^{+}_{+}(\tau;\bar{\xi}_1,\bar{\eta}))
-\p_{\eta}\delta Y^{*}_{+}(\tau,\Upsilon^{+}_{+}(\tau;\bar{\xi}_2,\bar{\eta}))\big|\dd \tau\\
&+\frac{|a^{+}_{+}(\underline{\mathcal{U}}_{+})|+|a^{-}_{+}(\underline{\mathcal{U}}_{+})|}{2}\int_0^{\bar{\xi}_2}
\big|\big(\frac{\p_{\eta}\Lambda_{+}}{\Lambda_{+}}Y_{+}^{*}\big)(\tau,\Upsilon^{+}_{+}(\tau;\bar{\xi}_1,\bar{\eta}))
-\big(\frac{\p_{\eta}\Lambda_{+}}{\Lambda_{+}}Y_{+}^{*}\big)(\tau,\Upsilon^{+}_{+}(\tau;\bar{\xi}_2,\bar{\eta}))\big|\dd\tau\\
&+\int_0^{\bar{\xi}_2}\big|\p_{\eta}f^{+}_{+}(\delta\mathcal{V}_{+}, \delta Y_{+})(\tau,\Upsilon^{+}_{+}(\tau;\bar{\xi}_1,\bar{\eta}))
-\p_{\eta}f^{+}_{+}(\delta\mathcal{V}_{+}, \delta Y_{+})(\tau,\Upsilon^{+}_{+}(\tau;\bar{\xi}_2,\bar{\eta}))\big|\dd \tau\\
&+\sum_{k=\pm}\int_0^{\bar{\xi}_2}\big|\big(\frac{\p_{\eta} \Lambda_{+}}{2\Lambda_{+}}f^{k}_{+}(\delta\mathcal{V}_{+}, \delta Y_{+})\big)(\tau,\Upsilon^{+}_{+}(\tau;\bar{\xi}_1,\bar{\eta}))
-\big(\frac{\p_{\eta} \Lambda_{+}}{2\Lambda_{+}}f^{k}_{+}(\delta\mathcal{V}_{+}, \delta Y_{+})\big)(\tau,\Upsilon^{+}_{+}(\tau;\bar{\xi}_2,\bar{\eta}))\big|\dd \tau,
\end{split}
\end{eqnarray*}
and
\begin{eqnarray*}
\begin{split}
\tilde{\mathcal{J}}^{+}_{\rm{I}^{+}_{2}}&\triangleq\int_{\bar{\xi}_2}^{\bar{\xi}_1}\Big|\Big(\frac{\p_{+}\Lambda_{+}-2\Lambda_{+}\p_{\eta}\lambda^{+}_{+}}{2\Lambda_{+}}\mathcal {Z}_{+}^{+,*}-\frac{\p_{-}\Lambda_{+}}{2\Lambda_{+}}\mathcal {Z}_{+}^{-,*}\Big)(\tau,\Upsilon^{+}_{+}(\tau;\bar{\xi}_1,\bar{\eta}))\Big|\dd \tau\\
&+\int_{\bar{\xi}_2}^{\bar{\xi}_1}\Big|\Big(a^{-}_{+}(\underline{\mathcal{U}}_{+})\p_{\eta}\delta Y^{*}_{+}-\frac{[a^{-}_{+}(\underline{\mathcal{U}}_{+})-a^{+}_{+}(\underline{\mathcal{U}}_{+})]\p_{\eta} \Lambda_{+}}{2\Lambda_{+}}\delta Y^{*}_{+} \Big)(\tau,\Upsilon^{+}_{+}(\tau;\bar{\xi}_1,\bar{\eta}))\Big|\dd \tau\\
&+\int_{\bar{\xi}_2}^{\bar{\xi}_1}\Big|\Big(\p_{\eta}f^{-}_{+}(\delta\mathcal{V}_{+}, \delta Y_{+})-\frac{\p_{\eta} \Lambda_{+}}{2\Lambda_{+}}\big[f^{-}_{+}(\delta\mathcal{V}_{+}, \delta Y_{+})
-f^{+}_{+}(\delta\mathcal{V}_{+}, \delta Y_{+})\big]\Big)(\tau,\Upsilon^{+}_{+}(\tau;\bar{\xi}_1,\bar{\eta}))\Big| \dd \tau.
\end{split}
\end{eqnarray*}

Now, we can estimate $\tilde{\mathcal {J}}^{+}_{\rm{I}^{+}_{1}}$ and $\tilde{\mathcal {J}}^{+}_{\rm{I}^{+}_{2}}$ one by one. 
Following a similar argument as the one for the estimates in \eqref{eq:prop-3.2-23}, we can deduce that
\begin{align}\label{eq:prop-3.2-28}
\tilde{\mathcal {J}}^{+}_{\rm{I}^{+}_{1}}\leq \mathcal{O}(1)\Big([\mathcal {Z}^{+}_{+,\rm in}]_{0,\alpha;\tilde{\Gamma}^{+}_{\rm in}}+\|\delta Y^{*}_{+}\|_{1,\alpha; \tilde{\mathcal{N}}^{\rm{I}}_{+}}+\sum_{k=\pm}\|f^{k}_{+}\|_{1,\alpha; \tilde{\mathcal{N}}^{\rm{I}}_{+}}+\int_0^{\bar{\xi}_{2}}\sum_{k=\pm}[\mathcal {Z}^{k,*}_{\rm{I}^{+}}]_{\eta}(\tau)\dd\tau\Big)\cdot|\bar{\xi}_1-\bar{\xi}_2|^{\alpha},
\end{align}
and
\begin{align}\label{eq:prop-3.2-29}
\begin{split}
\tilde{\mathcal {J}}^{+}_{\rm{I}^{+}_{2}}&\leq\Big\{\big\|\frac{\p_{+}\Lambda_{+}-2\Lambda_{+}\p_{\eta}\lambda^{+}_{+}}{2\Lambda_{+}}\big\|_{0,0; \tilde{\mathcal{N}}^{\rm{I}}_{+}}\cdot\|\mathcal {Z}_{+}^{+,*}\|_{0,0; \tilde{\mathcal{N}}^{\rm{I}}_{+}}+\big\|\frac{\p_{-}\Lambda_{+}}{2\Lambda_{+}}\big\|_{0,0; \tilde{\mathcal{N}}^{\rm{I}}_{+}}\cdot\|\mathcal {Z}_{+}^{-,*}\|_{0,0; \tilde{\mathcal{N}}^{\rm{I}}_{+}}\\
\quad&+\big\|\frac{(a^{+}_{+}(\underline{\mathcal{U}}_{+})-a^{-}_{+}(\underline{\mathcal{U}}_{+}))\p_{\eta} \Lambda_{+}}{2\Lambda_{+}}\big\|_{0,0; \tilde{\mathcal{N}}^{I}_{+}}\cdot\|\delta Y^{*}_{+}\|_{0,0; \tilde{\mathcal{N}}^{\rm{I}}_{+}}+|a^{+}_{+}(\underline{\mathcal{U}}_{+})|\cdot\|\p_{\eta}\delta Y^{*}_{+}\|_{0,0; \tilde{\mathcal{N}}^{\rm{I}}_{+}}\\
\quad&+\big\|\frac{\p_{\eta} \Lambda_{+}}{2\Lambda_{+}}\big\|_{0,0; \tilde{\mathcal{N}}^{\rm{I}}_{+}}\cdot\big(\sum_{k=\pm}\|f^{k}_{+}\|_{0,0; \tilde{\mathcal{N}}^{\rm{I}}_{+}}\big)
+\|\p_{\eta}f^{+}_{+}\|_{0,0; \tilde{\mathcal{N}}^{\rm{I}}_{+}}\Big\}\cdot|\bar{\xi}_1-\bar{\xi}_2|\\
&\leq \mathcal {O}(1)\Big\{\sum_{k=\pm}\big(\|\mathcal {Z}_{+}^{k,*}\|_{0,0; \tilde{\mathcal{N}}^{\rm{I}}_{+}}+\|f^{k}_{+}\|_{0,0; \tilde{\mathcal{N}}^{\rm{I}}_{+}}\big)+\|\delta Y^{*}_{+}\|_{1,0; \tilde{\mathcal{N}}^{\rm{I}}_{+}}+\|\p_{\eta}f^{+}_{+}\|_{0,0; \tilde{\mathcal{N}}^{\rm{I}}_{+}}\Big\}\cdot|\bar{\xi}_1-\bar{\xi}_2|^{\alpha}.
\end{split}
\end{align}

Then 
by letting $\bar{\xi}_{2}=\kappa$, we have
\begin{eqnarray*}
\begin{split}
&[\mathcal {Z}^{+,*}_{\rm{I}^{+}}]_{\xi}(\kappa)\leq \mathcal{O}(1)\Big([\mathcal {Z}^{+}_{+,\rm in}]_{0,\alpha;\tilde{\Gamma}^{+}_{\rm in}}+\|\delta Y^{*}_{+}\|_{1,\alpha; \tilde{\mathcal{N}}^{\rm{I}}_{+}}+\sum_{k=\pm}\|f^{k}_{+}\|_{1,\alpha; \tilde{\mathcal{N}}^{\rm{I}}_{+}}+\int_0^{\kappa}\sum_{k=\pm}[\mathcal {Z}^{k,*}_{\rm{I}^{+}}]_{\eta}(\tau)\dd\tau\Big).
\end{split}
\end{eqnarray*}

By \eqref{eq:prop-3.2-26}, we further have
\begin{eqnarray}\label{eq:prop-3.2-30}
\begin{split}
&[\mathcal {Z}^{+,*}_{\rm{I}^{+}}]_{\xi}(\kappa)\leq \mathcal{O}(1)\Big(\sum_{k=\pm}\big([\mathcal {Z}^{k}_{+,\rm in}]_{0,\alpha;\tilde{\Gamma}^{+}_{\rm in}}+\|f^{k}_{+}\|_{1,\alpha; \tilde{\mathcal{N}}^{\rm{I}}_{+}}\big)+\|\delta Y^{*}_{+}\|_{1,\alpha; \tilde{\mathcal{N}}^{\rm{I}}_{+}}\Big).
\end{split}
\end{eqnarray}

Similarly, we also have
\begin{eqnarray}\label{eq:prop-3.2-31}
\begin{split}
[\mathcal {Z}^{-,*}_{\rm{I}^{+}}]_{\xi}(\kappa)\leq \mathcal{O}(1)\Big(\sum_{k=\pm}\big([\mathcal {Z}^{k}_{+,\rm in}]_{0,\alpha;\tilde{\Gamma}^{+}_{\rm in}}+\|f^{k}_{+}\|_{1,\alpha; \tilde{\mathcal{N}}^{\rm{I}}_{+}}\big)+\|\delta Y^{*}_{+}\|_{1,\alpha; \tilde{\mathcal{N}}^{\rm{I}}_{+}}\Big).
\end{split}
\end{eqnarray}

Combining the estimates \eqref{eq:prop-3.2-26},\eqref{eq:prop-3.2-30}-\eqref{eq:prop-3.2-31}, and by  \eqref{eq:prop-3.2-22} and \eqref{eq:prop-3.2-21}, we thus obtain
\begin{eqnarray}\label{eq:prop-3.2-32}
\begin{split}
&\sum_{k=\pm}[\mathcal {Z}^{k,*}_{+}]_{0,\alpha,\tilde{\mathcal{N}}^{\rm{I}}_{+}}\leq \mathcal{O}(1)\Big(\sum_{k=\pm}\big([\mathcal {Z}^{k}_{+,\rm in}]_{0,\alpha;\tilde{\Gamma}^{+}_{\rm in}}+\|f^{k}_{+}\|_{1,\alpha; \tilde{\mathcal{N}}^{\rm{I}}_{+}}\big)+\|\delta Y^{*}_{+}\|_{1,\alpha; \tilde{\mathcal{N}}^{\rm{I}}_{+}}\Big).
\end{split}
\end{eqnarray}


By \eqref{eq:prop-3.2-17}-\eqref{eq:prop-3.2-18}, 
there exists a constant $\tilde{C}^{**}_{\rm{I}^{+}}>0$ depending only on $\underline{\mathcal{U}}_{+}$, $\alpha$ and $L$ such that
\begin{align}\label{eq:prop-3.2-34}
\begin{split}
[(\nabla\delta \omega^{*}_{+},\nabla\delta p^{*}_{+})]_{0,\alpha; \tilde{\mathcal{N}}^{\rm{I}}_{+}}\leq \tilde{C}^{**}_{\rm{I}^{+}}
\Big(\|(\delta \omega^{+}_{\rm in},\delta p^{+}_{\rm in})\|_{0,\alpha; \tilde{\Gamma}^{+}_{\rm in}}+\|\delta Y^{*}_{+}\|_{1,\alpha; \tilde{\mathcal{N}}^{\rm{I}}_{+}}
+\sum_{k=\pm}\|f^{k}_{+}\|_{1,\alpha; \tilde{\mathcal{N}}^{\rm{I}}_{+}}\Big).
\end{split}
\end{align}


Finally, combing the estimates \eqref{eq:prop-3.2-11}, \eqref{eq:prop-3.2-20} and \eqref{eq:prop-3.2-34}, we can get estimate \eqref{eq:prop-3.2-1} by choosing $\tilde{C_{31}}=\max\{\tilde{C}_{\rm{I}^{+}}, \tilde{C}^{*}_{\textrm{I}^{+}}, \tilde{C}^{**}_{\rm{I}^{+}}\}$. It completes the proof of proposition.
\end{proof}

Next, we will consider the problem  $(\mathbf{IBVP})^{*}$ for $(\delta\omega^{*}_{\pm}, \delta p^{*}_{\pm})$ in $\tilde{\mathcal{N}}^{\textrm{I}}_{\pm}\cup\tilde{\mathcal{N}}^{\textrm{II}}_{\pm}$
which is the following initial-boundary value problem.
\begin{eqnarray}\label{eq:IBVP-II}
(\mathbf{IBVP})^{*}_{\textrm{II}}\
\begin{cases}
\partial_{\xi}\delta\omega^{*}_{+}+\lambda^{\pm}_{+}\partial_{\eta}\delta\omega^{*}_{+}\pm\Lambda_{+}\big(\partial_{\xi}\delta p^{*}_{+}+\lambda^{\pm}_{+}\partial_{\eta}\delta p^{*}_{+}\big)\\
\qquad\qquad\qquad\qquad\quad=a^{\pm}_{+}(\underline{\mathcal{U}}_{+})\delta Y^{*}_{+}+f^{\pm}_{+}(\delta\mathcal{V}_{+}, \delta Y_{+}), &\quad  \mbox{in} \quad \tilde{\mathcal{N}}^{\textrm{I}}_{+}\cup\tilde{\mathcal{N}}^{\textrm{II}}_{+}, \\
\partial_{\xi}\delta\omega^{*}_{-}+\lambda^{\pm}_{-}\partial_{\eta}\delta\omega^{*}_{-}\pm\Lambda_{-}\big(\partial_{\xi}\delta p^{*}_{-}+\lambda^{\pm}_{-}\partial_{\eta}\delta p^{*}_{-}\big)\\
\qquad\qquad\qquad\qquad\quad=a^{\pm}_{-}(\underline{\mathcal{U}}_{-})\delta Y^{*}_{-}+f^{\pm}_{-}(\delta\mathcal{V}_{-}, \delta Y_{-}), &\quad \mbox{in}\quad \tilde{\mathcal{N}}^{\textrm{I}}_{-}\cup\tilde{\mathcal{N}}^{\textrm{II}}_{-},\\
(\delta\omega^{*}_{+},\delta p^{*}_{+})=(\delta\omega^{+}_{\rm in},\delta p^{+}_{\rm in}), &\quad  \mbox{on} \quad \tilde{\Gamma}^{+}_{\rm in},\\
(\delta\omega^{*}_{-},\delta p^{*}_{-})=(\delta\omega^{-}_{\rm in},\delta p^{-}_{\rm in}), &\quad  \mbox{on} \quad \tilde{\Gamma}^{-}_{\rm in},\\
\delta\omega^{*}_{\pm}=g'_{\pm}(\xi), &\quad \mbox{on} \quad \tilde{\Gamma}_{\pm}\cap \overline{\tilde{\mathcal{N}}^{\textrm{II}}_{\pm}}.
\end{cases}
\end{eqnarray}

We have the following proposition for $(\mathbf{IBVP})^{*}_{\textrm{II}}$.
\begin{proposition}\label{prop:3.3}
For $\sigma>0$ sufficiently small, if $(\delta\mathcal{U}_{+},\delta\mathcal{U}_{-})\in\mathcal{X}_{\sigma}$, the problem $(\mathbf{IBVP})^{*}_{\rm{II}}$ admits a unique $C^{1,\alpha}$-solution $(\delta \omega^{*}_{\pm}, \delta p^{*}_{\pm})$ satisfying
\begin{eqnarray}\label{eq:prop-3.3-1}
\begin{split}
\|(\delta \omega^{*}_{k},\delta p^{*}_{k})\|_{1,\alpha; \tilde{\mathcal{N}}^{\rm{I}}_{k}\cup\tilde{\mathcal{N}}^{\rm{II}}_{k}}
&\leq \tilde{C}_{32}\Big(\|(\delta \omega^{k}_{\rm {in}},\delta p^{k}_{\rm {in}})\|_{1,\alpha; \tilde{\Gamma}^{k}_{\rm {in}}}
+\big\|g_{k}-\underline{g}_{k}\big\|_{2,\alpha; \tilde{\Gamma}_{k}}\\
&\qquad\qquad+\|\delta Y^{*}_{k}\|_{1,\alpha; \tilde{\mathcal{N}}^{\rm{I}}_{k}\cup\tilde{\mathcal{N}}^{\rm{II}}_{k}}+\sum_{j=\pm}\|f^{j}_{k}\|_{1,\alpha; \tilde{\mathcal{N}}^{\rm{I}}_{k}\cup\tilde{\mathcal{N}}^{\rm{II}}_{k}}\Big),
\end{split}
\end{eqnarray}
where $k=``\pm"$, and the constant $\tilde{C}_{32}>0$ depends only on $\underline{\mathcal{U}}_{+}$, $\alpha$ and $L$.
\end{proposition}

\begin{proof}
Since \eqref{eq:IBVP-II} is linear, we can also follow the methods in \cite{li-yu} to employ the \emph{Picard iteration scheme} to establish the existence and uniqueness of solutions of \eqref{eq:IBVP-II}.
So the remaining task is to derive the {\color{black}$C^{1,\alpha}$-\emph{priori}} estimates of $(\delta \omega^{*}_{\pm},\delta p^{*}_{\pm})$.
Without loss of the generality, we only estimate $(\delta \omega^{*}_{+},\delta p^{*}_{+})$ in $\tilde{\mathcal{N}}^{\rm{I}}_{+}\cup\tilde{\mathcal{N}}^{\rm{II}}_{+}$
since $(\delta \omega^{*}_{-},\delta p^{*}_{-})$ in $\tilde{\mathcal{N}}^{\rm{I}}_{-}\cup\tilde{\mathcal{N}}^{\rm{II}}_{-}$ can be dealt with in the same way.

\begin{figure}[ht]
\begin{center}
\begin{tikzpicture}[scale=1.1]

\draw [line width=0.04cm](-2.5,0.8)--(0.8,0.8);

\draw [line width=0.02cm](-2.5,-2.0)--(-2.5,0.8);
\draw [line width=0.04cm](-2.5,-2.0)to[out=20, in=-120](0.8,0.8);
\draw [line width=0.03cm][blue](-2.5,-1.2)to[out=10, in=-120](-0.2,0.2);
\draw [line width=0.03cm][blue](-2.5,-0.3)to[out=10, in=-110](-1.2,0.8);
\draw [line width=0.03cm][red](-1.2,0.8)to[out=-10, in=120](-0.2,0.2);
\draw [line width=0.01cm](-2.5,0.8)to[out=-20, in=110](-0.5,-0.9);

\node at (-1.8, -0.5) {$\tilde{\mathcal{N}}^{\rm{I}}_{+}$};
\node at (0.2, 0.4) {$\tilde{\mathcal{N}}^{\rm{II}}_{+}$};

\node at (0.8,0.8) {$\bullet$};
\node at (-2.5,0.8) {$\bullet$};
\node at (-2.5,-2.0) {$\bullet$};
\node at (-2.5,-1.2) {$\bullet$};
\node at (-2.5,-0.3) {$\bullet$};
\node at (-1.2,0.8) {$\bullet$};
\node at (-0.2,0.2) {$\bullet$};
\node at (-0.5,-0.9) {$\bullet$};

\node at (1.2,0.8) {$\tilde{\Gamma}_{+}$};
\node at (-2.8, -0.6) {$\tilde{\Gamma}^{+}_{\textrm{in}}$};
\node at (-2.8,0.8) {$\tilde{P}_+$};
\node at (-2.8,-2.0) {$\tilde{O}$};
\end{tikzpicture}
\end{center}
\caption{Estimates for $(\delta \omega^{*}_{+},\delta p^{*}_{+})$ in $\tilde{\mathcal{N}}^{\rm{I}}_{+}\cup\tilde{\mathcal{N}}^{\rm{II}}_{+}$}\label{fig3.6}
\end{figure}
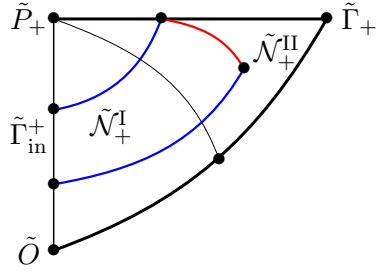

\rm{1}.\emph{\ Estimates on $\|(\delta \omega^{*}_{+},\delta p^{*}_{+})\|_{0,0;\tilde{\mathcal{N}}^{\rm{I}}_{+}\cup\tilde{\mathcal{N}}^{\rm{II}}_{+}}$.} Without loss of the generality, we only consider the stimates in $\tilde{\mathcal{N}}^{\rm{II}}_{+}$.
Recalling the definition of $z^{+,*}_{+}$ and $z^{-,*}_{+}$ in \eqref{eq:prop-3.2-2} and by $\eqref{eq:IBVP-II}_{3}-\eqref{eq:IBVP-II}_{5}$,
we know that $z^{+,*}_{+}$ and $z^{-,*}_{+}$ satisfy equations $\eqref{eq:prop-3.2-3}_{1}$-$\eqref{eq:prop-3.2-3}_{2}$ in 
$\tilde{\mathcal{N}}^{\rm{II}}_{k}$ with initial conditions $\eqref{eq:prop-3.2-3}_{3}$ and boundary condition
\begin{eqnarray}\label{eq:prop-3.3-2}
z^{+,*}_{+}+z^{-,*}_{+}=2g'_{+}(\xi), &\quad \mbox{on} \quad \tilde{\Gamma}_{+}\cap \overline{\tilde{\mathcal{N}}^{\rm{II}}_{+}}.
\end{eqnarray}

Then, as shown in Fig.\ref{fig3.6}, for any point $(\bar{\xi},\bar{\eta})\in \tilde{\mathcal{N}}^{\rm{II}}_{+}$, as in \eqref{eq:prop-3.2-4}, we can define a backward characteristic curve $\eta=\Upsilon^{+}_{+}(\xi;\bar{\xi},\bar{\eta})$ corresponding to $\lambda^{+}_{+}$, passing through $(\bar{\xi},\bar{\eta})$ and intersecting $\tilde{\Gamma}^{+}_{\rm{in}}$ at the point $(0,\gamma^{+}_{+}(\bar{\xi},\bar{\eta}))$.
Similarly, let $\eta=\Upsilon^{-, \rm{b}}_{+}(\xi;\bar{\xi},\bar{\eta})$ be a backward characteristic curve corresponding to $\lambda^{-}_{+}$, passing through $(\bar{\xi},\bar{\eta})$ and hitting $\tilde{\Gamma}_{+}\cap \overline{\tilde{\mathcal{N}}^{\rm{II}}_{+}}$ at the point $(\zeta^{-,\rm{b}}_{+}(\bar{\xi},\bar{\eta}),\textrm{m}_{+})$. Then, by the compatibility condition $\eqref{eq:2.11}_1$, along $\eta=\Upsilon^{+}_{+}(\xi;\bar{\xi},\bar{\eta})$, $z^{+,*}_{+}(\bar{\xi},\bar{\eta})$ satisfies $\eqref{eq:prop-3.2-5}_1$,
and along $\eta=\Upsilon^{-, \rm{b}}_{+}(\xi;\bar{\xi},\bar{\eta})$, $z^{-,*}_{+}$ satisfies
\begin{eqnarray}\label{eq:prop-3.3-3}
\begin{split}
z^{-,*}_{+}(\bar{\xi},\bar{\eta})&=z^{-,*}_{+}(\zeta^{-,\rm{b}}_{+}(\bar{\xi},\bar{\eta}),\textrm{m}_{+})
+\int_{\zeta^{-,\rm{b}}_{+}(\bar{\xi},\bar{\eta})}^{\bar{\xi}}\frac{\p_{-}\Lambda_{+}}{2\Lambda_{+}}\big(z^{-,*}_{+}-z^{+,*}_{+}\big)
(\tau,\Upsilon^{-,\rm{b}}_{+}(\tau;\bar{\xi},\bar{\eta}))\mathrm{d} \tau\\
&\quad\ +\int^{\bar{\xi}}_{\zeta^{-,\rm{b}}_{+}(\bar{\xi},\bar{\eta})}\big(a^{-}_{+}(\underline{\mathcal{U}}_{+})\delta Y^{*}_{+}+f^{-}_{+}(\delta\mathcal{V}_{+}, \delta Y_{+})\big)(\tau,\Upsilon^{-, \rm{b}}_{+}(\tau;\bar{\xi},\bar{\eta}))\mathrm{d} \tau.
\end{split}
\end{eqnarray}

Define
\begin{eqnarray}\label{eq:prop-3.3-4}
R^{\rm{II}}_{+}(\kappa)=\{(\xi,\eta)| 0\leq\xi\leq\kappa,\ \Upsilon^{+,0}_{+}(\xi;\bar{\xi}_{+}^{2},\textrm{m}_{+})\leq\eta\leq\textrm{m}_{+}\},\  \kappa\in [0,\bar{\xi}^{2}_{+}],
\end{eqnarray}
and
\begin{eqnarray}\label{eq:prop-3.3-5}
z^{\pm,*}_{\rm{II}^{+}}(\kappa)\triangleq\sup\limits_{(\xi,\eta)\in R^{\rm{II}}_{+}(\kappa)}|z^{\pm,*}_{+}(\xi,\eta)|.
\end{eqnarray}

Obviously, we have
\begin{eqnarray}\label{eq:prop-3.3-6}
\|z_{+}^{\pm,*}\|_{0,0; \tilde{\mathcal{N}}^{\rm{I}}_{+}\cup\tilde{\mathcal{N}}^{\rm{II}}_{+}}=z_{\rm{II}^{+}}^{\pm,*}(\bar{\xi}^{2}_{+})\quad\mbox{and}\quad z^{\pm,*}_{\rm{II}^{+}}(\kappa)\geq z^{\pm,*}_{\rm{I}^{+}}(\kappa),\quad \kappa\in [0,\bar{\xi}^{1}_{+}].
\end{eqnarray}

Following \emph{Step 1} in the proof of Proposition \ref{prop:3.2}, we can deduce from $\eqref{eq:prop-3.2-5}_1$ and \eqref{eq:prop-3.3-5} that
\begin{align}\label{eq:prop-3.3-7}
\begin{split}
z_{\rm{II}^{+}}^{+,*}(\kappa)\leq  \mathcal{O}(1)\Big(\|z_{+,\rm in}^{+}\|_{0,0; \tilde{\Gamma}^{+}_{\rm in}}+\|\delta Y^{*}_{+}\|_{0,0; \tilde{\mathcal{N}}^{\rm{I}}_{+}\cup\tilde{\mathcal{N}}^{\rm{II}}_{+}}+\|f^{+}_{+}\|_{0,0; \tilde{\mathcal{N}}^{\rm{I}}_{+}\cup\tilde{\mathcal{N}}^{\rm{II}}_{+}}+\int_0^{\kappa}\sum_{k=\pm}z_{\rm{II}^{+}}^{\pm,*}(\varsigma)\dd\varsigma\Big),
\end{split}
\end{align}
and
\begin{align}\label{eq:prop-3.3-8}
\begin{split}
|z^{-,*}_{+}(\bar{\xi},\bar{\eta})|
&\leq \mathcal{O}(1)\Big(\|\delta Y^{*}_{+}\|_{0,0; \tilde{\mathcal{N}}^{\rm{II}}_{+}}
+\|f^{-}_{+}\|_{0,0; \tilde{\mathcal{N}}^{\rm{II}}_{+}}+\int_{\zeta^{-,\rm{b}}_{+}(\bar{\xi},\bar{\eta})}^{\bar{\xi}}\sum_{k=\pm}z_{\rm{II}^{+}}^{k,*}(\varsigma)\dd\varsigma\Big)\\
&\quad + |z^{-,*}_{+}(\zeta^{-,\rm{b}}_{+}(\bar{\xi},\bar{\eta}),\textrm{m}_{+})|,
\end{split}
\end{align}
where the constants $\mathcal{O}(1)$ depend only on $\underline{\mathcal{U}}_{+}$.

According to the boundary condition \eqref{eq:prop-3.3-2}, we have
\begin{eqnarray}\label{eq:prop-3.3-9}
\begin{split}
z^{-,*}_{+}(\zeta^{-,\rm{b}}_{+}(\bar{\xi},\bar{\eta}),\textrm{m}_{+})=2g'_{+}(\zeta^{-,\rm{b}}_{+}(\bar{\xi},\bar{\eta}))-z^{+,*}_{+}(\zeta^{-,\rm{b}}_{+}(\bar{\xi},\bar{\eta}),\textrm{m}_{+}).
\end{split}
\end{eqnarray}

Then, by the compatibility condition $\eqref{eq:2.11}_1$, along the characteristic $\eta=\Upsilon^{+}_{+}(\xi;\zeta^{-,\rm{b}}_{+}(\bar{\xi},\bar{\eta}),\textrm{m}_{+})$,
\begin{eqnarray}\label{eq:prop-3.3-10}
\begin{split}
&z^{+,*}_{+}(\zeta^{-,\rm{b}}_{+}(\bar{\xi},\bar{\eta}),\textrm{m}_{+})\\
&=z^{+}_{+,\rm in}(\gamma^{+}_{+}(\zeta^{-,\rm{b}}_{+}(\bar{\xi},\bar{\eta}),\textrm{m}_{+}))
+\int_0^{\zeta^{-,\rm{b}}_{+}(\bar{\xi},\bar{\eta})}\frac{\p_{+}\Lambda_{+}}{2\Lambda_{+}}\big(z^{+,*}_{+}-z^{-,*}_{+}\big)(\tau,\Upsilon^{+}_{+}(\tau;\zeta^{-,\rm{b}}_{+}(\bar{\xi},\bar{\eta}),\textrm{m}_{+}))\mathrm{d} \tau\\
&\quad+\int_0^{\zeta^{-,\rm{b}}_{+}(\bar{\xi},\bar{\eta})}\big(a^{+}_{+}(\underline{\mathcal{U}}_{+})\delta Y^{*}_{+}+f^{+}_{+}(\delta\mathcal{V}_{+}, \delta Y_{+})\big)(\tau,\Upsilon^{+}_{+}(\tau;\zeta^{-,\rm{b}}_{+}(\bar{\xi},\bar{\eta}),\textrm{m}_{+}))\mathrm{d} \tau.
\end{split}
\end{eqnarray}

Thus, it follows from \eqref{eq:prop-3.3-9} and \eqref{eq:prop-3.3-10} that
\begin{eqnarray*}
\begin{split}
|z^{-,*}_{+}(\zeta^{-,\rm{b}}_{+}(\bar{\xi},\bar{\eta}),\textrm{m}_{+})|
&\leq \mathcal{O}(1)\Big(\|g_{+}-1\|_{1,0;\tilde{\Gamma}_{+}}+\|z_{+,\rm in}^{+}\|_{0,0; \tilde{\Gamma}^{+}_{\rm in}}+\|\delta Y^{*}_{+}\|_{0,0; \tilde{\mathcal{N}}^{\rm{I}}_{+}\cup\tilde{\mathcal{N}}^{\rm{II}}_{+}}\Big)\\
&\quad +\mathcal{O}(1)\Big(\|f^{+}_{+}\|_{0,0; \tilde{\mathcal{N}}^{\rm{I}}_{+}\cup\tilde{\mathcal{N}}^{\rm{II}}_{+}}+\int_{0}^{\zeta^{-,\rm{b}}_{+}(\bar{\xi},\bar{\eta})}\sum_{k=\pm}z^{k,*}_{\rm{II}^{+}}(\varsigma)\dd\varsigma\Big).
\end{split}
\end{eqnarray*}

Plugging it into \eqref{eq:prop-3.3-8} and by \eqref{eq:prop-3.3-5}, one can get
\begin{align}\label{eq:prop-3.3-11}
\begin{split}
z_{\rm{II}^{+}}^{-,*}(\kappa)&\leq \mathcal{O}(1)\Big(\|g_{+}-1\|_{1,0; \tilde{\Gamma}_{+}}+\|z_{+,\rm in}^{+}\|_{0,0; \tilde{\Gamma}^{+}_{\rm in}}+\|\delta Y^{*}_{+}\|_{0,0; \tilde{\mathcal{N}}^{\rm{I}}_{+}\cup\tilde{\mathcal{N}}^{\rm{II}}_{+}}\Big)\\
&\quad  +\mathcal{O}(1)\Big(\sum_{k=\pm}\|f^{k}_{+}\|_{0,0; \tilde{\mathcal{N}}^{\rm{I}}_{+}\cup\tilde{\mathcal{N}}^{\rm{II}}_{+}}
+\int_{0}^{\kappa}\sum_{k=\pm}z_{\rm{II}^{+}}^{k,*}(\varsigma)\dd\varsigma\Big),
\end{split}
\end{align}
where the constants $\mathcal{O}(1)$ depend only on $\underline{\mathcal{U}}_{+}$.

Adding up \eqref{eq:prop-3.3-7} and \eqref{eq:prop-3.3-11}, and by the Gronwall's inequality, we conclude for $\kappa=\bar{\xi}^{2}_{+}$ that
\begin{eqnarray*}
\begin{split}
&\sum_{k=\pm}\|z_{+}^{k,*}\|_{0,0; \tilde{\mathcal{N}}^{\rm{I}}_{+}\cup\tilde{\mathcal{N}}^{\rm{II}}_{+}}\\
&\leq\mathcal{O}(1)\Big(\|g_{+}-1\|_{1,0; \tilde{\Gamma}_{+}}+\|z_{+,\rm in}^{+}\|_{0,0; \tilde{\Gamma}^{+}_{\rm in}}
+\|\delta Y^{*}_{+}\|_{0,0; \tilde{\mathcal{N}}^{\rm{I}}_{+}\cup\tilde{\mathcal{N}}^{\rm{II}}_{+}}+\sum_{k=\pm}\|f^{k}_{+}\|_{0,0; \tilde{\mathcal{N}}^{\rm{I}}_{+}\cup\tilde{\mathcal{N}}^{\rm{II}}_{+}}\Big).
\end{split}
\end{eqnarray*}


By \eqref{eq:prop-3.2-2a}, we can choose a constant $\tilde{C}_{\rm{II}}>0$ depending only on $\underline{\mathcal{U}}_{+}$ such that
\begin{eqnarray}\label{eq:prop-3.3-12}
\begin{split}
\|(\delta \omega^{*}_{+}, \delta p^{*}_{+})\|_{0,0; \tilde{\mathcal{N}}^{\rm{I}}_{+}\cup\tilde{\mathcal{N}}^{\rm{II}}_{+}}
&\leq\tilde{C}_{\rm{II}}\Big(\|(\delta \omega^{+}_{\rm{in}}, \delta p^{+}_{\rm{in}})\|_{0,0; \tilde{\Gamma}^{+}_{\rm in}}+\|g_{+}-1\|_{1,0; \tilde{\Gamma}_{+}}\\
&\qquad\quad\ +\|\delta Y^{*}_{+}\|_{0,0; \tilde{\mathcal{N}}^{\rm{I}}_{+}\cup\tilde{\mathcal{N}}^{\rm{II}}_{+}}+\sum_{k=\pm}\|f^{k}_{+}\|_{0,0; \tilde{\mathcal{N}}^{\rm{I}}_{+}\cup\tilde{\mathcal{N}}^{\rm{II}}_{+}}\Big).
\end{split}
\end{eqnarray}

2.\ \emph{Estimates on $\|(\nabla\delta \omega^{*}_{+}, \nabla\delta p^{*}_{+})\|_{0,0;\tilde{\mathcal{N}}^{\rm{I}}_{+}\cup\tilde{\mathcal{N}}^{\rm{II}}_{+}}$.}
We also employ functions $\mathcal {Z}_{+}^{+,*}$ and $\mathcal {Z}_{+}^{-,*}$ as defined in \eqref{eq:prop-3.2-12}, which satisfy equations \eqref{eq:prop-3.2-13} in $\tilde{\mathcal{N}}^{\rm{I}}_{+}\cup\tilde{\mathcal{N}}^{\rm{II}}_{+}$ with the following initial and boundary conditions
\begin{eqnarray}\label{eq:prop-3.3-13}
\begin{cases}
(\mathcal {Z}_{+}^{+,*}, \mathcal {Z}_{+}^{-,*})=(\mathcal {Z}_{+}^{+,\rm in}, \mathcal {Z}_{+}^{-,\rm in}), &\quad  \mbox{on} \quad \tilde{\Gamma}^{+}_{\rm in},\\
\lambda_{+}^{+}\mathcal {Z}_{+}^{+,*}+\lambda_{+}^{-}\mathcal {Z}_{+}^{-,*}=-2g''_{+}(\xi)+\big(a^{+}_{+}(\underline{\mathcal{U}}_{+})+a_{+}^{-}(\underline{\mathcal{U}}_{+})\big)\delta Y^{*}_{+}\\
\qquad\qquad\qquad\qquad+f^{+}_{+}(\delta\mathcal{V}_{+}, \delta Y_{+})+f^{-}_{+}(\delta\mathcal{V}_{+}, \delta Y_{+}), &\quad \mbox{on} \quad \tilde{\Gamma}_{+}\cap \overline{\tilde{\mathcal{N}}^{\rm{II}}_{+}}.
\end{cases}
\end{eqnarray}
By the compatibility condition $\eqref{eq:2.11}_2$, along the characteristic curve $\eta=\Upsilon^{+}_{+}(\xi;\bar{\xi},\bar{\eta})$, $\mathcal {Z}_{+}^{+,*}(\bar{\xi},\bar{\eta})$ satisfies $\eqref{eq:prop-3.2-14}_{1}$,
and along the characteristic $\eta=\Upsilon^{-,\rm{b}}_{+}(\xi; \bar{\xi},\bar{\eta})$, $\mathcal {Z}_{+}^{-,*}(\bar{\xi},\bar{\eta})$ satisfies
\begin{align}\label{eq:prop-3.3-14}
\begin{split}
&\mathcal {Z}_{+}^{-,*}(\bar{\xi},\bar{\eta})\\
&=\mathcal{Z}_{+}^{-,*}(\zeta^{-,\rm{b}}_{+}(\bar{\xi},\bar{\eta}),\textrm{m}_{+})+\int_{\zeta^{-,\rm{b}}_{+}(\bar{\xi},\bar{\eta})}^{\bar{\xi}}\Big(\frac{\p_{-}\Lambda_{+}-2\Lambda_{+}\p_{\eta}\lambda^{-}_{+}}{2\Lambda_{+}}\mathcal {Z}_{+}^{-,*}-\frac{\p_{+}\Lambda_{+}}{2\Lambda_{+}}\mathcal {Z}_{+}^{+,*}\Big)(\tau,\Upsilon^{-,\rm{b}}_{+}(\tau;\bar{\xi},\bar{\eta}))\dd \tau\\
&\quad\ +\int_{\zeta^{-,\rm{b}}_{+}(\bar{\xi},\bar{\eta})}^{\bar{\xi}}\Big(a^{-}_{+}(\underline{\mathcal{U}}_{+})\p_{\eta}\delta Y^{*}_{+}-\frac{[a^{-}_{+}(\underline{\mathcal{U}}_{+})-a^{+}_{+}(\underline{\mathcal{U}}_{+})]\p_{\eta} \Lambda_{+}}{2\Lambda_{+}}\delta Y^{*}_{+} \Big)(\tau,\Upsilon^{-,\rm{b}}_{+}(\tau;\bar{\xi},\bar{\eta}))\dd \tau\\
&\quad\ +\int_{\zeta^{-,\rm{b}}_{+}(\bar{\xi},\bar{\eta})}^{\bar{\xi}}\Big(\p_{\eta}f^{-}_{+}(\delta\mathcal{V}_{+}, \delta Y_{+})-\frac{\p_{\eta} \Lambda_{+}}{2\Lambda_{+}}\big[f^{-}_{+}(\delta\mathcal{V}_{+}, \delta Y_{+})
-f^{+}_{+}(\delta\mathcal{V}_{+}, \delta Y_{+})\big]\Big)(\tau,\Upsilon^{-,\rm{b}}_{+}(\tau;\bar{\xi},\bar{\eta}))\dd \tau.
\end{split}
\end{align}

Now, we introduce the following norm to estimate $\|(\nabla\delta \omega^{*}_{+}, \nabla\delta
p^{*}_{+})\|_{0,0;\tilde{\mathcal{N}}^{\rm{I}}_{+}\cup\tilde{\mathcal{N}}^{\rm{II}}_{+}}$.
\begin{eqnarray}\label{3.57a}
\mathcal {Z}_{\rm{II}^{+}}^{\pm,*}(\kappa)\triangleq\sup\limits_{(\xi,\eta)\in R_{+}^{\rm{II}}(\kappa)}|\mathcal {Z}_{+}^{\pm,*}(\xi,\eta)|.
\end{eqnarray}

Obviously, we have
\begin{eqnarray}\label{eq:prop-3.3-15}
\|\mathcal {Z}_{+}^{\pm,*}\|_{0,0; \tilde{\mathcal{N}}^{\rm{I}}_{+}\cup\tilde{\mathcal{N}}^{\rm{II}}_{+}}=\mathcal {Z}_{\rm{II}^{+}}^{\pm,*}(\bar{\xi}^{2}_{+})\quad\text{and}\quad\mathcal{Z}^{\pm,*}_{\rm{II}^{+}}(\kappa)\geq \mathcal {Z}^{\pm,*}_{\rm{I}^{+}}(\kappa),\quad \kappa\in [0,\bar{\xi}^{1}_{+}].
\end{eqnarray}

Therefore, for sufficiently small $\sigma>0$, by \eqref{3.57a}, we can mimic \emph{Step 2} in the proof of Proposition \ref{prop:3.2} to derive from $\eqref{eq:prop-3.2-14}_1$ that
\begin{align}\label{eq:prop-3.3-16}
\begin{split}
&\mathcal {Z}_{\rm{II}^{+}}^{+,*}(\kappa)\\
&\leq \mathcal{O}(1)\Big(\|\mathcal {Z}_{+,\rm in}^{+}\|_{0,0; \tilde{\Gamma}^{+}_{\rm in}}+\sum_{k=\pm}\|f^{k}_{+}\|_{1,0; \tilde{\mathcal{N}}^{\rm{I}}_{+}\cup\tilde{\mathcal{N}}^{\rm{II}}_{+}}+\|\delta Y^{*}_{+}\|_{1,0; \tilde{\mathcal{N}}^{\rm{I}}_{+}\cup\tilde{\mathcal{N}}^{\rm{II}}_{+}}+\int_0^{\kappa}\sum_{k=\pm}\mathcal {Z}_{\rm{II}^{+}}^{k,*}(\varsigma)\dd\varsigma\Big),
\end{split}
\end{align}
where the constant $\mathcal{O}(1)$ depends only on $\underline{\mathcal{U}}_{+}$ and $L$.

For $\mathcal {Z}_{+}^{-,*}(\bar{\xi},\bar{\eta})$, in light of \eqref{eq:prop-3.3-14}, one needs to estimate $|\mathcal {Z}_{+}^{-,*}(\zeta^{-,\rm{b}}_{+}(\bar{\xi},\bar{\eta}),\textrm{m}_{+})|$ first. 
By the boundary condition $\eqref{eq:prop-3.3-13}_{2}$ on $\tilde{\Gamma}_{+}\cap \overline{\tilde{\mathcal{N}}^{\rm{II}}_{+}}$,
\begin{eqnarray}\label{eq:prop-3.3-17}
\begin{split}
\mathcal{Z}_{+}^{-,*}(\zeta^{-,\rm{b}}_{+}(\bar{\xi},\bar{\eta}),\textrm{m}_{+})
&=-\frac{2\lambda_{+}^{+}}{\lambda_{+}^{-}}g''_{+}(\zeta^{-,\rm{b}}_{+}(\bar{\xi},\bar{\eta}))-\frac{\lambda_{+}^{+}}{\lambda_{+}^{-}}\mathcal {Z}_{+}^{+,*}(\zeta^{-,\rm{b}}_{+}(\bar{\xi},\bar{\eta}),\textrm{m}_{+})\\
&\quad\ +\big[a^{+}_{+}(\underline{\mathcal{U}}_{+})+a_{+}^{-}(\underline{\mathcal{U}}_{+})\big]\frac{\lambda_{+}^{+}}{\lambda_{+}^{-}}
\delta Y^{*}_{+}(\zeta^{-,\rm{b}}_{+}(\bar{\xi},\bar{\eta}),\textrm{m}_{+})\\
&\quad\ +\frac{\lambda_{+}^{+}}{\lambda_{+}^{-}}\Big(f^{+}_{+}(\delta\mathcal{V}_{+}, \delta Y_{+})+f^{-}_{+}(\delta\mathcal{V}_{+}, \delta Y_{+})\Big)(\zeta^{-,\rm{b}}_{+}(\bar{\xi},\bar{\eta}),\textrm{m}_{+}).
\end{split}
\end{eqnarray}

For $\mathcal{Z}_{+}^{+,*}(\zeta^{-,\rm{b}}_{+}(\bar{\xi},\bar{\eta}),\textrm{m}_{+})$ in \eqref{eq:prop-3.3-17}, by the compatibility condition $\eqref{eq:2.11}_2$, we can follow the ways in \emph{Step 2} of the proof of Proposition \ref{prop:3.2} to integrate along the characteristic curve $\eta=\Upsilon^{+}_{+}(\xi;\zeta^{-,\rm{b}}_{+}(\bar{\xi},\bar{\eta}),\textrm{m}_{+})$
which 
intersects $\tilde{\Gamma}^{+}_{\rm{in}}$ at $(0,\gamma^{+}_{+}(\zeta^{-,\rm{b}}_{+}(\bar{\xi},\bar{\eta}),\textrm{m}_{+}))$ to derive that
\begin{align}\label{eq:prop-3.3-18}
\begin{split}
&\mathcal {Z}_{+}^{+,*}(\zeta^{-,\rm{b}}_{+}(\bar{\xi},\bar{\eta}),\textrm{m}_{+})\\
& =\mathcal {Z}_{+,\rm in}^{+}(\gamma^{+}_{+}(\zeta^{-,\rm{b}}_{+}(\bar{\xi},\bar{\eta}),\textrm{m}_{+}))\\
&\quad+\int_0^{\zeta^{-,\rm{b}}_{+}(\bar{\xi},\bar{\eta})}\Big(\frac{\p_{+}\Lambda_{+}-2\Lambda_{+}\p_{\eta}\lambda^{+}_{+}}{2\Lambda_{+}}\mathcal {Z}_{+}^{+,*}-\frac{\p_{-}\Lambda_{+}}{2\Lambda_{+}}\mathcal {Z}_{+}^{-,*}\Big)(\tau,\Upsilon^{+}_{+}(\tau;\zeta^{-,\rm{b}}_{+}(\bar{\xi},\bar{\eta}),\textrm{m}_{+}))\dd \tau\\
&\quad+\int^{\zeta^{-,\rm{b}}_{+}(\bar{\xi},\bar{\eta})}_{0}\Big(a^{+}_{+}(\underline{\mathcal{U}}_{+})\p_{\eta}\delta Y^{*}_{+}-\frac{[a^{+}_{+}(\underline{\mathcal{U}}_{+})-a^{-}_{+}(\underline{\mathcal{U}}_{+})]\p_{\eta} \Lambda_{+}}{2\Lambda_{+}}\delta Y^{*}_{+}\Big)(\tau,\Upsilon^{+}_{+}(\tau;\zeta^{-,\rm{b}}_{+}(\bar{\xi},\bar{\eta}),\textrm{m}_{+}))\dd \tau\\
&\quad+\int^{\zeta^{-,\rm{b}}_{+}(\bar{\xi},\bar{\eta})}_{0}\Big(\p_{\eta}f^{+}_{+}(\delta\mathcal{V}_{+}, \delta Y_{+})-\frac{\p_{\eta} \Lambda_{+}}{2\Lambda_{+}}\big[f^{+}_{+}(\delta\mathcal{V}_{+}, \delta Y_{+})-f^{-}_{+}(\delta\mathcal{V}_{+}, \delta Y_{+})\big]\Big)\\
&\qquad\qquad \times (\tau,\Upsilon^{+}_{+}(\tau;\zeta^{-,\rm{b}}_{+}(\bar{\xi},\bar{\eta}),\textrm{m}_{+}))\dd \tau.
\end{split}
\end{align}

Then, plugging \eqref{eq:prop-3.3-18} into \eqref{eq:prop-3.3-17} and letting $\sigma>0$ sufficiently small, we obtain
\begin{eqnarray}\label{eq:prop-3.3-19}
\begin{split}
|\mathcal {Z}_{+}^{-,*}(\zeta^{-,\rm{b}}_{+}(\bar{\xi},\bar{\eta}),\textrm{m}_{+})|
&\leq  \mathcal{O}(1)\Big(\|g''_{+}\|_{1,0;\tilde{\Gamma}_{+}}+\|\mathcal {Z}_{+,\rm in}^{+}\|_{0,0; \tilde{\Gamma}^{+}_{\rm in}}
+\|\delta Y^{*}_{+}\|_{1,0; \tilde{\mathcal{N}}^{\rm{I}}_{+}\cup\tilde{\mathcal{N}}^{\rm{II}}_{+}}\Big)\\
&\quad\ + \mathcal{O}(1)\Big(\sum_{k=\pm}\|f^{k}_{+}\|_{1,0; \tilde{\mathcal{N}}^{\rm{I}}_{+}\cup\tilde{\mathcal{N}}^{\rm{II}}_{+}}+\int^{\zeta^{-,\rm{b}}_{+}(\bar{\xi},\bar{\eta})}_{0}\sum_{k=\pm}\mathcal {Z}^{k,*}_{\rm{II}^{+}}(\varsigma)\dd\varsigma\Big),
\end{split}
\end{eqnarray}
where the constants $\mathcal{O}(1)$ depend only on $\underline{\mathcal{U}}_{+}$ and $L$.
Therefore, for sufficiently small $\sigma>0$, 
\begin{eqnarray*}
\begin{split}
|\mathcal {Z}_{+}^{-,*}(\bar{\xi},\bar{\eta})|
&\leq |\mathcal {Z}_{+}^{-,*}(\zeta^{-,\rm{b}}_{+}(\bar{\xi},\bar{\eta}),\textrm{m}_{+})|+\mathcal{O}(1)\sum_{k=\pm}\|f^{k}_{+}\|_{1,0; \tilde{\mathcal{N}}^{\rm{I}}_{+}\cup\tilde{\mathcal{N}}^{\rm{II}}_{+}}\\
&\quad\ +\mathcal{O}(1)\Big\{\|\delta Y^{*}_{+}\|_{1,0; \tilde{\mathcal{N}}^{\rm{I}}_{+}\cup\tilde{\mathcal{N}}^{\rm{II}}_{+}}+\int_{\zeta^{-,\rm{b}}_{+}(\bar{\xi},\bar{\eta})}^{\bar{\xi}}
\sum_{k=\pm}\mathcal {Z}_{\rm{II}^{+}}^{k,*}(\varsigma)\dd \varsigma\Big\}\\
&\leq \mathcal{O}(1)\Big(\|g''_{+}\|_{1,0;\tilde{\Gamma}_{+}}+\|\mathcal {Z}_{+,\rm in}^{+}\|_{0,0; \tilde{\Gamma}^{+}_{\rm in}}
+\|\delta Y^{*}_{+}\|_{1,0; \tilde{\mathcal{N}}^{\rm{I}}_{+}\cup\tilde{\mathcal{N}}^{\rm{II}}_{+}}\Big)\\
&\qquad\qquad\ + \mathcal{O}(1)\Big(\sum_{k=\pm}\|f^{k}_{+}\|_{1,0; \tilde{\mathcal{N}}^{\rm{I}}_{+}\cup\tilde{\mathcal{N}}^{\rm{II}}_{+}}+\int^{\bar{\xi}}_{0}\sum_{k=\pm}\mathcal {Z}^{k,*}_{\rm{II}^{+}}(\varsigma)\dd\varsigma\Big).
\end{split}
\end{eqnarray*}
So 
\begin{eqnarray}\label{eq:prop-3.3-20}
\begin{split}
\mathcal {Z}_{\rm{II}^{+}}^{-,*}(\kappa)
&\leq \mathcal{O}(1)\Big(\|\mathcal {Z}_{+,\rm in}^{+}\|_{0,0; \tilde{\Gamma}^{+}_{\rm in}}+\|g''_{+}\|_{1,0;\tilde{\Gamma}_{+}}
+\|\delta Y^{*}_{+}\|_{1,0; \tilde{\mathcal{N}}^{\rm{I}}_{+}\cup\tilde{\mathcal{N}}^{\rm{II}}_{+}}\\
&\qquad\qquad\ + \sum_{k=\pm}\|f^{k}_{+}\|_{1,0; \tilde{\mathcal{N}}^{\rm{I}}_{+}\cup\tilde{\mathcal{N}}^{\rm{II}}_{+}}+\int^{\kappa}_{0}\sum_{k=\pm}\mathcal {Z}^{k,*}_{\rm{II}^{+}}(\varsigma)\dd\varsigma\Big),
\end{split}
\end{eqnarray}
where the constant $\mathcal{O}(1)$ depends only on $\underline{\mathcal{U}}_{+}$ and $L$.
Combining \eqref{eq:prop-3.3-16} and \eqref{eq:prop-3.3-20} and applying the Gronwall's inequality together with \eqref{eq:prop-3.3-15}, we have 
\begin{align}\label{eq:prop-3.3-21}
\begin{split}
&\sum_{k=\pm}\|\mathcal {Z}_{+}^{\pm,*}\|_{0,0;\tilde{\mathcal{N}}^{\rm{I}}_{+}\cup\tilde{\mathcal{N}}^{\rm{II}}_{+}}\\
&\leq \mathcal{O}(1)\Big\{\sum_{k=\pm}\big(\|\mathcal {Z}_{+,\rm in}^{k}\|_{0,0; \tilde{\Gamma}^{+}_{\rm in}}+\|f^{k}_{+}\|_{1,0; \tilde{\mathcal{N}}^{\rm{I}}_{+}\cup\tilde{\mathcal{N}}^{\rm{II}}_{+}}\big)+\|g''_{+}\|_{1,0;\tilde{\Gamma}_{+}}
+\|\delta Y^{*}_{+}\|_{1,0; \tilde{\mathcal{N}}^{\rm{I}}_{+}\cup\tilde{\mathcal{N}}^{\rm{II}}_{+}}\Big\}.
\end{split}
\end{align}

Then 
by \eqref{eq:prop-3.2-18}, 
we finally obtain
\begin{eqnarray}\label{eq:prop-3.3-23}
\begin{split}
\|(\nabla\delta \omega^{*}_{+},\nabla\delta p^{*}_{+})\|_{0,0; \tilde{\mathcal{N}}^{\rm{I}}_{+}\cup\tilde{\mathcal{N}}^{\rm{II}}_{+}}
&\leq \tilde{C}^{*}_{\rm{II}^{+}}\Big(\|(\nabla\delta \omega^{+}_{\rm{in}},\nabla\delta p^{+}_{\rm{in}})\|_{0,0; \tilde{\Gamma}^{+}_{\rm in}}+\|g_{+}-1\|_{2,0; \tilde{\Gamma}_{+}}\\
&\qquad\quad\ +\|\delta Y^{*}_{+}\|_{1,0; \tilde{\mathcal{N}}^{\rm{I}}_{+}\cup\tilde{\mathcal{N}}^{\rm{II}}_{+}}+\sum_{k=\pm}\|f^{k}_{+}\|_{1,0; \tilde{\mathcal{N}}^{\rm{I}}_{+}\cup\tilde{\mathcal{N}}^{\rm{II}}_{+}}\Big),
\end{split}
\end{eqnarray}
where the constant $\tilde{C}^{*}_{\rm{II}^{+}}>0$ depends only on $\underline{\mathcal{U}}_{+}$ and $L$.

3.\ \emph{Estimates on $\|(\nabla\delta \omega^{*}_{+},\nabla\delta p^{*}_{+})\|_{0,\alpha;\tilde{\mathcal{N}}^{\rm{I}}_{+}\cup\tilde{\mathcal{N}}^{\rm{II}}_{+}}$.}
Introduce semi-H\"{o}lder norms for $\mathcal {Z}^{\pm,*}_{+}$ in $\tilde{\mathcal{N}}^{\rm{I}}_{+}\cup\tilde{\mathcal{N}}^{\rm{II}}_{+}$ as follows 
\begin{eqnarray}\label{eq:prop-3.3-24}
\begin{cases}
[\mathcal {Z}^{\pm,*}_{\rm{II}^{+}}](\kappa)=\sup\limits_{(\xi_1,\eta_1),(\xi_2,\eta_2)\in R_{+}^{\rm{II}}(\kappa)}\frac{|\mathcal {Z}^{\pm,*}_{+}(\xi_1,\eta_1)-\mathcal {Z}^{\pm,*}_{+}(\xi_2,\eta_2)|}{|(\xi_1,\eta_1)-(\xi_2,\eta_2)|^{\alpha}},\\
[\mathcal {Z}^{\pm,*}_{\rm{II}^{+}}]_{\xi}(\kappa)=\sup\limits_{(\xi_1,\eta),(\xi_2,\eta)\in R_{+}^{\rm{II}}(\kappa)}\frac{|\mathcal {Z}^{\pm,*}_{+}(\xi_1,\eta)-\mathcal {Z}^{\pm,*}_{+}(\xi_2,\eta)|}{|\xi_1-\xi_2|^{\alpha}},\\
[\mathcal {Z}^{\pm,*}_{\rm{II}^{+}}]_{\eta}(\kappa)=\sup\limits_{(\xi,\eta_1),(\xi,\eta_2)\in R_{+}^{\rm{II}}(\kappa)}\frac{|\mathcal {Z}^{\pm,*}_{+}(\xi,\eta_1)-\mathcal {Z}^{\pm,*}_{+}(\xi,\eta_2)|}{|\eta_1-\eta_2|^{\alpha}}.
\end{cases}
\end{eqnarray}

Then, we have
\begin{eqnarray}\label{eq:prop-3.3-25}
\begin{split}
[\mathcal {Z}^{\pm,*}_{\rm{II}^{+}}](\kappa) \leq [\mathcal {Z}^{\pm,*}_{\rm{II}^{+}}]_{\xi}(\kappa)+[\mathcal {Z}^{\pm,*}_{\rm{II}^{+}}]_{\eta}(\kappa),
\quad [\mathcal {Z}^{\pm,*}_{\rm{II}^{+}}]_{0,\alpha; \tilde{\mathcal{N}}^{\rm{I}}_{+}\cup\tilde{\mathcal{N}}^{\rm{II}}_{+}}=[\mathcal {Z}^{\pm,*}_{\rm{II}^{+}}](\bar{\xi}^{2}_{+}),
\end{split}
\end{eqnarray}
and for any $\kappa \in[0,\bar{\xi}_+^1]$,
\begin{eqnarray}\label{eq:prop-3.3-26}
\begin{split}
[\mathcal {Z}^{\pm,*}_{\rm{II}^{+}}](\kappa) \geq [\mathcal {Z}^{\pm,*}_{\rm{I}^{+}}](\kappa),\quad
[\mathcal {Z}^{\pm,*}_{\rm{II}^{+}}]_{\xi}(\kappa) \geq [\mathcal {Z}^{\pm,*}_{\rm{I}^{+}}]_{\xi}(\kappa),\quad
[\mathcal {Z}^{\pm,*}_{\rm{II}^{+}}]_{\eta}(\kappa) \geq [\mathcal {Z}^{\pm,*}_{\rm{I}^{+}}]_{\eta}(\kappa).
\end{split}
\end{eqnarray}

Now let us estimate $[\mathcal {Z}^{\pm,*}_{\rm{II}^{+}}]_{\eta}(\kappa)$. For any $(\bar{\xi},\bar{\eta}_1),\ (\bar{\xi},\bar{\eta}_2)\in\tilde{\mathcal{N}}^{\rm{I}}_{+}\cup\tilde{\mathcal{N}}^{\rm{II}}_{+}$ with $\bar{\eta}_1\neq \bar{\eta}_2$ and $\bar{\xi}$ being fixed, there are three possible cases. 

\emph{Case 1. $(\bar{\xi}, \bar{\eta}_1), (\bar{\xi},\bar{\eta}_2)\in \tilde{\mathcal{N}}^{\rm {I}}_{+}$}. By \eqref{eq:prop-3.2-24}-\eqref{eq:prop-3.2-25}
and \eqref{eq:prop-3.3-26}, we have
\begin{eqnarray}\label{eq:prop-3.3-27}
\begin{split}
&\sum_{k=\pm}\frac{|\mathcal {Z}^{k,*}_{+}(\bar{\xi}, \bar{\eta}_1)-\mathcal {Z}^{k,*}_{+}(\bar{\xi}, \bar{\eta}_2)|}{|\bar{\eta}_1-\bar{\eta}_2|^{\alpha}}\\
&\leq \mathcal{O}(1)\Big(\sum_{k=\pm}\big(\|\mathcal {Z}^{k}_{+,\rm in}\|_{0,\alpha;\tilde{\Gamma}^{+}_{\rm in}}+\|f^{k}_{+}\|_{1,\alpha; \tilde{\mathcal{N}}^{\rm{I}}_{+}\cup\tilde{\mathcal{N}}^{\rm{II}}_{+}}\big)
+\|\delta Y^{*}_{+}\|_{1,\alpha; \tilde{\mathcal{N}}^{\rm{I}}_{+}\cup\tilde{\mathcal{N}}^{\rm{II}}_{+}}+\int_0^{\bar{\xi}}\sum_{k=\pm}[\mathcal {Z}^{k,*}_{\rm{II}^{+}}]_{\eta}(\varsigma)\dd \varsigma\Big),
\end{split}
\end{eqnarray}
where the constant $\mathcal {O}(1)$ depends only on $\underline{\mathcal{U}}_{+}$, $\alpha$ and $L$.

\emph{Case 2. $(\bar{\xi},\bar{\eta}_1), (\bar{\xi},\bar{\eta}_2)\in \tilde{\mathcal{N}}^{\rm{II}}_{+}$.} 
By \eqref{eq:prop-3.3-21}, we follow the argument for $[\mathcal {Z}^{+,*}_{\rm{I}^{+}}](\kappa)$ in \emph{Step 3} of the proof of Proposition \ref{prop:3.2} to derive that
\begin{eqnarray}\label{eq:prop-3.3-28}
\begin{split}
&\frac{|\mathcal {Z}^{+,*}_{+}(\bar{\xi},\bar{\eta}_1)-\mathcal {Z}^{+,*}_{+}(\bar{\xi},\bar{\eta}_2)|}{|\bar{\eta}_1-\bar{\eta}_2|^{\alpha}}\\
&\leq \mathcal{O}(1)\Big(\sum_{k=\pm}\big(\|\mathcal {Z}^{k}_{+,\rm in}\|_{0,\alpha;\tilde{\Gamma}^{+}_{\rm in}}+\|f^{k}_{+}\|_{1,\alpha; \tilde{\mathcal{N}}^{\rm{I}}_{+}\cup\tilde{\mathcal{N}}^{\rm{II}}_{+}}\big)
+\|\delta Y^{*}_{+}\|_{1,\alpha; \tilde{\mathcal{N}}^{\rm{I}}_{+}\cup\tilde{\mathcal{N}}^{\rm{II}}_{+}}+\int_0^{\bar{\xi}}\sum_{k=\pm}[\mathcal {Z}^{k,*}_{\rm{II}^{+}}]_{\eta}(\varsigma)\dd \varsigma\Big),
\end{split}
\end{eqnarray}
where the constant $\mathcal {O}(1)$ depends only on $\underline{\mathcal{U}}_{+}$, $\alpha$ and $L$.

Next, for $\mathcal {Z}^{-,*}_{+}$, without loss of the generality, we assume $\zeta^{-,\rm{b}}_{+}(\bar{\xi},\bar{\eta}_1)<\zeta^{-,\rm{b}}_{+}(\bar{\xi},\bar{\eta}_2)$. Applying \eqref{eq:prop-3.3-14} for points $(\bar{\xi}, \bar{\eta}_1)$ and $(\bar{\xi},\bar{\eta}_2)$, respectively, and taking difference, we have
\begin{align}\label{eq:prop-3.3-29}
\begin{split}
\big|\mathcal {Z}^{-,*}_{+}(\bar{\xi},\bar{\eta}_1)-\mathcal {Z}^{-,*}_{+}(\bar{\xi},\bar{\eta}_2)\big|\leq\sum_{k=1}^{3}\mathcal {J}^{-}_{{\rm II}^{+}_{k}},
\end{split}
\end{align}
where
{
\begin{align*}
\begin{split}
\mathcal {J}^{-}_{\rm{II}^{+}_{1}}&\triangleq\big|\mathcal {Z}^{-,*}_{+}(\zeta^{-,\rm{b}}_{+}(\bar{\xi},\bar{\eta}_2),\textrm{m}_{+})-\mathcal {Z}^{+,*}_{+}(\zeta^{-,\rm{b}}_{+}(\bar{\xi},\bar{\eta}_1),\textrm{m}_{+})\big|,
\end{split}
\end{align*}
and
\begin{align*}
\begin{split}
&\mathcal {J}^{-}_{\rm{II}^{+}_{2}}
\triangleq
\int_{\zeta^{-,\rm{b}}_{+}(\bar{\xi},\bar{\eta}_2)}^{\bar{\xi}}\Big|\big(\frac{\p_{-}\Lambda_{+}-2\Lambda_{+}\p_{\eta}\lambda^{-}_{+}}{2\Lambda_{+}}\mathcal {Z}_{+}^{-,*}\big)(\tau,\Upsilon^{-,\textrm{b}}_{+}(\tau;\bar{\xi},\bar{\eta}_1))\\
&\qquad\qquad\qquad\qquad -\big(\frac{\p_{+}\Lambda_{+}-2\Lambda_{+}\p_{\eta}\lambda^{+}_{+}}{2\Lambda_{+}}\mathcal {Z}_{+}^{-,*}\big)(\tau,\Upsilon^{-,\textrm{b}}_{+}(\tau;\bar{\xi},\bar{\eta}_2))\Big|\dd\tau\\
&+\int_{\zeta^{-,\rm{b}}_{+}(\bar{\xi},\bar{\eta}_2)}^{\bar{\xi}}\Big|\big(\frac{\p_{-}\Lambda_{+}}{2\Lambda_{+}}\mathcal {Z}_{+}^{-,*}\big)(\tau,\Upsilon^{-,\textrm{b}}_{+}(\tau;\bar{\xi},\bar{\eta}_1))-\big(\frac{\p_{-}\Lambda_{+}}{2\Lambda_{+}}\mathcal {Z}_{+}^{-,*}\big)(\tau,\Upsilon^{-,,\textrm{b}}_{+}(\tau;\bar{\xi},\bar{\eta}_2))\Big|\dd\tau\\
&+|a^{-}_{+}(\underline{\mathcal{U}}_{+})|\int_{\zeta^{-,\rm{b}}_{+}(\bar{\xi},\bar{\eta}_1)}^{\bar{\xi}}\Big|\p_{\eta}\delta Y^{*}_{+}(\tau,\Upsilon^{-,\textrm{b}}_{+}(\tau;\bar{\xi},\bar{\eta}_1))
-\p_{\eta}\delta Y^{*}_{+}(\tau,\Upsilon^{-,\textrm{b}}_{+}(\tau;\bar{\xi},\bar{\eta}_2))\Big|\dd \tau\\
&+\frac{|a^{+}_{+}(\underline{\mathcal{U}}_{+})|+|a^{-}_{+}(\underline{\mathcal{U}}_{+})|}{2}\\
&\qquad \times\int_{\zeta^{-,\rm{b}}_{+}(\bar{\xi},\bar{\eta}_1)}^{\bar{\xi}}
\Big|\big(\frac{\p_{\eta}\Lambda_{+}}{\Lambda_{+}}Y_{+}^{*}\big)(\tau,\Upsilon^{+}_{+}(\tau;\bar{\xi},\bar{\eta}_1))
-\big(\frac{\p_{\eta}\Lambda_{+}}{\Lambda_{+}}Y_{+}^{*}\big)(\tau,\Upsilon^{-,\textrm{b}}_{+}(\tau;\bar{\xi},\bar{\eta}_2))\Big|\dd\tau\\
&+\int_{\zeta^{-,\rm{b}}_{+}(\bar{\xi},\bar{\eta}_2)}^{\bar{\xi}}\big|\p_{\eta}f^{-}_{+}(\delta\mathcal{V}_{+}, \delta Y_{+})(\tau,\Upsilon^{-,\textrm{b}}_{+}(\tau;\bar{\xi},\bar{\eta}_1))
-\p_{\eta}f^{-}_{+}(\delta\mathcal{V}_{+}, \delta Y_{+})(\tau,\Upsilon^{-,\textrm{b}}_{+}(\tau;\bar{\xi},\bar{\eta}_2))\big|\dd \tau\\
&+\sum_{k=\pm}\int_{\zeta^{-,\rm{b}}_{+}(\bar{\xi},\bar{\eta}_2)}^{\bar{\xi}}\Big|\big(\frac{\p_{\eta} \Lambda_{+}}{2\Lambda_{+}}f^{k}_{+}(\delta\mathcal{V}_{+}, \delta Y_{+})\big)(\tau,\Upsilon^{-,\textrm{b}}_{+}(\tau;\bar{\xi},\bar{\eta}_1))\\
&\qquad\qquad\qquad\qquad\qquad -\big(\frac{\p_{\eta} \Lambda_{+}}{2\Lambda_{+}}f^{k}_{+}(\delta\mathcal{V}_{+}, \delta Y_{+})\big)(\tau,\Upsilon^{-,\textrm{b}}_{+}(\tau;\bar{\xi},\bar{\eta}_2))\Big|\dd \tau,
\end{split}
\end{align*}
}
and
\begin{align*}
\begin{split}
&\mathcal{J}^{-}_{\rm{II}^{+}_{3}}\triangleq\int_{\zeta^{-,\rm{b}}_{+}(\bar{\xi},\bar{\eta}_1)}^{\zeta^{-,\rm{b}}_{+}(\bar{\xi},\bar{\eta}_2)}
\Big|\Big(\frac{\p_{-}\Lambda_{+}-2\Lambda_{+}\p_{\eta}\lambda^{+}_{+}}{2\Lambda_{+}}\mathcal {Z}_{+}^{-,*}-\frac{\p_{+}\Lambda_{+}}{2\Lambda_{+}}\mathcal {Z}_{+}^{+,*}\Big)(\tau,\Upsilon^{-,\textrm{b}}_{+}(\tau;\bar{\xi},\bar{\eta}_{2}))\Big|\dd \tau\\
&+\int_{\zeta^{-,\rm{b}}_{+}(\bar{\xi},\bar{\eta}_1)}^{\zeta^{-,\rm{b}}_{+}(\bar{\xi},\bar{\eta}_2)}\Big|\big(a^{-}_{+}(\underline{\mathcal{U}}_{+})\p_{\eta}\delta Y^{*}_{+}-\frac{[a^{-}_{+}(\underline{\mathcal{U}}_{+})-a^{+}_{+}(\underline{\mathcal{U}}_{+})]\p_{\eta} \Lambda_{+}}{2\Lambda_{+}}\delta Y^{*}_{+} \Big)(\tau,\Upsilon^{-,\textrm{b}}_{+}(\tau;\bar{\xi},\bar{\eta}_2))\Big|\dd \tau\\
&+\int_{\zeta^{-,\rm{b}}_{+}(\bar{\xi},\bar{\eta}_1)}^{\zeta^{-,\rm{b}}_{+}(\bar{\xi},\bar{\eta}_2)}\Big|\Big(\p_{\eta}f^{-}_{+}(\delta\mathcal{V}_{+}, \delta Y_{+})-\frac{\p_{\eta} \Lambda_{+}}{2\Lambda_{+}}\big[f^{-}_{+}(\delta\mathcal{V}_{+}, \delta Y_{+})
-f^{+}_{+}(\delta\mathcal{V}_{+}, \delta Y_{+})\big]\Big)(\tau,\Upsilon^{-,\textrm{b}}_{+}(\tau;\bar{\xi},\bar{\eta}_{2}))\Big| \dd \tau.
\end{split}
\end{align*}

By Lemma \ref{lem:A2} and following the argument in \emph{Step 3} of the proof of Proposition \ref{prop:3.2}, we have
\begin{eqnarray*}
\begin{split}
\mathcal{J}^{-}_{\rm{II}^{+}_{2}}&\leq\mathcal {O}(1)\Big\{\sum_{k=\pm}\|\mathcal {Z}_{+}^{\pm,*}\|_{0,\alpha; \tilde{\mathcal{N}}^{\rm{II}}_{+}}
+\|\p_{\eta}f^{-}_{+}\|_{0,\alpha; \tilde{\mathcal{N}}^{\rm{II}}_{+}}
+\|\p_{\eta}f^{+}_{+}\|_{0,0; \tilde{\mathcal{N}}^{\rm{II}}_{+}}\\
&\qquad\qquad\  +\|\delta Y^{*}_{+}\|_{1,\alpha; \tilde{\mathcal{N}}^{\rm{II}}_{+}}+\int_{\zeta^{-,\rm{b}}_{+}(\bar{\xi},\bar{\eta}_2)}^{\bar{\xi}}\sum_{k=\pm}[\mathcal {Z}^{k,*}_{\rm{II}^{+}}]_{\eta}(\varsigma)\dd \varsigma\Big\}\cdot \big|\bar{\eta}_{1}-\bar{\eta}_{2}\big|^{\alpha},
\end{split}
\end{eqnarray*}
and
\begin{align*}
\begin{split}
\mathcal{J}^{-}_{\rm{II}^{+}_{3}}&\leq\mathcal {O}(1)\Big\{\sum_{k=\pm}\|\mathcal {Z}_{+}^{\pm,*}\|_{0,0; \tilde{\mathcal{N}}^{\rm{II}}_{+}}
+\|f^{-}_{+}\|_{1,0; \tilde{\mathcal{N}}^{\rm{II}}_{+}}+\|f^{+}_{+}\|_{0,0; \tilde{\mathcal{N}}^{\rm{II}}_{+}}
+\|\delta Y^{*}_{+}\|_{1,0; \tilde{\mathcal{N}}^{\rm{II}}_{+}}\Big\}\\
&\qquad\quad \times \big|\zeta^{-,\rm{b}}_{+}(\bar{\xi},\bar{\eta}_1)-\zeta^{-,\rm{b}}_{+}(\bar{\xi},\bar{\eta}_2)\big|\\
&\leq\mathcal {O}(1)\Big\{\sum_{k=\pm}\|\mathcal {Z}_{+}^{\pm,*}\|_{0,0; \tilde{\mathcal{N}}^{\rm{II}}_{+}}+\|f^{-}_{+}\|_{1,0; \tilde{\mathcal{N}}^{\rm{II}}_{+}}+\|f^{+}_{+}\|_{0,0; \tilde{\mathcal{N}}^{\rm{II}}_{+}}
+\|\delta Y^{*}_{+}\|_{1,0; \tilde{\mathcal{N}}^{\rm{II}}_{+}}\Big\}\cdot\big|\bar{\eta}_{1}-\bar{\eta}_{2}\big|^{\alpha},
\end{split}
\end{align*}
where the constants $\mathcal {O}(1)$ depend only on $\underline{\mathcal{U}}_{+}$, $\alpha$ and $L$.

For $\mathcal {J}^{-}_{\rm{II}^{+}_{1}}$, by boundary condition \eqref{eq:prop-3.3-17}, we know that
\begin{eqnarray*}
\begin{split}
\mathcal {J}^{-}_{\rm{II}^{+}_{1}}&\leq 2\bigg|\Big(\frac{\lambda_{+}^{+}}{\lambda_{+}^{-}}g''_{+}\Big)(\zeta^{-,\rm{b}}_{+}(\bar{\xi},\bar{\eta}_1))
-\Big(\frac{\lambda_{+}^{+}}{\lambda_{+}^{-}}g''_{+}\Big)(\zeta^{-,\rm{b}}_{+}(\bar{\xi},\bar{\eta}_2))\bigg|\\
&\quad\ +\sum_{k=\pm}|a^{k}_{+}(\underline{\mathcal{U}}_{+})|\bigg|\Big(\frac{\lambda_{+}^{+}}{\lambda_{+}^{-}}\delta Y^{*}_{+}\Big)(\zeta^{-,\rm{b}}_{+}(\bar{\xi},\bar{\eta}_1), \textrm{m}_{+})
-\Big(\frac{\lambda_{+}^{+}}{\lambda_{+}^{-}}\delta Y^{*}_{+}\Big)(\zeta^{-,\rm{b}}_{+}(\bar{\xi},\bar{\eta}_2), \textrm{m}_{+})\bigg|\\
&\quad\ +\sum_{k=\pm}\bigg|\Big(\frac{\lambda_{+}^{+}}{\lambda_{+}^{-}}f^{k}_{+}(\delta\mathcal{V}_{+}, \delta Y_{+})\Big)(\zeta^{-,\rm{b}}_{+}(\bar{\xi},\bar{\eta}_1), \textrm{m}_{+})
-\Big(\frac{\lambda_{+}^{+}}{\lambda_{+}^{-}}f^{k}_{+}(\delta\mathcal{V}_{+}, \delta Y_{+})\Big)(\zeta^{-,\rm{b}}_{+}(\bar{\xi},\bar{\eta}_2), \textrm{m}_{+})\bigg|\\
&\quad\ +\bigg|\Big(\frac{\lambda_{+}^{+}}{\lambda_{+}^{-}}\mathcal{Z}^{+,*}_{+}\Big)(\zeta^{-,\rm{b}}_{+}(\bar{\xi},\bar{\eta}_1), \textrm{m}_{+})-\Big(\frac{\lambda_{+}^{+}}{\lambda_{+}^{-}}\mathcal{Z}^{+,*}_{+}\Big)(\zeta^{-,\rm{b}}_{+}(\bar{\xi},\bar{\eta}_2), \textrm{m}_{+})\bigg|\\
&\triangleq \sum^{4}_{\rm{i}=1}\mathcal {J}^{-}_{\rm{II}^{+}_{1,i}}.
\end{split}
\end{eqnarray*}

For sufficiently small $\sigma>0$, by Lemma \ref{lem:A2}, we have
\begin{eqnarray*}
\begin{split}
\mathcal{J}^{-}_{\rm{II}^{+}_{1,1}}
&\leq2\Big\|\p_{\xi}\big(\frac{\lambda_{+}^{+}}{\lambda_{+}^{-}}\big)\Big\|_{0,0;\tilde{\Gamma}_{+}}
\|\p_{\eta}\zeta^{-,\rm{b}}_{+}\|_{0,0;\tilde{\mathcal{N}}^{\rm{II}}_{+}}\|g''_{+}\|_{0,0;\tilde{\Gamma}_{+}}\cdot|\bar{\eta}_1-\bar{\eta}_{2}|\\
&\quad\ +2\Big\|\frac{\lambda_{+}^{+}}{\lambda_{+}^{-}}\Big\|_{0,0;\tilde{\Gamma}_{+}}\|\p_{\eta}\zeta^{-,\rm{b}}_{+}\|^{\alpha}_{0,0;\tilde{\mathcal{N}}^{\rm{II}}_{+}}\|g''_{+}\|_{0,\alpha;\tilde{\Gamma}_{+}}\cdot |\bar{\eta}_1-\bar{\eta}_{2}|^{\alpha}\\
&\leq \mathcal {O}(1)\|g''_{+}\|_{0,\alpha;\tilde{\Gamma}_{+}}\cdot|\bar{\eta}_1-\bar{\eta}_{2}|^{\alpha},
\end{split}
\end{eqnarray*}
and
\begin{eqnarray*}
\begin{split}
\mathcal{J}^{-}_{\rm{II}^{+}_{1,2}}+\mathcal{J}^{-}_{\rm{II}^{+}_{1,3}}
\leq \mathcal {O}(1)\Big(\|\nabla\delta Y^{*}_{+}\|_{0,0;\tilde{\mathcal{N}}^{\rm{II}}_{+}}
+\sum_{k=\pm}\|f^{k}_{+}\|_{1,0;\tilde{\mathcal{N}}^{\rm{II}}_{+}}\Big)\cdot|\bar{\eta}_1-\bar{\eta}_{2}|^{\alpha},
\end{split}
\end{eqnarray*}
where the constants $\mathcal {O}(1)$ depend only on $\underline{\mathcal{U}}_{+}$, $\alpha$ and $L$.

For the last term $\mathcal {J}^{-}_{\rm{II}^{+}_{1,4}}$, we employ \eqref{eq:prop-3.3-18} and \eqref{eq:prop-3.3-21} and follow the argument in \emph{Step 3} in the proof of Proposition \ref{prop:3.2} for $[\mathcal {Z}^{+,*}_{\rm{I}^{+}}](\kappa)$ to derive that
\begin{eqnarray*}
\begin{split}
\mathcal{J}^{-}_{\rm{II}^{+}_{1,4}}
&\leq2\Big\|\p_{\xi}\big(\frac{\lambda_{+}^{+}}{\lambda_{+}^{-}}\big)\Big\|_{0,0;\tilde{\Gamma}_{+}}
\|\p_{\eta}\zeta^{-,\rm{b}}_{+}\|_{0,0;\tilde{\mathcal{N}}^{\rm{II}}_{+}}\|\mathcal{Z}^{+,*}_{+}\|_{0,0;\tilde{\mathcal{N}}^{\rm{II}}_{+}}\cdot|\bar{\eta}_1-\bar{\eta}_{2}|\\
&\quad\ +2\Big\|\frac{\lambda_{+}^{+}}{\lambda_{+}^{-}}\Big\|_{0,0;\tilde{\Gamma}_{+}}
\big|\mathcal{Z}^{+,*}_{+}(\zeta^{-,\rm{b}}_{+}(\bar{\xi},\bar{\eta}_1), \textrm{m}_{+})-\mathcal{Z}^{+,*}_{+}(\zeta^{-,\rm{b}}_{+}(\bar{\xi},\bar{\eta}_2), \textrm{m}_{+})\big|\\
&\leq \mathcal{O}(1)\Big(\|\mathcal {Z}^{+}_{+,\rm in}\|_{0,\alpha;\tilde{\Gamma}^{+}_{\rm in}}+\sum_{k=\pm}\|\mathcal {Z}^{k,*}_{+}\|_{0,0;\tilde{\mathcal{N}}^{\rm{I}}_{+}\cup\tilde{\mathcal{N}}^{\rm{II}}_{+}}+\|\p_{\eta}f^{+}_{+}\|_{0,\alpha; \tilde{\mathcal{N}}^{\rm{I}}_{+}\cup\tilde{\mathcal{N}}^{\rm{II}}_{+}}+\|\p_{\eta}f^{-}_{+}\|_{0,0; \tilde{\mathcal{N}}^{\rm{I}}_{+}\cup\tilde{\mathcal{N}}^{\rm{II}}_{+}}\\
&\qquad\qquad\  +\|\delta Y^{*}_{+}\|_{1,\alpha; \tilde{\mathcal{N}}^{\rm{I}}_{+}\cup\tilde{\mathcal{N}}^{\rm{II}}_{+}}+\int_0^{\zeta^{-,\rm{b}}_{+}(\bar{\xi},\bar{\eta}_1)}\sum_{k=\pm}[\mathcal {Z}^{k,*}_{\rm{II}^{+}}]_{\eta}(\varsigma)\dd \varsigma\Big)\cdot |\bar{\eta}_1-\bar{\eta}_{2}|^{\alpha},
\end{split}
\end{eqnarray*}
where the constant $\mathcal {O}(1)$ depends only on $\underline{\mathcal{U}}_{+}$, $\alpha$ and $L$.

Combining the estimates on $\mathcal {J}^{-}_{\rm{II}^{+}_{1,\rm{i}}}$, $(\rm{i}=1,2, 3,4)$, we can deduce that
\begin{eqnarray*}
\begin{split}
\mathcal {J}^{-}_{\rm{II}^{+}_{1}}&\leq\mathcal{O}(1)\Big(\|\mathcal {Z}^{+}_{+,\rm in}\|_{0,\alpha;\tilde{\Gamma}^{+}_{\rm in}}+\|g''_{+}\|_{0,\alpha;\tilde{\Gamma}_{+}}+\|f^{+}_{+}\|_{1,\alpha; \tilde{\mathcal{N}}^{\rm{I}}_{+}\cup\tilde{\mathcal{N}}^{\rm{II}}_{+}}
+\|f^{-}_{+}\|_{1,0; \tilde{\mathcal{N}}^{\rm{I}}_{+}\cup\tilde{\mathcal{N}}^{\rm{II}}_{+}}\\
&\quad\ +\sum_{k=\pm}\|\mathcal {Z}^{k,*}_{+}\|_{0,0;\tilde{\mathcal{N}}^{\rm{I}}_{+}\cup\tilde{\mathcal{N}}^{\rm{II}}_{+}}+\|\delta Y^{*}_{+}\|_{1,\alpha; \tilde{\mathcal{N}}^{\rm{I}}_{+}\cup\tilde{\mathcal{N}}^{\rm{II}}_{+}}+\int_0^{\zeta^{-,\rm{b}}_{+}(\bar{\xi},\bar{\eta}_1)}\sum_{k=\pm}[\mathcal {Z}^{k,*}_{\rm{II}^{+}}]_{\eta}(\varsigma)\dd \varsigma\Big)
 \cdot |\bar{\eta}_1-\bar{\eta}_{2}|^{\alpha}.
\end{split}
\end{eqnarray*}

Hence substituting the estimates of $\mathcal {J}^{-}_{\rm{II}^{+}_{i}}$, $(\rm{i}=1,2,3,4)$
into \eqref{eq:prop-3.3-29}, we have
\begin{align}\label{eq:prop-3.3-30}
\begin{split}
&\frac{\big|\mathcal {Z}^{-,*}_{+}(\bar{\xi},\bar{\eta}_1)-\mathcal {Z}^{-,*}_{+}(\bar{\xi},\bar{\eta}_2)\big|}{|\bar{\eta}_1-\bar{\eta}_{2}|^{\alpha}}\\
\leq&\mathcal{O}(1)\Big(\|\mathcal {Z}^{+}_{+,\rm in}\|_{0,\alpha;\tilde{\Gamma}^{+}_{\rm in}}+\|g''_{+}\|_{0,\alpha;\tilde{\Gamma}_{+}}
+\|\delta Y^{*}_{+}\|_{1,\alpha; \tilde{\mathcal{N}}^{\rm{I}}_{+}\cup\tilde{\mathcal{N}}^{\rm{II}}_{+}}\\
&\qquad\quad +\sum_{k=\pm}\big(\|\mathcal {Z}^{k,*}_{+}\|_{0,0;\tilde{\mathcal{N}}^{\rm{I}}_{+}\cup\tilde{\mathcal{N}}^{\rm{II}}_{+}}+\|f^{k}_{+}\|_{1,\alpha; \tilde{\mathcal{N}}^{\rm{I}}_{+}\cup\tilde{\mathcal{N}}^{\rm{II}}_{+}}\big)+\int_0^{\bar{\xi}}\sum_{k=\pm}[\mathcal {Z}^{k,*}_{\rm{II}^{+}}]_{\eta}(\varsigma)\dd \varsigma\Big),
\end{split}
\end{align}
where the constant $\mathcal {O}(1)$ depends only on $\underline{\mathcal{U}}_{+}$, $\alpha$ and $L$.

\emph{Case 3.  $(\bar{\xi},\bar{\eta}_1)\in \tilde{\mathcal{N}}^{\rm{I}}_{+}$ and $(\bar{\xi},\bar{\eta}_2)\in \tilde{\mathcal{N}}^{\rm{II}}_{+}$}.
In this case, $\bar{\eta}_1\leq\Upsilon^{-}_{+}(\bar{\xi};\bar{\xi}^{1}_{+},\bar{\eta}^{1}_{+})\leq\bar{\eta}_2$
for $\bar{\xi} \in[0, \bar{\xi}^{1}_{+}]$. Thus, we can apply the estimates \eqref{eq:prop-3.3-27}, \eqref{eq:prop-3.3-28} and \eqref{eq:prop-3.3-30} to obtain
\begin{align}\label{eq:prop-3.3-31}
\begin{split}
&\sum_{k=\pm}\frac{|\mathcal {Z}^{k,*}_{+}(\bar{\xi},\bar{\eta}_1)-\mathcal {Z}^{k,*}_{+}(\bar{\xi},\bar{\eta}_2)|}{|\bar{\eta}_1-\bar{\eta}_2|^{\alpha}}\\
& \leq\sum_{k=\pm}\frac{|\mathcal {Z}^{k,*}_{+}(\bar{\xi},\bar{\eta}_1)
-\mathcal {Z}^{k,*}_{+}(\bar{\xi},\Upsilon^{-}_{+}(\bar{\xi};\bar{\xi}^{1}_{+},\bar{\eta}^{1}_{+}))|}{|\bar{\eta}_1-\Upsilon^{-}_{+}(\bar{\xi};\bar{\xi}^{1}_{+},\bar{\eta}^{1}_{+}))|^{\alpha}}
+\sum_{k=\pm}\frac{|\mathcal {Z}^{k,*}_{+}(\bar{\xi},\Upsilon^{-}_{+}(\bar{\xi};\bar{\xi}^{1}_{+},\bar{\eta}^{1}_{+}))-\mathcal {Z}^{k,*}_{+}(\bar{\xi},\bar{\eta}_2)|}{|\bar{\eta}_2-\Upsilon^{-}_{+}(\bar{\xi};\bar{\xi}^{1}_{+},\bar{\eta}^{1}_{+}))|^{\alpha}}\\
&\leq \mathcal{O}(1)\Big(\|\mathcal {Z}^{+}_{+,\rm in}\|_{0,\alpha;\tilde{\Gamma}^{+}_{\rm in}}+\|g''_{+}\|_{0,\alpha;\tilde{\Gamma}_{+}}
+\|\delta Y^{*}_{+}\|_{1,\alpha; \tilde{\mathcal{N}}^{\rm{I}}_{+}\cup\tilde{\mathcal{N}}^{\rm{II}}_{+}}\\
&\qquad\qquad\ +\sum_{k=\pm}\big(\|\mathcal {Z}^{k,*}_{+}\|_{0,0;\tilde{\mathcal{N}}^{\rm{I}}_{+}\cup\tilde{\mathcal{N}}^{\rm{II}}_{+}}+\|f^{k}_{+}\|_{1,\alpha; \tilde{\mathcal{N}}^{\rm{I}}_{+}\cup\tilde{\mathcal{N}}^{\rm{II}}_{+}}\big)+\int_0^{\bar{\xi}}\sum_{k=\pm}[\mathcal {Z}^{k,*}_{\rm{II}^{+}}]_{\eta}(\varsigma)\dd \varsigma\Big),
\end{split}
\end{align}
where the constant $\mathcal {O}(1)$ depends only on $\underline{\mathcal{U}}_{+}$, $\alpha$ and $L$.

With the estimates \eqref{eq:prop-3.3-21}, \eqref{eq:prop-3.3-27}, \eqref{eq:prop-3.3-28}, \eqref{eq:prop-3.3-30} and \eqref{eq:prop-3.3-31}, we can get by the Gronwall's inequality that
\begin{align*}
\begin{split}
&\sum_{k=\pm}[\mathcal {Z}^{k,*}_{\rm{II}^{+}}]_{\eta}(\kappa)\leq \mathcal{O}(1)\Big(\sum_{k=\pm}\big(\|\mathcal {Z}^{k}_{+,\rm in}\|_{0,\alpha;\tilde{\Gamma}^{+}_{\rm in}}+\|f^{k}_{+}\|_{1,\alpha; \tilde{\mathcal{N}}^{\rm{I}}_{+}\cup\tilde{\mathcal{N}}^{\rm{II}}_{+}}\big)
+\|g''_{+}\|_{0,\alpha;\tilde{\Gamma}_{+}}+\|\delta Y^{*}_{+}\|_{1,\alpha; \tilde{\mathcal{N}}^{\rm{I}}_{+}\cup\tilde{\mathcal{N}}^{\rm{II}}_{+}}\Big),
\end{split}
\end{align*}
where the constant $\mathcal {O}(1)$ depends only on $\underline{\mathcal{U}}_{+}$, $\alpha$ and $L$.

For the estimate of $[\mathcal {Z}^{\pm,*}_{\rm{II}^{+}}]_{\xi}(\kappa)$, we can also follow the argument in \emph{Step 3} in the proof of Proposition \ref{prop:3.2} and in the proof of $[\mathcal {Z}^{\pm,*}_{\rm{II}^{+}}]_{\eta}(\kappa)$ above, by taking any two points $(\bar{\xi}_1,\bar{\eta}),\ (\bar{\xi}_2,\bar{\eta})\in \tilde{\mathcal{N}}^{\rm{I}}_{+}\cup\tilde{\mathcal{N}}^{\bar{II}}_{+}$ with $\bar{\eta}$ being fixed and $\bar{\xi}_1\neq \bar{\xi}_2$, 
to derive that
\begin{eqnarray*}
\begin{split}
&\sum_{k=\pm}[\mathcal {Z}^{k,*}_{\rm{II}^{+}}]_{\xi}(\kappa)\leq \mathcal{O}(1)\Big(\sum_{k=\pm}\big(\|\mathcal {Z}^{k}_{+,\rm in}\|_{0,\alpha;\tilde{\Gamma}^{+}_{\rm in}}+\|f^{k}_{+}\|_{1,\alpha; \tilde{\mathcal{N}}^{\rm{I}}_{+}\cup\tilde{\mathcal{N}}^{\rm{II}}_{+}}\big)
+\|g''_{+}\|_{0,\alpha;\tilde{\Gamma}_{+}}+\|\delta Y^{*}_{+}\|_{1,\alpha; \tilde{\mathcal{N}}^{\rm{I}}_{+}\cup\tilde{\mathcal{N}}^{\rm{II}}_{+}}\Big),
\end{split}
\end{eqnarray*}
where the constant $\mathcal {O}(1)$ depends only on $\underline{\mathcal{U}}_{+}$, $\alpha$ and $L$.

Therefore, by \eqref{eq:prop-3.3-25} and the estimates of $[\mathcal {Z}^{\pm,*}_{\rm{II}^{+}}]_{\eta}(\kappa)$ and $[\mathcal {Z}^{\pm,*}_{\rm{II}^{+}}]_{\xi}(\kappa)$,
we let $\kappa=\bar{\xi}^{2}_{+}$ to have
\begin{eqnarray*}
\begin{split}
&\sum_{k=\pm}[\mathcal {Z}^{k,*}_{+}]_{0,\alpha;\tilde{\mathcal{N}}^{\rm{I}}_{+}\cup\tilde{\mathcal{N}}^{\rm{II}}_{+}}\\
&\leq\mathcal{O}(1)\Big(\sum_{k=\pm}\big(\|\mathcal {Z}^{k}_{+,\rm in}\|_{0,\alpha;\tilde{\Gamma}^{+}_{\rm in}}+\|f^{k}_{+}\|_{1,\alpha; \tilde{\mathcal{N}}^{\rm{I}}_{+}\cup\tilde{\mathcal{N}}^{\rm{II}}_{+}}\big)
+\|g''_{+}\|_{0,\alpha;\tilde{\Gamma}_{+}}+\|\delta Y^{*}_{+}\|_{1,\alpha; \tilde{\mathcal{N}}^{\rm{I}}_{+}\cup\tilde{\mathcal{N}}^{\rm{II}}_{+}}\Big),
\end{split}
\end{eqnarray*}
where the constant $\mathcal {O}(1)$ depends only on $\underline{\mathcal{U}}_{+}$, $\alpha$ and $L$.

Then, it follows from \eqref{eq:prop-3.2-12} that
\begin{eqnarray}\label{eq:prop-3.3-34}
\begin{split}
&[(\nabla\delta \omega^{*}_{+}, \nabla\delta p^{*}_{+})]_{0,\alpha,\tilde{\mathcal{N}}^{\rm{I}}_{+}\cup\tilde{\mathcal{N}}^{\rm{II}}_{+}}\\
&\leq \tilde{C}^{**}_{\rm{II}^{+}}\Big(\|\nabla(\delta\omega^{+}_{\rm in}, \delta p^{+}_{\rm in})\|_{0,\alpha;\tilde{\Gamma}^{+}_{\rm in}}+\|g''_{+}\|_{0,\alpha;\tilde{\Gamma}_{+}}
+\|\delta Y^{*}_{+}\|_{1,\alpha; \tilde{\mathcal{N}}^{\rm{I}}_{+}\cup\tilde{\mathcal{N}}^{\rm{II}}_{+}}
+\sum_{k=\pm}\|f^{k}_{+}\|_{1,\alpha; \tilde{\mathcal{N}}^{\rm{I}}_{+}\cup\tilde{\mathcal{N}}^{\rm{II}}_{+}}\Big),
\end{split}
\end{eqnarray}
where the constant $\tilde{C}^{**}_{\rm{II}^{+}}>0$ depends only on $\underline{\mathcal{U}}_{+}$, $\alpha$ and $L$.

Finally, combining the estimates \eqref{eq:prop-3.3-12},\eqref{eq:prop-3.3-23} and \eqref{eq:prop-3.3-34}, and choosing $$\tilde{C}_{33}=\max\{\tilde{C}_{\rm{II}^{+}}, \tilde{C}^{*}_{\rm{II}^{+}}, \tilde{C}^{**}_{\rm{II}^{+}} \},$$ we can then establish the estimate \eqref{eq:prop-3.3-1}.
\end{proof}

At last, we will consider the problem $(\mathbf{IBVP})^{*}$ in $\tilde{\mathcal{N}}^{\rm{I}}_{\pm}\cup\tilde{\mathcal{N}}^{\rm{III}}_{\pm}$ 
with the contact discontinuity as an interface inside.
\begin{eqnarray}\label{eq:IBVP-III}
(\mathbf{IBVP})^{*}_{\rm{III}}\
\begin{cases}
\partial_{\xi}\delta\omega^{*}_{+}+\lambda^{\pm}_{+}\partial_{\eta}\delta\omega^{*}_{+}\pm\Lambda_{+}\big(\partial_{\xi}\delta p^{*}_{+}+\lambda^{\pm}_{+}\partial_{\eta}\delta p^{*}_{+}\big)\\
\qquad\qquad\qquad\qquad =a^{\pm}_{+}(\underline{\mathcal{U}}_{+})\delta Y^{*}_{+}+f^{\pm}_{+}(\delta\mathcal{V}_{+}, \delta Y_{+}), &\quad  \mbox{in} \quad \tilde{\mathcal{N}}^{\rm{I}}_{+}\cup\tilde{\mathcal{N}}^{\rm{III}}_{+}, \\
\partial_{\xi}\delta\omega^{*}_{-}+\lambda^{\pm}_{-}\partial_{\eta}\delta\omega^{*}_{-}\pm\Lambda_{-}\big(\partial_{\xi}\delta p^{*}_{-}+\lambda^{\pm}_{-}\partial_{\eta}\delta p^{*}_{-}\big)\\
\qquad\qquad\qquad\qquad =a^{\pm}_{-}(\underline{\mathcal{U}}_{-})\delta Y^{*}_{-}+f^{\pm}_{-}(\delta\mathcal{V}_{-}, \delta Y_{-}), &\quad \mbox{in}\quad \tilde{\mathcal{N}}^{\rm{I}}_{-}\cup\tilde{\mathcal{N}}^{\rm{III}}_{-},\\
(\delta\omega^{*}_{+}, \delta p^{*}_{+})=(\delta\omega^{+}_{\rm in},\delta p^{+}_{\rm in}), &\quad  \mbox{on} \quad \tilde{\Gamma}^{+}_{\rm in},\\
(\delta\omega^{*}_{-}, \delta p^{*}_{-})=(\delta\omega^{-}_{\rm in},\delta p^{-}_{\rm in}), &\quad  \mbox{on} \quad \tilde{\Gamma}^{-}_{\rm in},\\
\delta\omega^{*}_{+}=\delta\omega^{*}_{-},\quad \delta p^{*}_{+}=\delta p^{*}_{-}, &\quad \mbox{on} \quad \tilde{\Gamma}_{\rm cd}\cap\overline{\tilde{\mathcal{N}}^{\rm{III}}_{\pm}}.
\end{cases}
\end{eqnarray}
We have the following proposition. 
\begin{proposition}\label{prop:3.4}
For sufficiently small $\sigma>0$, if $(\delta\mathcal{U}_{+},\delta\mathcal{U}_{-})\in\mathcal{X}_{\sigma}$, the problem $(\mathbf{IBVP})^{*}_{\rm{III}}$
admits a unique $C^{1,\alpha}$-solution $(\delta \omega^{*}_{\pm}, \delta p^{*}_{\pm})$ satisfying
\begin{align}\label{eq:prop-3.4-1}
\begin{split}
&\sum_{k=\pm}\|(\delta \omega^{*}_{k}, \delta p^{*}_{k})\|_{1,\alpha; \tilde{\mathcal{N}}^{\rm{I}}_{k}\cup\tilde{\mathcal{N}}^{\rm{III}}_{k}}\\
&\leq  \tilde{C}_{34}\sum_{k=\pm}\Big(\|(\delta \omega^{k}_{\rm in},\delta p^{k}_{\rm in})\|_{1,\alpha; \tilde{\Gamma}^{k}_{\rm in}}
+\|\delta Y^{*}_{k}\|_{1,\alpha; \tilde{\mathcal{N}}^{\rm{I}}_{k}\cup\tilde{\mathcal{N}}^{\rm{III}}_{k}}
 +\sum_{j=\pm}\|f^{j}_{k}\|_{1,\alpha; \tilde{\mathcal{N}}^{\rm{I}}_{k}\cup\tilde{\mathcal{N}}^{\rm{III}}_{k}}\Big),
\end{split}
\end{align}
where the constant $\tilde{C}_{34}>0$ depends only on $\underline{\mathcal{U}}_{+}$, $\alpha$ and $L$.
\end{proposition}

\begin{proof}
Since the contact discontinuity is a free interface, we need to solve $(\mathbf{IBVP})^{*}_{\rm{III}}$ in
$\tilde{\mathcal{N}}^{\rm{I}}_{+}\cup\tilde{\mathcal{N}}^{\rm{III}}_{+}$ and $\tilde{\mathcal{N}}^{\rm{I}}_{-}\cup\tilde{\mathcal{N}}^{\rm{III}}_{-}$ at same time.
First, we can follow \cite{li-yu} to employ the \emph{Picard iteration scheme} to show the existence and uniqueness of \eqref{eq:prop-3.4-1}. We omit it for the shortness and only establish estimate \eqref{eq:prop-3.4-1} below.

First, we consider the $C^0$-estimate of $(\delta \omega^{*}_{\pm},\delta p^{*}_{\pm})$ in $\tilde{\mathcal{N}}^{\rm{I}}_{\pm}\cup\tilde{\mathcal{N}}^{\rm{III}}_{\pm}$.
Consider the first two equations in $\eqref{eq:prop-3.2-3}$ for $z^{\pm,*}_{+}$ in $\tilde{\mathcal{N}}^{\rm{I}}_{+}\cup\tilde{\mathcal{N}}^{\rm{III}}_{+}$.
In $\tilde{\mathcal{N}}^{\rm{I}}_{-}\cup\tilde{\mathcal{N}}^{\rm{III}}_{-}$, we introduce
\begin{eqnarray}\label{eq:prop-3.4-2}
z^{+,*}_{-}=\delta \omega^{*}_{+}+\Lambda_{-}\delta p^{*}_{+}, \quad z^{-,*}_{-}=\delta \omega^{*}_{+}-\Lambda_{-}\delta p^{*}_{+}.
\end{eqnarray}
Then, $z^{\pm,*}_{-}$ satisfy the following equations
\begin{eqnarray}\label{eq:prop-3.4-3}
\begin{cases}
\partial_{\xi}z^{+,*}_{-}+\lambda^{+}_{-}\partial_{\eta}z^{+,*}_{-}-\frac{\partial^{+}_{-}\Lambda_{-}}{2\Lambda_{-}}\big(z^{+,*}_{-}-z^{-,*}_{-}\big)\\
\qquad\qquad\qquad\quad =a^{+}_{-}(\underline{\mathcal{U}}_{-})\delta Y^{*}_{-}+f^{+}_{-}(\delta\mathcal{V}_{-}, \delta Y_{-}),&\quad  \mbox{in}
\quad \tilde{\mathcal{N}}^{\rm{I}}_{-}\cup\tilde{\mathcal{N}}^{\rm{III}}_{-}, \\
\partial_{\xi}z^{-,*}_{-}+\lambda^{-}_{-}\partial_{\eta}z^{-,*}_{-}+\frac{\partial^{-}_{-}\Lambda_{-}}{2\Lambda_{-}}\big(z^{+,*}_{-}-z^{-,*}_{-}\big)\\
\qquad\qquad\qquad\quad =a^{-}_{-}(\underline{\mathcal{U}}_{-})\delta Y^{*}_{-}+f^{-}_{-}(\delta\mathcal{V}_{-}, \delta Y_{-}),&\quad  \mbox{in} \quad \tilde{\mathcal{N}}^{\rm{I}}_{-}\cup\tilde{\mathcal{N}}^{\rm{III}}_{-},
\end{cases}
\end{eqnarray}
where $\p^{\pm}_{-}\triangleq \p_{\xi}+\lambda^{\pm}_{-}\p_{\eta}$. In addition, $z^{\pm,*}_{+}$ and $z^{\pm,*}_{-}$ satisfy the following initial-boundary conditions
\begin{eqnarray}\label{eq:prop-3.4-4}
\begin{cases}
(z^{+,*}_{+},z^{-,*}_{+})=(z^{+}_{+,\rm in}, z^{-}_{+,\rm in}),  &\quad  \mbox{on} \quad \tilde{\Gamma}^{+}_{\rm in}, \\
(z^{+,*}_{-}, z^{-,*}_{-})=(z^{+}_{-,\rm in}, z^{-}_{-,\rm in}),  &\quad  \mbox{on} \quad \tilde{\Gamma}^{-}_{\rm in}, \\
z^{+,*}_{+}+z^{-,*}_{+}=z^{+,*}_{-}+z^{-,*}_{-},\quad\frac{z^{+,*}_{+}-z^{-,*}_{+}}{\Lambda_{+}}=\frac{z^{+,*}_{-}-z^{-,*}_{+}}{\Lambda_{-}}, &\quad \mbox{on} \quad \tilde{\Gamma}_{\rm{cd}}\cap \overline{\tilde{\mathcal{N}}^{\rm{III}}_{+}}.
\end{cases}
\end{eqnarray}

We further reduce the boundary condition on $\tilde{\Gamma}_{\rm{cd}}\cap \overline{\tilde{\mathcal{N}}^{\rm{III}}_{+}}$ as
\begin{eqnarray}\label{eq:prop-3.4-5}
z^{+,*}_{+}=\frac{\Lambda_{-}-\Lambda_{+}}{\Lambda_{+}+\Lambda_{-}}z^{-,*}_{+}+\frac{2\Lambda_{+}}{\Lambda_{+}+\Lambda_{+}}z^{+,*}_{-},\quad z^{-,*}_{-}=\frac{2\Lambda_{-}}{\Lambda_{+}+\Lambda_{+}}z^{-,*}_{+}+\frac{\Lambda_{+}-\Lambda_{-}}{\Lambda_{+}+\Lambda_{-}}z^{+,*}_{-}.
\end{eqnarray}

As shown in Fig.\ref{fig3.7}, for any $(\bar{\xi}, \bar{\eta})\in \tilde{\mathcal{N}}^{\rm{III}}_{+}$, we denote by $\eta=\Upsilon^{+}_{+}(\xi;\bar{\xi},\bar{\eta})$ with \eqref{eq:prop-3.2-4}, the backward characteristic curve corresponding to $\lambda^{+}_{+}$ and passing through $(\bar{\xi},\bar{\eta})$.
Let it intersects $\tilde{\Gamma}^{+}_{\rm{in}}$ at the point $(0,\gamma^{-}_{+}(\bar{\xi},\bar{\eta}))$.
Denote by $\eta=\Upsilon^{+,\rm{cd}}_{+}(\tau;\bar{\xi}, \bar{\eta})$ the backward characteristic curve, corresponding to $\lambda^{+}_{+}$, passing through $(\bar{\xi}, \bar{\eta})$ and intersecting $\tilde{\Gamma}_{\rm{cd}}\cap \overline{\tilde{\mathcal{N}}^{\rm{III}}_{+}}$ at $(\zeta^{+,\rm{cd}}_{+}(\bar{\xi},\bar{\eta}),0)$.
Moreover, for any $(\bar{\xi}, \bar{\eta})\in \tilde{\mathcal{N}}^{\rm{III}}_{-}$, let $\eta=\Upsilon^{+}_{-}(\xi;\bar{\xi},\bar{\eta})$ be the backward characteristic curve corresponding to $\lambda^{+}_{-}$,
passing through $(\bar{\xi},\bar{\eta})$ and intersecting $\tilde{\Gamma}^{-}_{\rm{in}}$ at $(0,\gamma^{+}_{-}(\bar{\xi},\bar{\eta}))$.
Let $\eta=\Upsilon^{-,\rm{cd}}_{-}(\xi;\bar{\xi}, \bar{\eta})$ be the backward characteristic curve corresponding to $\lambda^{-}_{-}$,
passing through $(\bar{\xi}, \bar{\eta})$ and intersecting $\tilde{\Gamma}_{\rm{cd}}\cap \overline{\tilde{\mathcal{N}}^{\rm{III}}_{+}}$ at $(\zeta^{-,\rm{cd}}_{-}(\bar{\xi},\bar{\eta}),0)$.

\begin{figure}[ht]
\begin{center}
\begin{tikzpicture}[scale=1.0]
\draw [line width=0.02cm](-2.5,0.8)--(-2.5,3.6);
\draw [line width=0.04cm](-2.5,3.6)to[out=-20, in=120](0.8,0.8);
\draw [line width=0.01cm](-2.5,0.8)to[out=20, in=-110](-0.8,2.7);

\draw [line width=0.03cm][red](-1.2,0.8)to[out=10, in=-120](-0.2,1.4);
\draw [line width=0.03cm][blue](-2.5,2.8)to[out=-10, in=120](-0.2,1.4);
\draw [line width=0.03cm][green](-2.5,1.9)to[out=-10, in=110](-1.2,0.8);

\draw [line width=0.04cm][dashed][red](-2.5,0.8)--(1.5,0.8);
\draw [line width=0.02cm](-2.5,-2.0)--(-2.5,0.8);
\draw [line width=0.04cm](-2.5,-2.0)to[out=20, in=-120](0.8,0.8);
\draw [line width=0.03cm][green](-2.5,-1.2)to[out=10, in=-120](-0.2,0.2);
\draw [line width=0.03cm][blue](-2.5,-0.3)to[out=10, in=-110](-1.2,0.8);
\draw [line width=0.03cm][red](-1.2,0.8)to[out=-10, in=120](-0.2,0.2);
\draw [line width=0.01cm](-2.5,0.8)to[out=-20, in=110](-0.5,-0.9);

\node at (-1.8, 2.1) {$\tilde{\mathcal{N}}^{\rm{I}}_{+}$};
\node at (-0.8, 1.4) {$\tilde{\mathcal{N}}^{\rm{III}}_{+}$};

\node at (-1.8, -0.5) {$\tilde{\mathcal{N}}^{\rm{I}}_{-}$};
\node at (0.2, 0.4) {$\tilde{\mathcal{N}}^{\rm{III}}_{-}$};

\node at (-2.5,3.6) {$\bullet$};
\node at (-0.8,2.7) {$\bullet$};
\node at (-0.2,1.4) {$\bullet$};
\node at (-2.5,2.8) {$\bullet$};
\node at (-2.5,1.9) {$\bullet$};

\node at (0.8,0.8) {$\bullet$};
\node at (-2.5,0.8) {$\bullet$};
\node at (-2.5,-2.0) {$\bullet$};
\node at (-2.5,-1.2) {$\bullet$};
\node at (-2.5,-0.3) {$\bullet$};
\node at (-1.2,0.8) {$\bullet$};
\node at (-0.2,0.2) {$\bullet$};
\node at (-0.5,-0.9) {$\bullet$};

\node at (1.9, 0.8) {$\tilde{\Gamma}_{\textrm{cd}}$};
\node at (-2.9, 2.4) {$\tilde{\Gamma}^{+}_{\textrm{in}}$};
\node at (-2.9, -0.6) {$\tilde{\Gamma}^{-}_{\textrm{in}}$};
\node at (-2.9,3.6) {$\tilde{P}_+$};
\node at (-2.9,0.8) {$\tilde{O}$};
\node at (-2.9,-2.0) {$\tilde{P}_-$};
\end{tikzpicture}
\end{center}
\caption{Estimates for $(\delta \omega^{*}_{\pm},\delta p^{*}_{\pm})$ in $\tilde{\mathcal{N}}^{\rm{I}}_{\pm}\cup\tilde{\mathcal{N}}^{\rm{III}}_{\pm}$}\label{fig3.7}
\end{figure}

Then integrating $\eqref{eq:prop-3.2-3}_1$ and $\eqref{eq:prop-3.2-3}_2$ along the characteristic curves $\eta=\Upsilon^{+}_{+}(\xi;\bar{\xi},\bar{\eta})$ and $\eta=\Upsilon^{+,\rm{cd}}_{+}(\xi;\bar{\xi}, \bar{\eta})$, respectively, we obtain
\begin{eqnarray}\label{eq:prop-3.4-6}
\begin{cases}
z^{+,*}_{+}(\bar{\xi},\bar{\eta})
=z^{+,*}_{+}(\zeta^{+,\rm{cd}}_{+}(\bar{\xi},\bar{\eta}),0)
+\int_{\zeta^{+,\rm{cd}}_{+}(\bar{\xi},\bar{\eta})}^{\bar{\xi}}\frac{\p^{+}_{+}\Lambda_{+}}{2\Lambda_{+}}\big(z^{+,*}_{+}-z^{-,*}_{+}\big)(\tau,\Upsilon^{+,\rm{cd}}_{+}(\tau;\bar{\xi},\bar{\eta}))\mathrm{d} \tau\\
\qquad\qquad\qquad\qquad\qquad+\int_{\zeta^{+,\rm{cd}}_{+}(\bar{\xi},\bar{\eta})}^{\bar{\xi}}\big(a^{+}_{+}(\underline{\mathcal{U}}_{+})\delta Y^{*}_{+}+f^{+}_{+}(\delta\mathcal{V}_{+}, \delta Y_{+})\big)(\tau,\Upsilon^{+,\rm{cd}}_{+}(\tau;\bar{\xi},\bar{\eta}))\dd \tau,\\
z^{-,*}_{+}(\bar{\xi},\bar{\eta})
=z^{-}_{+,\rm in}(\gamma_{+}^{-}(\bar{\xi},\bar{\eta}))+\int_0^{\bar{\xi}}\frac{\p^{-}_{+}\Lambda_{+}}{2\Lambda_{+}}\big(z^{-,*}_{+}-z^{+,*}_{+}\big)
(\tau,\Upsilon^{-}_{+}(\tau;\bar{\xi},\bar{\eta}))\mathrm{d} \tau\\
\qquad\qquad\qquad\qquad\qquad\qquad+\int^{\bar{\xi}}_{0}\big(a^{-}_{+}(\underline{\mathcal{U}}_{+})\delta Y^{*}_{+}+f^{-}_{+}(\delta\mathcal{V}_{+}, \delta Y_{+})\big)(\tau,\Upsilon^{-}_{+}(\tau;\bar{\xi},\bar{\eta}))\mathrm{d} \tau.
\end{cases}
\end{eqnarray}
Integrating $\eqref{eq:prop-3.4-3}_1$ and $\eqref{eq:prop-3.4-3}_2$ along the characteristic curves $\eta=\Upsilon^{+}_{-}(\xi;\bar{\xi},\bar{\eta})$ and $\eta=\Upsilon^{-,\rm{cd}}_{-}(\xi;\bar{\xi}, \bar{\eta})$, respectively, together with the compatibility condition $\eqref{eq:2.10}_1$, we have
\begin{eqnarray}\label{eq:prop-3.4-7}
\begin{cases}
z^{+,*}_{-}(\bar{\xi},\bar{\eta})
=z^{-}_{+,\rm in}(\gamma_{-}^{+}(\bar{\xi},\bar{\eta}))+\int_0^{\bar{\xi}}\frac{\p^{+}_{-}\Lambda_{-}}{2\Lambda_{-}}\big(z^{+,*}_{-}-z^{-,*}_{-}\big)
(\tau,\Upsilon^{+}_{-}(\tau;\bar{\xi},\bar{\eta}))\mathrm{d} \tau\\
\qquad\qquad\qquad\qquad\qquad\qquad+\int^{\bar{\xi}}_{0}\big(a^{+}_{-}(\underline{\mathcal{U}}_{+})\delta Y^{*}_{-}+f^{+}_{-}(\delta\mathcal{V}_{-}, \delta Y_{-})\big)(\tau,\Upsilon^{+}_{-}(\tau;\bar{\xi},\bar{\eta}))\mathrm{d} \tau,\\
z^{-,*}_{-}(\bar{\xi},\bar{\eta})
=z^{+,*}_{-}(\zeta^{-,\rm{cd}}_{-}(\bar{\xi},\bar{\eta}),0)
+\int_{\zeta^{-,\rm{cd}}_{-}(\bar{\xi},\bar{\eta})}^{\bar{\xi}}\frac{\p^{-}_{-}\Lambda_{-}}{2\Lambda_{-}}\big(z^{+,*}_{-}-z^{-,*}_{-}\big)(\tau,\Upsilon^{-,\rm{cd}}_{-}(\tau;\bar{\xi},\bar{\eta}))\mathrm{d} \tau\\
\qquad\qquad\qquad\qquad\qquad+\int_{\zeta^{-,\rm{cd}}_{-}(\bar{\xi},\bar{\eta})}^{\bar{\xi}}\big(a^{-}_{-}(\underline{\mathcal{U}}_{-})\delta Y^{*}_{-}+f^{-}_{-}(\delta\mathcal{V}_{-}, \delta Y_{-})\big)(\tau,\Upsilon^{-,\rm{cd}}_{-}(\tau;\bar{\xi},\bar{\eta}))\dd \tau.
\end{cases}
\end{eqnarray}

To estimate $z^{\pm,*}_{+}(\bar{\xi},\bar{\eta})$ and  $z^{\pm,*}_{-}(\bar{\xi},\bar{\eta})$, let us define
\begin{eqnarray*}
R^{\rm{III}}_{+}(\kappa)=\{(\xi,\eta)| 0\leq\xi\leq\kappa,\ 0\leq\eta\leq\Upsilon^{-,1}_{+}(\xi; \bar{\xi}^{2}_{\rm{cd}},0)\}, \ \kappa\in[0,\bar{\xi}^{2}_{\rm{cd}}],
\end{eqnarray*}
and
\begin{eqnarray*}
R^{\rm{III}}_{-}(\kappa)=\{(\xi,\eta)|0\leq\xi\leq\kappa,\ \Upsilon^{+,1}_{-}(\xi;\bar{\xi}^{2}_{\rm{cd}},0)\leq\eta\leq0\},\ \kappa\in[0, \bar{\xi}^{2}_{\rm{cd}}].
\end{eqnarray*}

Denote
\begin{eqnarray}\label{eq:prop-3.4-8}
z^{\pm,*}_{\rm{III}^{+}}(\kappa)\triangleq\sup\limits_{(\xi,\eta)\in R^{\rm{III}}_{+}(\kappa)}|z^{\pm,*}_{+}(\xi,\eta)|,
\quad z^{\pm,*}_{\rm{III}^{-}}(\kappa)\triangleq\sup\limits_{(\xi,\eta)\in R^{\rm{III}}_{-}(\kappa)}|z^{\pm,*}_{-}(\xi,\eta)|.
\end{eqnarray}

Obviously, we have
\begin{eqnarray}\label{eq:prop-3.4-9}
\|z_{+}^{\pm,*}\|_{0,0; \tilde{\mathcal{N}}^{\rm{I}}_{+}\cup\tilde{\mathcal{N}}^{\rm{III}}_{+}}=z_{\rm{III}^{+}}^{\pm,*}(\bar{\xi}^{2}_{\rm{cd}}),\quad
\|z_{-}^{\pm,*}\|_{0,0; \tilde{\mathcal{N}}^{\rm{I}}_{-}\cup\tilde{\mathcal{N}}^{\rm{III}}_{-}}=z_{\rm{III}^{-}}^{\pm,*}(\bar{\xi}^{2}_{\rm{cd}}).
\end{eqnarray}

By \eqref{eq:prop-3.4-6} and direct computation,  for sufficiently small $\sigma>0$, we have
\begin{eqnarray}\label{eq:prop-3.4-10}
\begin{split}
|z^{+,*}_{+}(\bar{\xi},\bar{\eta})|&\leq \mathcal{O}(1)\Big(\|\delta Y^{*}_{+}\|_{0,0; \tilde{\mathcal{N}}^{\rm{III}}_{+}}+\|f^{+}_{+}\|_{0,0; \tilde{\mathcal{N}}^{\rm{III}}_{+}}+\int_{\zeta^{+,\rm{cd}}_{+}(\bar{\xi},\bar{\eta})}^{\bar{\xi}}\sum_{k=\pm}z^{\pm,*}_{\rm{III}^{+}}(\varsigma)\dd\varsigma\Big)\\
&\qquad\qquad\quad +|z^{+,*}_{+}(\zeta^{+,\rm{cd}}_{+}(\bar{\xi},\bar{\eta}),0)|,
\end{split}
\end{eqnarray}
and
\begin{eqnarray}\label{eq:prop-3.4-11}
\begin{split}
|z^{-,*}_{+}(\bar{\xi},\bar{\eta})|\leq \|z^{-}_{+,\rm in}\|_{0,0; \tilde{\Gamma}^{+}_{\rm in}}+\mathcal{O}(1)\Big(\|\delta Y^{*}_{+}\|_{0,0; \tilde{\mathcal{N}}^{\rm{III}}_{+}}+\|f^{-}_{+}\|_{0,0; \tilde{\mathcal{N}}^{\rm{III}}_{+}}
+\int_{0}^{\bar{\xi}}\sum_{k=\pm}z^{k,*}_{\rm{III}^{+}}(\varsigma)\dd\varsigma\Big).
\end{split}
\end{eqnarray}

Thus, in order to estimate $|z^{+,*}_{+}(\bar{\xi},\bar{\eta})|$, one needs to estimate $z^{+,*}_{+}(\zeta^{+,\rm{cd}}_{+}(\bar{\xi},\bar{\eta}),0)$.
By \eqref{eq:prop-3.4-5},
\begin{eqnarray}\label{eq:prop-3.4-12}
\begin{split}
z^{+,*}_{+}(\zeta^{+,\rm{cd}}_{+}(\bar{\xi},\bar{\eta}),0)=\Big(\frac{\Lambda_{-}-\Lambda_{+}}{\Lambda_{+}+\Lambda_{-}}z^{-,*}_{+}\Big)(\zeta^{+,\rm{cd}}_{+}(\bar{\xi},\bar{\eta}),0)
+\Big(\frac{2\Lambda_{+}}{\Lambda_{+}+\Lambda_{-}}z^{+,*}_{-}\Big)(\zeta^{+,\rm{cd}}_{+}(\bar{\xi},\bar{\eta}),0).
\end{split}
\end{eqnarray}

For $z^{-,*}_{+}(\zeta^{+,\rm{cd}}_{+}(\bar{\xi},\bar{\eta}),0)$ and $z^{+,*}_{-}(\zeta^{+,\rm{cd}}_{+}(\bar{\xi},\bar{\eta}),0)$, by \eqref{eq:prop-3.4-6} and the compatibility condition $\eqref{eq:2.10}_1$,
\begin{eqnarray}\label{eq:prop-3.4-13}
\begin{split}
&z^{-,*}_{+}(\zeta^{+,\rm{cd}}_{+}(\bar{\xi},\bar{\eta}),0)\\
&=z^{-}_{+,\rm in}(\gamma_{+}^{-}(\zeta^{+,\rm{cd}}_{+}(\bar{\xi},\bar{\eta}),0))
+\int_0^{\zeta^{+,\rm{cd}}_{+}(\bar{\xi},\bar{\eta})}
\frac{\p^{-}_{+}\Lambda_{+}}{2\Lambda_{+}}\big(z^{-,*}_{+}-z^{+,*}_{+}\big)
(\tau,\Upsilon^{-}_{+}(\tau;\zeta^{+,\rm{cd}}_{+}(\bar{\xi},\bar{\eta}),0))\dd \tau\\
&\quad\ +\int^{\zeta^{+,\rm{cd}}_{+}(\bar{\xi},\bar{\eta})}_{0}\big(a^{-}_{+}(\underline{\mathcal{U}}_{+})\delta Y^{*}_{+}+f^{-}_{+}(\delta\mathcal{V}_{+}, \delta Y_{+})\big)(\tau,\Upsilon^{-}_{+}(\tau;\zeta^{+,\rm{cd}}_{+}(\bar{\xi},\bar{\eta}),0))\dd \tau,
\end{split}
\end{eqnarray}
and by \eqref{eq:prop-3.4-7},
\begin{eqnarray}\label{eq:prop-3.4-14}
\begin{split}
&z^{+,*}_{-}(\zeta^{+,\rm{cd}}_{+}(\bar{\xi},\bar{\eta}),0)\\
&=z^{+}_{-,\rm in}(\gamma_{-}^{+}(\zeta^{+,\rm{cd}}_{+}(\bar{\xi},\bar{\eta}),0))
+\int_0^{\zeta^{+,\rm{cd}}_{+}(\bar{\xi},\bar{\eta})}
\frac{\p^{+}_{-}\Lambda_{+}}{2\Lambda_{+}}\big(z^{+,*}_{-}-z^{-,*}_{-}\big)
(\tau,\Upsilon^{+}_{-}(\tau;\zeta^{+,\rm{cd}}_{+}(\bar{\xi},\bar{\eta}),0))\dd \tau\\
&\quad\ +\int^{\zeta^{+,\rm{cd}}_{+}(\bar{\xi},\bar{\eta})}_{0}\big(a^{+}_{-}(\underline{\mathcal{U}}_{-})\delta Y^{*}_{-}+f^{+}_{-}(\delta\mathcal{V}_{-}, \delta Y_{-})\big)(\tau,\Upsilon^{+}_{-}(\tau;\zeta^{+,\rm{cd}}_{+}(\bar{\xi},\bar{\eta}),0))\dd \tau.
\end{split}
\end{eqnarray}

Therefore, for sufficiently small $\sigma>0$, it follows from \eqref{eq:prop-3.4-12}-\eqref{eq:prop-3.4-14} that
\begin{eqnarray}\label{eq:prop-3.4-15}
\begin{split}
|z^{+,*}_{+}(\zeta^{+,\rm{cd}}_{+}(\bar{\xi},\bar{\eta}),0)|&\leq\mathcal{O}(1)\Big\{\|z^{-}_{+,\rm in}\|_{0,0; \tilde{\Gamma}^{+}_{\rm in}}+\|z^{+}_{-,\rm in}\|_{0,0; \tilde{\Gamma}^{-}_{\rm in}}
+\|f^{-}_{+}\|_{0,0; \tilde{\mathcal{N}}^{\rm{I}}_{+}\cup\tilde{\mathcal{N}}^{\rm{III}}_{+}}+\|f^{+}_{-}\|_{0,0; \tilde{\mathcal{N}}^{\rm{I}}_{-}\cup\tilde{\mathcal{N}}^{\rm{III}}_{-}}\\
&\qquad\quad\ +\sum_{\rm{j}=\pm}\|\delta Y^{*}_{\rm{j}}\|_{0,0; \tilde{\mathcal{N}}^{\rm{I}}_{\rm{j}}\cup\tilde{\mathcal{N}}^{\rm{III}}_{\rm{j}}}
+\int_{0}^{\zeta^{+,\rm{cd}}_{+}(\bar{\xi},\bar{\eta})}\sum_{k=\pm}\big(z^{k,*}_{\rm{III}^{+}}(\varsigma)+z^{k,*}_{\rm{III}^{-}}(\varsigma)\big)\dd\varsigma\Big\}.
\end{split}
\end{eqnarray}

Plugging \eqref{eq:prop-3.4-15} into \eqref{eq:prop-3.4-10}, and by \eqref{eq:prop-3.4-8} and \eqref{eq:prop-3.4-11}, one can derive
\begin{eqnarray}\label{eq:prop-3.4-16}
\begin{split}
\sum_{k=\pm}z^{k,*}_{\rm{III}^{+}}(\kappa)&\leq \mathcal{O}(1)\Big\{\|z_{+,\rm in}^{-}\|_{0,0; \tilde{\Gamma}^{+}_{\rm in}}+\|z_{-,\rm in}^{+}\|_{0,0; \tilde{\Gamma}^{-}_{\rm in}}+\sum_{k=\pm}\|f^{k}_{+}\|_{0,0; \tilde{\mathcal{N}}^{\rm{I}}_{+}\cup\tilde{\mathcal{N}}^{\rm{III}}_{+}}
+\|f^{+}_{-}\|_{0,0; \tilde{\mathcal{N}}^{\rm{I}}_{-}\cup\tilde{\mathcal{N}}^{\rm{III}}_{-}}\\
&\quad\quad+\|\delta Y^{*}_{+}\|_{0,0; \tilde{\mathcal{N}}^{\rm{I}}_{+}\cup\tilde{\mathcal{N}}^{\rm{III}}_{+}}+\|\delta Y^{*}_{-}\|_{0,0; \tilde{\mathcal{N}}^{\rm{I}}_{-}\cup\tilde{\mathcal{N}}^{\rm{III}}_{-}}
+\int^{\bar{\xi}}_{0}\sum_{k=\pm}\big(z^{k,*}_{\rm{III}^{+}}(\varsigma)+z^{k,*}_{\rm{III}^{-}}(\varsigma)\big)\dd\varsigma\Big\}.
\end{split}
\end{eqnarray}

Meanwhile, by boundary condition \eqref{eq:prop-3.4-5} for $z^{-,*}_{-}$ and following the argument above for the estimate \eqref{eq:prop-3.4-16}, we can also deduce that
\begin{eqnarray}\label{eq:prop-3.4-17}
\begin{split}
\sum_{k=\pm}z^{k,*}_{\rm{III}^{-}}(\kappa)&\leq \mathcal{O}(1)\Big\{\|z_{-,\rm in}^{+}\|_{0,0; \tilde{\Gamma}^{-}_{\rm in}}+\|z_{+,\rm in}^{-}\|_{0,0; \tilde{\Gamma}^{+}_{\rm in}}
+\sum_{k=\pm}\|f^{k}_{-}\|_{0,0; \tilde{\mathcal{N}}^{\rm{I}}_{-}\cup\tilde{\mathcal{N}}^{\rm{III}}_{-}}
+\|f^{-}_{+}\|_{0,0; \tilde{\mathcal{N}}^{\rm{I}}_{+}\cup\tilde{\mathcal{N}}^{\rm{III}}_{+}}\\
&\quad\quad+\|\delta Y^{*}_{+}\|_{0,0; \tilde{\mathcal{N}}^{\rm{I}}_{+}\cup\tilde{\mathcal{N}}^{\rm{III}}_{+}}+\|\delta Y^{*}_{-}\|_{0,0; \tilde{\mathcal{N}}^{\rm{I}}_{-}\cup\tilde{\mathcal{N}}^{\rm{III}}_{-}}
+\int^{\bar{\xi}}_{0}\sum_{k=\pm}\big(z^{k,*}_{\rm{III}^{+}}(\varsigma)+z^{k,*}_{\rm{III}^{-}}(\varsigma)\big)\dd\varsigma\Big\}.
\end{split}
\end{eqnarray}

Adding \eqref{eq:prop-3.4-16} and $\eqref{eq:prop-3.4-17}$, and applying the Gronwall's inequality for $\kappa=\bar{\xi}^{3}_{\rm{cd}}$, we conclude
\begin{eqnarray*}
\begin{split}
&\sum_{k=\pm}\Big(\|z^{k,*}_{+}\|_{0,0; \tilde{\mathcal{N}}^{\rm{III}}_k}+\|z^{k,*}_{-}\|_{0,0; \tilde{\mathcal{N}}^{\rm{III}}_k}\Big)\\
&\quad\ \leq  \mathcal{O}(1)\sum_{j=\pm}\Big(\sum_{k}\|z_{j,\rm in}^{k}\|_{0,0; \tilde{\Gamma}^{j}_{\rm in}}+\|z_{j,\rm in}^{k}\|_{0,0; \tilde{\Gamma}^{j}_{\rm in}}
+\sum_{k=\pm}\|f^{k}_{j}\|_{0,0; \tilde{\mathcal{N}}^{\rm{I}}_{j}\cup\tilde{\mathcal{N}}^{\rm{III}}_{j}}
+\|\delta Y^{*}_{j}\|_{0,0; \tilde{\mathcal{N}}^{\rm{I}}_{j}\cup\tilde{\mathcal{N}}^{\rm{III}}_{j}}\Big).
\end{split}
\end{eqnarray*}
By \eqref{eq:prop-3.2-2} and \eqref{eq:prop-3.4-2}, it further implies that
\begin{eqnarray}\label{eq:prop-3.4-18}
\begin{split}
&\sum_{j=\pm}\big\|(\delta \omega^{*}_{j},\delta p^{*}_{j})\big\|_{0,0; \tilde{\mathcal{N}}^{\rm{I}}_{j}\cup\tilde{\mathcal{N}}^{\rm{III}}_{j}}\\
&\quad\leq  \tilde{C}_{\rm{III}}\sum_{j=\pm}\Big(\big\|(\delta \omega^{j}_{\rm{in}},\delta p^{j}_{\rm{in}})\big\|_{0,0; \tilde{\Gamma}^{{j}}_{\rm in}}
+\|\delta Y^{*}_{j}\|_{0,0; \tilde{\mathcal{N}}^{\rm{I}}_{j}\cup\tilde{\mathcal{N}}^{\rm{III}}_{j}}+\sum_{k=\pm}\|f^{k}_{j}\|_{0,0; \tilde{\mathcal{N}}^{\rm{I}}_{j}\cup\tilde{\mathcal{N}}^{\rm{III}}_{j}}\Big),
\end{split}
\end{eqnarray}
 where the constant $\tilde{C}_{\rm{III}}>0$ depends only on $\underline{\mathcal{U}}$, $\alpha$ and $L$.

Next, we will establish the $C^0$-estimates of $(\nabla\delta \omega^{*}_{\pm}, \nabla\delta p^{*}_{\pm})$ in $\tilde{\mathcal{N}}^{\rm{I}}_{\pm}\cup\tilde{\mathcal{N}}^{\rm{III}}_{\pm}$. In $\tilde{\mathcal{N}}^{\rm{I}}_{+}\cup\tilde{\mathcal{N}}^{\rm{III}}_{+}$,
we still use the notations $\mathcal {Z}_{+}^{\pm,*}$ introduced in \eqref{eq:prop-3.2-12} as the new qualities which satisfy equations \eqref{eq:prop-3.2-13}. While in $\tilde{\mathcal{N}}^{\rm{I}}_{-}\cup\tilde{\mathcal{N}}^{\rm{III}}_{-}$,
we introduce
\begin{eqnarray}\label{eq:prop-3.4-19}
\mathcal {Z}_{-}^{+,*}\triangleq\p_{\eta} \delta\omega^{*}_{-}+\Lambda_{-}\p_{\eta} \delta p^{*}_{-},
\quad \mathcal {Z}_{-}^{-,*}\triangleq\p_{\eta} \delta\omega^{*}_{-}-\Lambda_{-}\p_{\eta} \delta p^{*}_{-}.
\end{eqnarray}

By taking derivatives on $\eqref{eq:IBVP-I}_{2}$ with respect to $\eta$ and by \eqref{eq:prop-3.4-19}, $\mathcal {Z}_{-}^{\pm,*}$ in $\tilde{\mathcal{N}}^{\rm{I}}_{-}\cup\tilde{\mathcal{N}}^{\rm{III}}_{-}$ satisfy
\begin{eqnarray}\label{eq:prop-3.4-20}
\begin{cases}
\p_{\xi}\mathcal {Z}_{-}^{+,*}+\lambda^{+}_{-}\p_{\eta}\mathcal {Z}_{-}^{+,*}
=\frac{\p^{+}_{-}\Lambda_{-}-2\Lambda_{-}\p_{\eta}\lambda^{+}_{-}}{2\Lambda_{-}}\mathcal {Z}_{-}^{+,*}-\frac{\p^{-}_{-}\Lambda_{-}}{2\Lambda_{-}}\mathcal {Z}_{-}^{-,*}\\
\qquad\qquad\qquad+a^{+}_{-}(\underline{\mathcal{U}}_{-})\p_{\eta}\delta Y^{*}_{-}-\frac{[a^{+}_{-}(\underline{\mathcal{U}}_{-})-a^{-}_{-}(\underline{\mathcal{U}}_{-})]\p_{\eta} \Lambda_{-}}{2\Lambda_{-}}\delta Y^{*}_{-}\\
\qquad\qquad\qquad+\p_{\eta}f^{+}_{-}(\delta\mathcal{V}_{-}, \delta Y_{-})-\frac{\p_{\eta} \Lambda_{-}}{2\Lambda_{-}}\big[f^{+}_{-}(\delta\mathcal{V}_{-}, \delta Y_{-})-f^{-}_{-}(\delta\mathcal{V}_{-}, \delta Y_{-})\big], \\
\p_{\xi}\mathcal {Z}_{-}^{-,*}+\lambda^{-}_{+}\p_{\eta}\mathcal {Z}_{-}^{-,*}=\frac{\p_{-}\Lambda_{-}-2\Lambda_{-}\p_{\eta}\lambda^{-}_{-}}{2\Lambda_{-}}\mathcal {Z}_{-}^{-,*}-\frac{\p^{+}_{-}\Lambda_{-}}{2\Lambda_{-}}\mathcal {Z}_{-}^{+,*}\\
\qquad\qquad\qquad+a^{-}_{-}(\underline{\mathcal{U}}_{-})\p_{\eta}\delta Y^{*}_{-}-\frac{[a^{+}_{-}(\underline{\mathcal{U}}_{-})-a^{-}_{-}(\underline{\mathcal{U}}_{-})]\p_{\eta} \Lambda_{-}}{2\Lambda_{-}}\delta Y^{*}_{-}\\
\qquad\qquad\qquad+\p_{\eta}f^{-}_{-}(\delta\mathcal{V}_{-}, \delta Y_{-})
+\frac{\p_{\eta} \Lambda_{-}}{2\Lambda_{-}}\big[f^{+}_{-}(\delta\mathcal{V}_{-}, \delta Y_{-})
-f^{-}_{-}(\delta\mathcal{V}_{-}, \delta Y_{-})\big],
\end{cases}
\end{eqnarray}
with the following initial-boundary conditions
\begin{align}\label{eq:prop-3.4-21}
\begin{cases}
(\mathcal {Z}_{+}^{+,*},\mathcal {Z}_{+}^{-,*})=(\mathcal {Z}_{+, \rm in}^{+},\mathcal {Z}_{+, \rm in}^{-}), &  \mbox{on}\  \tilde{\Gamma}^{+}_{\rm in},\\
(\mathcal {Z}_{-}^{+,*},\mathcal {Z}_{-}^{-,*})=(\mathcal {Z}_{-, \rm in}^{+},\mathcal {Z}_{-, \rm in}^{-}), &   \mbox{on}\  \tilde{\Gamma}^{-}_{\rm in},\\
\lambda_{+}^{+}\mathcal {Z}_{+}^{+,*}+\lambda_{+}^{-}\mathcal {Z}_{+}^{-,*}-\big[a^{+}_{+}(\underline{\mathcal{U}}_{+})+a_{+}^{-}(\underline{\mathcal{U}}_{+})\big]\delta Y^{*}_{+}-\big[f^{+}_{+}(\delta\mathcal{V}_{+}, \delta Y_{+})+f^{-}_{+}(\delta\mathcal{V}_{+}, \delta Y_{+})\big]\\
=\lambda_{-}^{+}\mathcal {Z}_{-}^{+,*}+\lambda_{-}^{-}\mathcal {Z}_{-}^{-,*}
-\big[a^{+}_{-}(\underline{\mathcal{U}}_{-})+a_{-}^{-}(\underline{\mathcal{U}}_{-})\big]\delta Y^{*}_{-}-\big[f^{+}_{-}(\delta\mathcal{V}_{-}, \delta Y_{-})+f^{-}_{-}(\delta\mathcal{V}_{-}, \delta Y_{-})\big], & \mbox{on}\
 \tilde{\Gamma}_{\rm{cd}}\cap \overline{\tilde{\mathcal{N}}^{\rm{III}}},\\
-\frac{\lambda_{+}^{+}}{\Lambda_{+}}\mathcal {Z}_{+}^{+,*}+\frac{\lambda_{+}^{-}}{\Lambda_{+}}\mathcal {Z}_{+}^{-,*}+\frac{a^{+}_{+}(\underline{\mathcal{U}}_{+})-a^{-}_{+}(\underline{\mathcal{U}}_{+})}{\Lambda_{+}}\delta Y^{*}_{+}
+\frac{f^{+}_{+}(\delta\mathcal{V}_{+}, \delta Y_{+})-f^{-}_{+}(\delta\mathcal{V}_{+}, \delta Y_{+})}{\Lambda_{+}}\\
=-\frac{\lambda_{-}^{+}}{\Lambda_{-}}\mathcal {Z}_{-}^{+,*}+\frac{\lambda_{-}^{-}}{\Lambda_{-}}\mathcal {Z}_{-}^{-,*}
+\frac{a^{+}_{-}(\underline{\mathcal{U}}_{-})-a^{-}_{-}(\underline{\mathcal{U}}_{-})}{\Lambda_{-}}\delta Y^{*}_{-}+\frac{f^{+}_{-}(\delta\mathcal{V}_{-}, \delta Y_{-})-f^{-}_{-}(\delta\mathcal{V}_{-}, \delta Y_{-})}{\Lambda_{-}}, & \mbox{on}\ \tilde{\Gamma}_{\rm{cd}}\cap \overline{\tilde{\mathcal{N}}^{\rm{III}}},
\end{cases}
\end{align}
where $\mathcal {Z}_{-,\rm in}^{\pm}\triangleq(\delta\omega^{-}_{\rm in})'\pm\Lambda_{-,\rm in}(\delta p^{-}_{\rm in})'$ and $\Lambda_{-,\rm in}\triangleq\frac{\sqrt{u^2_{-,\rm in}+v^2_{-,\rm in}-c^2_{-,\rm in}}}{\rho_{-,\rm in}c_{-,\rm in}u^2_{-,\rm in}}$.
Then, by the compatibility condition $\eqref{eq:2.10}_2$, we integrate $\eqref{eq:prop-3.2-13}_1$ for $\mathcal {Z}_{+}^{+,*}$ along the characteristic $\eta=\Upsilon^{+,\rm{cd}}_{+}(\xi;\bar{\xi},\bar{\eta})$ from $\zeta^{+,\rm{cd}}_{+}(\bar{\xi},\bar{\eta})$ to $\bar{\xi}$, and integrate $\eqref{eq:prop-3.2-13}_2$ for $\mathcal {Z}_{+}^{-,*}$ along the characteristic $\eta=\Upsilon^{-}_{+}(\xi;\bar{\xi},\bar{\eta})$ from $0$ to $\bar{\xi}$, to obtain
\begin{align}\label{eq:prop-3.4-22}
\begin{cases}
\mathcal {Z}_{+}^{+,*}(\bar{\xi},\bar{\eta})
=\mathcal{Z}_{+}^{+,*}(\zeta^{+,\rm{cd}}_{+}(\bar{\xi},\bar{\eta}),0)
+\int_{\zeta^{+,\rm{cd}}_{+}(\bar{\xi},\bar{\eta})}^{\bar{\xi}}\big(\frac{\p_{+}\Lambda_{+}-2\Lambda_{+}\p_{\eta}\lambda^{+}_{+}}{2\Lambda_{+}}\mathcal {Z}_{+}^{+,*}-\frac{\p_{-}\Lambda_{+}}{2\Lambda_{+}}\mathcal {Z}_{+}^{-,*}\big)(\tau,\Upsilon^{+,\rm{cd}}_{+}(\tau;\bar{\xi},\bar{\eta}))\dd \tau\\
\qquad\qquad +\int_{\zeta^{+,\rm{cd}}_{+}(\bar{\xi},\bar{\eta})}^{\bar{\xi}}\big(a^{+}_{+}(\underline{\mathcal{U}}_{+})\p_{\eta}\delta Y^{*}_{+}-\frac{[a^{+}_{+}(\underline{\mathcal{U}}_{+})-a^{-}_{+}(\underline{\mathcal{U}}_{+})]\p_{\eta} \Lambda_{+}}{2\Lambda_{+}}\delta Y^{*}_{+}\big)(\tau,\Upsilon^{+,\rm{cd}}_{+}(\tau;\bar{\xi},\bar{\eta}))\dd \tau\\
\qquad\qquad +\int_{\zeta^{+,\rm{cd}}_{+}(\bar{\xi},\bar{\eta})}^{\bar{\xi}}
\big(\p_{\eta}f^{+}_{+}(\delta\mathcal{V}_{+}, \delta Y_{+})-\frac{\p_{\eta} \Lambda_{+}}{2\Lambda_{+}}\big[f^{+}_{+}(\delta\mathcal{V}_{+}, \underline{\mathcal{U}}_{+})-f^{-}_{+}(\delta\mathcal{V}_{+}, \delta Y_{+})\big]\big)(\tau,\Upsilon^{+,\rm{cd}}_{+}(\tau;\bar{\xi},\bar{\eta}))\dd \tau,\\
\mathcal {Z}_{+}^{-,*}(\bar{\xi},\bar{\eta})
=\mathcal {Z}_{+,\rm in}^{-}(\gamma^{-}_{+}(\bar{\xi},\bar{\eta}))+\int_0^{\bar{\xi}}\big(\frac{\p_{-}\Lambda_{+}-2\Lambda_{+}\p_{\eta}\lambda^{-}_{+}}{2\Lambda_{+}}\mathcal {Z}_{+}^{-,*}-\frac{\p_{+}\Lambda_{+}}{2\Lambda_{+}}\mathcal {Z}_{+}^{+,*}\big)(\tau,\Upsilon^{+}_{+}(\tau;\bar{\xi},\bar{\eta}))\dd \tau\\
\qquad\qquad  +\int_0^{\bar{\xi}}\big(a^{-}_{+}(\underline{\mathcal{U}}_{+})\p_{\eta}\delta Y^{*}_{+}-\frac{[a^{-}_{+}(\underline{\mathcal{U}}_{+})-a^{+}_{+}(\underline{\mathcal{U}}_{+})]\p_{\eta} \Lambda_{+}}{2\Lambda_{+}}\delta Y^{*}_{+} \big)(\tau,\Upsilon^{-}_{+}(\tau;\bar{\xi},\bar{\eta}))\dd \tau\\
\qquad\qquad +\int_0^{\bar{\xi}}\big(\p_{\eta}f^{-}_{+}(\delta\mathcal{V}_{+}, \delta Y_{+})-\frac{\p_{\eta} \Lambda_{+}}{2\Lambda_{+}}\big[f^{-}_{+}(\delta\mathcal{V}_{+}, \delta Y_{+})
-f^{+}_{+}(\delta\mathcal{V}_{+}, \delta Y_{+})\big]\big)(\tau,\Upsilon^{-}_{+}(\tau;\bar{\xi},\bar{\eta}))\dd \tau.
\end{cases}
\end{align}
Similarly, integrating $\eqref{eq:prop-3.4-20}_1$ for $\mathcal{Z}_{-}^{-,*}$ along the characteristic $\eta=\Upsilon^{-,\rm{cd}}_{-}(\xi;\bar{\xi},\bar{\eta})$ from $\zeta^{-,\rm{cd}}_{-}(\bar{\xi},\bar{\eta})$ to $\bar{\xi}$ and integrating $\eqref{eq:prop-3.4-20}_2$ for $\mathcal{Z}_{-}^{+,*}$ along the characteristic $\eta=\Upsilon^{+}_{-}(\xi;\bar{\xi},\bar{\eta})$ from $0$ to $\bar{\xi}$, and by the compatibility condition $\eqref{eq:2.10}_3$, we have 
\begin{align}\label{eq:prop-3.4-23}
\begin{cases}
\mathcal {Z}_{-}^{+,*}(\bar{\xi},\bar{\eta})=\mathcal {Z}_{-,\rm{in}}^{+}(\gamma^{+}_{-}(\bar{\xi},\bar{\eta}))
+\int_{0}^{\bar{\xi}}\big(\frac{\p_{-}\Lambda_{-}-2\Lambda_{-}\p_{\eta}\lambda^{-}_{-}}{2\Lambda_{-}}\mathcal {Z}_{-}^{-,*}-\frac{\p^{+}_{-}\Lambda_{-}}{2\Lambda_{-}}\mathcal {Z}_{-}^{+,*}\big)(\tau,\Upsilon^{+}_{-}(\tau;\bar{\xi},\bar{\eta}))\dd\tau \\
\qquad\qquad +\int_{0}^{\bar{\xi}}\big(a^{-}_{-}(\underline{\mathcal{U}}_{-})\p_{\eta}\delta Y^{*}_{-}-\frac{[a^{+}_{-}(\underline{\mathcal{U}}_{-})-a^{-}_{-}(\underline{\mathcal{U}}_{-})]\p_{\eta} \Lambda_{-}}{2\Lambda_{-}}\delta Y^{*}_{-}\big)(\tau,\Upsilon^{+}_{-}(\tau;\bar{\xi},\bar{\eta}))\dd \tau\\
\qquad\qquad +\int_{0}^{\bar{\xi}}\big(\p_{\eta}f^{-}_{-}(\delta\mathcal{V}_{-}, \delta Y_{-})
+\frac{\p_{\eta} \Lambda_{-}}{2\Lambda_{-}}\big[f^{+}_{-}(\delta\mathcal{V}_{-}, \delta Y_{-})
-f^{-}_{-}(\delta\mathcal{V}_{-}, \delta Y_{-})\big]\big)(\tau,\Upsilon^{+}_{-}(\tau;\bar{\xi},\bar{\eta}))\dd \tau,\\
\mathcal{Z}_{-}^{-,*}(\bar{\xi},\bar{\eta})=\mathcal{Z}_{-}^{-,*}(\zeta^{-,\rm{cd}}_{-}(\bar{\xi},\bar{\eta}),0)
+\int_{\zeta^{-,\rm{cd}}_{-}(\bar{\xi},\bar{\eta})}^{\bar{\xi}}\big(\frac{\p^{+}_{-}\Lambda_{-}-2\Lambda_{-}\p_{\eta}\lambda^{+}_{-}}{2\Lambda_{-}}\mathcal {Z}_{-}^{+,*}-\frac{\p^{-}_{-}\Lambda_{-}}{2\Lambda_{-}}\mathcal {Z}_{-}^{-,*}\big)(\tau, \Upsilon^{-,\rm{cd}}_{-}(\tau;\bar{\xi},\bar{\eta}))\dd\tau \\
\qquad\qquad +\int_{\zeta^{-,\rm{cd}}_{-}(\bar{\xi},\bar{\eta})}^{\bar{\xi}}\big(a^{+}_{-}(\underline{\mathcal{U}}_{-})\p_{\eta}\delta Y^{*}_{-}-\frac{[a^{+}_{-}(\underline{\mathcal{U}}_{-})-a^{-}_{-}(\underline{\mathcal{U}}_{-})]\p_{\eta} \Lambda_{-}}{2\Lambda_{-}}\delta Y^{*}_{-}\big)(\tau, \Upsilon^{-,\rm{cd}}_{-}(\tau;\bar{\xi},\bar{\eta}))\dd\tau \\
\qquad\qquad +\int_{\zeta^{-,\rm{cd}}_{-}(\bar{\xi},\bar{\eta})}^{\bar{\xi}}\big(\p_{\eta}f^{+}_{-}(\delta\mathcal{V}_{-}, \delta Y_{-})-\frac{\p_{\eta} \Lambda_{-}}{2\Lambda_{-}}\big[f^{+}_{-}(\delta\mathcal{V}_{-}, \delta Y_{-})-f^{-}_{-}(\delta\mathcal{V}_{-}, \delta Y_{-})\big]\big)(\tau,\Upsilon^{-,\rm{cd}}_{-}(\tau;\bar{\xi},\bar{\eta}))\dd \tau.
\end{cases}
\end{align}

In order to estimate $\mathcal {Z}_{+}^{\pm,*}$ and $\mathcal {Z}_{-}^{\pm,*}$, we introduce the following norms:
\begin{eqnarray*}
\mathcal {Z}_{\rm{III}^{+}}^{\pm,*}(\kappa)\triangleq\sup\limits_{(\xi,\eta)\in R_{+}^{\rm{III}}(\kappa)}|\mathcal {Z}_{+}^{\pm,*}(\xi,\eta)|,\quad
\mathcal {Z}_{\rm{III}^{-}}^{\pm,*}(\kappa)\triangleq\sup\limits_{(\xi,\eta)\in R_{-}^{\rm{III}}(\kappa)}|\mathcal {Z}_{-}^{\pm,*}(\xi,\eta)|.
\end{eqnarray*}

Obviously,
\begin{eqnarray}\label{eq:prop-3.4-24}
\|\mathcal {Z}_{+}^{\pm,*}\|_{0,0; \tilde{\mathcal{N}}^{\rm{I}}_{+}\cup\tilde{\mathcal{N}}^{\rm{III}}_{+}}=\mathcal {Z}_{\rm{III}^{+}}^{\pm,*}(\bar{\xi}^{2}_{\rm {cd}})\quad\text{and}\quad\mathcal{Z}^{\pm,*}_{\rm{III}^{+}}(\kappa)\geq \mathcal {Z}^{\pm,*}_{\rm{I}^{+}}(\kappa),\quad \kappa\in [0,\bar{\xi}^{1}_{+}],
\end{eqnarray}
and
\begin{eqnarray}\label{eq:prop-3.4-25}
\|\mathcal {Z}_{-}^{\pm,*}\|_{0,0; \tilde{\mathcal{N}}^{\rm{I}}_{-}\cup\tilde{\mathcal{N}}^{\rm{III}}_{-}}=\mathcal {Z}_{\rm{III}^{-}}^{\pm,*}(\bar{\xi}^{2}_{\rm {cd}})\quad\text{and}\quad\mathcal{Z}^{\pm,*}_{\rm{III}^{-}}(\kappa)\geq \mathcal {Z}^{\pm,*}_{\rm{I}^{-}}(\kappa),\quad \kappa\in [0,\bar{\xi}^{1}_{-}].
\end{eqnarray}

Then, similar as the argument to estimate $z^{-,*}_{+}$, we obtain for $\mathcal {Z}_{+}^{-,*}$ that
\begin{eqnarray}\label{eq:prop-3.4-26}
\begin{split}
\mathcal {Z}_{\rm{III}^{+}}^{-,*}(\kappa)&\leq \|\mathcal {Z}_{+,\rm in}^{-}\|_{0,0; \tilde{\Gamma}^{+}_{\rm in}}
+\mathcal{O}(1)\Big\{\sum_{k=\pm}\|f^{k}_{+}\|_{0,0; \tilde{\mathcal{N}}^{\rm{I}}_{+}\cup\tilde{\mathcal{N}}^{\rm{III}}_{+}}+\|\p_{\eta}f^{+}_{+}\|_{0,0; \tilde{\mathcal{N}}^{\rm{I}}_{+}\cup\tilde{\mathcal{N}}^{\rm{III}}_{+}}\\
&\qquad\qquad\qquad\quad +\|\delta Y^{*}_{+}\|_{1,0; \tilde{\mathcal{N}}^{\rm{I}}_{+}\cup\tilde{\mathcal{N}}^{\rm{III}}_{+}}+\int_0^{\kappa}\sum_{k=\pm}\mathcal {Z}_{\rm{III}^{+}}^{k,*}(\varsigma)\dd\varsigma\Big\},
\end{split}
\end{eqnarray}
and
\begin{eqnarray}\label{eq:prop-3.4-27}
\begin{split}
|\mathcal{Z}_{+}^{+,*}(\bar{\xi},\bar{\eta})|&\leq |\mathcal{Z}_{+}^{+,*}(\zeta^{+,\rm{cd}}_{+}(\bar{\xi},\bar{\eta}),0)|
+\mathcal{O}(1)\Big\{\sum_{k=\pm}\|f^{k}_{+}\|_{0,0; \tilde{\mathcal{N}}^{\rm{I}}_{+}\cup\tilde{\mathcal{N}}^{\rm{III}}_{+}}+\|\p_{\eta}f^{-}_{+}\|_{0,0; \tilde{\mathcal{N}}^{\rm{I}}_{+}\cup\tilde{\mathcal{N}}^{\rm{III}}_{+}}\\
&\qquad\qquad\ +\|\delta Y^{*}_{+}\|_{1,0; \tilde{\mathcal{N}}^{\rm{I}}_{+}\cup\tilde{\mathcal{N}}^{\rm{II}}_{+}}+\int_{\zeta^{+,\rm{cd}}_{+}(\bar{\xi},\bar{\eta})}^{\bar{\xi}}
\sum_{k=\pm}\mathcal {Z}_{\rm{III}^{+}}^{k,*}(\varsigma)\dd \varsigma\Big\}.
\end{split}
\end{eqnarray}

By the boundary conditions \eqref{eq:prop-3.4-21} on $\tilde{\Gamma}_{\rm{cd}}\cap \overline{\tilde{\mathcal{N}}^{\rm{III}}_{+}}$, we have 
\begin{eqnarray}\label{eq:prop-3.4-28a}
\begin{split}
&\mathcal {Z}_{+}^{+,*}(\zeta^{+,\rm{cd}}_{+}(\bar{\xi},\bar{\eta}),0)\\
&=-\Big(\frac{(\Lambda_{+}-\Lambda_{-})\lambda_{+}^{-}}{(\Lambda_{+}+\Lambda_{-})\lambda_{+}^{+}}\mathcal {Z}^{-,*}_{+}+\frac{2\Lambda_{+}\lambda_{+}^{-}}{(\Lambda_{+}+\Lambda_{-})\lambda_{+}^{+}}\mathcal {Z}^{+,*}_{-}\Big)(\zeta^{+,\rm{cd}}_{+}(\bar{\xi},\bar{\eta}),0)\\
&\quad+\Big(\frac{(a _{+}^{+}+a _{+}^{-})\Lambda_{+}+(a _{+}^{+}-a _{+}^{-})\Lambda_{-}}{(\Lambda_{+}+\Lambda_{-})\lambda_{+}^{+}}\delta Y^{*}_{+}-\frac{2a_{-}^{-}\Lambda_{+}}{(\Lambda_{+}+\Lambda_{-})\lambda_{+}^{+}}\delta Y^{*}_{-}\Big)(\zeta^{+,\rm{cd}}_{+}(\bar{\xi},\bar{\eta}),0)\\
&\quad +\Big(\frac{\Lambda_{+}-\Lambda_{-}}{(\Lambda_{+}+\Lambda_{-})\lambda_{+}^{+}}f^{+}_{+}(\delta\mathcal{V}_{+}, \delta Y_{+})+\frac{1}{\lambda_{+}^{+}}f^{-}_{+}(\delta\mathcal{V}_{+}, \delta Y_{+})\Big)(\zeta^{+,\rm{cd}}_{+}(\bar{\xi},\bar{\eta}),0)\\
&\quad -\Big(\frac{2\Lambda_{+}}{(\Lambda_{+}+\Lambda_{-})\lambda_{+}^{+}}f^{-}_{-}(\delta\mathcal{V}_{-}, \delta Y_{-})\Big)(\zeta^{+,\rm{cd}}_{+}(\bar{\xi},\bar{\eta}),0).
\end{split}
\end{eqnarray}

Note that $\mathcal{Z}^{-,*}_{+}(\zeta^{+,\rm{cd}}_{+}(\bar{\xi},\bar{\eta}),0)$ satisfies \eqref{eq:prop-3.4-22} at $(\zeta^{+,\rm{cd}}_{+}(\bar{\xi},\bar{\eta}),0)$, so
\begin{eqnarray}\label{eq:prop-3.4-29}
\begin{split}
|\mathcal{Z}^{-,*}_{+}(\zeta^{+,\rm{cd}}_{+}(\bar{\xi},\bar{\eta}),0)|&\leq \|\mathcal {Z}_{+,\rm in}^{-}\|_{0,0; \tilde{\Gamma}^{+}_{\rm in}}
+\mathcal{O}(1)\Big\{\sum_{k=\pm}\|f^{k}_{+}\|_{0,0; \tilde{\mathcal{N}}^{\rm{I}}_{+}\cup\tilde{\mathcal{N}}^{\rm{III}}_{+}}+\|\p_{\eta}f^{+}_{+}\|_{0,0; \tilde{\mathcal{N}}^{\rm{I}}_{+}\cup\tilde{\mathcal{N}}^{\rm{III}}_{+}}\\
&\qquad\qquad\quad +\|\delta Y^{*}_{+}\|_{1,0; \tilde{\mathcal{N}}^{\rm{I}}_{+}\cup\tilde{\mathcal{N}}^{\rm{III}}_{+}}
+\int_0^{\zeta^{+,\rm{cd}}_{+}(\bar{\xi},\bar{\eta})}\sum_{k=\pm}\mathcal {Z}_{\rm{III}^{+}}^{k,*}(\varsigma)\dd\varsigma\Big\}.
\end{split}
\end{eqnarray}

Note that $\mathcal{Z}^{+,*}_{-}(\zeta^{+,\rm{cd}}_{+}(\bar{\xi},\bar{\eta}),0)$ satisfies $\eqref{eq:prop-3.4-23}_1$ at $(\zeta^{+,\rm{cd}}_{+}(\bar{\xi},\bar{\eta}),0)$, so
for sufficiently small $\sigma>0$, 
\begin{eqnarray}\label{eq:prop-3.4-30}
\begin{split}
|\mathcal{Z}^{+,*}_{-}(\zeta^{+,\rm{cd}}_{+}(\bar{\xi},\bar{\eta}),0)|&\leq \|\mathcal {Z}_{-,\rm in}^{+}\|_{0,0; \tilde{\Gamma}^{-}_{\rm in}}
+\mathcal{O}(1)\Big\{\sum_{k=\pm}\|f^{k}_{-}\|_{0,0; \tilde{\mathcal{N}}^{\rm{I}}_{-}\cup\tilde{\mathcal{N}}^{\rm{III}}_{-}}+\|\p_{\eta}f^{-}_{-}\|_{0,0; \tilde{\mathcal{N}}^{\rm{I}}_{-}\cup\tilde{\mathcal{N}}^{\rm{III}}_{-}}\\
&\qquad\qquad\quad +\|\delta Y^{*}_{-}\|_{1,0; \tilde{\mathcal{N}}^{\rm{I}}_{-}\cup\tilde{\mathcal{N}}^{\rm{III}}_{-}}
+\int_0^{\zeta^{+,\rm{cd}}_{+}(\bar{\xi},\bar{\eta})}\sum_{k=\pm}\mathcal {Z}_{\rm{III}^{-}}^{k,*}(\varsigma)\dd\varsigma\Big\}.
\end{split}
\end{eqnarray}

Plugging \eqref{eq:prop-3.4-29}-\eqref{eq:prop-3.4-30} into \eqref{eq:prop-3.4-28a}, and let  $\sigma>0$ sufficiently small, we can deduce that
\begin{eqnarray*}
\begin{split}
&|\mathcal {Z}_{+}^{+,*}(\zeta^{+,\rm{cd}}_{+}(\bar{\xi},\bar{\eta}),0)|\\
&\leq\mathcal{O}(1)\Big\{ \|\mathcal {Z}_{+,\rm in}^{-}\|_{0,0; \tilde{\Gamma}^{+}_{\rm in}}+\|\mathcal {Z}_{-,\rm in}^{+}\|_{0,0; \tilde{\Gamma}^{-}_{\rm in}}
+\sum_{k=\pm}\big(\|f^{k}_{+}\|_{1,0; \tilde{\mathcal{N}}^{\rm{I}}_{+}\cup\tilde{\mathcal{N}}^{\rm{III}}_{+}}+\|f^{k}_{-}\|_{1,0; \tilde{\mathcal{N}}^{\rm{I}}_{-}\cup\tilde{\mathcal{N}}^{\rm{III}}_{-}}\big)\\
&\quad +\|\delta Y^{*}_{+}\|_{1,0; \tilde{\mathcal{N}}^{\rm{I}}_{+}\cup\tilde{\mathcal{N}}^{\rm{III}}_{+}}+\|\delta Y^{*}_{-}\|_{1,0; \tilde{\mathcal{N}}^{\rm{I}}_{-}\cup\tilde{\mathcal{N}}^{\rm{III}}_{-}}
+\int_0^{\zeta^{+,\rm{cd}}_{+}(\bar{\xi},\bar{\eta})}\sum_{k=\pm}\big(\mathcal {Z}_{\rm{III}^{+}}^{k,*}(\varsigma)+
\mathcal{Z}_{\rm{III}^{-}}^{k,*}(\varsigma)\big)\dd\varsigma \Big\},
\end{split}
\end{eqnarray*}
which together with \eqref{eq:prop-3.4-26}-\eqref{eq:prop-3.4-27} yields that
\begin{align}\label{eq:prop-3.4-31}
\begin{split}
&\sum_{k=\pm}\mathcal {Z}_{\textrm{III}^{+}}^{k,*}(\kappa)\\
&\leq \mathcal{O}(1)\Big\{\sum_{k=\pm}\|\mathcal {Z}_{+,\rm in}^{\pm}\|_{0,0; \tilde{\Gamma}^{+}_{\rm in}}
+ \|\mathcal {Z}_{-,\rm in}^{+}\|_{0,0; \tilde{\Gamma}^{-}_{\rm in}}
+\sum_{k=\pm}\big(\|f^{k}_{+}\|_{1,0; \tilde{\mathcal{N}}^{\rm{I}}_{+}\cup\tilde{\mathcal{N}}^{\rm{III}}_{+}}+\|f^{k}_{-}\|_{1,0; \tilde{\mathcal{N}}^{\rm{I}}_{-}\cup\tilde{\mathcal{N}}^{\rm{III}}_{-}}\big)\\
&\quad +\|\delta Y^{*}_{+}\|_{1,0; \tilde{\mathcal{N}}^{\rm{I}}_{+}\cup\tilde{\mathcal{N}}^{\rm{III}}_{+}}+\|\delta Y^{*}_{-}\|_{1,0; \tilde{\mathcal{N}}^{\rm{I}}_{-}\cup\tilde{\mathcal{N}}^{\rm{III}}_{-}}
+\int_0^{\bar{\xi}}\sum_{k=\pm}\big(\mathcal {Z}_{\rm{III}^{+}}^{k,*}(\varsigma)+
\mathcal{Z}_{\rm{III}^{-}}^{k,*}(\varsigma)\big)\dd\varsigma \Big\}.
\end{split}
\end{align}

Following the same argument and by using the following identity derived from boundary condition \eqref{eq:prop-3.4-21} on $\tilde{\Gamma}_{\rm{cd}}\cap \overline{\tilde{\mathcal{N}}^{\rm{III}}_{+}}$ that
\begin{eqnarray}\label{eq:prop-3.4-32}
\begin{split}
\mathcal {Z}^{-,*}_{-}&=\frac{2\Lambda_{-}\lambda_{+}^{-}}{(\Lambda_{+}+\Lambda_{-})\lambda_{-}^{-}}\mathcal {Z}^{+,*}_{-}+\frac{(\Lambda_{+}-\Lambda_{-})\lambda_{-}^{+}}{(\Lambda_{+}+\Lambda_{-})\lambda_{-}^{-}}\mathcal {Z}^{+,*}_{-}\\
&\quad -\frac{2a_{+}^{-}\Lambda_{-}}{(\Lambda_{+}+\Lambda_{-})\lambda_{-}^{-}}\delta Y^{*}_{+}-\frac{(a_{-}^{+}-a_{-}^{-})\Lambda_{+}-(a_{-}^{+}-a_{-}^{-})\Lambda_{-}}{(\Lambda_{+}+\Lambda_{-})\lambda_{-}^{-}}\delta Y^{*}_{-}\\
&\quad-\frac{2\Lambda_{-}}{(\Lambda_{+}+\Lambda_{-})\lambda_{-}^{-}}f^{-}_{+}(\delta\mathcal{V}_{+}, \delta Y_{+})-\frac{\Lambda_{+}-\Lambda_{-}}{(\Lambda_{+}+\Lambda_{-})\lambda_{-}^{-}}f^{+}_{-}(\delta\mathcal{V}_{-}, \delta Y_{-})+\frac{1}{\lambda_{-}^{-}}f^{-}_{-}(\delta\mathcal{V}_{-}, \delta Y_{-}),
\end{split}
\end{eqnarray}
one can also derive for $\sigma>0$ sufficiently small that
\begin{align}\label{eq:prop-3.4-33}
\begin{split}
&\sum_{k=\pm}\mathcal {Z}_{\textrm{III}^{-}}^{k,*}(\kappa)\\
&\leq \mathcal{O}(1)\Big\{\sum_{k=\pm}\|\mathcal {Z}_{-,\rm in}^{k}\|_{0,0; \tilde{\Gamma}^{-}_{\rm in}}
+ \|\mathcal {Z}_{+,\rm in}^{-}\|_{0,0; \tilde{\Gamma}^{+}_{\rm in}}
+\sum_{k=\pm}\big(\|f^{k}_{+}\|_{1,0; \tilde{\mathcal{N}}^{\rm{I}}_{+}\cup\tilde{\mathcal{N}}^{\rm{III}}_{+}}+\|f^{k}_{-}\|_{1,0; \tilde{\mathcal{N}}^{\rm{I}}_{-}\cup\tilde{\mathcal{N}}^{\rm{III}}_{-}}\big)\\
&\quad +\|\delta Y^{*}_{+}\|_{1,0; \tilde{\mathcal{N}}^{\rm{I}}_{+}\cup\tilde{\mathcal{N}}^{\rm{III}}_{+}}+\|\delta Y^{*}_{-}\|_{1,0; \tilde{\mathcal{N}}^{\rm{I}}_{-}\cup\tilde{\mathcal{N}}^{\rm{III}}_{-}}
+\int_0^{\bar{\xi}}\sum_{k=\pm}\big(\mathcal {Z}_{\rm{III}^{+}}^{k,*}(\varsigma)+
\mathcal{Z}_{\rm{III}^{-}}^{k,*}(\varsigma)\big)\dd\varsigma \Big\}.
\end{split}
\end{align}

Based on \eqref{eq:prop-3.4-31} and \eqref{eq:prop-3.4-33}, we apply the Gronwall's inequality and let $\kappa=\bar{\xi}^{2}_{\rm{cd}}$ to get
\begin{eqnarray}\label{eq:prop-3.4-33a}
\begin{split}
\sum_{j,k=\pm}\|\mathcal {Z}_{j}^{k,*}\|_{0,0; \tilde{\mathcal{N}}^{\rm{III}}_{j}}
 \leq O(1)\sum_{j,k=\pm}\Big(\|\mathcal {Z}_{j,\rm in}^{k}\|_{0,0; \tilde{\Gamma}^{+}_{\rm in}}
 +\|f^{k}_{j}\|_{1,0; \tilde{\mathcal{N}}^{\rm{I}}_{j}\cup\tilde{\mathcal{N}}^{\rm{III}}_{j}}
+\|\delta Y^{*}_{j}\|_{1,0; \tilde{\mathcal{N}}^{\rm{I}}_{j}\cup\tilde{\mathcal{N}}^{\rm{III}}_{j}}\Big).
\end{split}
\end{eqnarray}

Combing it with the relations
\begin{eqnarray}\label{eq:prop-3.4-34}
\p_{\eta} \delta \omega^{*}_{-}=\frac{\mathcal {Z}_{-}^{+,*}+\mathcal {Z}_{-}^{-,*}}{2},\quad \p_{\eta} \delta p^{*}_{-}=\frac{\mathcal {Z}_{-}^{+,*}-\mathcal {Z}_{-}^{-,*}}{2\Lambda_{-}},
\end{eqnarray}
and \eqref{eq:prop-3.2-17}-\eqref{eq:prop-3.2-18} and $\eqref{eq:IBVP-III}_2$, we can further conclude that
\begin{align}\label{eq:prop-3.4-35}
\begin{split}
&\sum_{j=\pm}\|(\nabla\delta\omega^{*}_{j},\nabla\delta p^{*}_{j})\|_{0,0; \tilde{\mathcal{N}}^{\rm{III}}_j}\\
&\leq \tilde{C}^{*}_{\rm{III}}\sum_{j=\pm}\Big(\|(\delta\omega^{j}_{\rm in},\delta p^{j}_{\rm in})\|_{1,0; \tilde{\Gamma}^{j}_{\rm in}}
+\|\delta Y^{*}_{j}\|_{1,0; \tilde{\mathcal{N}}^{\rm{I}}_{j}\cup\tilde{\mathcal{N}}^{\rm{III}}_{j}}+\sum_{k=\pm}\|f^{k}_{j}\|_{1,0; \tilde{\mathcal{N}}^{\rm{I}}_{j}\cup\tilde{\mathcal{N}}^{\rm{III}}_{j}}\Big),
\end{split}
\end{align}
where the constant $\tilde{C}^*_{\rm{III}}>0$ depends only on $\underline{\mathcal{U}}$ and $L$.

Finally, we will show the $C^{\alpha}$-estimates of $(\nabla\delta \omega^{*}_{\pm},\nabla\delta p^{*}_{\pm})$ in $\tilde{\mathcal{N}}^{\rm{I}}_{\pm}\cup\tilde{\mathcal{N}}^{\rm{III}}_{\pm}$. Introduce 
\begin{eqnarray}\label{eq:prop-3.4-36}
\begin{cases}
[\mathcal {Z}^{\pm,*}_{\textrm{III}^{k}}](\kappa)=\sup\limits_{(\xi_1,\eta_1),(\xi_2,\eta_2)\in R_{k}^{\rm{III}}(\kappa)}\frac{|\mathcal {Z}^{\pm,*}_{k}(\xi_1,\eta_1)-\mathcal {Z}^{\pm,*}_{k}(\xi_2,\eta_2)|}{|(\xi_1,\eta_1)-(\xi_2,\eta_2)|^{\alpha}},\\
[\mathcal {Z}^{\pm,*}_{\textrm{III}^{k}}]_{\xi}(\kappa)=\sup\limits_{(\xi_1,\eta),(\xi_2,\eta)\in R_{k}^{\rm{III}}(\kappa)}\frac{|\mathcal {Z}^{\pm,*}_{k}(\xi_1,\eta)-\mathcal {Z}^{\pm,*}_{k}(\xi_2,\eta)|}{|\xi_1-\xi_2|^{\alpha}},\\
[\mathcal {Z}^{\pm,*}_{\textrm{III}^{k}}]_{\eta}(\kappa)=\sup\limits_{(\xi,\eta_1),(\xi,\eta_2)\in R_{k}^{\rm{III}}(\kappa)}\frac{|\mathcal {Z}^{\pm,*}_{k}(\xi,\eta_1)-\mathcal {Z}^{\pm,*}_{k}(\xi,\eta_2)|}{|\eta_1-\eta_2|^{\alpha}},
\end{cases}
\end{eqnarray}
where $k=\pm$. Then, for $k=\pm$ {\color{black}and for any $\kappa \in[0,\bar{\xi}_+^1]$}, we have
\begin{eqnarray}\label{eq:prop-3.4-37}
\begin{split}
[\mathcal {Z}^{\pm,*}_{\textrm{III}^{k}}](\kappa) \leq [\mathcal {Z}^{\pm,*}_{\textrm{III}^{k}}]_{\xi}(\kappa)+[\mathcal {Z}^{\pm,*}_{\textrm{III}^{k}}]_{\eta}(\kappa),
\quad [\mathcal {Z}^{\pm,*}_{\textrm{III}^{k}}]_{0,\alpha; \tilde{\mathcal{N}}^{\textrm{I}}_{k}\cup\tilde{\mathcal{N}}^{\textrm{III}}_{k}}=[\mathcal {Z}^{\pm,*}_{\textrm{III}^{k}}](\bar{\xi}^{2}_{\rm{cd}}),
\end{split}
\end{eqnarray}
and
\begin{eqnarray}\label{eq:prop-3.4-38}
\begin{split}
[\mathcal {Z}^{\pm,*}_{\textrm{III}^{k}}](\kappa) \geq [\mathcal {Z}^{\pm,*}_{\textrm{I}^{k}}](\kappa),\quad
[\mathcal {Z}^{\pm,*}_{\textrm{III}^{k}}]_{\xi}(\kappa) \geq [\mathcal {Z}^{\pm,*}_{\textrm{I}^{k}}]_{\xi}(\kappa),\quad
[\mathcal {Z}^{\pm,*}_{\textrm{III}^{k}}]_{\eta}(\kappa) \geq [\mathcal {Z}^{\pm,*}_{\textrm{I}^{k}}]_{\eta}(\kappa).
\end{split}
\end{eqnarray}

According to \eqref{eq:prop-3.4-37}, we need to derive the estimates of $[\mathcal {Z}^{\pm,*}_{\rm{III}^{+}}]_{\xi}(\kappa)$ and $[\mathcal {Z}^{\pm,*}_{\rm{III}^{+}}]_{\eta}(\kappa)$. First for $[\mathcal {Z}^{\pm,*}_{\rm{III}^{+}}]_{\eta}(\kappa)$, without loss of the generality, we consider 
the case that $(\bar{\xi},\bar{\eta}_1),\ (\bar{\xi},\bar{\eta}_2)\in \tilde{\mathcal{N}}^{\rm{III}}_{+}$ with $\bar{\eta}_1\neq \bar{\eta}_2$ and $\bar{\xi}$ being fixed, and assume $\zeta^{+,\rm{cd}}_{+}(\bar{\xi},\bar{\eta}_1)<\zeta^{+, \rm{cd}}_{+}(\bar{\xi},\bar{\eta}_2)$.
Then, applying $\eqref{eq:prop-3.4-22}_{1}$ for $(\bar{\xi}, \bar{\eta}_1)$ and $(\bar{\xi},\bar{\eta}_2)$, respectively, and taking difference to get
\begin{align}\label{eq:prop-3.4-39}
\begin{split}
\big|\mathcal {Z}^{+,*}_{+}(\bar{\xi},\bar{\eta}_1)-\mathcal {Z}^{+,*}_{+}(\bar{\xi},\bar{\eta}_2)\big|\leq\sum_{k=1}^{3}\mathcal {J}^{+}_{\textrm{III}^{+}_{k}},
\end{split}
\end{align}
where
{
\[
\mathcal {J}^{+}_{\rm{III}^{+}_{1}}\triangleq\big|\mathcal {Z}^{+,*}_{+}(\zeta^{+,\rm{cd}}_{+}(\bar{\xi},\bar{\eta}_2),0)-\mathcal {Z}^{+,*}_{+}(\zeta^{+,\rm{cd}}_{+}(\bar{\xi},\bar{\eta}_1),0)\big|,
\]}
{
\begin{align*}
\begin{split}
&\mathcal {J}^{+}_{\rm{III}^{+}_{2}}
\triangleq
\int_{\zeta^{+,\rm{cd}}_{+}(\bar{\xi},\bar{\eta}_2)}^{\bar{\xi}}\Big|\big(\frac{\p^{+}_{+}\Lambda_{+}-2\Lambda_{+}\p_{\eta}\lambda^{+}_{+}}{2\Lambda_{+}}\mathcal {Z}_{+}^{+,*}\big)(\tau,\Upsilon^{+,\textrm{cd}}_{+}(\tau;\bar{\xi},\bar{\eta}_1))\\
&\quad -\big(\frac{\p^{+}_{+}\Lambda_{+}-2\Lambda_{+}\p_{\eta}\lambda^{+}_{+}}{2\Lambda_{+}}\mathcal {Z}_{+}^{+,*}\big)(\tau,\Upsilon^{+,\textrm{cd}}_{+}(\tau;\bar{\xi},\bar{\eta}_2))\Big|\dd\tau\\
&+\int_{\zeta^{+,\rm{cd}}_{+}(\bar{\xi},\bar{\eta}_2)}^{\bar{\xi}}\Big|\big(\frac{\p^-_{+}\Lambda_{+}}{2\Lambda_{+}}\mathcal {Z}_{+}^{-,*}\big)(\tau,\Upsilon^{+,\textrm{cd}}_{+}(\tau;\bar{\xi},\bar{\eta}_1))-\big(\frac{\p^-_{+}\Lambda_{+}}{2\Lambda_{+}}\mathcal {Z}_{+}^{-,*}\big)(\tau,\Upsilon^{+,\textrm{cd}}_{+}(\tau;\bar{\xi},\bar{\eta}_2))\Big|\dd\tau\\
&+|a^{+}_{+}(\underline{\mathcal{U}}_{+})|
\int_{\zeta^{+,\rm{cd}}_{+}(\bar{\xi},\bar{\eta}_1)}^{\bar{\xi}}\Big|\p_{\eta}\delta Y^{*}_{+}(\tau,\Upsilon^{+,\textrm{cd}}_{+}(\tau;\bar{\xi},\bar{\eta}_1))
-\p_{\eta}\delta Y^{*}_{+}(\tau,\Upsilon^{+,\textrm{cd}}_{+}(\tau;\bar{\xi},\bar{\eta}_2))\Big|\dd \tau\\
&+\frac{|a^{+}_{+}(\underline{\mathcal{U}}_{+})|+|a^{-}_{+}(\underline{\mathcal{U}}_{+})|}{2}\\
&\qquad \times\int_{\zeta^{+,\rm{cd}}_{+}(\bar{\xi},\bar{\eta}_1)}^{\bar{\xi}}
\Big|\big(\frac{\p_{\eta}\Lambda_{+}}{\Lambda_{+}}Y_{+}^{*}\big)(\tau,\Upsilon^{+,\rm{cd}}_{+}(\tau;\bar{\xi},\bar{\eta}_1))
-\big(\frac{\p_{\eta}\Lambda_{+}}{\Lambda_{+}}Y_{+}^{*}\big)(\tau,\Upsilon^{+,\textrm{cd}}_{+}(\tau;\bar{\xi},\bar{\eta}_2))\Big|\dd\tau\\
&+\int_{\zeta^{+,\rm{cd}}_{+}(\bar{\xi},\bar{\eta}_2)}^{\bar{\xi}}\big|\p_{\eta}f^{+}_{+}(\delta\mathcal{V}_{+}, \delta Y_{+})(\tau,\Upsilon^{+,\textrm{cd}}_{+}(\tau;\bar{\xi},\bar{\eta}_1))
-\p_{\eta}f^{+}_{+}(\delta\mathcal{V}_{+}, \delta Y_{+})(\tau,\Upsilon^{+,\textrm{cd}}_{+}(\tau;\bar{\xi},\bar{\eta}_2))\big|\dd \tau\\
&+\sum_{k=\pm}\int_{\zeta^{+,\rm{cd}}_{+}(\bar{\xi},\bar{\eta}_2)}^{\bar{\xi}}\Big|\big(\frac{\p_{\eta} \Lambda_{+}}{2\Lambda_{+}}f^{k}_{+}(\delta\mathcal{V}_{+}, \delta Y_{+})\big)(\tau,\Upsilon^{+,\textrm{cd}}_{+}(\tau;\bar{\xi},\bar{\eta}_1))\\
&\qquad\qquad\qquad\qquad\qquad -\big(\frac{\p_{\eta} \Lambda_{+}}{2\Lambda_{+}}f^{k}_{+}(\delta\mathcal{V}_{+}, \delta Y_{+})\big)(\tau,\Upsilon^{+,\textrm{cd}}_{+}(\tau;\bar{\xi},\bar{\eta}_2))\Big|\dd \tau,
\end{split}
\end{align*}
}
and
\begin{align*}
\begin{split}
&\mathcal{J}^{+}_{\rm{III}^{+}_{3}}
\triangleq\int_{\zeta^{+,\rm{cd}}_{+}(\bar{\xi},\bar{\eta}_1)}^{\zeta^{+,\rm{cd}}_{+}(\bar{\xi},\bar{\eta}_2)}
\Big|\Big(\frac{\p^{+}_{+}\Lambda_{+}-2\Lambda_{+}\p_{\eta}\lambda^{+}_{+}}{2\Lambda_{+}}\mathcal {Z}_{+}^{+,*}-\frac{\p^{-}_{+}\Lambda_{+}}{2\Lambda_{+}}\mathcal {Z}_{+}^{-,*}\Big)(\tau,\Upsilon^{+,\textrm{cd}}_{+}(\tau;\bar{\xi},\bar{\eta}_{2}))\Big|\dd \tau\\
&+\int_{\zeta^{+,\rm{cd}}_{+}(\bar{\xi},\bar{\eta}_1)}^{\zeta^{+,\rm{cd}}_{+}(\bar{\xi},\bar{\eta}_2)}\Big|\big(a^{-}_{+}(\underline{\mathcal{U}}_{+})\p_{\eta}\delta Y^{*}_{+}-\frac{[a^{-}_{+}(\underline{\mathcal{U}}_{+})-a^{+}_{+}(\underline{\mathcal{U}}_{+})]\p_{\eta} \Lambda_{+}}{2\Lambda_{+}}\delta Y^{*}_{+} \Big)(\tau,\Upsilon^{+,\textrm{cd}}_{+}(\tau;\bar{\xi},\bar{\eta}_2))\Big|\dd \tau\\
&+\int_{\zeta^{+,\rm{cd}}_{+}(\bar{\xi},\bar{\eta}_1)}^{\zeta^{+,\rm{cd}}_{+}(\bar{\xi},\bar{\eta}_2)}\Big|\Big(\p_{\eta}f^{+}_{+}(\delta\mathcal{V}_{+}, \delta Y_{+})-\frac{\p_{\eta} \Lambda_{+}}{2\Lambda_{+}}\big[f^{+}_{+}(\delta\mathcal{V}_{+}, \delta Y_{+})
-f^{-}_{+}(\delta\mathcal{V}_{+}, \delta Y_{+})\big]\Big)(\tau,\Upsilon^{+,\textrm{cd}}_{+}(\tau;\bar{\xi},\bar{\eta}_{2}))\Big| \dd \tau.
\end{split}
\end{align*}

For $\mathcal{J}^{+}_{\rm{III}^{+}_{2}}$ and $\mathcal{J}^{+}_{\rm{III}^{+}_{3}}$, we can follow the ways in \emph{Step 3} in the proof of Proposition \ref{prop:3.3} and by Lemma \ref{lem:A3} to deduce that
\begin{eqnarray*}
\begin{split}
\mathcal{J}^{+}_{\rm{III}^{+}_{2}}+\mathcal{J}^{+}_{\rm{III}^{+}_{3}}
&\leq\mathcal {O}(1)\Big\{\sum_{k=\pm}\|\mathcal {Z}_{+}^{\pm,*}\|_{0,\alpha; \tilde{\mathcal{N}}^{\rm{I}}_{+}\cup\tilde{\mathcal{N}}^{\rm{III}}_{+}}
+\|f^{+}_{+}\|_{1,\alpha; \tilde{\mathcal{N}}^{\rm{I}}_{+}\cup\tilde{\mathcal{N}}^{\rm{III}}_{+}}
+\|f^{-}_{+}\|_{1,0; \tilde{\mathcal{N}}^{\rm{I}}_{+}\cup\tilde{\mathcal{N}}^{\rm{III}}_{+}}\\
&\qquad\qquad\  +\|\delta Y^{*}_{+}\|_{1,\alpha; \tilde{\mathcal{N}}^{\rm{I}}_{+}\cup\tilde{\mathcal{N}}^{\rm{III}}_{+}}+\int_{\zeta^{+,\rm{cd}}_{+}(\bar{\xi},\bar{\eta}_2)}^{\bar{\xi}}\sum_{k=\pm}[\mathcal {Z}^{k,*}_{\rm{III}^{+}}]_{\eta}(\varsigma)\dd \varsigma\Big\}\cdot \big|\bar{\eta}_{1}-\bar{\eta}_{2}\big|^{\alpha},
\end{split}
\end{eqnarray*}
where the constant $\mathcal {O}(1)$ depends only on $\underline{\mathcal{U}}_{+}$, $\alpha$ and $L$.

For $\mathcal {J}^{+}_{\rm{III}^{+}_{1}}$, by \eqref{eq:prop-3.4-32} on $\tilde{\Gamma}_{\rm{cd}}\cap \overline{\tilde{\mathcal{N}}^{\rm{III}}_{+}}$, we can get
\begin{eqnarray*}\label{eq:prop-3.4-28}
\begin{split}
&\mathcal {J}^{+}_{\rm{III}^{+}_{1}}\\
&\leq
\bigg|\Big(\frac{(\Lambda_{+}-\Lambda_{-})\lambda_{+}^{-}}{(\Lambda_{+}+\Lambda_{-})\lambda_{+}^{+}}\mathcal {Z}^{-,*}_{+}\Big)(\zeta^{+,\rm{cd}}_{+}(\bar{\xi},\bar{\eta}_1),0)
-\Big(\frac{(\Lambda_{+}-\Lambda_{-})\lambda_{+}^{-}}{(\Lambda_{+}+\Lambda_{-})\lambda_{+}^{+}}\mathcal {Z}^{-,*}_{+}\Big)(\zeta^{+,\rm{cd}}_{+}(\bar{\xi},\bar{\eta}_2),0)\bigg|\\
&+\bigg|\Big(\frac{2\Lambda_{+}\lambda_{+}^{-}}{(\Lambda_{+}+\Lambda_{-})\lambda_{+}^{+}}\mathcal {Z}^{+,*}_{-}\Big)(\zeta^{+,\rm{cd}}_{+}(\bar{\xi},\bar{\eta}_1),0)-\Big(\frac{2\Lambda_{+}\lambda_{+}^{-}}{(\Lambda_{+}+\Lambda_{-})\lambda_{+}^{+}}\mathcal {Z}^{+,*}_{-}\Big)(\zeta^{+,\rm{cd}}_{+}(\bar{\xi},\bar{\eta}_2),0)\bigg|\\
&+\bigg|\Big(\frac{(a _{+}^{+}+a _{+}^{-})\Lambda_{+}+(a _{+}^{+}-a _{+}^{-})\Lambda_{-}}{(\Lambda_{+}+\Lambda_{-})\lambda_{+}^{+}}\delta Y^{*}_{+}\Big)(\zeta^{+,\rm{cd}}_{+}(\bar{\xi},\bar{\eta}_1),0)\\
&\qquad\ -\Big(\frac{(a _{+}^{+}+a _{+}^{-})\Lambda_{+}+(a _{+}^{+}-a _{+}^{-})\Lambda_{-}}{(\Lambda_{+}+\Lambda_{-})\lambda_{+}^{+}}\delta Y^{*}_{+}\Big)(\zeta^{+,\rm{cd}}_{+}(\bar{\xi},\bar{\eta}_2),0)\bigg|\\
&+\bigg|\Big(\frac{2a_{-}^{-}\Lambda_{+}}{(\Lambda_{+}+\Lambda_{-})\lambda_{+}^{+}}\delta Y^{*}_{-}\Big)(\zeta^{+,\rm{cd}}_{+}(\bar{\xi},\bar{\eta}_1),0)-\Big(\frac{2a_{-}^{-}\Lambda_{+}}{(\Lambda_{+}+\Lambda_{-})\lambda_{+}^{+}}\delta Y^{*}_{-}\Big)(\zeta^{+,\rm{cd}}_{+}(\bar{\xi},\bar{\eta}_2),0)\bigg|\\
&+\bigg|\Big(\frac{\Lambda_{+}-\Lambda_{-}}{(\Lambda_{+}+\Lambda_{-})\lambda_{+}^{+}}f^{+}_{+}(\delta\mathcal{V}_{+}, \delta Y_{+})\Big)(\zeta^{+,\rm{cd}}_{+}(\bar{\xi},\bar{\eta}_1),0)-\Big(\frac{\Lambda_{+}-\Lambda_{-}}{(\Lambda_{+}+\Lambda_{-})\lambda_{+}^{+}}
f^{+}_{+}(\delta\mathcal{V}_{+}, \delta Y_{+})\Big)(\zeta^{+,\rm{cd}}_{+}(\bar{\xi},\bar{\eta}_2),0)\bigg|\\
&+\bigg|\Big(\frac{1}{\lambda_{+}^{+}}f^{-}_{+}(\delta\mathcal{V}_{+}, \delta Y_{+})\Big)(\zeta^{+,\rm{cd}}_{+}(\bar{\xi},\bar{\eta}_1),0)-\Big(\frac{1}{\lambda_{+}^{+}}f^{-}_{+}(\delta\mathcal{V}_{+}, \delta Y_{+})\Big)(\zeta^{+,\rm{cd}}_{+}(\bar{\xi},\bar{\eta}_2),0)\bigg|\\
&+\bigg|\Big(\frac{2\Lambda_{+}}{(\Lambda_{+}+\Lambda_{-})\lambda_{+}^{+}}f^{-}_{-}(\delta\mathcal{V}_{-}, \delta Y_{-})\Big)(\zeta^{+,\rm{cd}}_{+}(\bar{\xi},\bar{\eta}_1),0)-\Big(\frac{2\Lambda_{+}}{(\Lambda_{+}+\Lambda_{-})\lambda_{+}^{+}}
f^{-}_{-}(\delta\mathcal{V}_{-}, \delta Y_{-})\Big)(\zeta^{+,\rm{cd}}_{+}(\bar{\xi},\bar{\eta}_2),0)\bigg|\\
\triangleq& \sum^{7}_{n=1}\mathcal {J}^{+}_{\textrm{III}^{+}_{1,n}}.
\end{split}
\end{eqnarray*}

We can follow the procedure in \emph{Step 3} of Proposition \ref{prop:3.2} and by Lemma \ref{lem:A3} to get that
\begin{align}\label{eq:prop-3.4-40}
\begin{split}
&\mathcal {J}^{+}_{\textrm{III}^{+}_{1,1}}+\mathcal {J}^{+}_{\textrm{III}^{+}_{1,2}}\\
&\leq \mathcal{O}(1)\bigg\{[\mathcal {Z}^{-}_{+,\rm in}]_{0,\alpha;\tilde{\Gamma}^{+}_{\rm in}}+[\mathcal {Z}^{+}_{-,\rm in}]_{0,\alpha;\tilde{\Gamma}^{-}_{\rm in}}
+\|\delta Y^{*}_{+}\|_{1,\alpha; \tilde{\mathcal{N}}^{\rm{I}}_{+}\cup\tilde{\mathcal{N}}^{\rm{III}}_{+}}+\|\delta Y^{*}_{-}\|_{1,\alpha; \tilde{\mathcal{N}}^{\rm{I}}_{-}\cup\tilde{\mathcal{N}}^{\rm{III}}_{-}}\\
&\quad\ +\sum_{k=\pm}\Big(\|f^{k}_{+}\|_{1,\alpha; \tilde{\mathcal{N}}^{I}_{+}\cup\tilde{\mathcal{N}}^{\rm{I}}_{+}}+\|f^{k}_{-}\|_{1,\alpha; \tilde{\mathcal{N}}^{\rm{I}}_{-}\cup\tilde{\mathcal{N}}^{\rm{III}}_{-}}\Big)\\
&\quad\ +\int_0^{\zeta^{+,\rm{cd}}_{+}(\bar{\xi},\bar{\eta}_1)}\sum_{k=\pm}\Big([\mathcal {Z}^{k,*}_{\rm{III}^{+}}]_{\eta}(\tau)+[\mathcal {Z}^{k,*}_{\rm{III}^{-}}]_{\eta}(\tau)\Big)\dd\tau\bigg\}\cdot |\bar{\eta}_1-\bar{\eta}_2|^{\alpha},
\end{split}
\end{align}
where the constant $\mathcal{O}(1)$ depends only on $\underline{\mathcal{U}}$, $\alpha$ and $L$.
By direct computation, we can also get 
\begin{eqnarray*}
\begin{split}
\sum^{7}_{n=3}\mathcal{J}^{+}_{\textrm{III}^{+}_{1,n}}
&\leq \mathcal {O}(1)\Big(\sum_{j=\pm}\|\delta Y^{*}_{j}\|_{1,0;\tilde{\mathcal{N}}^{\rm{I}}_{j}\cup\tilde{\mathcal{N}}^{\rm{III}}_{j}}
+\sum_{k=\pm}\|f^{k}_{+}\|_{1,0;\tilde{\mathcal{N}}^{\rm{I}}_{+}\cup\tilde{\mathcal{N}}^{\rm{II}}_{+}}+\|f^{-}_{-}\|_{1,0;\tilde{\mathcal{N}}^{\rm{I}}_{-}\cup\tilde{\mathcal{N}}^{\rm{III}}_{-}}\Big) \cdot |\bar{\eta}_1-\bar{\eta}_{2}|^{\alpha},
\end{split}
\end{eqnarray*}
where the constant $\mathcal {O}(1)$ also depends only on $\underline{\mathcal{U}}_{+}$, $\alpha$ and $L$.
Therefore,
\begin{eqnarray*}
\begin{split}
\mathcal {J}^{+}_{\rm{III}^{+}_{1}}&\leq\mathcal{O}(1)\bigg\{[\mathcal {Z}^{-}_{+,\rm in}]_{0,\alpha;\tilde{\Gamma}^{+}_{\rm in}}+[\mathcal {Z}^{+}_{-,\rm in}]_{0,\alpha;\tilde{\Gamma}^{-}_{\rm in}}+\sum_{j=\pm}\|\delta Y^{*}_{j}\|_{1,\alpha;\tilde{\mathcal{N}}^{\rm{I}}_{j}\cup\tilde{\mathcal{N}}^{\rm{III}}_{j}}
+\sum_{j,k=\pm}\|f^{k}_{+}\|_{1,\alpha;\tilde{\mathcal{N}}^{\rm{I}}_{j}\cup\tilde{\mathcal{N}}^{\rm{III}}_{j}}\\
&\qquad\qquad\quad +\int_0^{\zeta^{+,\rm{cd}}_{+}(\bar{\xi},\bar{\eta}_1)}\sum_{j,k=\pm}[\mathcal {Z}^{k,*}_{\textrm{III}^{j}}]_{\eta}(\varsigma)\dd \varsigma\bigg\}
 \cdot |\bar{\eta}_1-\bar{\eta}_{2}|^{\alpha}.
\end{split}
\end{eqnarray*}

Hence it follows from \eqref{eq:prop-3.4-39} that
\begin{eqnarray}\label{eq:prop-3.4-41}
\begin{split}
[\mathcal {Z}^{+,*}_{\rm{III}^{+}}]_{\eta}(\kappa)&\leq\mathcal{O}(1)\bigg\{\sum_{k=\pm}[\mathcal {Z}^{k}_{+,\rm in}]_{0,\alpha;\tilde{\Gamma}^{+}_{\rm in}}+[\mathcal {Z}^{+}_{-,\rm in}]_{0,\alpha;\tilde{\Gamma}^{-}_{\rm in}}+\sum_{j=\pm}\|\delta Y^{*}_{j}\|_{1,\alpha;\tilde{\mathcal{N}}^{\rm{I}}_{j}\cup\tilde{\mathcal{N}}^{\rm{III}}_{j}}\\
&\qquad\qquad\ +\sum_{j,k=\pm}\|f^{k}_{+}\|_{1,\alpha;\tilde{\mathcal{N}}^{\rm{I}}_{j}\cup\tilde{\mathcal{N}}^{\rm{III}}_{j}}
+\int_0^{\bar{\xi}}\sum_{j,k=\pm}[\mathcal {Z}^{k,*}_{\textrm{III}^{j}}]_{\eta}(\varsigma)\dd \varsigma\bigg\}.
\end{split}
\end{eqnarray}

For $\mathcal {Z}^{-,*}_{+}$, we can follow the ways to estimate $[\mathcal {Z}^{+,*}_{\rm{I}^{+}}](\kappa)$
in \emph{Step 3} of the proof of Proposition \ref{prop:3.2}, and apply the estimate \eqref{eq:prop-3.3-21} to derive that
\begin{align}\label{eq:prop-3.4-42}
\begin{split}
[\mathcal {Z}^{-,*}_{\rm{III}^{+}}]_{\eta}(\kappa)
&\leq \mathcal{O}(1)\Big\{\sum_{k=\pm}\big(\|\mathcal {Z}^{k}_{+,\rm in}\|_{0,\alpha;\tilde{\Gamma}^{+}_{\rm in}}+\|f^{k}_{+}\|_{1,\alpha; \tilde{\mathcal{N}}^{\rm{I}}_{+}\cup\tilde{\mathcal{N}}^{\rm{III}}_{+}}\big)\\
&\qquad\qquad\quad +\|\delta Y^{*}_{+}\|_{1,\alpha; \tilde{\mathcal{N}}^{\rm{I}}_{+}\cup\tilde{\mathcal{N}}^{\rm{II}}_{+}}
+\int_0^{\kappa}\sum_{k=\pm}[\mathcal {Z}^{k,*}_{\rm{III}^{+}}]_{\eta}(\varsigma)\dd \varsigma\Big\},
\end{split}
\end{align}
where the constant $\mathcal {O}(1)$ depends only on $\underline{\mathcal{U}}_{+}$, $\alpha$ and $L$.

Similarly, by boundary condition \eqref{eq:prop-3.4-32} and estimate \eqref{eq:prop-3.4-33a},
we also have
\begin{align}\label{eq:prop-3.4-43}
\begin{split}
[\mathcal {Z}^{+,*}_{\rm{III}^{-}}]_{\eta}(\kappa)
&\leq \mathcal{O}(1)\Big\{\sum_{k=\pm}\big(\|\mathcal {Z}^{k}_{-,\rm in}\|_{0,\alpha;\tilde{\Gamma}^{-}_{\rm in}}+\|f^{k}_{-}\|_{1,\alpha; \tilde{\mathcal{N}}^{\rm{I}}_{+}\cup\tilde{\mathcal{N}}^{\rm{III}}_{+}}\big)\\
&\qquad\qquad\ +\|\delta Y^{*}_{-}\|_{1,\alpha; \tilde{\mathcal{N}}^{\rm{I}}_{-}\cup\tilde{\mathcal{N}}^{\rm{III}}_{-}}
+\int_0^{\kappa}\sum_{k=\pm}[\mathcal {Z}^{k,*}_{\rm{III}^{-}}]_{\eta}(\varsigma)\dd \varsigma\Big\},
\end{split}
\end{align}
and
\begin{eqnarray}\label{eq:prop-3.4-44}
\begin{split}
[\mathcal {Z}^{-,*}_{\rm{III}^{-}}]_{\eta}(\kappa)&\leq\mathcal{O}(1)\bigg\{\sum_{k=\pm}[\mathcal {Z}^{k}_{-,\rm in}]_{0,\alpha;\tilde{\Gamma}^{-}_{\rm in}}+[\mathcal {Z}^{-}_{+,\rm in}]_{0,\alpha;\tilde{\Gamma}^{+}_{\rm in}}+\sum_{j=\pm}\|\delta Y^{*}_{j}\|_{1,\alpha;\tilde{\mathcal{N}}^{\rm{I}}_{j}\cup\tilde{\mathcal{N}}^{\rm{III}}_{j}}
\\
&\qquad\qquad\ +\sum_{j,k=\pm}\|f^{k}_{+}\|_{1,\alpha;\tilde{\mathcal{N}}^{\rm{I}}_{j}\cup\tilde{\mathcal{N}}^{\rm{III}}_{j}}+\int_0^{\bar{\xi}}\sum_{j,k=\pm}[\mathcal {Z}^{k,*}_{\textrm{III}^{j}}]_{\eta}(\varsigma)\dd \varsigma\bigg\}.
\end{split}
\end{eqnarray}
So combining the estimates \eqref{eq:prop-3.4-41}-\eqref{eq:prop-3.4-44} and applying the Gronwall's inequality, we arrive at
\begin{align}\label{eq:prop-3.4-45}
\begin{split}
&\sum_{j,k=\pm}[\mathcal {Z}^{k,*}_{\textrm{III}^{j}}]_{\eta}(\kappa)\\
&\leq \mathcal{O}(1)\sum_{j=\pm}\Big\{\sum_{k=\pm}[\mathcal {Z}^{k}_{j,\rm in}]_{0,\alpha;\tilde{\Gamma}^{-}_{\rm in}}+\|\delta Y^{*}_{j}\|_{1,\alpha;\tilde{\mathcal{N}}^{\rm{I}}_{j}\cup\tilde{\mathcal{N}}^{\rm{III}}_{j}}
+\sum_{k=\pm}\|f^{k}_{+}\|_{1,\alpha;\tilde{\mathcal{N}}^{\rm{I}}_{j}\cup\tilde{\mathcal{N}}^{\rm{III}}_{j}}\Big\}.
\end{split}
\end{align}

Next, in a similar way as done above, we can also get
\begin{align}\label{eq:prop-3.4-46}
\begin{split}
&\sum_{j,k=\pm}[\mathcal {Z}^{k,*}_{\textrm{III}^{j}}]_{\xi}(\kappa)\\
&\leq \mathcal{O}(1)\sum_{j=\pm}\Big\{\sum_{k=\pm}[\mathcal {Z}^{k}_{j,\rm in}]_{0,\alpha;\tilde{\Gamma}^{-}_{\rm in}}+\|\delta Y^{*}_{j}\|_{1,\alpha;\tilde{\mathcal{N}}^{\rm{I}}_{j}\cup\tilde{\mathcal{N}}^{\rm{III}}_{j}}
+\sum_{k=\pm}\|f^{k}_{+}\|_{1,\alpha;\tilde{\mathcal{N}}^{\rm{I}}_{j}\cup\tilde{\mathcal{N}}^{\rm{III}}_{j}}\Big\}.
\end{split}
\end{align}

So, by \eqref{eq:prop-3.4-37}, \eqref{eq:prop-3.4-45} and \eqref{eq:prop-3.4-46}, and letting $\kappa=\bar{\xi}^{2}_{\rm{cd}}$, it follows that
\begin{eqnarray*}
\begin{split}
\sum_{j,k=\pm}[\mathcal {Z}^{k,*}_{j}]_{0,\alpha;\tilde{\mathcal{N}}^{\rm{I}}_{j}\cup\tilde{\mathcal{N}}^{\rm{III}}_{j}}\leq \mathcal{O}(1)\sum_{j=\pm}\Big\{\sum_{k=\pm}[\mathcal {Z}^{k}_{j,\rm in}]_{0,\alpha;\tilde{\Gamma}^{-}_{\rm in}}+\|\delta Y^{*}_{j}\|_{1,\alpha;\tilde{\mathcal{N}}^{\rm{I}}_{j}\cup\tilde{\mathcal{N}}^{\rm{III}}_{j}}
+\sum_{k=\pm}\|f^{k}_{j}\|_{1,\alpha;\tilde{\mathcal{N}}^{\rm{I}}_{j}\cup\tilde{\mathcal{N}}^{\rm{III}}_{j}}\Big\},
\end{split}
\end{eqnarray*}
which further implies that
\begin{eqnarray}\label{eq:prop-3.4-47}
\begin{split}
&\sum_{j=\pm}[(\nabla \delta\omega^{*}_{j},\nabla \delta p^{*}_{j})]_{0,\alpha;\tilde{\mathcal{N}}^{\rm{I}}_{j}\cup\tilde{\mathcal{N}}^{\rm{III}}_{j}}\\
& \leq \tilde{C}^{**}_{\rm{III}}\sum_{j=\pm}\Big\{\|(\delta\omega^{j}_{\rm{in}},\delta p^{j}_{\rm{in}})\|_{1,\alpha;\tilde{\mathcal{N}}^{\rm{I}}_{j}\cup\tilde{\mathcal{N}}^{\rm{III}}_{j}}+\|\delta Y^{*}_{j}\|_{1,\alpha;\tilde{\mathcal{N}}^{\rm{I}}_{j}\cup\tilde{\mathcal{N}}^{\rm{III}}_{j}}
 +\sum_{k=\pm}\|f^{k}_{+}\|_{1,\alpha;\tilde{\mathcal{N}}^{\rm{I}}_{j}\cup\tilde{\mathcal{N}}^{\rm{III}}_{j}}\Big\},
\end{split}
\end{eqnarray}
where the constant $\tilde{C}^{**}_{\rm{III}}>0$ depends only on $\underline{\mathcal{U}}_{+}$, $\alpha$ and $L$.

Finally, set $\tilde{C}^{**}_{\rm{III}}=\max\{\tilde{C}_{\rm{I}}, \tilde{C}^{*}_{\rm{II}}, \tilde{C}^{**}_{\rm{III}}\}$.
Then, with the estimates \eqref{eq:prop-3.4-18}, \eqref{eq:prop-3.4-35} and \eqref{eq:prop-3.4-47}, we established the estimate \eqref{eq:prop-3.4-1}.
\end{proof}

\subsection{Proof of the Theorem \ref{thm:3.1}}
In this subsection, we are going to complete the proof of Theorem \ref{thm:3.1}. First of all, from Proposition \ref{prop:3.1}, we can get the existence and uniqueness of a
$C^{1,\alpha}$-solution $(\delta B^{*}_{\pm}, \delta S^{*}_{\pm},\delta Y^{*}_{\pm})$ in $\tilde{\mathcal{N}}_{\pm}$ with the estimates \eqref{eq:prop-3.1-1}.
Next, we can construct a solution $(\delta \omega^{*}_{\pm}, \delta p^{*}_{\pm})$ 
in $\tilde{\mathcal{N}}_{\pm}$ by an induction procedure as follows. 
First, with $(\delta B^{*}_{\pm}, \delta S^{*}_{\pm},\delta Y^{*}_{\pm})$ in hand, by Propositions \ref{prop:3.2}-\ref{prop:3.4}, we can obtain the existence and uniqueness of a solution $(\delta \omega^{*}_{\pm}, \delta p^{*}_{\pm})$ to $(\mathbf{IBVP})^{*}$ in $\tilde{\mathcal{N}}^{\rm{I}}_{\pm}\cup\tilde{\mathcal{N}}^{\rm{II}}_{\pm}\cup\tilde{\mathcal{N}}^{\rm{III}}_{\pm}$ with the estimate
\begin{eqnarray}\label{eq:prop-3.4-48}
\begin{split}
&\sum_{k=\pm}\|(\delta \omega^{*}_{k},\delta p^{*}_{k})\|_{1,0; \tilde{\mathcal{N}}^{\rm{I}}_{k}\cup\tilde{\mathcal{N}}^{\rm{II}}_{k}\cup\tilde{\mathcal{N}}^{\rm{III}}_{k}}\\
&\quad \leq \tilde{C}_{35}\sum_{k=\pm}\Big(\|(\delta\omega^{k}_{\rm in},\delta p^{k}_{\rm in})\|_{1,\alpha;\tilde{\Gamma}^{k}_{\rm in}}
+\|g_k-\underline{g}_{k}\|_{2,\alpha; \tilde{\Gamma}_{k}\cap\{0\leq\xi\leq\bar{\xi}_{k}^{2}\}}\\
&\qquad\qquad\qquad\quad  +\|\delta Y^{*}_{k}\|_{1,\alpha;  \tilde{\mathcal{N}}^{\rm{I}}_{k}\cup\tilde{\mathcal{N}}^{\rm{II}}_{k}\cup\tilde{\mathcal{N}}^{\rm{III}}_{k}}
+\sum_{j=\pm}\|f^{j}_{k}\|_{1,\alpha; \tilde{\mathcal{N}}^{\rm{I}}_{k}\cup\tilde{\mathcal{N}}^{\rm{II}}_{k}\cup\tilde{\mathcal{N}}^{\rm{III}}_{k}}\Big),
\end{split}
\end{eqnarray}
where the constant $\tilde{C}_{35}>0$ depends only on $\underline{\mathcal{U}}$, $\alpha$ and $L$.

However, the semi-H\"{o}lder norm estimates for $(\nabla\delta \omega^{*}_{\pm},\nabla\delta p^{*}_{\pm})$ on the whole region $\tilde{\mathcal{N}}^{\rm{I}}_{\pm}\cup\tilde{\mathcal{N}}^{\rm{II}}_{\pm}\cup\tilde{\mathcal{N}}^{\rm{III}}_{\pm}$ can not be obtained directly by adding the estimates \eqref{eq:prop-3.3-1} and \eqref{eq:prop-3.4-1}
altogether. In order to get it, we need to do some ``bonding", that is, for any $(\bar{\xi}_1,\bar{\eta}_1)\in \tilde{\mathcal{N}}^{\rm{II}}_{\pm}$ and $(\bar{\xi}_2,\bar{\eta}_2)\in\tilde{\mathcal{N}}^{\rm{III}}_{\pm}$, 
{\small
\begin{eqnarray}\label{eq:prop-3.4-49}
\begin{split}
&\sup\limits_{{(\bar{\xi}_1,\bar{\eta}_1)\in \tilde{\mathcal{N}}^{\rm{II}}_{\pm}  \atop
(\bar{\xi}_2,\bar{\eta}_2)\in\tilde{\mathcal{N}}^{\rm{III}}_{\pm}}}\frac{|\nabla\delta \omega^{*}_{\pm}(\bar{\xi}_1,\bar{\eta}_1)-\nabla\delta \omega^{*}_{\pm}(\bar{\xi}_2,\bar{\eta}_2)|}
{|(\bar{\xi}_1,\bar{\eta}_1)-(\bar{\xi}_2,\bar{\eta}_2)|^{\alpha}}\\
&\quad\leq\sup\limits_{(\bar{\xi}_1,\bar{\eta}_1)\in \tilde{\mathcal{N}}^{\rm{II}}_{\pm}}\frac{|\nabla\delta \omega^{*}_{\pm}(\xi_1,\eta_1)-\nabla\delta \omega^{*}_{\pm}(\bar{\xi}^1_{\pm},\bar{\eta}^1_{\pm})|}{|(\bar{\xi}_1,\bar{\eta}_1)-(\bar{\xi}^1_{\pm},\bar{\eta}^1_{\pm})|^{\alpha}}
+\sup\limits_{(\bar{\xi}_2,\bar{\eta}_2)\in \tilde{\mathcal{N}}^{\rm{III}}_{\pm}}\frac{|\nabla\delta \omega^{*}_{\pm}(\bar{\xi}^1_{\pm},\bar{\eta}^1_{\pm})-\nabla\delta \omega^{*}_{\pm}(\bar{\xi}_2,\bar{\eta}_2)|}{|(\bar{\xi}^1_{\pm},\bar{\eta}^1_{\pm})-(\bar{\xi}_2,\bar{\eta}_2)|^{\alpha}}\\
&\quad\leq [{\color{black}\nabla\delta \omega^{*}_{\pm}}]_{0,\alpha; \tilde{\mathcal{N}}^{\rm{I}}_{\pm}\cup\tilde{\mathcal{N}}^{\rm{II}}_{\pm}}+[{\color{black}\nabla\delta \omega^{*}_{\pm}}]_{0,\alpha; \tilde{\mathcal{N}}^{\rm{I}}_{\pm}\cup\tilde{\mathcal{N}}^{\rm{III}}_{\pm}}\\
&\quad \leq \tilde{C}_{35}\sum_{k=\pm}\Big(\|(\delta\omega^{k}_{\rm in},\delta p^{k}_{\rm in})\|_{1,\alpha;\tilde{\Gamma}^{k}_{\rm in}}
+\|g_k-\underline{g}_{k}\|_{2,\alpha; \tilde{\Gamma}_{k}\cap\{0\leq\xi\leq\bar{\xi}_{k}^{2}\}}\\
&\qquad\qquad\qquad\quad
+\|\delta Y^{*}_{k}\|_{1,\alpha;  \tilde{\mathcal{N}}^{\rm{I}}_{k}\cup\tilde{\mathcal{N}}^{\rm{II}}_{k}\cup\tilde{\mathcal{N}}^{\rm{III}}_{k}}
+\sum_{j=\pm}\|f^{j}_{k}\|_{1,\alpha; \tilde{\mathcal{N}}^{\rm{I}}_{k}\cup\tilde{\mathcal{N}}^{\rm{II}}_{k}\cup\tilde{\mathcal{N}}^{\rm{III}}_{k}}\Big),
\end{split}
\end{eqnarray}}
where $(\bar{\xi}^1_{\pm},\bar{\eta}^1_{\pm})\in\overline{\tilde{\mathcal{N}}^{\rm{II}}_{\pm}}\cap\overline{\tilde{\mathcal{N}}^{\rm{III}}_{\pm}}$.
In the same way, we can also show that $[\delta p^{*}_{\pm}]_{0,\alpha; \tilde{\mathcal{N}}^{\rm{I}}_{\pm}\cup\tilde{\mathcal{N}}^{\rm{II}}_{\pm}\cup\tilde{\mathcal{N}}^{\rm{III}}_{\pm}}$ can be controlled by the right hand side of \eqref{eq:prop-3.4-49}.
Let $\bar{\xi}_1^*\triangleq\min\{\bar{\xi}_{+}^1,\bar{\xi}_{-}^1\}$ and $\tilde{\mathcal{N}}^{1}_{\pm}\triangleq\tilde{\mathcal{N}}_{\pm}\cap\{0< \xi\leq\bar{\xi}_1^*\}$.
Then, 
\begin{align*}
\begin{split}
&\sum_{k=\pm}\|(\delta \omega^{*}_{k},\delta p^{*}_{k})\|_{1,\alpha; \tilde{\mathcal{N}}^{1}_{k}}\\
&\leq \tilde{C}_{1*}\sum_{k=\pm}\Big(\|(\delta\omega^{k}_{\rm in},\delta p^{k}_{\rm in})\|_{1,\alpha;\tilde{\Gamma}^{k}_{\rm in}}
+\|g_k-\underline{g}_{k}\|_{2,\alpha; \tilde{\Gamma}_{k}}
+\|\delta Y^{*}_{k}\|_{1,\alpha;  \tilde{\mathcal{N}}_{k}}
+\sum_{j=\pm}\|f^{j}_{k}\|_{1,\alpha; \tilde{\mathcal{N}}_{k}}\Big),
\end{split}
\end{align*}
where the constant $\tilde{C}_{1*}>0$ depends only on $\underline{\mathcal{U}}$, $\alpha$ and $L$.

Next, we consider the solution of the problem in the rectangular region $\tilde{\mathcal{N}}_{\pm}\cap\{\bar{\xi}_1^*<\xi\}$ by repeating the previous process, and determine a constant $\bar{\xi}_2^*>\bar{\xi}_1^*$
so that
\begin{align*}
\begin{split}
\sum_{k=\pm}\|(\delta \omega^{*}_{k},\delta p^{*}_{k})\|_{1,\alpha; \tilde{\mathcal{N}}^{2}_{k}}
&\leq \tilde{C}_{2*}\sum_{k=\pm}\Big(\|(\delta \omega^{*}_{k},\delta p^{*}_{k})\|_{1,\alpha; \tilde{\mathcal{N}}^{1}_{k}\cap\{\xi=\bar{\xi}_1^*\}}+\|g_k-\underline{g}_{k}\|_{2,\alpha; \tilde{\Gamma}_{k}}\\
&\qquad\qquad\ +\|\delta Y^{*}_{k}\|_{1,\alpha; \tilde{\mathcal{N}}_{k}}+\sum_{j=\pm}\|f^{j}_{k}\|_{1,\alpha; \tilde{\mathcal{N}}_{k}}\Big),
\end{split}
\end{align*}
where $\tilde{\mathcal{N}}^{2}_{\pm}\triangleq\tilde{\mathcal{N}}_{\pm}\cap\{\bar{\xi}_1^*<\xi\leq\bar{\xi}_2^*\}$, and the constant $\tilde{C}_{2*}>0$ depends only on $\underline{\mathcal{U}}$, $\alpha$ and $L$.

Then, we can use the ``bond" method as above again to deduce that
\begin{eqnarray*}
\begin{split}
&\sum_{k=\pm}\|(\delta \omega^{*}_{k},\delta p^{*}_{k})\|_{1,\alpha; \tilde{\mathcal{N}}^{1}_{k}\cup\tilde{\mathcal{N}}^{2}_{k}}\\
&\leq \tilde{C}_{12*}\sum_{k=\pm}\Big(\|(\delta\omega^{k}_{\rm in},\delta p^{k}_{\rm in})\|_{1,\alpha;\tilde{\Gamma}^{k}_{\rm in}}
+\|g_k-\underline{g}_k\|_{2,\alpha; \tilde{\Gamma}_{k}}
+\|\delta Y^{*}_{k}\|_{1,\alpha;  \tilde{\mathcal{N}}_{k}}
+\sum_{j=\pm}\|f^{j}_{k}\|_{1,\alpha; \tilde{\mathcal{N}}_{k}}\Big),
\end{split}
\end{eqnarray*}
where the constant $\tilde{C}_{1,2*}>0$ depends only on $\underline{\mathcal{U}}$, $\alpha$ and $L$.

Inductively, for sufficiently small $\sigma>0$, there exist finite constants $\bar{\xi}^{*}_{j}$ for $j=1, 2,\cdot\cdot\cdot N$ with $N$ depending only on $\underline{U}$, $\alpha$ and $L$ satisfying $0=\bar{\xi}^{*}_{0}<\bar{\xi}^{*}_{1}<\bar{\xi}^{*}_2<\cdot\cdot\cdot<\bar{\xi}^{*}_N=L$.
Set $\tilde{\mathcal{N}}_{\pm}=\bigcup^{N}_{j=1}\tilde{\mathcal{N}}^{j}_{\pm}$ with $\tilde{\mathcal{N}}^{j}_{\pm}=\tilde{\mathcal{N}}_{\pm}\cap\{\bar{\xi}^{*}_{j}< \xi\leq \bar{\xi}^{*}_{j} \}$.
Then, on each $\tilde{\mathcal{N}}^{j}_{\pm}$, we have the estimate
\begin{align}\label{eq:prop-3.4-50}
\begin{split}
&\sum_{k=\pm}\|(\delta \omega^{*}_{k},\delta p^{*}_{k})\|_{1,\alpha; \tilde{\mathcal{N}}^{j-1}_{k}\cup\tilde{\mathcal{N}}^{j}_{k}}\\
&\leq \tilde{C}_{j-1,j*}\sum_{k=\pm}\Big(\|(\delta\omega^{k}_{\rm in},\delta p^{k}_{\rm in})\|_{1,\alpha;\tilde{\Gamma}^{k}_{\rm in}}+\|g_k-\underline{g}_{k}\|_{2,\alpha; \tilde{\Gamma}_{k}}+\|\delta Y^{*}_{k}\|_{1,\alpha; \tilde{\mathcal{N}}_{k}}+\sum_{j=\pm}\|f^{j}_{k}\|_{1,\alpha; \tilde{\mathcal{N}}_{k}}\Big),
\end{split}
\end{align}
where the constant $\tilde{C}_{j-1,j*}>0$ depends only on $\underline{\mathcal{U}}$, $\alpha$ and $L$.
Therefore,
\begin{align}\label{eq:prop-3.4-51}
\begin{split}
&\sum_{k=\pm}\|(\delta \omega^{*}_{k},\delta p^{*}_{k})\|_{1,\alpha; \tilde{\mathcal{N}}_{k}}\\
&\leq \tilde{C}_{*}\sum_{k=\pm}\Big(\|(\delta\omega^{k}_{\rm in},\delta p^{k}_{\rm in})\|_{1,\alpha;\tilde{\Gamma}^{k}_{\rm in}}+\|g_k-\underline{g}_{k}\|_{2,\alpha; \tilde{\Gamma}_{k}} +\|\delta Y^{*}_{k}\|_{1,\alpha; \tilde{\mathcal{N}}_{k}}+\sum_{j=\pm}\|f^{j}_{k}\|_{1,\alpha; \tilde{\mathcal{N}}_{k}}\Big).
\end{split}
\end{align}

Since $\phi\in C^{1,1}$, for $\sigma>0$ sufficiently small,  
\begin{align}\label{eq:prop-3.4-52}
\begin{split}
\sum_{j=\pm}\|f^{j}_{k}\|_{1,\alpha; \tilde{\mathcal{N}}_{k}}\leq \mathcal{O}(1)\|\delta\mathcal{V}_{k}\|_{1,\alpha; \tilde{\mathcal{N}}_{k}}\|\delta Y_{k}\|_{1,\alpha; \tilde{\mathcal{N}}_{k}},
\end{split}
\end{align}
where the constant $\mathcal{O}(1)$ depends only on $\underline{\mathcal{U}}$, $\alpha$ and $L$.
Then
\begin{align}\label{eq:prop-3.4-54}
\begin{split}
&\sum_{k=\pm}\|(\delta \omega^{*}_{k},\delta p^{*}_{k})\|_{1,\alpha; \tilde{\mathcal{N}}_{k}}\\
&\leq \tilde{C}_{**}\sum_{k=\pm}\Big(\|(\delta\omega^{k}_{\rm in},\delta p^{k}_{\rm in})\|_{1,\alpha;\tilde{\Gamma}^{k}_{\rm in}}+\|g_k-\underline{g}_{k}\|_{2,\alpha; \tilde{\Gamma}_{k}}+\|\delta Y^{*}_{k}\|_{1,\alpha; \tilde{\mathcal{N}}_{k}}+\|\delta\mathcal{V}_{k}\|_{1,\alpha; \tilde{\mathcal{N}}_{k}}\|\delta Y_{k}\|_{1,\alpha; \tilde{\mathcal{N}}_{k}}\Big),
\end{split}
\end{align}
where the constant $\tilde{C}_{**}>0$ depends only on $\underline{\mathcal{U}}$, $\alpha$ and $L$.

Finally, combining \eqref{eq:prop-3.1-1} and \eqref{eq:prop-3.4-54}, one can choose a constant $\tilde{C}_{3}>0$ depending only on $\underline{\mathcal{U}}$, $\alpha$ and $L$,
such that the estimate \eqref{eq:3.5} holds, which completes the proof of Theorem \ref{thm:3.1}.

\section{Solutions to Problem $(\mathbf{IBVP})$ }\setcounter{equation}{0}
In this section, we will establish the existence and uniqueness of solutions of the initial-boundary value problem $(\mathbf{IBVP})$ based on Theorem \ref{thm:3.1} to show Theorem \ref{thm:2.2}. For the existence, we will employ the \emph{Schauder fixed point theorem} (see \cite[p.279, Theorem 11.1]{gilbarg-trudinger}), which is 

\emph{Let $\mathcal{X}$ be a compact convex set in a Banach space $\mathfrak{B}$ and let $\mathcal{T}$ be a continuous map of $\mathcal{X}$ into itself. Then $\mathcal{T}$ has a fixed point, that is, $\mathcal{T}z=z$ for some $z\in \mathcal{X}$.}

In order to apply the \emph{Schauder fixed point theorem}, one needs to do the following two things:

 \emph{$\bullet$ Construction of the map $\mathcal{T}$.}\ By the linearized problem $(\mathbf{IBVP})^{*}$, we can define a map $\mathcal{T}$ as
\begin{eqnarray}\label{eq:4.1}
(\delta\mathcal{U}^{*}_{+}, \delta\mathcal{U}^{*}_{-})\triangleq\mathcal{T}(\delta\mathcal{U}_{+}, \delta\mathcal{U}_{-}).
\end{eqnarray}

Set $\mathfrak{B}=C^{1,\frac{\alpha}{2}}(\tilde{\mathcal{N}}_{+})\times C^{1,\frac{\alpha}{2}}(\tilde{\mathcal{N}}_{-})$.
Clearly, $\mathcal{X}_{\sigma}$ is a compact convex set of $\mathfrak{B}$. Moreover, for any $(\delta\mathcal{U}_{+}, \delta\mathcal{U}_{-})\in \mathcal{X}_{\sigma}$, if 
$$
\sum_{k=\pm}\Big(\|\delta \mathcal{U}^{k}_{\rm in}\|_{1,\alpha; \tilde{\Gamma}^{k}_{\rm in}}+\|g_k-\underline{g}_k\|_{2,\alpha; \tilde{\Gamma}_{k}}\Big)\leq \varepsilon,
$$
it follows from Theorem \ref{thm:3.1} that
\begin{eqnarray}\label{eq:4.2}
\begin{split}
\|\delta \mathcal{U}^{*}_{+}\|_{1,\alpha; \tilde{\mathcal{N}}_{+}}+\|\delta\mathcal{U}^{*}_{-}\|_{1,\alpha; \tilde{\mathcal{N}}_{-}}\leq \tilde{C}_{4}(\varepsilon+\sigma^2),
\end{split}
\end{eqnarray}
where the constant $\tilde{C}_{4}>0$ depends only on $\underline{\mathcal{U}}$, $\alpha$ and $L$. Therefore, if we choose $\varepsilon$ and $\sigma$ sufficiently small so that $\tilde{C}_{4}\varepsilon=\frac{1}{2}\sigma$ and $\tilde{C}_{4}\sigma\leq \frac{1}{2}$, it follows from \eqref{eq:4.2} that
\begin{eqnarray*}\label{eq:3.4}
\|\delta \mathcal{U}^{*}_{+}\|_{1,\alpha; \tilde{\mathcal{N}}_{+}}+\|\delta\mathcal{U}^{*}_{-}\|_{1,\alpha; \tilde{\mathcal{N}}_{-}}\leq \sigma.
\end{eqnarray*}
So the map $\mathcal{T}$ is from $\mathcal{X}_{\sigma}$ to $\mathcal{X}_{\sigma}$.

\emph{$\bullet$ Continuity of the map $\mathcal{T}$.}\ 
 Let $\{\delta \mathcal{U}_{\rm n}\}\subset \mathcal{X}_{\sigma}$ be a given sequence with
$\delta \mathcal{U}_{\rm n}=(\delta \mathcal{U}_{+,\rm{n}}, \delta \mathcal{U}_{-,\rm{n}})$ and satisfying
\begin{eqnarray}\label{eq:n-1}
\delta \mathcal{U}_{\rm n}\rightarrow \delta \mathcal{U}_{0}, \quad  \mbox{in} \quad  \mathfrak{B}, \quad \mbox{for}\quad \delta \mathcal{U}_{0}=(\delta \mathcal{U}_{+,0}, \delta \mathcal{U}_{-,0}).
\end{eqnarray}

By Theorem \ref{thm:3.1}, let $\delta\mathcal{U}^{*}_{\rm {n}}\triangleq\mathcal{T}(\delta \mathcal{U}_{\rm n})$ and
$\delta\mathcal{U}^{*}_{0}\triangleq\mathcal{T}(\delta \mathcal{U}_{0})$. 
We claim that
\begin{eqnarray}\label{eq:n-2}
\delta\mathcal{U}^{*}_{\rm {n}}\rightarrow \delta \mathcal{U}^{*}_{0}, \qquad  \mbox{in} \quad  \mathfrak{B}.
\end{eqnarray}

Suppose that the claim \eqref{eq:n-2} were not true. Then there would exist a constant $d^{*}>0$ and a subsequence $\{\delta\mathcal{U}^{*}_{\rm {n}_{k}}\}$
such that
\begin{eqnarray}\label{eq:n-3}
\|\delta\mathcal{U}^{*}_{\rm {n}_k}- \delta \mathcal{U}^{*}_{0}\|_{\mathfrak{B}}>d^{*}.
\end{eqnarray}

On the other hand, since the sequence $\{\delta\mathcal{U}^{*}_{\rm {n}_k}\}\subset \mathcal{X}_{\sigma}$ is compact in $\mathfrak{B}$, we can find a subsequence of $\{\delta\mathcal{U}^{*}_{\rm {n}_k}\}$,  which we still denote by $\{\delta\mathcal{U}^{*}_{\rm {n}_k}\}$, such that it converges to a function $\delta\mathcal{U}^{*}_{1}=(\delta\mathcal{U}^{*}_{-,1}, \delta\mathcal{U}^{*}_{+,1})\in \mathcal{X}_{\sigma}$ in $\mathfrak{B}$. It follows from \eqref{eq:n-3} that
\begin{eqnarray}\label{eq:n-4}
\|\delta\mathcal{U}^{*}_{1}- \delta \mathcal{U}^{*}_{0}\|_{\mathfrak{B}}>d^{*}.
\end{eqnarray}

Note that both $ \delta \mathcal{U}^{*}_{0}$ and $\delta\mathcal{U}^{*}_{1}$ are solutions to the same linearized problem $(\mathbf{IBVP})^{*}$ (i.e., \eqref{eq:3.2}), \emph{i.e.}, whose coefficients are determined completely by $\delta\mathcal{U}_{0}$. Then, by the uniquenss in Theorem \ref{thm:3.1}, $\delta\mathcal{U}^{*}_{1}=\delta \mathcal{U}^{*}_{0}$, which contradicts to the \eqref{eq:n-4}. Thus, claim \eqref{eq:n-2} holds and the map $\mathcal{T}$ is continuous.

Applying the \emph{Schauder fixed point theorem}, we know that the iteration scheme admits a fixed point $\delta \mathcal{U}_{\pm}$. Let $\mathcal{U}_{\pm}\triangleq\underline{\mathcal{U}}_{\pm}+\delta \mathcal{U}_{\pm}$. Then $\mathcal{U}_{\pm}$ is a solution of $(\mathbf{IBVP})$. Moreover, by Theorem \ref{thm:3.1}, we can further show that $\mathcal{U}_{\pm}$ satisfies the estimate \eqref{eq:2.35}.

Now, we will show that the solution $\mathcal{U}_{\pm}$ to the nonlinear problem $(\mathbf{IBVP})$ is unique. Suppose $\mathcal{U}_{\pm, 1}=(\omega_{\pm,1}, p_{\pm,1}, B_{\pm,1}, S_{\pm,1}, Y_{\pm,1})^{\top}$
and $\mathcal{U}_{\pm, 2}=(\omega_{\pm,2}, p_{\pm,2}, B_{\pm,2}, S_{\pm,2}, Y_{\pm,2})^{\top}$ are two solutions to $(\mathbf{IBVP})$ with the same datum and satisfying \eqref{eq:2.35}.
Denote
\begin{eqnarray*}
\delta\mathcal{U}^{\natural}_{\pm}=(\delta\omega^{\natural}_{\pm}, \delta p^{\natural}_{\pm}, \delta B^{\natural}_{\pm}, \delta S^{\natural}_{\pm}, \delta Y^{\natural}_{\pm})^{\top}
\triangleq\mathcal{U}_{\pm,1}-\mathcal{U}_{\pm,2}.
\end{eqnarray*}

Then, $\delta\mathcal{U}^{\natural}_{\pm}$ satisfies the following initial-boundary value problem:
\begin{eqnarray*}
\begin{cases}
\partial_{\xi}\delta\omega^{\natural}_{+}+\lambda^{\pm}_{+,1}\partial_{\eta}\delta\omega^{\natural}_{+}\pm\Lambda_{+,1}\big(\partial_{\xi}\delta p^{\natural}_{+}+\lambda^{\pm}_{+,1}\partial_{\eta}\delta p^{\natural}_{+}\big)\\
\qquad\qquad\qquad\qquad\qquad=a^{\pm,\natural}_{+}(\mathcal{U}_{+,1})\delta Y^{\natural}_{+}
+f^{\pm, \natural}_{+}(\delta\mathcal{V}^{\natural}_{+},\delta Y^{\natural}_{+}, \mathcal{U}_{+,2}), &\quad  \mbox{in} \quad \tilde{\mathcal{N}}_{+}, \\
\partial_{\xi}\delta B^{\natural}_{+}=a^{\natural}_{B_{+}}(\mathcal{U}_{+,1})\delta Y^{\natural}_{+}+f^{\natural}_{B_{+}}(\delta\mathcal{V}^{\natural}_{+}, \delta Y^{\natural}_{+}, \mathcal{U}_{+,2}), &\quad \mbox{in}\quad \tilde{\mathcal{N}}_{+},\\
\partial_{\xi}\delta S^{\natural}_{+}=a^{\natural}_{S_{+}}(\mathcal{U}_{+,1})\delta Y^{\natural}_{+}+f^{\natural}_{S_{+}}(\delta\mathcal{V}^{\natural}_{+}, \delta Y^{\natural}_{+}, \mathcal{U}_{+,2}), &\quad \mbox{in}\quad \tilde{\mathcal{N}}_{+},\\
\partial_{\xi}\delta Y^{\natural}_{+}=a^{\natural}_{Y_{+}}(\mathcal{U}_{+,1})\delta Y^{\natural}_{+}+f^{\natural}_{Y_{+}}(\delta\mathcal{V}^{\natural}_{+},\delta Y^{\natural}_{+},\mathcal{U}_{+,2}), &\quad \mbox{in}\quad \tilde{\mathcal{N}}_{+},\\
\partial_{\xi}\delta\omega^{\natural}_{-}+\lambda^{\pm}_{-,1}\partial_{\eta}\delta \omega^{\natural}_{-}\pm\Lambda_{-,1}\big(\partial_{\xi}\delta p^{\natural}_{-}+\lambda^{\pm}_{-,1}\partial_{\eta}\delta p^{\natural}_{-}\big)\\
\qquad\qquad\qquad\qquad\qquad=a^{\pm,\natural}_{-}(\mathcal{U}_{-,1})\delta Y^{\natural}_{-}+f^{\pm,\natural}_{-}(\delta\mathcal{V}^{\natural}_{-},\delta Y^{\natural}_{-},\mathcal{U}_{-,2}) &\quad  \mbox{in} \quad \tilde{\mathcal{N}}_{-},\\
\partial_{\xi}\delta B^{\natural}_{-}=a^{\natural}_{B_{-}}(\mathcal{U}_{-,1})\delta Y^{\natural}_{-}+f^{\natural}_{B_{-}}(\delta\mathcal{V}^{\natural}_{-},\delta Y^{\natural}_{-}, \mathcal{U}_{-,2}), &\quad  \mbox{in} \quad \tilde{\mathcal{N}}_{-},\\
\partial_{\xi}\delta S^{\natural}_{-}=a^{\natural}_{S_{-}}(\mathcal{U}_{-,1})\delta Y^{\natural}_{-}+f^{\natural}_{S_{-}}(\delta\mathcal{V}^{\natural}_{-},\delta Y^{\natural}_{-}, \mathcal{U}_{-,2}), &\quad  \mbox{in} \quad \tilde{\mathcal{N}}_{-},\\
\partial_{\xi}\delta Y^{\natural}_{-}=a^{\natural}_{Y_{-}}(\mathcal{U}_{-,1})\delta Y^{\natural}_{-}+f^{\natural}_{Y_{-}}(\delta\mathcal{V}^{\natural}_{-},\delta Y^{\natural}_{-},\mathcal{U}_{-,2}), &\quad \mbox{in}\quad \tilde{\mathcal{N}}_{-},\\
\delta \mathcal{U}^{\natural}_{+}=0,  &\quad  \mbox{on} \quad \tilde{\Gamma}^{+}_{\rm in},\\
\delta \mathcal{U}^{\natural}_{-}=0,  &\quad  \mbox{on} \quad \tilde{\Gamma}^{-}_{\rm in},\\
\delta\omega^{\natural}_{+}=0, &\quad \mbox{on} \quad \tilde{\Gamma}_{+},\\
\delta\omega^{\natural}_{+}=\delta\omega^{\natural}_{-},\quad \delta p^{\natural}_{+}=\delta p^{\natural}_{-},&\quad \mbox{on} \quad \tilde{\Gamma}_{\rm cd},\\
\delta\omega^{\natural}_{-}=0, &\quad \mbox{on} \quad \tilde{\Gamma}_{-},
\end{cases}
\end{eqnarray*}
where $\delta \mathcal{V}^{\natural}_{\pm}=(\delta\omega^{\natural}_{\pm}, \delta p^{\natural}_{\pm}, \delta B^{\natural}_{\pm}, \delta S^{\natural}_{\pm})^{\top}$, the terms $a^{\pm, \natural}_{\pm}$, $a^{\natural}_{B_{\pm}}$, $a^{\natural}_{S_{\pm}}$, $a^{\natural}_{Y_{\pm}}$ are defined by
\begin{eqnarray*}
\begin{cases}
a^{\pm, \natural}_{+}(\mathcal{U}_{+, 1})=\frac{(\gamma-1)\mathfrak{q}_{0}\phi(T_{+,1})\lambda^{\pm}_{+,1}}{\rho_{+ ,1}c^{2}_{+,1}u^{2}_{+,1}},\quad
a^{\pm, \natural}_{-}(\mathcal{U}_{-, 1})=\frac{(\gamma-1)\mathfrak{q}_{0}\phi(T_{-,1})\lambda^{\pm}_{-,1}}{\rho_{- ,1}c^{2}_{-,1}u^{2}_{-,1}},\\
a^{\natural}_{B_{\pm}}(\mathcal{U}_{\pm, 1})=\frac{\mathfrak{q}_{0}\phi(T_{\pm,1})}{u_{\pm,1}},
\quad
a^{\natural}_{S_{\pm}}(\mathcal{U}_{\pm, 1})=-\frac{\gamma \mathcal{R}\mathfrak{q}_{0}\phi(T_{\pm,1})}{c^2_{\pm,1}u_{\pm,1}},
\quad a^{\natural}_{Y_{\pm}}(\mathcal{U}_{\pm, 1})=-\frac{\phi(T_{\pm,1})}{u_{\pm,1}},
\end{cases}
\end{eqnarray*}
and $f^{\pm, \natural}_{\pm}$, $f^{\natural}_{B_{\pm}}$, $f^{\natural}_{S_{\pm}}$, $f^{\natural}_{Y_{\pm}}$ are defined by
\begin{eqnarray*}
\begin{cases}
f^{\pm, \natural}_{+}(\delta \mathcal{V}^{\natural}_{+}, \delta Y^{\natural}_{+}, \mathcal{U}_{+, 2})&\triangleq(\gamma-1)\mathfrak{q}_{0}\cdot\bigg(\frac{\phi(T_{+,1})\lambda^{\pm}_{+,1}}{\rho_{+ ,1}c^{2}_{+,1}u^{2}_{+,1}}-\frac{\phi(T_{+,2})\lambda^{\pm}_{+,2}}{\rho_{+,2}c^{2}_{+,2}u^{2}_{+,2}}\bigg)Y_{+,2}
+\big(\lambda^{\pm}_{+,2}-\lambda^{\pm}_{+,1}\big)\partial_{\eta}\omega_{+, 2}\\
&\quad \pm\big(\lambda^{\pm}_{+,2}-\lambda^{\pm}_{+,1}\big)\Lambda_{+,1}\partial_{\eta}p_{+,2}
\pm\big(\Lambda_{+,2}-\Lambda_{+,1}\big)\big(\partial_{\xi}p_{+,2}+\lambda^{\pm}_{+,2}\partial_{\eta}p_{+, 2}\big),\\
f^{\pm, \natural}_{-}(\delta \mathcal{V}^{\natural}_{-}, \delta Y^{\natural}_{-}, \mathcal{U}_{-, 2})&\triangleq(\gamma-1)\mathfrak{q}_{0}\cdot\bigg(\frac{\phi(T_{-,1})\lambda^{\pm}_{-,1}}{\rho_{- ,1}c^{2}_{-,1}u^{2}_{-,1}}-\frac{\phi(T_{-,2})\lambda^{\pm}_{-,2}}{\rho_{-,2}c^{2}_{-,2}u^{2}_{-,2}}\bigg)Y_{-,2}
+\big(\lambda^{\pm}_{-,2}-\lambda^{\pm}_{-,1}\big)\partial_{\eta}\omega_{-,2}\\
&\quad \pm\big(\lambda^{\pm}_{-,2}-\lambda^{\pm}_{-,1}\big)\Lambda_{-,1}\partial_{\eta}p_{-,2}
\pm\big(\Lambda_{-,2}-\Lambda_{-,1}\big)\big(\partial_{\xi}p_{-,2}+\lambda^{\pm}_{-,2}\partial_{\eta}p_{-,2}\big),
\end{cases}
\end{eqnarray*}
and
\begin{eqnarray*}
\begin{cases}
f^{\natural}_{B_{\pm}}(\delta \mathcal{V}^{\natural}_{\pm},\delta Y^{\natural}_{\pm}, \mathcal{U}_{\pm, 2})\triangleq-\mathfrak{q}_{0}\cdot\bigg(\frac{\phi(T_{\pm,1})}{u_{\pm,1}}
-\frac{\phi(T_{\pm,2})}{u_{\pm,2}}\bigg)Y_{\pm,2},\\
f^{\natural}_{S_{\pm}}(\delta \mathcal{V}^{\natural}_{\pm},\delta Y^{\natural}_{\pm}, \mathcal{U}_{\pm, 2})\triangleq-\gamma\mathcal{R}\mathfrak{q}_{0}\cdot\bigg(\frac{\phi(T_{\pm,1})}{c^{2}_{\pm,1}u_{\pm,1}}
-\frac{\phi(T_{\pm,2})}{c^{2}_{\pm,2}u_{\pm,2}}\bigg)Y_{\pm,2}, \\
f^{\natural}_{Y_{\pm}}(\delta \mathcal{V}^{\natural}_{\pm}, \delta Y^{\natural}_{\pm},\mathcal{U}_{\pm, 2})\triangleq
-\bigg(\frac{\phi(T_{\pm,1})}{u_{\pm,1}}-\frac{\phi(T_{\pm,2})}{u_{\pm,2}}\bigg)Y_{\pm,2}.
\end{cases}
\end{eqnarray*}

Then we need to show that $\delta \mathcal{U}^{\natural}_{+}=\delta \mathcal{U}^{\natural}_{-}\equiv0$. 

First by the Gronwall's inequality and taking $\sigma>0$ sufficiently small, we can get that
\begin{align}\label{eq:4.15}
\begin{split}
\sum_{k=\pm}\|\delta Y^{\natural}_{k}\|_{0,0;\tilde{\mathcal{N}}_{k}}\leq \mathcal{O}(1)
\sum_{k=\pm}\|f^{\natural}_{Y_{k}}\|_{0,0;\tilde{\mathcal{N}}_{k}}
\leq \mathcal{O}(1)\sum_{k=\pm}\|\delta Y_{k,2}\|_{0,0;\tilde{\mathcal{N}}_{k}}\|\delta \mathcal{V}^{\natural}_{k}\|_{0,0;\tilde{\mathcal{N}}_{k}},
\end{split}
\end{align}
where $\mathcal{O}(1)$ depends only on $\underline{\mathcal{U}}$ and $L$.
Then
\begin{align}\label{eq:4.16}
\begin{split}
\sum_{k=\pm}\|(\delta B^{\natural}_{k}, \delta S^{\natural}_{k})\|_{0,0;\tilde{\mathcal{N}}_{k}}
&\leq \mathcal{O}(1)\sum_{k=\pm}\Big(\|\delta Y^{\natural}_{k}\|_{0,0;\tilde{\mathcal{N}}_{k}}+\|\delta f^{\natural}_{B_{k}}\|_{0,0;\tilde{\mathcal{N}}_{k}}+\|\delta f^{\natural}_{S_{k}}\|_{0,0;\tilde{\mathcal{N}}_{k}}\Big)\\
&\leq \mathcal{O}(1)\sum_{k=\pm}\|\delta Y_{k,2}\|_{0,0;\tilde{\mathcal{N}}_{k}}\|\delta \mathcal{V}^{\natural}_{k}\|_{0,0;\tilde{\mathcal{N}}_{k}},
\end{split}
\end{align}
where $\mathcal{O}(1)$ depends only on $\underline{\mathcal{U}}$ and $L$.

For $(\delta \omega^{\natural}_{k}, \delta p^{\natural}_{k})$, by employing the characteristic as done in the proof of Theorem \ref{thm:3.1}, we have
\begin{align}\label{eq:4.17}
\begin{split}
\sum_{k=\pm}\Big(\|\delta \omega^{\natural}_{k}\|_{0,0;\tilde{\mathcal{N}}_{k}}+\|\delta p^{\natural}_{k}\|_{0,0;\tilde{\mathcal{N}}_{k}}\Big)
\leq \mathcal{O}(1)\sum_{k=\pm}\Big(\|\delta Y^{\natural}_{k}\|_{0,0;\tilde{\mathcal{N}}_{k}}+\|f^{\pm, \natural}_{k}\|_{0,0;\tilde{\mathcal{N}}_{k}}\Big).
\end{split}
\end{align}

Note that
\begin{align}\label{eq:4.18}
\begin{split}
&\|f^{\pm, \natural}_{k}\|_{0,0;\tilde{\mathcal{N}}_{k}}\\
&\leq \mathcal{O}(1)\bigg\|\frac{\phi(T_{k,1})\lambda^{\pm}_{k,1}}{\rho_{k ,1}c^{2}_{k,1}u^{2}_{k,1}}-\frac{\phi(T_{k,2})\lambda^{\pm}_{k,2}}{\rho_{k,2}c^{2}_{k,2}u^{2}_{k,2}}\bigg\|_{0,\tilde{\mathcal{N}}_{k}} \big\|Y_{k,2}\big\|_{0,0;\tilde{\mathcal{N}}_{k}}
+\big\|\big(\lambda^{\pm}_{k,2}-\lambda^{\pm}_{k,1}\big)\partial_{\eta}\omega_{k, 2}\big\|_{0,0;\tilde{\mathcal{N}}_{k}}\\
&\quad +\mathcal{O}(1)\Big(\big\|\big(\lambda^{\pm}_{k,2}-\lambda^{\pm}_{k,1}\big)\Lambda_{k,1}\partial_{\eta}p_{k,2}\big\|_{0,0;\tilde{\mathcal{N}}_{k}}
+\big\|\big(\Lambda_{k,2}-\Lambda_{k,1}\big)\big(\partial_{\xi}p_{k,2}+\lambda^{\pm}_{k,2}\partial_{\eta}p_{k, 2}\big)\big\|_{0,0;\tilde{\mathcal{N}}_{k}}\Big)\\
&\leq \mathcal{O}(1)\Big(\big\|\delta Y_{k,2}\big\|_{0,0;\tilde{\mathcal{N}}_{k}}\big\|\delta\mathcal{V}^{\natural}_{k}\big\|_{0,0;\tilde{\mathcal{N}}_{k}}
+\big(\|\nabla\delta\omega_{k,2}\|_{0,0;\tilde{\mathcal{N}}_{k}}+\|\nabla\delta p_{k,2}\|_{0,0;\tilde{\mathcal{N}}_{k}}\big)\big\|\delta\mathcal{V}^{\natural}_{k}\big\|_{0,0;\tilde{\mathcal{N}}_{k}}\Big)\\
&\leq \mathcal{O}(1)\tilde{\varepsilon}\big\|\delta\mathcal{V}^{\natural}_{k}\big\|_{0,0;\tilde{\mathcal{N}}_{k}}.
\end{split}
\end{align}
%
Then, 
\begin{align}\label{eq:4.19}
\begin{split}
\sum_{k=\pm}\Big(\|\delta \omega^{\natural}_{k}\|_{0,0;\tilde{\mathcal{N}}_{k}}+\|\delta p^{\natural}_{k}\|_{0,0;\tilde{\mathcal{N}}_{k}}\Big)
\leq \mathcal{O}(1)\tilde{\varepsilon}\sum_{k=\pm}\big\|\delta\mathcal{U}^{\natural}_{k}\big\|_{0,0;\tilde{\mathcal{N}}_{k}},
\end{split}
\end{align}
where constant $\mathcal{O}(1)$ depends only on $\underline{\mathcal{U}}$ and $L$.

Combining the estimates \eqref{eq:4.15}, \eqref{eq:4.16} and \eqref{eq:4.19} altogether, we arrive at
\begin{eqnarray*}
\begin{split}
\sum_{k=\pm}\|\delta \mathcal{U}^{\natural}_{k}\|_{0,0;\tilde{\mathcal{N}}_{k}}
\leq \mathcal{O}(1)\tilde{\varepsilon}\sum_{k=\pm}\|\delta \mathcal{U}^{\natural}_{k}\|_{0,0;\tilde{\mathcal{N}}_{k}},
\end{split}
\end{eqnarray*}
which implies $\delta \mathcal{U}^{\natural}_{+}=\delta \mathcal{U}^{\natural}_{-}=0$ by choosing $\tilde{\varepsilon}>0$ sufficiently small.

\begin{proof}[\underline{Proof of Theorem \ref{thm:2.1}}]
Since $(\omega_{\pm}, p_{\pm}, B_{\pm}, S_{\pm}, Y_{\pm})$ and $(u_{\pm}, v_{\pm},\rho_{\pm})\in C^{1,\alpha}(\tilde{\mathcal{N}}_{\pm})$, by the relations \eqref{eq:1.4}, \eqref{eq:1.8}, \eqref{eq:2.23}
and \eqref{eq:2.25}, we have in $\tilde{\mathcal{N}}_{\pm}$ that
\begin{eqnarray}\label{eq:4.20}
\begin{split}
\rho_{\pm}=\big(A(S_{\pm})\big)^{-\frac{1}{\gamma}}(p_{\pm})^{\frac{1}{\gamma}},
\end{split}
\end{eqnarray}
and
\begin{align}\label{eq:4.21}
\begin{split}
&u_{\pm}=\sqrt{\frac{2\Big[(\gamma-1)B_{\pm}-\gamma A^{\frac{1}{\gamma}}(S_{\pm})p^{1-\frac{1}{\gamma}}_{\pm}\Big]}
{(\gamma-1)\big(1+\omega^{2}_{\pm}\big)}},\
v_{\pm}=\omega_{\pm}\sqrt{\frac{2\Big[(\gamma-1)B_{\pm}-\gamma A^{\frac{1}{\gamma}}(S_{\pm})p^{1-\frac{1}{\gamma}}_{\pm}\Big]}
{(\gamma-1)\big(1+\omega^{2}_{\pm}\big)}}.
\end{split}
\end{align}

Thus, we can reverse the procedures in deriving the equations \eqref{eq:2.19}-\eqref{eq:2.27} for $(\omega_{\pm}, p_{\pm}, B_{\pm}, S_{\pm}, Y_{\pm})$ to know that $(u_{\pm},v_{\pm}, p_{\pm},\rho_{\pm}, Y_{\pm})$ solve \emph{Problem 2.1}. Moreover, the estimate \eqref{eq:2.16} can be obtained by using \eqref{eq:4.20}-\eqref{eq:4.21} together with the estimate \eqref{eq:2.35}.
\end{proof}

\section{Quasi-One-Dimensional Approximation for Steady Supersonic Combustion Contact Discontinuity Flow in Nozzles}\setcounter{equation}{0}

In this section, we will consider the quasi-one-dimensional approximation problem for the two-dimensional steady supersonic combustion contact discontinuity flow in the nozzle $\mathcal{N}$
and show that the integral average of states $U_{\pm}$ in $\mathcal{N}_{\pm}$ can be approximated by the quasi-one-dimensional flow.
To achieve it, we shall first study the quasi-one-dimansional flow problem. 

\subsection{Quasi-one-dimensional flow problem}
We now consider the following problem:
\begin{eqnarray}\label{eq:5.1}
\begin{cases}
\frac{\dd }{\dd x}\big(\rho^{\textrm{A}}_{\pm} u^{\textrm{A}}_{\pm} \textrm{A}_{\pm}(x)\big)=0, &\quad  \mbox{in} \quad (0,L),\\
\frac{\dd }{\dd x}\big((\rho^{\textrm{A}}_{\pm} (u^{\textrm{A}} _{\pm})^{2}+p^{\textrm{A}}_{\pm})\textrm{A}_{\pm}(x)\big)=\textrm{A}'_{\pm}(x)p^{\textrm{A}}_{\pm}, &\quad  \mbox{in} \quad (0,L),\\
\frac{\dd }{\dd x}\Big(\big(\frac{1}{2}(u^{\textrm{A}}_{\pm})^2+\frac{\gamma p^{\textrm{A}}_{\pm}}{(\gamma-1)\rho^{\textrm{A}}_{\pm}}\big)\rho^{\textrm{A}}_{\pm} u^{\textrm{A}}_{\pm}\textrm{A}_{\pm}(x)\Big)=\mathfrak{q}_0\textrm{A}_{\pm}(x)\rho^{\textrm{A}}_{\pm}\phi(T^{\textrm{A}}_{\pm})Y^{\textrm{A}}_{\pm},&\quad  \mbox{in} \quad (0,L),\\
\frac{\dd }{\dd x}\big(\rho^{\textrm{A}}_{\pm} u^{\textrm{A}}_{\pm}Y^{\textrm{A}}_{\pm}\textrm{A}_{\pm}(x)\big)=-\textrm{A}_{\pm}(x)\rho^{\textrm{A}}_{\pm}\phi(T^{\textrm{A}}_{\pm})Y^{\textrm{A}}_{\pm},
&\quad  \mbox{in} \quad (0,L),
\end{cases}
\end{eqnarray}
with initial data
\begin{eqnarray}\label{eq:5.2}
(u^{\textrm{A}} _{\pm}, p^{\textrm{A}} _{\pm}, \rho^{\textrm{A}}_{\pm}, Y^{\textrm{A}} _{\pm})(x)=(\bar{u}^{\pm}_{\rm in}, \bar{p}^{\pm}_{\rm in}, \bar{\rho}^{\pm}_{\rm in}, \bar{Y}^{\pm}_{\rm in}),
\quad \mbox{on}\quad x=0.
\end{eqnarray}
Here, $T^{\textrm{A}}_{\pm}=\frac{p^{\textrm{A}} _{\pm}}{\mathcal{R}\rho^{\textrm{A}}_{\pm}}$ and $p^{\textrm{A}} _{\pm}=A(S^{\textrm{A}} _{\pm})(\rho^{\textrm{A}} _{\pm})^{\gamma}$.
Denote $V^{\textrm{A}}_{\pm}=(u^{\textrm{A}} _{\pm}, p^{\rm{A}} _{\pm}, \rho^{\textrm{A}}_{\pm}, Y^{\rm{A}} _{\pm})$, $V^{\rm{A}, \pm}_{\rm in}=(\bar{u}^{\pm}_{\rm in}, \bar{p}^{\pm}_{\rm in}, \bar{\rho}^{\pm}_{\rm in}, \bar{Y}^{\pm}_{\rm in})$ and $\underline{V}^{\pm}_{\rm{in}}=(\underline{u}^{\pm}_{\rm in}, \underline{p}^{\pm}_{\rm in}, \underline{\rho}^{\pm}_{\rm in},0)^{\top}$. Then, we have the following proposition. 
\begin{proposition}\label{prop:5.1}
Assume that the reaction rate function $\phi$ is $C^{1,1}$ with respect to $T^{\textrm{A}}_{\pm}>0$ and $\textrm{A}_{\pm}(x)\in C^{2,\alpha}((0,L))$.
Then there exist constants $\tilde{\epsilon}_{0}>0$ and $\tilde{C}_{4}>0$ depending only on $\underline{V}^{\pm}_{\rm{in}}$, $\alpha$ and $L$ such that for $\tilde{\epsilon}\in (0,\tilde{\epsilon}_{0})$, if
\begin{eqnarray}\label{eq:5.3}
\sum_{k=\pm}\Big(|\bar{V}^{k}_{\rm in}-\underline{V}^{k}_{\rm{in}}|+\|\textrm{A}_{k}-1\|_{2,\alpha; (0,L)}\Big)\leq \tilde{\epsilon},
\end{eqnarray}
problem \eqref{eq:5.1}-\eqref{eq:5.2} admits a unique solution $V^{\rm{A}}_{\pm}\in C^{1,\alpha}((0,L))$ satisfying
\begin{eqnarray}\label{eq:5.4}
\begin{split}
\|V^{\rm{A}}_{+}-\underline{V}^{+}_{\rm{in}}\|_{1,\alpha; (0,L)}+\|V^{\rm{A}}_{-}-\underline{V}^{-}_{\rm{in}}\|_{1,\alpha; (0,L)}\leq \tilde{C}_{4}\tilde{\epsilon}.
\end{split}
\end{eqnarray}
\end{proposition}

\begin{remark}\label{rem:5.1}
The smallness assumption \eqref{eq:5.3} in Proposition \ref{prop:5.1} can be ensured by assumption \eqref{eq:1.22} in Theorem \ref{thm:1.1}. In fact, by \eqref{eq:1.28},
there exists a positive constant $\tilde{C}_{5}$ depending only on $\underline{U}_{\pm}$ and $\alpha$ such that
\begin{eqnarray*}
\sum_{k=\pm}\big(|\bar{V}^{k}_{\textrm {in}}-\underline{V}^{k}_{\textrm{in}}|+\|\textrm{A}_{k}-1\|_{2,\alpha; (0,L)}\big)\leq
\tilde{C}_{5}\sum_{k=\pm}\big(\|U^{k}_{\rm in}-\underline{U}^{k}_{\textrm{in}}\|_{1,\alpha; \Gamma^{k}_{\textrm{in}}}+\|g_k-\underline{g}_{k}\|_{2,\alpha; \Gamma_{k}}\big)<\tilde{C}_{5}\epsilon.
\end{eqnarray*}
So we set $\tilde{\epsilon}=\tilde{C}_{5}\epsilon$.
\end{remark}
\begin{proof}[Proof of Proposition \ref{prop:5.1}]
First, we will develop an iteration scheme to solve the problem \eqref{eq:5.1}-\eqref{eq:5.2}. Set $M^{\textrm{A}}_{\pm}\triangleq{u^{\textrm{A}}_{\pm}}/{c^{\textrm{A}}_{\pm}}$ and $c^{\textrm{A}}_{\pm}=\sqrt{\frac{\gamma p^{\textrm{A}}_{\pm}}{\rho^{\textrm{A}}_{\pm}}}$. We rewrite system \eqref{eq:5.1} as
\begin{eqnarray}\label{eq:5.5}
\begin{cases}
\frac{\dd u^{\textrm{A}}_{\pm}}{\dd x} = \frac{(\gamma-1)\mathfrak{q}_0\phi(T^{\textrm{A}}_{\pm})Y^{\textrm{A}}_{\pm}}{(c^{\textrm{A}}_{\pm})^2-(u^{\textrm{A}}_{\pm})^2}
-\frac{u^{\textrm{A}}_{\pm}\textrm{A}'_{\pm}}{\textrm{A}_{\pm}[1-(M^{\textrm{A}}_{\pm})^2]},\\
\frac{\dd p^{\textrm{A}}_{\pm}}{\dd x} =- \frac{(\gamma-1)\mathfrak{q}_0\rho^{\textrm{A}}_{\pm}u^{\textrm{A}}_{\pm}\phi(T^{\textrm{A}}_{\pm})Y^{\textrm{A}}_{\pm}}{(c^{\textrm{A}}_{\pm})^2-(u^{\textrm{A}}_{\pm})^2}
+\frac{\rho^{\textrm{A}}_{\pm}(u^{\textrm{A}}_{\pm})^2\textrm{A}'_{\pm}}{\textrm{A}_{\pm}[1-(M^{\textrm{A}}_{\pm})^2]},\\
\frac{\dd \rho^{\textrm{A}}_{\pm}}{\dd x} =- \frac{(\gamma-1)\mathfrak{q}_0\rho^{\textrm{A}}_{\pm}\phi(T^{\textrm{A}}_{\pm})Y^{\textrm{A}}_{\pm}}{u^{\textrm{A}}_{\pm}[(c^{\textrm{A}}_{\pm})^2-(u^{\textrm{A}}_{\pm})^2]}-\frac{\rho^{\textrm{A}}_{\pm}
(M^{\textrm{A}}_{\pm})^2\textrm{A}'_{\pm}}{\textrm{A}_{\pm}[1-(M^{\textrm{A}}_{\pm})^2]},\\
\frac{\dd Y^{\textrm{A}}_{\pm}}{\dd x} =- \frac{\phi(T^{\textrm{A}}_{\pm})Y^{\textrm{A}}_{\pm}}{u^{\textrm{A}}_{\pm}}.
\end{cases}
\end{eqnarray}

For $\tilde{\sigma}\in(0,1)$, we define the iteration set as
\begin{eqnarray}\label{eq:5.6}
\mathcal{X}^{\textrm{A}}_{\tilde{\sigma}}=\big\{(\delta V^{\textrm{A}}_{+}, \delta V^{\textrm{A}}_{-})\in C^{1,\alpha}((0,L))\times C^{1,\alpha}((0,L))|\ \|\delta V^{\textrm{A}}_{+}\|_{1,\alpha;(0,L)}+\|\delta V^{\textrm{A}}_{-}\|_{1,\alpha;(0,L)} \leq \tilde{\sigma} \big\},
\end{eqnarray}
where $\delta V^{\textrm{A}}_{\pm}=\big(\delta u^{\textrm{A}}_{\pm}, \delta p^{\textrm{A}}_{\pm}, \delta \rho^{\textrm{A}}_{\pm}, \delta Y^{\textrm{A}}_{\pm}\big)^{\top}$.
Then, for any given $\delta V^{\textrm{A}}_{\pm}\in \mathcal{X}^{\textrm{A}}_{\tilde{\sigma}}$, let $V^{\textrm{A}}_{\pm}\triangleq\underline{V}_{\pm}+\delta V^{\textrm{A}}_{\pm}$.
We update $\delta V^{\textrm{A}}_{\pm}$ via $\delta \tilde{V}^{\textrm{A}}_{\pm}=\big(\delta \tilde{u}^{\textrm{A}}_{\pm}, \delta \tilde{p}^{\textrm{A}}_{\pm}, \delta \tilde{\rho}^{\textrm{A}}_{\pm}, \delta \tilde{Y}^{\textrm{A}}_{\pm}\big)^{\top}$
by solving the following linearized problem:
\begin{eqnarray}\label{eq:5.7}
\begin{cases}
\frac{\dd \delta \tilde{u}^{\textrm{A}}_{\pm}}{\dd x} = a_{u^{\textrm{A}}_{\pm}}(\underline{V}_{\pm})\delta \tilde{Y}^{\textrm{A}}_{\pm}+f_{u^{\textrm{A}}_{\pm}}(\underline{V}_{\pm},\delta V^{\textrm{A}}_{\pm},\textrm{A}_{\pm}),&\quad  \mbox{in} \quad (0,L), \\
\frac{\dd \delta \tilde{p}^{\textrm{A}}_{\pm}}{\dd x} = a_{p^{\textrm{A}}_{\pm}}(\underline{V}_{\pm})\delta \tilde{Y}^{\textrm{A}}_{\pm}+f_{p^{\textrm{A}}_{\pm}}(\delta V^{\textrm{A}}_{\pm},\textrm{A}'_{\pm}),&\quad  \mbox{in} \quad (0,L), \\
\frac{\dd \delta \tilde{\rho}^{\textrm{A}}_{\pm}}{\dd x} = a_{\rho^{\textrm{A}}_{\pm}}(\underline{V}_{\pm})\delta \tilde{Y}^{\textrm{A}}_{\pm}+f_{\rho^{\textrm{A}}_{\pm}}(\delta V^{\textrm{A}}_{\pm},\textrm{A}'_{\pm}),&\quad  \mbox{in} \quad (0,L), \\
\frac{\dd \delta \tilde{Y}^{\textrm{A}}_{\pm}}{\dd x} = a_{Y^{\textrm{A}}_{\pm}}(\underline{V}_{\pm})\delta \tilde{Y}^{\textrm{A}}_{\pm}+f_{Y^{\textrm{A}}_{\pm}}(\delta V^{\textrm{A}}_{\pm}),&\quad  \mbox{in} \quad (0,L), \\
(\delta \tilde{u}^{\textrm{A}}_{\pm}, \delta \tilde{p}^{\textrm{A}}_{\pm},\delta \tilde{\rho}^{\textrm{A}}_{\pm},\delta \tilde{Y}^{\textrm{A}}_{\pm})
=(\delta\bar{u}^{\pm}_{\rm in}, \delta\bar{p}^{\pm}_{\rm in}, \delta\bar{\rho}^{\pm}_{\rm in}, \delta\bar{Y}^{\pm}_{\rm in}), &\quad \mbox{on} \quad x=0,
\end{cases}
\end{eqnarray}
where $a_{u^{\textrm{A}}_{\pm}}$, $a_{p^{\textrm{A}}_{\pm}}$, $a_{\rho^{\textrm{A}}_{\pm}}$, $a_{Y^{\textrm{A}}_{\pm}}$ and $f_{u^{\textrm{A}}_{\pm}}$, $f_{p^{\textrm{A}}_{\pm}}$, $f_{\rho^{\textrm{A}}_{\pm}}$, $f_{Y^{\textrm{A}}_{\pm}}$ are defined by
\begin{eqnarray}\label{eq:5.8a}
\begin{cases}
a_{u^{\textrm{A}}_{\pm}}(\underline{V}_{\pm})\triangleq\frac{(\gamma-1)\mathfrak{q}_0\phi(\underline{T}^{\pm})}{(\underline{c}^{\pm})^2-(\underline{u}^{\pm})^2},\quad
a_{p^{\textrm{A}}_{\pm}}(\underline{V}_{\pm})\triangleq- \frac{(\gamma-1)\mathfrak{q}_0\underline{\rho}^{\pm}\underline{u}^{\pm}\phi(\underline{T}^{\pm})}{(\underline{c}^{\pm})^2-(\underline{u}^{\pm})^2}, \\
a_{\rho^{\textrm{A}}_{\pm}}(\underline{V}_{\pm})\triangleq-\frac{(\gamma-1)\mathfrak{q}_{0}\underline{\rho}^{\pm}\phi(\underline{T}^{\pm})}{\underline{u}^{\pm}[(\underline{c}^{\pm})^2-(\underline{u}^{\pm})^2]}, \quad a_{Y^{\textrm{A}}_{\pm}}(\underline{V}_{\pm})\triangleq-\frac{\phi(\underline{T}^{\pm})}{\underline{u}^{\pm}},
\end{cases}
\end{eqnarray}
and
\begin{align}\label{eq:5.9}
\begin{cases}
f_{u^{\textrm{A}}_{\pm}}(\delta V^{\textrm{A}}_{\pm},\textrm{A}'_{\pm})\triangleq(\gamma-1)\mathfrak{q}_{0}\cdot
\Big(\frac{\phi(T^{\textrm{A}}_{\pm})}{(c^{\textrm{A}}_{\pm})^2-(u^{\textrm{A}}_{\pm})^2}
-\frac{\phi(\underline{T}^{\pm})}{(\underline{c}^{\pm})^2-(\underline{u}^{\pm})^2}\Big)\delta Y^{\textrm{A}}_{\pm}-\frac{u^{\textrm{A}}_{\pm}\textrm{A}'_{\pm}(x)}{\textrm{A}_{\pm}[1-(M^{\textrm{A}}_{\pm})^2]},\\
f_{p^{\textrm{A}}_{\pm}}(\delta V^{\textrm{A}}_{\pm},\textrm{A}'_{\pm})\triangleq-(\gamma-1)\mathfrak{q}_0 \cdot \Big(\frac{\rho^{\textrm{A}}_{\pm}u^{\textrm{A}}_{\pm}\phi(T^{\textrm{A}}_{\pm})Y^{\textrm{A}}_{\pm}}{(c^{\textrm{A}}_{\pm})^2-(u^{\textrm{A}}_{\pm})^2}- \frac{\underline{\rho}^{\pm}\underline{u}^{\pm}\phi(\underline{T}^{\pm})}{(\underline{c}^{\pm})^2-(\underline{u}^{\pm})^2}\Big) \delta Y^{\textrm{A}}_{\pm}+\frac{\rho^{\textrm{A}}_{\pm}(u^{\textrm{A}}_{\pm})^2\textrm{A}'_{\pm}}{\textrm{A}_{\pm}[1-(M^{\textrm{A}}_{\pm})^2]},\\
f_{\rho^{\textrm{A}}_{\pm}}(\delta V^{\textrm{A}}_{\pm},\textrm{A}'_{\pm})\triangleq
(\gamma-1)\mathfrak{q}_{0}\cdot\Big(\frac{\phi(T^{\textrm{A}}_{\pm})}{u^{\textrm{A}}_{\pm}[(c^{\textrm{A}}_{\pm})^2-(u^{\textrm{A}}_{\pm})^2]}
-\frac{\phi(\underline{T}^{\pm})}{\underline{u}^{\textrm{A}}_{\pm}[(\underline{c}^{\pm})^2-(\underline{u}^{\pm})^2]}\Big)\delta Y^{\textrm{A}}_{\pm}
-\frac{u^{\textrm{A}}_{\pm}(M^{\textrm{A}}_{\pm})^2\textrm{A}'_{\pm}(x)}{\textrm{A}_{\pm}[1-(M^{\textrm{A}}_{\pm})^2]},\\
f_{Y^{\textrm{A}}_{\pm}}(\delta V^{\textrm{A}}_{\pm})\triangleq-\Big(\frac{\phi(T^{\textrm{A}}_{\pm})}{u^{\textrm{A}}_{\pm}}-\frac{\phi(\underline{T}^{\pm})}{\underline{u}^{\pm}}\Big) \delta Y^{\textrm{A}}_{\pm}.
\end{cases}
\end{align}

So we can follow the procedures as done in the proof of Proposition \ref{prop:3.1} to establish the solvability and \emph{a priori} estimates for the linearized problem \eqref{eq:5.7}.
\begin{lemma}\label{lem:5.1}
Assume the assumptions in Proposition \ref{prop:5.1} hold. For any given state $(\delta V^{\rm{A}}_{+},\delta V^{\rm{A}}_{-})\in\mathcal{X}^{\rm{A}}_{\tilde{\sigma}}$ and for sufficiently small $\tilde{\sigma}$,
the linearized problem \eqref{eq:5.7} admits a unique solution $(\delta \tilde{V}^{\rm{A}}_{+},\delta\tilde{V}^{\rm{A}}_{-})\in C^{1,\alpha}((0,L))\times C^{1,\alpha}((0,L))$
satisfying
\begin{eqnarray}\label{eq:5.10}
\begin{split}
&\sum_{k=\pm}\|\delta \tilde{V}^{\rm{A}}_{k}\|_{1,\alpha; (0,L)}\leq \tilde{C}_{6}\sum_{k=\pm}\Big(|\bar{V}^{k}_{\rm in}-\underline{V}^{k}_{\rm{in}}|+\|\textrm{A}_{k}-1\|_{2,\alpha; (0,L)}+\|\delta V^{\textrm{A}}_{k}\|^2_{1,\alpha; (0,L)}\Big),
\end{split}
\end{eqnarray}
where the constant $\tilde{C}_{6}$ depends only on $\underline{V}_{\pm}$, $\alpha$ and $L$.
\end{lemma}

Define a map $\mathcal{T}^{\rm{A}}$ as
\begin{eqnarray}\label{eq:5.11}
\delta\tilde{\mathbf{V}}^{\textrm{A}}=\mathcal{T}^{\rm{A}}(\delta \mathbf{V}^{\textrm{A}}),\quad\mbox{for}\qquad  \delta \mathbf{V}^{\textrm{A}}\triangleq(\delta V^{\textrm{A}}_{+}, \delta V^{\textrm{A}}_{-})\quad\mbox{and}\quad
\delta\tilde{\mathbf{V}}^{\textrm{A}}\triangleq(\delta\tilde{V}^{\textrm{A}}_{+}, \delta\tilde{V}^{\textrm{A}}_{-}).
\end{eqnarray}

By the assumption \eqref{eq:5.3} and the estimate \eqref{eq:5.10}, we can get that
\begin{eqnarray}\label{eq:5.12}
\begin{split}
\|\delta \tilde{V}^{\rm{A}}_{+}\|_{1,\alpha; (0,L)}+\|\delta\tilde{V}^{\rm{A}}_{-}\|_{1,\alpha; (0,L)}\leq \tilde{C}_{6}(\tilde{\epsilon}+\tilde{\sigma}^2).
\end{split}
\end{eqnarray}
So, if we choose $\tilde{C}_{6}\tilde{\epsilon}=\frac{1}{2}\tilde{\sigma}$ and $\tilde{C}_{6}\tilde{\sigma}\leq \frac{1}{2}$ in \eqref{eq:5.12}, it holds that
\begin{eqnarray*}
\|\delta \tilde{V}^{\rm{A}}_{+}\|_{1,\alpha; (0,L)}+\|\delta\tilde{V}_{-}\|_{1,\alpha; (0,L)}\leq \tilde{\sigma}.
\end{eqnarray*}
It implies that $\mathcal{T}^{\rm{A}}$ maps from $\mathcal{X}^{\rm{A}}_{\tilde{\sigma}}$ to $\mathcal{X}^{\rm{A}}_{\tilde{\sigma}}$.

Finally, we are going to show $\mathcal{T}^{\rm{A}}$ is a contraction map. Given two states $(\delta V^{\textrm{A}}_{+,1}, \delta V^{\textrm{A}}_{-,1})$ and $(\delta V^{\textrm{A}}_{+,2}, \delta V^{\textrm{A}}_{-,2})$ in $\mathcal{X}_{\tilde{\sigma}}$, denote by $\delta \tilde{\mathbf{V}}^{\textrm{A}}_{1}\triangleq(\delta \tilde{V}^{\textrm{A}}_{+,1}, \delta \tilde{V}^{\textrm{A}}_{-,1})$ and $\delta \tilde{\mathbf{V}}^{\textrm{A}}_{2}\triangleq(\delta \tilde{V}^{\textrm{A}}_{+,2}, \delta \tilde{V}^{\textrm{A}}_{-,2})$ the corresponding two solutions of problem \eqref{eq:5.7} with the same initial data. 
%
Set $\Delta\tilde{V}^{\textrm{A}}_{\pm}=(\Delta\tilde{u}^{\textrm{A}}_{\pm}, \Delta\tilde{p}^{\textrm{A}}_{\pm}, \Delta\tilde{\rho}^{\textrm{A}}_{\pm}, \Delta\tilde{Y}^{\textrm{A}}_{\pm})\triangleq\delta \tilde{V}^{\textrm{A}}_{\pm,1}-\delta \tilde{V}^{\textrm{A}}_{\pm,2}$ and $\Delta V^{\textrm{A}}_{\pm}=(\Delta u^{\textrm{A}}_{\pm}, \Delta p^{\textrm{A}}_{\pm}, \Delta\rho^{\textrm{A}}_{\pm}, \Delta Y^{\textrm{A}}_{\pm})\triangleq\delta V^{\textrm{A}}_{\pm,1}-\delta V^{\textrm{A}}_{\pm,2}$. Then, 
\begin{eqnarray}\label{eq:5.13}
\begin{cases}
\frac{\dd \Delta \tilde{u}^{\textrm{A}}_{\pm}}{\dd x} = a_{u^{\textrm{A}}_{\pm}}(\underline{V}_{\pm})\Delta \tilde{Y}^{\textrm{A}}_{\pm}+f^{\flat}_{u^{\textrm{A}}_{\pm}}(\Delta V^{\textrm{A}}_{\pm},\textrm{A}'_{\pm}),&\quad  \mbox{in} \quad (0,L), \\
\frac{\dd \Delta \tilde{p}^{\textrm{A}}_{\pm}}{\dd x} = a_{p^{\textrm{A}}_{\pm}}(\underline{V}_{\pm})\Delta \tilde{Y}^{\textrm{A}}_{\pm}+f^{\flat}_{p^{\textrm{A}}_{\pm}}(\Delta V^{\textrm{A}}_{\pm},\textrm{A}'_{\pm}),&\quad  \mbox{in} \quad (0,L), \\
\frac{\dd \Delta \tilde{\rho}^{\textrm{A}}_{\pm}}{\dd x} = a_{\rho^{\textrm{A}}_{\pm}}(\underline{V}_{\pm})\Delta \tilde{Y}^{\textrm{A}}_{\pm}+f^{\flat}_{\rho^{\textrm{A}}_{\pm}}(\Delta V^{\textrm{A}}_{\pm},\textrm{A}'_{\pm}),&\quad  \mbox{in} \quad (0,L), \\
\frac{\dd \Delta \tilde{Y}^{\textrm{A}}_{\pm}}{\dd x} = a_{Y^{\textrm{A}}_{\pm}}(\underline{V}_{\pm})\Delta \tilde{Y}^{\textrm{A}}_{\pm}+f^{\flat}_{Y^{\textrm{A}}_{\pm}}(\Delta V^{\textrm{A}}_{\pm}),&\quad  \mbox{in} \quad (0,L), \\
(\Delta \tilde{u}^{\textrm{A}}_{\pm}, \Delta \tilde{p}^{\textrm{A}}_{\pm}, \Delta \tilde{\rho}^{\textrm{A}}_{\pm}, \Delta \tilde{Y}^{\textrm{A}}_{\pm})=(0,0,0,0), &\quad \mbox{on} \quad x=0,
\end{cases}
\end{eqnarray}
where the terms $a_{u^{\textrm{A}}_{\pm}}$, $a_{p^{\textrm{A}}_{\pm}}$, $a_{\rho^{\textrm{A}}_{\pm}}$, $a_{Y^{\textrm{A}}_{\pm}}$ are given by \eqref{eq:5.8a}, and $f^{\flat}_{u^{\textrm{A}}_{\pm}}$, $f^{\flat}_{p^{\textrm{A}}_{\pm}}$, $f^{\flat}_{\rho^{\textrm{A}}_{\pm}}$, $f^{\flat}_{Y^{\textrm{A}}_{\pm}}$ are defined by
\begin{eqnarray*}
\begin{cases}
f^{\flat}_{u^{\textrm{A}}_{\pm}}(\Delta V^{\textrm{A}}_{\pm},\textrm{A}'_{\pm})\triangleq f_{u^{\textrm{A}}_{\pm}}(\delta V^{\textrm{A}}_{\pm,1},\textrm{A}'_{\pm})
-f_{u^{\textrm{A}}_{\pm}}(\delta V^{\textrm{A}}_{\pm,2},\textrm{A}'_{\pm}),\\
f^{\flat}_{p^{\textrm{A}}_{\pm}}(\Delta V^{\textrm{A}}_{\pm},\textrm{A}'_{\pm})\triangleq f_{p^{\textrm{A}}_{\pm}}(\delta V^{\textrm{A}}_{\pm,1},\textrm{A}'_{\pm})
-f_{p^{\textrm{A}}_{\pm}}(\delta V^{\textrm{A}}_{\pm,2},\textrm{A}'_{\pm}),\\
f^{\flat}_{\rho^{\textrm{A}}_{\pm}}(\Delta V^{\textrm{A}}_{\pm},\textrm{A}'_{\pm})\triangleq f_{\rho^{\textrm{A}}_{\pm}}(\delta V^{\textrm{A}}_{\pm,1},\textrm{A}'_{\pm})
-f_{\rho^{\textrm{A}}_{\pm}}(\delta V^{\textrm{A}}_{\pm,2},\textrm{A}'_{\pm}),\\
f^{\flat}_{Y^{\textrm{A}}_{\pm}}(\Delta V^{\textrm{A}}_{\pm})\triangleq f_{Y^{\textrm{A}}_{\pm}}(\delta V^{\textrm{A}}_{\pm,1})
-f_{Y^{\textrm{A}}_{\pm}}(\delta V^{\textrm{A}}_{\pm,2}).
\end{cases}
\end{eqnarray*}

Then, mimicing the proof of Proposition \ref{prop:3.1}, we have
\begin{eqnarray}\label{eq:5.14}
\sum_{k=\pm}\|\Delta \tilde{V}^{\textrm{A}}_{k}\|_{1,\alpha; (0,L)}\leq \mathcal{O}(1)\sum_{k=\pm}\big\|(f^{\flat}_{u^{\textrm{A}}_{k}},f^{\flat}_{p^{\textrm{A}}_{k}}, f^{\flat}_{\rho^{\textrm{A}}_{k}}, f^{\flat}_{Y^{\textrm{A}}_{k}})\big\|_{1,\alpha; (0,L)},
\end{eqnarray}
where the constant $\mathcal{O}(1)$ depends only on $\underline{V}$, $\alpha$ and $L$.
Notice that
{
\begin{align}\label{eq:5.15}
\begin{split}
&\sum_{k=\pm}\|f^{\flat}_{u^{\textrm{A}}_{k}}\|_{1,\alpha; (0,L)}\\
&\leq\mathcal{O}(1)\sum_{k=\pm}\Big\|\Big(\frac{\phi(T^{\textrm{A}}_{k,1})}{(c^{\textrm{A}}_{k,1})^2-(u^{\textrm{A}}_{k,1})^2}-\frac{\phi(T^{\textrm{A}}_{k,2})}{(c^{\textrm{A}}_{k,2})^2-(u^{\textrm{A}}_{k,2})^2}\Big)\delta Y^{\textrm{A}}_{k,1}\Big\|_{1,\alpha; (0,L)}\\
&\quad+\mathcal{O}(1)\sum_{k=\pm}\Big\|\Big(\frac{\phi(T^{\textrm{A}}_{k,2})}{(c^{\textrm{A}}_{k,2})^2-(u^{\textrm{A}}_{k,2})^2}-\frac{\phi(\underline{T}_{k})}{\underline{c}_{k}^2-\underline{u}_{k}^2}\Big)
\Delta Y^{\textrm{A}}_{k}\Big\|_{1,\alpha; (0,L)}\\
&\quad+\mathcal{O}(1)\sum_{k=\pm}\Big\|\Big(\frac{u^{\textrm{A}}_{k,1}}{1-(M^{\textrm{A}}_{k,1})^2}-\frac{u^{\textrm{A}}_{k,2}}{1-(M^{\textrm{A}}_{k,2})^2}\Big)\textrm{A}'_{k}\Big\|_{1,\alpha; (0,L)}\\
& \leq \mathcal{O}(1)\sum_{k=\pm}\|\delta Y^{\textrm{A}}_{k,1}\|_{1,\alpha; (0,L)}\|\Delta V^{\textrm{A}}_{k}\|_{1,\alpha; (0,L)}
+\mathcal{O}(1)\sum_{k=\pm}\|\delta V^{\textrm{A}}_{k,2}\|_{1,\alpha; (0,L)}\|\Delta Y^{\textrm{A}}_{k}\|_{1,\alpha; (0,L)}\\
&\quad
+\mathcal{O}(1)\sum_{k=\pm}\|\textrm{A}_{k}-1\|_{1,\alpha; (0,L)}\|\Delta V^{\textrm{A}}_{k}\|_{1,\alpha; (0,L)}\\
& \leq \mathcal{O}(1)(\tilde{\sigma}+\tilde{\epsilon})\sum_{k=\pm}\|\Delta V^{\textrm{A}}_{k}\|_{1,\alpha; (0,L)},
\end{split}
\end{align}}
where the constant $\mathcal{O}(1)$ depends only on $\underline{V}$, $\alpha$ and $L$.
Similarly, we can also have 
\begin{eqnarray}\label{eq:5.16}
\sum_{k=\pm}\big\|(f^{\flat}_{p^{\textrm{A}}_{k}},f^{\flat}_{\rho^{\textrm{A}}_{k}},f^{\flat}_{Y^{\textrm{A}}_{k}})\big\|_{1,\alpha; (0,L))}
\leq\mathcal{O}(1)(\tilde{\sigma}+\tilde{\epsilon})\sum_{k=\pm}\|\Delta V^{\textrm{A}}_{k}\|_{1,\alpha; (0,L)}.
\end{eqnarray}

Therefore,
$$
\sum_{k=\pm}\|\Delta \tilde{V}^{\textrm{A}}_{k}\|_{1,\alpha; (0,L)}\leq\mathcal{O}(1)(\tilde{\sigma}+\tilde{\epsilon})\sum_{k=\pm}\|\Delta V^{\textrm{A}}_{k}\|_{1,\alpha; (0,L)}.
$$
So if we choose $\tilde{\epsilon}=\tilde{\sigma}$ and $\mathcal{O}(1)\tilde{\sigma}\leq \frac{1}{4}$, then
\begin{eqnarray*}
\sum_{k=\pm}\|\Delta \tilde{V}^{\textrm{A}}_{k}\|_{1,\alpha; (0,L)}\leq\frac{1}{2}\sum_{k=\pm}\|\Delta V^{\textrm{A}}_{k}\|_{1,\alpha; (0,L)}.
\end{eqnarray*}
Hence $\mathcal{T}^{\rm{A}}$ is contract map from $\mathcal{X}^{\textrm{A}}_{\tilde{\sigma}}$ to $\mathcal{X}^{\textrm{A}}_{\tilde{\sigma}}$.
Then it follows from the \emph{Bannach fixed point theory} that there is a unique fixed point  $\delta V^{\textrm{A}, *}_{k}\in \mathcal{X}^{\textrm{A}}_{\tilde{\sigma}}$ of the map $\mathcal{T}^{\rm{A}}$.
Let $V^{\textrm{A}, *}_{\pm}\triangleq \underline{V}_{\pm} +\delta V^{\textrm{A}, *}_{\pm}$. Then, 
$(V^{\textrm{A}, *}_{+}, V^{\textrm{A}, *}_{-})$ is the unique solution of the problem \eqref{eq:5.1} and \eqref{eq:5.2} with the estimate \eqref{eq:5.4} by Lemma \ref{lem:5.1}.
\end{proof}

\subsection{Integral identities of solution $U$ in $\mathcal{N}$}
In this subsection, we will derive the equations of the integral identities of solution $U$ in nozzle $\mathcal{N}$. 
Given the states $U_{\pm}\in C^{1,\alpha}(\mathcal{N}_{\pm})$, 
we integrate the system \eqref{eq:1.1} over the intervals $(g_{\rm cd}(\tau), g_{+}(\tau))$ and $( g_{-}(\tau),g_{\rm cd}(\tau))$, respectively, with respect to $y$ for any fixed $\tau\in (0,x)$. Then using the initial conditions, 
we further integrate them with respect to $\tau$ from $0$ to $x$ to get the equations for $\bar{V}(x)$ as follows
\begin{eqnarray}\label{eq:5.17}
\begin{cases}
\bar{\rho}_{\pm}\bar{u}_{\pm}\textrm{A}_{\pm}(x)=\bar{\rho}^{\pm}_{\rm in}\bar{u}^{\pm}_{\rm in}\textrm{A}_{\pm}(0)+\mathbb{E}^{\pm}_{1}(x),\\
(\bar{\rho}_{+}\bar{u}^2_{+}+\bar{p}_{+})\textrm{A}_{+}(x)=\big[\bar{\rho}^{+}_{\rm in}(\bar{u}^{+}_{\rm in})^2+\bar{p}^{+}_{\rm in}\big]\textrm{A}_{+}(0)+\int^{x}_{0}\bar{p}_{\pm}(\tau)\textrm{A}'_{\pm}(\tau)\dd \tau+\mathbb{E}^{\pm}_{2}(x), \\
\big(\frac{1}{2}\bar{\rho}_{\pm}\bar{u}^3_{\pm}+\frac{\gamma \bar{p}_{\pm}\bar{u}_{\pm}}{\gamma-1}\big)\textrm{A}_{\pm}(x)=\big(\frac{1}{2}\bar{\rho}^{\pm}_{\rm in}(\bar{u}^{\pm}_{\rm in})^3
+\frac{\gamma \bar{p}^{\pm}_{\rm in}\bar{u}^{\pm}_{\rm in}}{\gamma-1}\big)\textrm{A}_{\pm}(0)\\
\qquad\qquad\qquad\qquad\qquad\qquad+\mathfrak{q}_0\int^{x}_{0}\bar{\rho}_{\pm}(\tau)\phi(\bar{T}_{\pm}(\tau))\bar{Y}_{\pm}(\tau)\textrm{A}_{\pm}(\tau)\dd\tau+\mathbb{E}^{\pm}_{3}(x),\\
\bar{\rho}_{\pm} \bar{u}_{\pm}\bar{Y}_{\pm}\textrm{A}_{\pm}(x)=\bar{\rho}^{\pm}_{\rm in} \bar{u}^{\pm}_{\rm in}\bar{Y}^{\pm}_{\rm in}\textrm{A}_{\pm}(0)
-\int^{x}_{0}\bar{\rho}_{\pm}(\tau)\phi(\bar{T}_{\pm}(\tau))\bar{Y}_{\pm}(\tau)\textrm{A}_{\pm}(\tau)\dd\tau+\mathbb{E}^{\pm}_{4}(x),
\end{cases}
\end{eqnarray}
where 
\begin{eqnarray*}
\begin{split}
\mathbb{E}^{\pm}_{1}(x)&\triangleq\pm\int^{g_{\pm}(0)}_{0}(\rho^{\pm}_{\rm in}-\bar{\rho}^{\pm}_{\rm in})(u^{\pm}_{\rm in}-\bar{u}^{\pm}_{\rm in})\dd y\pm\int^{g_{\pm}(x)}_{g_{\rm cd}(x)}(\bar{\rho}_{\pm}-\rho_{\pm})(u_{\pm}-\bar{u}_{\pm})\dd y\\
\mathbb{E}^{\pm}_{2}(x)&\triangleq\pm\int^{g_{\pm}(0)}_{0}\big(\rho^{\pm}_{\rm in}u^{\pm}_{\rm in}\pm\bar{\rho}^{\pm}_{\rm in}\bar{u}^{\pm}_{\rm in}\big)(u^{\pm}_{\rm in}-\bar{u}^{\pm}_{\rm in})\dd y
\pm\bar{u}^{\pm}_{\rm in}\int^{g_{\pm}(0)}_{0}(\rho^{\pm}_{\rm in}-\bar{\rho}^{\pm}_{\rm in})(u^{\pm}_{\rm in}-\bar{u}^{\pm}_{\rm in})\dd y\\
&\quad\pm\int^{g_{\pm}(x)}_{g_{\rm cd}(x)}(\bar{\rho}_{\pm}\bar{u}_{\pm}-\rho_{\pm} u_{\pm})(u_{\pm}-\bar{u}_{\pm})\dd y
\pm\bar{u}_{\pm}\int^{g_{\pm}(x)}_{g_{\rm cd}(x)}(\bar{\rho}_{\pm}-\rho_{\pm})(u_{\pm}-\bar{u}_{\pm})\dd y\\
&\quad+\int^{x}_{0}\big[p_{\pm}(\tau, g_{\rm cd}(\tau))-\bar{p}_{\pm}\big]\textrm{A}'_{\pm}(\tau)\dd \tau+\int^{x}_{0}\big[p_{\pm}(\tau, g_{\pm}(\tau))-p_{\pm}(\tau, g_{\rm cd}(\tau))\big]g'_{\pm}(\tau)d\tau,
\end{split}
\end{eqnarray*}
{
\begin{eqnarray*}
\begin{split}
\mathbb{E}^{\pm}_{3}(x)&\triangleq\pm\frac{1}{2}\int^{g_{\pm}(0)}_{0}\Big(\rho^{\pm}_{\rm in} u^{\pm}_{\rm in}-\overline{\rho^{\pm}_{\rm in} u^{\pm}_{\rm in}}\Big)\Big[(u^{\pm}_{\rm in})^2+(v^{\pm}_{\rm in})^{2}-\overline{(u^{\pm}_{\rm in})^2+(v^{\pm}_{\rm in})^2}\Big]\dd y \\
&\quad\pm\frac{\overline{\big(u^{\pm}_{\rm in}-\bar{u}^{\pm}_{\rm in}\big)^2+(v^{\pm}_{\rm in})^2}}{2}\int^{g_{\pm}(0)}_{0}\rho^{\pm}_{\rm in}u^{\pm}_{\rm in}\dd y
\pm\frac{(\bar{u}^{\pm}_{\rm in})^2}{2}\int^{g_{\pm}(0)}_{0}(\rho^{\pm}_{\rm in}-\bar{\rho}^{\pm}_{\rm in})(u^{\pm}_{\rm in}-\bar{u}^{\pm}_{\rm in})\dd y\\
&\quad\pm\frac{\overline{(u_{\pm}-\bar{u}_{\pm})^2+v^2_{\pm}}}{2}\int^{g_{\pm}(x)}_{g_{\rm cd}(x)}(-\rho_{\pm}u_{\pm})\dd y\pm\frac{\bar{u}^2_{\pm}}{2}\int^{g_{\pm}(x)}_{g_{\rm cd}(x)}(\bar{\rho}_{\pm}-\rho_{\pm})(u_{\pm}-\bar{u}_{\pm})\dd y\\
&\quad\pm\frac{1}{2}\int^{g_{\pm}(x)}_{g_{\rm cd}(x)}\big(\overline{\rho_{\pm} u_{\pm}}-\rho_{\pm} u_{\pm}\big)\big[u^2_{\pm}+v^{2}_{\pm}-\overline{u^2_{\pm}+v^{2}_{\pm}}\big]\dd y \\
&\quad\pm\mathfrak{q}_{0}\int^{x}_{0}\int^{g_{\pm}(\tau)}_{g_{\rm cd}(\tau)}\big(\bar{\rho}_{\pm}-\rho_{\pm}\big)\big(Y_{\pm}-\bar{Y}_{\pm}\big)\dd y \dd\tau\mp\mathfrak{q}_{0}\int^{x}_{0}\int^{g_{\pm}(\tau)}_{g_{\rm cd}(\tau)}\big(\phi(T_{\pm})-\phi(\bar{T}_{\pm})\big)\rho_{\pm}Y_{\pm}\dd y\dd\tau,
\end{split}
\end{eqnarray*}
and
\begin{eqnarray*}
\begin{split}
\mathbb{E}^{\pm}_{4}(x)&\triangleq\pm\frac{1}{2}\int^{g_{\pm}(0)}_{0}\big(\rho^{\pm}_{\rm in} u^{\pm}_{\rm in}-\overline{\rho^{\pm}_{\rm in} u^{\pm}_{\rm in}}\big)\big(Y^{\pm}_{\rm in}-\bar{Y}^{\pm}_{\rm in}\big)dy
\pm\bar{Y}^{\pm}_{\rm in}\int^{g_{\pm}(0)}_{0}\big(\rho^{\pm}_{\rm in}-\bar{\rho}^{\pm}_{\rm in}\big)\big(u^{\pm}_{\rm in}-\bar{u}^{\pm}_{\rm in}\big)\dd y\\
&\quad\pm\int^{g_{\rm cd}(x)}_{g_{-}(x)}\big(\overline{\rho_{\pm} u_{\pm}}-\rho_{\pm} u_{\pm}\big)\big(Y_{\pm}-\bar{Y}_{\pm}\big)\dd y-\bar{Y}_{\pm}\int^{g_{\rm cd}(x)}_{g_{\pm}(x)}\big(\rho_{\pm}-\bar{\rho}_{\pm}\big)\big(u_{\pm}-\bar{u}_{\pm}\big)\dd y\\
&\quad\pm\int^{x}_{0}\int^{g_{\pm}(\tau)}_{g_{\rm cd}(\tau)}\big(\phi(\bar{T}_{\pm})-\phi(T_{\pm})\big)\rho_{\pm}Y_{\pm}\dd y \dd\tau
\pm\int^{x}_{0}\int^{g_{\pm}(\tau)}_{g_{\rm cd}(\tau)}\phi(\bar{T}_{\pm})\big(\bar{\rho}_{\pm}-\rho_{\pm}\big)\big(Y_{\pm}-\bar{Y}_{\pm}\big)\dd y\dd\tau.
\end{split}
\end{eqnarray*}
}

\subsection{Proof of the Theorem \ref{thm:1.2}}
By Proposition \ref{prop:5.1},
we first integrate the equations \eqref{eq:5.1} with respect to $x$ to derive that
\begin{eqnarray}\label{eq:5.18}
\begin{cases}
\rho^{\textrm{A}}_{\pm} u^{\textrm{A}}_{\pm} \textrm{A}_{\pm}(x)=\bar{\rho}^{\pm}_{\rm in} u^{\pm}_{\rm in} \textrm{A}_{\pm}(0), \\
\big(\rho^{\textrm{A}}_{\pm} (u^{\textrm{A}}_{\pm})^{2}+p^{\textrm{A}}_{\pm}\big)\textrm{A}_{\pm}(x)=\big[\bar{\rho}^{\pm}_{\rm in} (\bar{u}^{\pm}_{\rm in})^{2}+\bar{p}^{\pm}_{\rm in}\big]\textrm{A}_{\pm}(0)+\int^{x}_{0}p^{\textrm{A}}_{\pm}(\tau)\textrm{A}'_{\pm}(\tau)\dd\tau, \\
\big(\frac{1}{2}\rho^{\textrm{A}}_{\pm}(u^{\textrm{A}}_{\pm})^3+\frac{\gamma p^{\textrm{A}}_{\pm}u^{\textrm{A}}_{\pm}}{\gamma-1}\big)\textrm{A}_{\pm}(x)=\big(\frac{1}{2}\bar{\rho}^{\pm}_{\rm in}(\bar{u}^{\pm}_{\rm in})^3
+\frac{\gamma \bar{p}^{\pm}_{\rm in}\bar{u}^{\pm}_{\rm in}}{\gamma-1}\big)\textrm{A}_{\pm}(0)\\
\qquad\qquad\qquad\qquad\qquad\qquad+\mathfrak{q}_0\int^{x}_{0}\rho^{\textrm{A}}_{\pm}(\tau)\phi(T^{\textrm{A}}_{\pm}(\tau))Y^{\textrm{A}}_{\pm}(\tau)\textrm{A}_{\pm}(\tau)\dd\tau,\\
\rho^{\textrm{A}}_{\pm} u^{\textrm{A}}_{\pm}Y^{\textrm{A}}_{\pm}\textrm{A}_{\pm}(x)=\bar{\rho}^{\pm}_{\rm in} \bar{u}^{\pm}_{\rm in}\bar{Y}^{\pm}_{\rm in}\textrm{A}_{\pm}(0)-\int^{x}_{0}\rho^{\textrm{A}}_{\pm}(\tau)\phi(T^{\textrm{A}}_{\pm}(\tau))Y^{\textrm{A}}_{\pm}(\tau)\textrm{A}_{\pm}(\tau)\dd\tau.
\end{cases}
\end{eqnarray}
%
Without loss of the generality, we  consider the error estimates
between $\bar{V}_{+}(x)=(\bar{u}_{+}, \bar{p}_{+}, \bar{\rho}_{+}, \bar{Y}_{+})^{\top}(x)$ and $V^{\textrm{A}}_{+}(x)=(u^{\textrm{A}}_{+}, p^{\textrm{A}}_{+}, \rho^{\textrm{A}}_{+}, Y^{\textrm{A}}_{+})^{\top}(x)$ only, since the argument for $\bar{V}_{-}(x)=(\bar{u}_{-}, \bar{p}_{-}, \bar{\rho}_{-}, \bar{Y}_{-})^{\top}(x)$ and $V^{\textrm{A}}_{-}(x)=(u^{\textrm{A}}_{-}, p^{\textrm{A}}_{-}, \rho^{\textrm{A}}_{-}, Y^{\textrm{A}}_{-})^{\top}(x)$ is similar. Set 
\begin{eqnarray*}
\delta\bar{V}^{\sharp}_{+}(x)\triangleq\bar{V}_{+}(x)-V^{\textrm{A}}_{+}(x)\triangleq\big(\delta\bar{u}^{\sharp}_{+}, \delta\bar{p}^{\sharp}_{+}, \delta\bar{\rho}^{\sharp}_{+}, \delta\bar{Y}^{\sharp}_{+}\big)^{\top}(x).
\end{eqnarray*}

It follows from \eqref{eq:5.17} and \eqref{eq:5.18} that $\delta\bar{V}^{\sharp}_{+}(x)$ satisfies
\begin{eqnarray}\label{eq:5.19}
\begin{cases}
\textrm{A}_{+}\bar{u}_{+}\delta \bar{\rho}^{\sharp}_{+}+\textrm{A}_{+}\rho^{\textrm{A}}_{+}\delta\bar{u}^{\sharp}_{+}=\mathbb{E}^{+}_{1},\\
\textrm{A}_{+}\rho^{\textrm{A}}_{+}u^{\textrm{A}}_{+}\delta\bar{u}^{\sharp}_{+}+\textrm{A}_{+}\delta\bar{p}^{\sharp}_{+}=\int^{x}_{0}\textrm{A}'_{+}(\tau)\delta\bar{p}^{\sharp}_{+}(\tau)\dd\tau
+\mathbb{E}^{+}_{2}-\bar{u}_{+}\mathbb{E}^{+}_{1},\\
\textrm{A}_{+}\big[\frac{\rho^{\textrm{A}}_{+}u^{\textrm{A}}_{+}(\bar{u}_{+}+u^{\textrm{A}}_{+})}{2}+\frac{\gamma}{\gamma-1}p^{\textrm{A}}_{+}\big]\delta\bar{u}^{\sharp}_{+}
+\frac{\gamma \textrm{A}_{+}\bar{u}_{+}}{\gamma-1}\delta\bar{p}^{\sharp}_{+}\\
\quad=\mathfrak{q}_{0}\int^{x}_{0}\textrm{A}_{+}(\tau)\rho^{\textrm{A}}_{+}(\tau)\phi(T^{\textrm{A}}_{+}(\tau))\delta\bar{Y}^{\sharp}_{+}(\tau)\dd\tau
 +\mathfrak{q}_{0}\int^{x}_{0}\textrm{A}_{+}(\tau)\phi(\bar{T}_{+})\bar{Y}_{+}(\tau)\delta\bar{\rho}^{\sharp}_{+}(\tau)\dd\tau\\
\qquad\quad +\mathfrak{q}_{0}\int^{x}_{0}\textrm{A}_{+}(\tau)\rho^{\textrm{A}}_{+}(\tau)\bar{Y}_{+}(\tau)\big[\phi(\bar{T}_{+}(\tau))-\phi(T^{\textrm{A}}_{+}(\tau))\big]\dd\tau+\mathbb{E}^{+}_{3}
-\frac{\bar{u}^2_{+}}{2}\mathbb{E}^{+}_{1},\\
\textrm{A}_{+}\rho^{\textrm{A}}_{+}u^{\textrm{A}}_{+}\delta \bar{Y}^{\sharp}_{+}=-\int^{x}_{0}\textrm{A}_{+}(\tau)\rho^{\textrm{A}}_{+}(\tau)\phi(T^{\textrm{A}}_{+}(\tau))\delta\bar{Y}^{\sharp}_{+}(\tau)\dd\tau
-\int^{x}_{0}\textrm{A}_{+}(\tau)\phi(\bar{T}_{+}(\tau))\bar{Y}_{+}(\tau)\delta \bar{\rho}^{\sharp}_{+}(\tau)\dd\tau\\
\qquad\qquad\qquad\quad-\int^{x}_{0}\textrm{A}_{+}(\tau)\rho^{\textrm{A}}_{+}(\tau)\bar{Y}_{+}(\tau)
\big[\phi(\bar{T}_{+}(\tau))-\phi(T^{\textrm{A}}_{+}(\tau))\big]\dd\tau+\mathbb{E}^{+}_{4}-\bar{Y}_{+}\mathbb{E}^{+}_{1}.
\end{cases}
\end{eqnarray}

So 
\begin{eqnarray}\label{eq:5.20}
\begin{cases}
\delta \bar{u}^{\sharp}_{+}=b^1_{u}(\bar{V}_{+},V^{\textrm{A}}_{+})\big[\int^{x}_{0}\delta\bar{p}^{\sharp}_{+}(\tau)\textrm{A}'_{+}(\tau)\dd\tau
+\mathbb{E}^{+}_{2}-\bar{u}_{+}\mathbb{E}^{+}_{1}\big]\\
\qquad\quad+b^2_{u}(\bar{V}_{+},V^{\textrm{A}}_{+})\big[\mathfrak{q}_{0}\int^{x}_{0}\textrm{A}_{+}(\tau)\rho^{\textrm{A}}_{+}(\tau)\phi(T^{\textrm{A}}_{+}(\tau))\delta\bar{Y}^{\sharp}_{+}(\tau)\dd\tau\\
\qquad\quad+\mathfrak{q}_{0}\int^{x}_{0}\textrm{A}_{+}(\tau)\phi(\bar{T}_{+})\bar{Y}_{+}(\tau)\delta\bar{\rho}^{\sharp}_{+}(\tau)\dd\tau\\
\qquad\quad+\mathfrak{q}_{0}\int^{x}_{0}\textrm{A}_{+}(\tau)\rho^{\textrm{A}}_{+}(\tau)\bar{Y}_{+}(\tau)[\phi(\bar{T}_{+}(\tau))-\phi(T^{\textrm{A}}_{+}(\tau))]\dd\tau
+\mathbb{E}^{+}_{3}-\frac{\bar{u}^2_{+}}{2}\mathbb{E}^{+}_{1}\big],\\
\delta\bar{p}^{\sharp}_{+}=b^1_{p}(\bar{V}_{+},V^{\textrm{A}}_{+})\big[\int^{x}_{0}\textrm{A}'_{+}(\tau)\delta\bar{p}^{\sharp}_{+}(\tau)\dd\tau
+\mathbb{E}^{+}_{2}-\bar{u}_{+}\mathbb{E}^{+}_{1}\big]\\
\qquad\quad+b^2_{p}(\bar{V}_{+},V^{\textrm{A}}_{+})\big[\mathfrak{q}_{0}\int^{x}_{0}\textrm{A}_{+}(\tau)\rho^{\textrm{A}}_{+}(\tau)\phi(T^{\textrm{A}}_{+}(\tau))\delta\bar{Y}^{\sharp}_{+}(\tau)\dd\tau\\
\qquad\quad+\mathfrak{q}_{0}\int^{x}_{0}\textrm{A}_{+}(\tau)\phi(\bar{T}_{+})\bar{Y}_{+}(\tau)\delta\bar{\rho}^{\sharp}_{+}(\tau)\dd\tau\\
\qquad\quad+\mathfrak{q}_{0}\int^{x}_{0}\textrm{A}_{+}(\tau)\rho^{\textrm{A}}_{+}(\tau)\bar{Y}_{+}(\tau)[\phi(\bar{T}_{+}(\tau))-\phi(T^{\textrm{A}}_{+}(\tau))]\dd\tau
+\mathbb{E}^{+}_{3}-\frac{\bar{u}^2_{+}}{2}\mathbb{E}^{+}_{1}\big]+\frac{1}{\textrm{A}_{+}\bar{u}_{+}}\mathbb{E}^{+}_{1},\\
\delta \bar{\rho}^{\sharp}_{+}=b^1_{\rho}(\bar{V}_{+},V^{\textrm{A}}_{+})\big[\int^{x}_{0}\textrm{A}'_{+}(\tau)\delta\bar{p}^{\sharp}_{+}(\tau)\dd\tau+\mathbb{E}^{+}_{2}-\bar{u}_{+}\mathbb{E}^{+}_{1}\big]\\
\qquad\quad+b^2_{\rho}(\bar{V}_{+},V^{\textrm{A}}_{+})\big[\mathfrak{q}_{0}\int^{x}_{0}\textrm{A}_{+}(\tau)\rho^{\textrm{A}}_{+}(\tau)\phi(T^{\textrm{A}}_{+}(\tau))\delta\bar{Y}^{\sharp}_{+}(\tau)\dd\tau\\
\qquad\quad+\mathfrak{q}_{0}\int^{x}_{0}\textrm{A}_{+}(\tau)\phi(\bar{T}_{+})\bar{Y}_{+}(\tau)\delta\bar{\rho}^{\sharp}_{+}(\tau)\dd\tau\\
\qquad\quad+\mathfrak{q}_{0}\int^{x}_{0}\textrm{A}_{+}(\tau)\rho^{\textrm{A}}_{+}(\tau)\bar{Y}_{+}(\tau)\big[\phi(\bar{T}_{+}(\tau))-\phi(T^{\textrm{A}}_{+}(\tau))\big]\dd\tau
+\mathbb{E}^{+}_{3}-\frac{\bar{u}^2_{+}}{2}\mathbb{E}^{+}_{1}\big]
+\frac{1}{\textrm{A}_{+}\bar{u}_{+}}\mathbb{E}^{+}_{1},\\
\delta \bar{Y}^{\sharp}_{+}=b^1_{Y}(\bar{V}_{+},V^{\textrm{A}}_{+})\big[-\int^{x}_{0}\textrm{A}_{+}(\tau)\rho^{\textrm{A}}_{+}(\tau)\phi(T^{\textrm{A}}_{+}(\tau))\delta\bar{Y}^{\sharp}_{+}(\tau)\dd\tau\\
\qquad\quad-\int^{x}_{0}\textrm{A}_{+}(\tau)\phi(\bar{T}_{+}(\tau))\bar{Y}_{+}(\tau)\delta \bar{\rho}^{\sharp}_{+}(\tau)\dd\tau\\
\qquad\quad-\int^{x}_{0}\textrm{A}_{+}(\tau)\rho^{\textrm{A}}_{+}(\tau)\bar{Y}_{+}(\tau)
[\phi(\bar{T}_{+}(\tau))-\phi(T^{\textrm{A}}_{+}(\tau))]\dd\tau+\mathbb{E}^{+}_{4}-\bar{Y}_{+}\mathbb{E}^{+}_{1}\big],
\end{cases}
\end{eqnarray}
where 
\begin{eqnarray}\label{eq:5.21}
\begin{cases}
b^1_{u}(\bar{V}_{+},V^{\textrm{A}}_{+})\triangleq\frac{\gamma \bar{u}_{+}}{\textrm{A}_{+}\big[\frac{\gamma+1}{2}\rho^{\textrm{A}}_{+}u^{\textrm{A}}_{+}\bar{u}_{+}-\frac{\gamma-1}{2}\rho^{\textrm{A}}_{+}(u^{\textrm{A}}_{+})^2-\gamma p^{\textrm{A}}_{+}\big]},\\
b^2_{u}(\bar{V}_{+},V^{\textrm{A}}_{+})\triangleq-\frac{1}{\textrm{A}_{+}\big[\frac{\gamma+1}{2(\gamma-1)}\rho^{\textrm{A}}_{+}u^{\textrm{A}}_{+}\bar{u}_{+}-\frac{1}{2}\rho^{\textrm{A}}_{+}(u^{\textrm{A}}_{+})^2-\frac{\gamma}{\gamma-1} p^{\textrm{A}}_{+}\big]}, \\
b^1_{p}(\bar{V}_{+},V^{\textrm{A}}_{+})\triangleq-\frac{\frac{\rho^{\textrm{A}}_{+}u^{\textrm{A}}_{+}(\bar{u}_{+}+u^{\textrm{A}}_{+})}{2}+\frac{\gamma}{\gamma-1}p^{\textrm{A}}_{+}}{\textrm{A}_{+}
\big[\frac{\gamma+1}{2(\gamma-1)}\rho^{\textrm{A}}_{+}u^{\textrm{A}}_{+}\bar{u}_{+}-\frac{1}{2}\rho^{\textrm{A}}_{+}(u^{\textrm{A}}_{+})^2-\frac{\gamma}{\gamma-1} p^{\textrm{A}}_{+}\big]},\\ b^2_{p}(\bar{V}_{+},V^{\textrm{A}}_{+})\triangleq\frac{\rho^{\textrm{A}}_{+}u^{\textrm{A}}_{+}}{\textrm{A}_{+}\bar{u}_{+}\big[\frac{\gamma+1}{2(\gamma-1)}\rho^{\textrm{A}}_{+}u^{\textrm{A}}_{+}\bar{u}_{+}-\frac{1}{2}\rho^{\textrm{A}}_{+}(u^{\textrm{A}}_{+})^2-\frac{\gamma}{\gamma-1} p^{\textrm{A}}_{+}\big]}, \\
b^1_{\rho}(\bar{V}_{+},V^{\textrm{A}}_{+})\triangleq-\frac{\gamma \rho^{\textrm{A}}_{+}}{\textrm{A}_{+}\big[\frac{\gamma+1}{2}\rho^{\textrm{A}}_{+}u^{\textrm{A}}_{+}\bar{u}_{+}-\frac{\gamma-1}{2}\rho^{\textrm{A}}_{+}(u^{\textrm{A}}_{+})^2-\gamma p^{\textrm{A}}_{+}\big]},\\
b^2_{\rho}(\bar{V}_{+},V^{\textrm{A}}_{+})\triangleq\frac{p^{\textrm{A}}_{+}}{\textrm{A}_{+}\bar{u}_{+}\big[\frac{\gamma+1}{2(\gamma-1)}\rho^{\textrm{A}}_{+}u^{\textrm{A}}_{+}\bar{u}_{+}-\frac{1}{2}\rho^{\textrm{A}}_{+}(u^{\textrm{A}}_{+})^2-\frac{\gamma}{\gamma-1} p^{\textrm{A}}_{+}\big]}, \\
b^1_{Y}(\bar{V}_{+},V^{\textrm{A}}_{+})\triangleq\frac{1}{\textrm{A}_{+}\rho^{\textrm{A}}_{+}u^{\textrm{A}}_{+}}.
\end{cases}
\end{eqnarray}

Now we will estimate $\|\delta\bar{V}^{\sharp}_{+}\|_{1,\alpha;(0,L)}$ in three steps.

\emph{1. Estimate on $\|\delta\bar{V}^{\sharp}_{+}\|_{0,0;(0,L)}$.} By Theorem \ref{thm:1.1} and \eqref{eq:5.20}-\eqref{eq:5.21}, we can get
\begin{eqnarray}\label{eq:5.22}
\begin{split}
|\delta\bar{V}^{\sharp}_{+}(x)|&\leq \mathcal{O}(1)\int^{x}_{0}\big[|(\delta\bar{\rho}^{\sharp}_{+}, \delta\bar{p}^{\sharp}_{+}, \delta\bar{Y}^{\sharp}_{+})(\tau)|+|\phi(\bar{T}_{+}(\tau))-\phi(T^{\textrm{A}}_{+}(\tau))|\big]\dd\tau
 +\sum^{4}_{k=1}\|\mathbb{E}^{+}_{k}\|_{0,0;(0,L)},
\end{split}
\end{eqnarray}
where the constant $\mathcal{O}(1)$ depends only on $\underline{U}_{+}$ and $L$.

Note that $\phi\in C^{1,1}$ and $\delta \bar{T}^{\sharp}_{+}=-\frac{\bar{p}_{+}}{\mathcal{R}\bar{\rho}_{+}\rho^{\textrm{A}}_{+}}\delta \bar{\rho}^{\sharp}_{+}+\frac{1}{\mathcal{R}\rho^{\textrm{A}}_{+}}\delta \bar{p}^{\sharp}_{+}$, so
\begin{eqnarray}\label{eq:5.23}
\begin{split}
\big|\phi(\bar{T}_{+}(\tau))-\phi(T^{\textrm{A}}_{+}(\tau))\big|\leq \|\phi'\|_{0,0;\mathbb{R}_{+}}|\delta \bar{T}^{\sharp}_{+}(\tau)|
\leq \mathcal{O}(1)\big(|\delta \bar{p}^{\sharp}_{+}(\tau)|+|\delta \bar{\rho}^{\sharp}_{+}(\tau)|\big),
\end{split}
\end{eqnarray}
where the constant $\mathcal{O}(1)$ depends only on $\underline{U}_{+}$ and $L$.
Then by the Gronwall's inequality, 
\begin{eqnarray}\label{eq:5.24}
\|\delta\bar{V}^{\sharp}_{+}\|_{0,0;(0,L)}\leq \mathcal{O}(1)\Big(\sum^{4}_{k=1}\|\mathbb{E}^{+}_{k}\|_{0,0;(0,L)}\Big),
\end{eqnarray}
where the constant $\mathcal{O}(1)$ depends only on $\underline{U}_{+}$ and $L$.
\emph{2. Estimate on $\big\|\frac{\dd \delta\bar{V}^{\sharp}_{+}}{\dd x}\big\|_{0,0;(0,L)}$.} Take the derivative on both sides of $\eqref{eq:5.20}_1$ with respect to $x$,
\begin{align}\label{eq:5.25}
\begin{split}
\frac{\dd\delta \bar{u}^{\sharp}_{+}}{\dd x}&=b^1_{u}(\bar{V}_{+},V^{\textrm{A}}_{+})\Big[\textrm{A}'_{+}(x)\delta\bar{p}^{\sharp}_{+}(x)
+\frac{\dd\mathbb{E}^{+}_{2}}{\dd x}-\frac{\dd(\bar{u}_{+}\mathbb{E}^{+}_{1})}{\dd x}\Big]\\
&\quad+\frac{\dd b^1_{u}(\bar{V}_{+},V^{\textrm{A}}_{+})}{\dd x}\Big[\int^{x}_{0}\textrm{A}'_{+}(\tau)\delta\bar{p}^{\sharp}_{+}(\tau)\dd\tau
+\mathbb{E}^{+}_{2}-\bar{u}_{+}\mathbb{E}^{+}_{1}\Big]\\
&\quad+b^2_{u}(\bar{V}_{+},V^{\textrm{A}}_{+})\Big[\mathfrak{q}_{0}\textrm{A}_{+}(x)\rho^{\textrm{A}}_{+}(x)\phi(T^{\textrm{A}}_{+}(x))\delta\bar{Y}^{\sharp}_{+}(x)
+\mathfrak{q}_{0}\textrm{A}_{+}(x)\phi(\bar{T}_{+}(x))\bar{Y}_{+}(x)\delta\bar{\rho}^{\sharp}_{+}(x)\\
&\quad+\mathfrak{q}_{0}\textrm{A}_{+}(x)\rho^{\textrm{A}}_{+}(x)\bar{Y}_{+}(x)\big[\phi(\bar{T}_{+}(x))-\phi(T^{\textrm{A}}_{+}(x))\big]+\frac{\dd\mathbb{E}^{+}_{3}}{\dd x}
-\frac{1}{2}\frac{\dd (\bar{u}^2_{+}\mathbb{E}^{+}_{1})}{\dd x}\Big]\\
&\quad+\frac{\dd b^2_{u}(\bar{V}_{+},V^{\textrm{A}}_{+})}{\dd x}
\Big[\mathfrak{q}_{0}\int^{x}_{0}\textrm{A}_{+}(\tau)\rho^{\textrm{A}}_{+}(\tau)\phi(T^{\textrm{A}}_{+}(\tau))\delta\bar{Y}^{\sharp}_{+}(\tau)\dd\tau\\
&\quad +\mathfrak{q}_{0}\int^{x}_{0}\textrm{A}_{+}(\tau)\phi(\bar{T}_{+})\bar{Y}_{+}(\tau)\delta\bar{\rho}^{\sharp}_{+}(\tau)\dd\tau\\
&\quad+\mathfrak{q}_{0}\int^{x}_{0}\textrm{A}_{+}(\tau)\rho^{\textrm{A}}_{+}(\tau)\bar{Y}_{+}(\tau)\big[\phi(\bar{T}_{+}(\tau))-\phi(T^{\textrm{A}}_{+}(\tau))\big]\dd\tau
+\mathbb{E}^{+}_{3}-\frac{\bar{u}^2_{+}}{2}\mathbb{E}^{+}_{1}\Big].
\end{split}
\end{align}

Then, by Theorem \ref{thm:1.1} and \eqref{eq:5.23}, we have for $x\in(0,L)$
\begin{eqnarray*}
\begin{split}
\bigg|\frac{\dd\delta \bar{u}^{\sharp}_{+}(x)}{\dd x}\bigg|&\leq \mathcal{O}(1)\Big(\|(\delta\bar{p}^{\sharp}_{+}, \delta\bar{\rho}^{\sharp}_{+}, \delta\bar{Y}^{\sharp}_{+} \|_{0,0;(0,L)}+|\phi(\bar{T}_{+}(x))-\phi(T^{\textrm{A}}_{+}(x))|+\sum^{3}_{k=1}\|\mathbb{E}^{+}_{k}\|_{1,0;(0,L)}\Big)\\
&\leq \mathcal{O}(1)\Big(\|\delta \bar{V}^{\sharp}_{+}\|_{0,0;(0,L)}+\sum^{3}_{k=1}\|\mathbb{E}^{+}_{k}\|_{1,0;(0,L)}\Big).
\end{split}
\end{eqnarray*}

So with the help of estimate \eqref{eq:5.24}, we have 
\begin{eqnarray}\label{eq:5.26}
\begin{split}
\bigg\|\frac{\dd\delta \bar{u}^{\sharp}_{+}}{\dd x}\bigg\|_{0,0;(0,L)}\leq \mathcal{O}(1)\sum^{4}_{k=1}\|\mathbb{E}^{+}_{k}\|_{1,0;(0,L)},
\end{split}
\end{eqnarray}
where the constant $\mathcal{O}(1)$ depends only on $\underline{U}_{+}$ and $L$.

In the same way, we can also obtain
\begin{eqnarray}\label{eq:5.27}
\begin{split}
\bigg\|\Big(\frac{\dd\delta \bar{p}^{\sharp}_{+}}{\dd x}, \frac{\dd\delta \bar{\rho}^{\sharp}_{+}}{\dd x}, \frac{\dd\delta \bar{Y}^{\sharp}_{+}}{\dd x}\Big)\bigg\|_{0,0;(0,L)}\leq \mathcal{O}(1)\sum^{4}_{k=1}\|\mathbb{E}^{+}_{k}\|_{1,0;(0,L)},
\end{split}
\end{eqnarray}
where the constant $\mathcal{O}(1)$ depends only on $\underline{U}_{+}$ and $L$.
Therefore, 
\begin{eqnarray}\label{eq:5.28}
\bigg\|\frac{\dd\delta \bar{V}^{\sharp}_{+}}{\dd x}\bigg\|_{0,0;(0,L)}\leq\mathcal{O}(1)\sum^{4}_{k=1}\|\mathbb{E}^{+}_{k}\|_{1,0;(0,L)},
\end{eqnarray}
where the constant $\mathcal{O}(1)$ depends only on $\underline{U}_{+}$ and $L$.

\emph{3. Estimate on $\big[\frac{\dd \delta\bar{V}^{\sharp}_{+}}{\dd x}\big]_{0,\alpha;(0,L)}$.} For any $x_1,\ x_2\in (0,L)$ with $x_1\neq x_2$, by \eqref{eq:5.25}, we have
\begin{eqnarray}\label{eq:5.29}
\begin{split}
&\bigg|\frac{\dd \delta\bar{u}^{\sharp}_{+}(x_1)}{\dd x}-\frac{\dd \delta\bar{u}^{\sharp}_{+}(x_2)}{\dd x}\bigg|\\
&\leq \mathcal{O}(1)\Big(\big\|(\delta\bar{p}^{\sharp}_{+}, \delta\bar{\rho}^{\sharp}_{+},\delta\bar{Y}^{\sharp}_{+})\big\|_{1,0;(0,L)}
+\sum^{3}_{k=1}\|\mathbb{E}^{+}_{k}\|_{1,\alpha;(0,L)}\Big) |x_1-x_2|^{\alpha}\\
&\quad+\mathcal{O}(1)\big|\phi(\bar{T}_{+}(x_1))-\phi(T^{\textrm{A}}_{+}(x_1))-[\phi(\bar{T}_{+}(x_2))-\phi(T^{\textrm{A}}_{+}(x_2))]\big|.
\end{split}
\end{eqnarray}

For the second term on the right hand side of inequality \eqref{eq:5.29}, it follows from the mean value theorem and the fact $\phi\in C^{1,1}$ that
\begin{eqnarray*}
\begin{split}
&\big|\phi(\bar{T}_{+}(x_1))-\phi(T^{\textrm{A}}_{+}(x_1))-[\phi(\bar{T}_{+}(x_2))-\phi(T^{\textrm{A}}_{+}(x_2))]\big|\\
&\quad \leq \Big|\big(\phi'(\bar{T}_{+})\frac{\dd \bar{T}_{+}}{\dd x}-\phi'(T^{\textrm{A}}_{+})\frac{\dd T^{\textrm{A}}_{+}}{\dd x}\big)(\varsigma)\Big||x_1-x_2|\\
&\quad \leq \Big[\Big|\phi'(\bar{T}_{+})\frac{\dd \delta\bar{T}^{\sharp}_{+}}{\dd x}\Big|+\Big|\frac{\dd T^{\textrm{A}}_{+}}{\dd x}\Big|
\big|\phi'(\bar{T}_{+})-\phi'(T^{\textrm{A}}_{+})\big|\Big]|x_1-x_2|\\
&\quad \leq \mathcal{O}(1)\|\phi'\|_{0,1;(0,L)}\|\delta\bar{T}^{\sharp}_{+}\|_{1,0;(0,L)}|x_1-x_2|^{\alpha}\\
&\quad \leq \mathcal{O}(1)\big\|(\delta\bar{p}^{\sharp}_{+}, \delta\bar{\rho}^{\sharp}_{+})\big\|_{1,0;(0,L)}|x_1-x_2|^{\alpha},
\end{split}
\end{eqnarray*}
where the constant $\mathcal{O}(1)$ depends only on $\underline{U}_{+}$, $\alpha$ and $L$.
So
\begin{eqnarray*}
\begin{split}
\bigg|\frac{\dd \delta\bar{u}^{\sharp}_{+}(x_1)}{\dd x}-\frac{\dd \delta\bar{u}^{\sharp}_{+}(x_2)}{\dd x}\bigg|
\leq \mathcal{O}(1)\Big(\sum^{4}_{k=1}\|\mathbb{E}^{+}_{k}\|_{1,\alpha;(0,L)}\Big) |x_1-x_2|^{\alpha},
\end{split}
\end{eqnarray*}
which implies that
\begin{eqnarray}\label{eq:5.30}
\Big[\frac{\dd \delta\bar{u}^{\sharp}_{+}}{\dd x}\Big]_{0,\alpha;(0,L)}\leq \mathcal{O}(1)\Big(\sum^{4}_{k=1}\|\mathbb{E}^{+}_{k}\|_{1,\alpha;(0,L)}\Big),
\end{eqnarray}
where the constant $\mathcal{O}(1)$ depends only on $\underline{U}_{+}$, $\alpha$ and $L$.

In the same way, one can also establish that
\begin{eqnarray}\label{eq:5.31}
\Big[\big(\frac{\dd\delta \bar{p}^{\sharp}_{+}}{\dd x}, \frac{\dd\delta \bar{\rho}^{\sharp}_{+}}{\dd x}, \frac{\dd\delta \bar{Y}^{\sharp}_{+}}{\dd x}\big)\Big]_{0,\alpha;(0,L)}\leq \mathcal{O}(1)\Big(\sum^{4}_{k=1}\|\mathbb{E}^{+}_{k}\|_{1,\alpha;(0,L)}\Big),
\end{eqnarray}
where the constant $\mathcal{O}(1)$ depends only on $\underline{U}_{+}$, $\alpha$ and $L$.

On the other hand, by  Theorem \ref{thm:1.1}, Proposition \ref{prop:5.1} and Remark \ref{rem:5.1}, we know that the terms $\mathbb{E}^{+}_{k}$,$(k=1,2,3,4)$ can be bounded by
\begin{eqnarray}\label{eq:5.32}
\sum^{4}_{k=1}\|\mathbb{E}^{+}_{k}\|_{1,\alpha; (0,L)}\leq \mathcal{O}(1)\epsilon^2,
\end{eqnarray}
where $\epsilon>0$ is sufficiently small, and the constant $\mathcal{O}(1)$ depends only on $\underline{U}_{+}$, $\alpha$ and $L$.

Finally, combining the estimates \eqref{eq:5.24}, \eqref{eq:5.28} and \eqref{eq:5.30}-\eqref{eq:5.32}, we can choose constants $C_{2}>0$ and $\epsilon^{*}_{0}>0$ depending only on $\underline{U}_{+}$, $\alpha$ and $L$
such that for $\epsilon\in(0,\epsilon^{*}_{0})$, it holds
\begin{eqnarray*}
\big\|\delta \bar{V}^{\sharp}_{+}\big\|_{1,\alpha;(0,L)}\leq C_{2}\epsilon^2,
\end{eqnarray*}
which is estimate \eqref{eq:1.29} in Theorem \ref{thm:1.2}.

\section{Proof of Theorem \ref{thm-global uniqueness}}\label{Sec6}

In this section, we are going to prove Theorem \ref{thm-global uniqueness} by establishing the global uniqueness 
of supersonic combustion contact discontinuity solution in a two-dimensional straight nozzle $\underline{\mathcal{N}}$ via the Lagrangian {\color{black}coordinate} transformation and the characteristic method.

Since $U_{\pm}\in C^{1,\alpha}(\underline{\mathcal{N}}_{\pm})$, and $u_{\pm}>0$, $u_{\pm}\neq c_{\pm}$ and $\rho_{\pm}>0$, so we can reformulate \emph{Problem 1.3} in the Lagrangian coordinates first. Define
\begin{eqnarray}\label{lagrangian-gloabl-unique-1}
\underline{\rm{m}}_{+}=\int^{1}_{0}\underline{\rho}^+ \underline{u}^{+}\dd\tau=\underline{\rho}^{+} \underline{u}^{+}>0, \quad \underline{\rm m}_{-}=\int^{0}_{-1}\underline{\rho}^{-}\underline{u}^{-}\dd\tau=\underline{\rho}^{-}\underline{u}^{-}>0.
\end{eqnarray}

Then, under the Lagrangian transformation $\mathcal{L}$ introduced by \eqref{eq:2.3} in section 2.1, 
the nozzle $\underline{\mathcal{N}}$ becomes
\begin{eqnarray*}
\underline{\tilde{\mathcal{N}}}\triangleq\{(\xi,\eta)\in \mathbb{R}^{2}| \xi\in(0,L),\ -\underline{\rm{m}}_{-}<\eta<\underline{\rm{m}}_{+}\},
\end{eqnarray*}
with its upper and lower boundaries being
\begin{eqnarray*}
\underline{\tilde{\Gamma}}_{+}\triangleq\{(\xi,\eta)\in \mathbb{R}^{2}| \xi\in(0,L),\ \eta=\underline{\rm{m}}_{+}\},\quad\
\underline{\tilde{\Gamma}}_{-}\triangleq\{(\xi,\eta)\in \mathbb{R}^{2}| \xi\in(0,L),\ \eta=-\underline{\rm{m}}_{-}\}.
\end{eqnarray*}

In the new coordinates, the contact discontinuity $\Gamma_{\rm{cd}}$ is
\begin{eqnarray*}
\tilde{\Gamma}_{\rm{cd}}\triangleq \{(\xi,\eta)\in \mathbb{R}^{2}| \xi\in(0,L),\ \eta=0\}.
\end{eqnarray*}
It divides $\underline{\tilde{\mathcal{N}}}$ into two supersonic regions 
\begin{eqnarray*}
\underline{\tilde{\mathcal{N}}}_{+}\triangleq \underline{\tilde{\mathcal{N}}}\cap\{0<\eta<\underline{\rm{m}}_{+}\},\quad\
\underline{\tilde{\mathcal{N}}}_{-}\triangleq\underline{\tilde{\mathcal{N}}}\cap\{-\underline{\rm{m}}_{-}<\eta<0\},
\end{eqnarray*}
with its corresponding entrances
\begin{eqnarray*}
\underline{\tilde{\Gamma}}^{+}_{\rm{in}}\triangleq \{(\xi,\eta)\in \mathbb{R}^{2}| \xi=0,\ \eta\in(0,\underline{\rm{m}}_{+}) \},\quad\
\underline{\tilde{\Gamma}}^{-}_{\rm{in}}\triangleq \{(\xi,\eta)\in \mathbb{R}^{2}| \xi=0,\ \eta\in(-\underline{\rm{m}}_{-},0)\}.
\end{eqnarray*}

Then, as done in section 2.2, we can reformulate \emph{Problem 1.3} into the following uniqueness problem for $(\omega_{\pm}, p_{\pm}, B_{\pm}, S_{\pm}, Y_{\pm})$ in $\underline{\tilde{\mathcal{N}}}_{\pm}$.

\smallskip
$\mathbf{Problem\ 6.1}$ 
Is the solution to the following initial-boundary value problem for $(\omega_{\pm}, p_{\pm}, B_{\pm}, S_{\pm}, Y_{\pm})$ in the Lagrangian {\color{black}coordinates} globally unique?
\begin{eqnarray}\label{gloabl-unique-problem-5.1}
\begin{cases}
\partial_{\xi}\omega_{+}+\lambda^{\pm}_{+}\partial_{\eta}\omega_{+}\pm\Lambda_{+}\big(\partial_{\xi}p_{+}+\lambda^{\pm}_{+}\partial_{\eta}p_{+}\big)
=\frac{(\gamma-1)\mathfrak{q}_{0}\phi(T_{+})\lambda^{\pm}_{+}Y_{+}}{\rho_{+}c^2_{+}u^2_{+}}, &\quad  \mbox{in} \quad \underline{\tilde{\mathcal{N}}}_{+}, \\
\partial_{\xi}B_{+}=\frac{\mathfrak{q}_0\phi(T_{+})Y_{+}}{u_{+}},\quad \partial_{\xi}S_{+}=-\frac{\gamma \mathcal{R}\mathfrak{q}_{0}\phi(T_{+})Y_{+}}{c^2_{+}u_{+}},\quad\ \partial_{\xi}Y_{+}=-\frac{\phi(T_{+})Y_{+}}{u_{+}}, &\quad \mbox{in}\quad \underline{\tilde{\mathcal{N}}}_{+},\\
\partial_{\xi}\omega_{-}+\lambda^{\pm}_{-}\partial_{\eta}p_{-}\pm\Lambda_{-}\big(\partial_{\xi}p_{-}+\lambda^{\pm}_{-}\partial_{\eta}p_{-}\big)
=\frac{(\gamma-1)\mathfrak{q}_{0}\phi(T_{-})\lambda^{\pm}_{-}Y_{-}}{\rho_{-}c^2_{-}u^2_{-}}, &\quad  \mbox{in} \quad \underline{\tilde{\mathcal{N}}}_{-},\\
\partial_{\xi}B_{-}=\frac{\mathfrak{q}_0\phi(T_{-})Y_{-}}{u_{-}},\quad\partial_{\xi}S_{-}=-\frac{\gamma \mathcal{R}\mathfrak{q}_{0}\phi(T_{-})Y_{-}}{c^2_{-}u_{-}}, \quad\ \partial_{\xi}Y_{-}=-\frac{\phi(T_{-})Y_{-}}{u_{-}}, &\quad  \mbox{in} \quad \underline{\tilde{\mathcal{N}}}_{-},\\
(\omega_{+}, p_{+}, B_{+}, S_{+}, Y_{+})=(0, \underline{p}^{+}, \underline{B}^{+}, \underline{S}^{+}, 0),  &\quad  \mbox{on} \quad \underline{\tilde{\Gamma}}^{+}_{\rm in},\\
(\omega_{-}, p_{-}, B_{-}, S_{-}, Y_{-})=(0, \underline{p}^{-}, \underline{B}^{-}, \underline{S}^{-}, 0),  &\quad  \mbox{on} \quad \underline{\tilde{\Gamma}}^{-}_{\rm in},\\
\omega_{+}=0, &\quad \mbox{on} \quad \underline{\tilde{\Gamma}}_{+},\\
\omega_{+}=\omega_{-},\ \ p_{+}=p_{-},&\quad \mbox{on} \quad \underline{\tilde{\Gamma}}_{\rm cd},\\
\omega_{-}=0, &\quad \mbox{on} \quad \underline{\tilde{\Gamma}}_{-},
\end{cases}
\end{eqnarray}
where $p$, $\rho$, $S$ and $B$ satisfy \eqref{eq:1.4} and \eqref{eq:1.8}.
$\lambda^{\pm}_{\pm}$ and $\Lambda_{\pm}$ are defined by \eqref{eq:2.18}, \eqref{eq:2.23} and \eqref{eq:2.25}.

\smallskip
Now, we will answer $\mathbf{Problem\ 6.1}$ by two main steps. First, let us consider
$(B_{\pm}, S_{\pm}, Y_{\pm})$.
\begin{proposition}\label{prop-unique-5.1}
Suppose that $\phi(T)$ is $C^{1,1}$ with respect to $T>0$ and $\alpha\in(0,1)$. Let $(U, \tilde{\Gamma}_{\rm{cd}})$ be a $C^{1,\alpha}$-smooth supersonic contact discontinuity solution of problem \eqref{gloabl-unique-problem-5.1} 
in $\underline{\tilde{\mathcal{N}}}$,
consisting of two $C^{1,\alpha}$-smooth supersonic flows $U_{+}=(u_{+}, v_{+}, p_{+},\rho_{+}, Y_{+})^{\top}$ in $\underline{\tilde{\mathcal{N}}}_{+}$
and $U_{-}=(u_{-}, v_{-}, p_{-},\rho_{-}, Y_{-})^{\top}$ in $\underline{\tilde{\mathcal{N}}}_{-}$
with the contact discontinuity $\tilde{\Gamma}_{\rm cd}$ in between and satisfying
\begin{eqnarray}\label{eq-global-unique-5.1}
\begin{split}
u_{+}>0, \ u_{+}\neq c_{+},\ \rho_{+}>0, \  \mbox{\rm{in}} \quad \underline{\tilde{\mathcal{N}}}_{+},\quad  \mbox{\rm{and}}\quad
u_{-}>0,\ u_{-}\neq c_{-},\  \rho_{-}>0, \  \mbox{\rm{in}} \quad \underline{\tilde{\mathcal{N}}}_{-}.
\end{split}
\end{eqnarray}
Then, $(B_{\pm}, S_{\pm}, Y_{\pm})$ to the following problem
\begin{eqnarray}\label{gloabl-unique-problem-5.2}
\begin{cases}
\partial_{\xi}B_{+}=\frac{\mathfrak{q}_0\phi(T_{+})Y_{+}}{u_{+}},\quad \partial_{\xi}S_{+}=-\frac{\gamma \mathcal{R}\mathfrak{q}_{0}\phi(T_{+})Y_{+}}{c^2_{+}u_{+}},\quad\ \partial_{\xi}Y_{+}=-\frac{\phi(T_{+})Y_{+}}{u_{+}}, &\quad \mbox{\rm{in}}\quad \underline{\tilde{\mathcal{N}}}_{+},\\
\partial_{\xi}B_{-}=\frac{\mathfrak{q}_0\phi(T_{-})Y_{-}}{u_{-}},\quad\partial_{\xi}S_{-}=-\frac{\gamma \mathcal{R}\mathfrak{q}_{0}\phi(T_{-})Y_{-}}{c^2_{-}u_{-}}, \quad\ \partial_{\xi}Y_{-}=-\frac{\phi(T_{-})Y_{-}}{u_{-}}, &\quad  \mbox{\rm{in}} \quad \underline{\tilde{\mathcal{N}}}_{-},\\
(B_{+}, S_{+}, Y_{+})=(\underline{B}^{+}, \underline{S}^{+}, 0),  &\quad  \mbox{\rm{on}} \quad \underline{\tilde{\Gamma}}^{+}_{\rm in},\\
(B_{-}, S_{-}, Y_{-})=(\underline{B}^{-}, \underline{S}^{-}, 0),  &\quad  \mbox{\rm{on}} \quad \underline{\tilde{\Gamma}}^{-}_{\rm in},
\end{cases}
\end{eqnarray}
satisfies
\begin{eqnarray}\label{eq-global-unique-5.2}
B_{\pm}\equiv \underline{B}^{\pm},\quad S_{\pm}\equiv \underline{S}^{\pm},\quad  Y_{\pm}\equiv 0,\quad \mbox{\rm{in}}\quad  \underline{\tilde{\mathcal{N}}}_{\pm}.
\end{eqnarray}
\end{proposition}

\begin{proof}
Since $\phi(T)$ is $C^{1,1}$ with respect to $T>0$ and $U_{\pm}\in C^{1,\alpha}(\underline{\tilde{\mathcal{N}}}_{\pm})$ with $u_{\pm}>0$ and $\rho_{\pm}>0$, we can get from the equations for $Y_{\rm}$ that
\begin{eqnarray*}
Y_{\pm}=Y_{\pm}(0,\eta)\rm{e}^{-\int^{\xi}_{0}\frac{\phi(T_{\pm})Y_{\pm}}{u_{\pm}}\dd\tau}=0,\qquad \mbox{\rm{in}}\quad \underline{\tilde{\mathcal{N}}}_{\pm}.
\end{eqnarray*}

Substituting it into equations for $B_{\pm}$ and $S_{\pm}$, we have $\partial_{\xi}B_{\pm}=\partial_{\xi}S_{\pm}=0$ in $\underline{\tilde{\mathcal{N}}}_{\pm}$.
So 
\begin{eqnarray*}
B_{\pm}=\underline{B}^{\pm},\quad S_{\pm}=\underline{S}^{\pm},\qquad \mbox{\rm{in}}\quad \underline{\tilde{\mathcal{N}}}_{\pm}.
\end{eqnarray*}
\end{proof}

Next, let us consider $(\omega_{\pm}, p_{\pm})$. By Proposition \ref{prop-unique-5.1}, we know that $(\omega_{\pm}, p_{\pm})$ satisfies the following initial-boundary value problem:
\begin{eqnarray}\label{gloabl-unique-problem-5.3}
\begin{cases}
\partial_{\xi}\omega_{+}+\lambda^{\pm}_{+}\partial_{\eta}\omega_{+}\pm\Lambda_{+}\big(\partial_{\xi}p_{+}+\lambda^{\pm}_{+}\partial_{\eta}p_{+}\big)
=0, &\quad  \mbox{in} \quad \underline{\tilde{\mathcal{N}}}_{+}, \\
\partial_{\xi}\omega_{-}+\lambda^{\pm}_{-}\partial_{\eta}p_{-}\pm\Lambda_{-}\big(\partial_{\xi}p_{-}+\lambda^{\pm}_{-}\partial_{\eta}p_{-}\big)
=0, &\quad  \mbox{in} \quad \underline{\tilde{\mathcal{N}}}_{-},\\
(\omega_{+}, p_{+})=(0, \underline{p}^{+}),  &\quad  \mbox{on} \quad \underline{\tilde{\Gamma}}^{+}_{\rm in},\\
(\omega_{-}, p_{-})=(0, \underline{p}^{-}),  &\quad  \mbox{on} \quad \underline{\tilde{\Gamma}}^{-}_{\rm in},\\
\omega_{+}=0, &\quad \mbox{on} \quad \underline{\tilde{\Gamma}}_{+},\\
\omega_{+}=\omega_{-},\ \ p_{+}=p_{-},&\quad \mbox{on} \quad \underline{\tilde{\Gamma}}_{\rm cd},\\
\omega_{-}=0, &\quad \mbox{on} \quad \underline{\tilde{\Gamma}}_{-}.
\end{cases}
\end{eqnarray}

Then, we have the following proposition.
\begin{proposition}\label{prop-unique-5.2}
Under the same assumptions in Proposition \ref{prop-unique-5.1}, the solutions $(\omega_{\pm}, p_{\pm})$ of problem \eqref{gloabl-unique-problem-5.3} satisfy
\begin{eqnarray}\label{eq-global-unique-5.3}
\omega_{\pm}=0,\quad p_{\pm}=\underline{p}^{\pm}, \quad  \mbox{\rm{in}} \quad \underline{\tilde{\mathcal{N}}}_{\pm}.
\end{eqnarray}
\end{proposition}

\begin{proof}
We divide the proof into three steps.

$\mathbf{Step\ 1}$. Set
\begin{eqnarray}\label{eq-global-unique-5.4}
z^{\pm,\flat}_{+}\triangleq\delta \omega^{\flat}_{+}\pm\Lambda_{+}\delta p^{\flat}_{+}, \quad z^{\pm,\flat}_{-}\triangleq\delta \omega^{\flat}_{-}\pm\Lambda_{-}\delta p^{\flat}_{-},
\end{eqnarray}
where $\delta\omega^{\flat}_{\pm}\triangleq \omega_{\pm}-0$ and $\delta p^{\flat}_{\pm}\triangleq p_{\pm}-\underline{p}^{\pm}$.
Then, by \eqref{gloabl-unique-problem-5.3}, $z^{\pm,\flat}_{+}$ and $z^{\pm,\flat}_{-}$ satisfy the following initial-boundary value problem:
\begin{eqnarray}\label{gloabl-unique-problem-5.4}
\begin{cases}
\partial_{\xi}z^{\pm,\flat}_{+}+\underline{\lambda}^{\pm}_{+}\partial_{\eta}z^{+,\flat}_{+}\mp\frac{\partial^{\pm,\flat}_{+}\Lambda_{+}}{2\Lambda_{+}}\big(z^{+,\flat}_{+}-z^{-,\flat}_{+}\big)\\
\qquad\qquad =-\frac{1}{2}(\lambda^{\pm}_{+}-\underline{\lambda}^{\pm}_{+})\Big[(\partial_{\eta}z^{+,\flat}_{+}+\partial_{\eta}z^{-,\flat}_{+})\pm\Lambda_{+}\partial_{\eta}\Big(\frac{z^{+,\flat}_{+}-z^{-,\flat}_{+}}{\Lambda_{+}}\Big)\Big],&\quad  \mbox{in}
\quad \underline{\tilde{\mathcal{N}}}_{+}, \\
\partial_{\xi}z^{\pm,\flat}_{-}+\underline{\lambda}^{\pm}_{-}\partial_{\eta}z^{+,\flat}_{-}\mp\frac{\partial^{\pm,\flat}_{-}\Lambda_{-}}{2\Lambda_{-}}\big(z^{+,\flat}_{-}-z^{-,\flat}_{-}\big)\\
\qquad\qquad =-\frac{1}{2}(\lambda^{\pm}_{-}-\underline{\lambda}^{\pm}_{-})\Big[(\partial_{\eta}z^{+,\flat}_{-}+\partial_{\eta}z^{-,\flat}_{-})\pm\Lambda_{-}\partial_{\eta}\Big(\frac{z^{+,\flat}_{-}-z^{-,\flat}_{-}}{\Lambda_{-}}\Big)\Big],&\quad  \mbox{in}
\quad \underline{\tilde{\mathcal{N}}}_{-}, \\
(z^{+,\flat}_{+},z^{-,\flat}_{+})=(0,0),  &\quad  \mbox{on} \quad \underline{\tilde{\Gamma}}^{+}_{\rm in},\\
(z^{+,\flat}_{-},z^{-,\flat}_{-})=(0,0),  &\quad  \mbox{on} \quad \underline{\tilde{\Gamma}}^{-}_{\rm in},\\
z^{+,\flat}_{+}+z^{-,\flat}_{+}=0, &\quad \mbox{on} \quad \underline{\tilde{\Gamma}}_{+},\\
z^{+,\flat}_{+}+z^{-,\flat}_{+}=z^{+,\flat}_{-}+z^{-,\flat}_{-},\quad \frac{z^{+,\flat}_{+}+z^{-,\flat}_{+}}{2\Lambda_{+}}=\frac{z^{+,\flat}_{-}+z^{-,\flat}_{-}}{2\Lambda_{-}},&\quad \mbox{on}
\quad \underline{\tilde{\Gamma}}_{\rm cd},\\
z^{+,\flat}_{-}+z^{-,\flat}_{-}=0, &\quad \mbox{on} \quad \underline{\tilde{\Gamma}}_{-},
\end{cases}
\end{eqnarray}
where $\p^{\pm,\flat}_{+}\triangleq \p_{\xi}+\underline{\lambda}^{\pm}_{+}\p_{\eta}$, $\p^{\pm,\flat}_{-}\triangleq \p_{\xi}+\underline{\lambda}^{\pm}_{-}\p_{\eta}$, and 
\begin{eqnarray}\label{eq-global-unique-5.5}
\underline{\lambda}^{\pm}_{+}\triangleq\lambda^{\pm}_{+}\big|_{U=U_{+}}=\pm\frac{\underline{\rho}^{+}\underline{u}^{+}\underline{c}^{+}}
{\sqrt{(\underline{u}^{+})^{2}-(\underline{c}^{+})^{2}}},\quad
 \underline{\lambda}^{\pm}_{-}\triangleq\lambda^{\pm}_{-}\big|_{U=U_{-}}=\pm\frac{\underline{\rho}^{-}\underline{u}^{-}\underline{c}^{-}}
{\sqrt{(\underline{u}^{-})^{2}-(\underline{c}^{-})^{2}}}.
\end{eqnarray}

$\mathbf{Step\ 2}$. We are going to derive $C^{0}$-estimates of $z^{\pm,\flat}_{+}$ and $z^{\pm,\flat}_{-}$ in $\underline{\tilde{\mathcal{N}}}_{+}\cup\underline{\tilde{\mathcal{N}}}_{-}$.
In order to make it, we once again employ the characteristic method. As the argument done before introducing Proposition \ref{eq:prop-3.2-1} in section 3.2 , we divide the upper and lower nozzles $\underline{\tilde{\mathcal{N}}}_{\pm}$ into several sub-domains (see Fig.\ref{fig5.1}). But the difference here is that all the characteristic curves are determined by $\underline{\lambda}^{\pm}_{+}$ and $\underline{\lambda}^{\pm}_{-}$ so that they are all straight lines. To avoid the repetition in the expression, we do the same argument as done in the section 3.2 by replacing the notations
$\Upsilon^{+,0}_{+}(\xi; \bar{\xi}^{2}_{+}, \rm{m}_{+})$ (or $\eta=\Upsilon^{-,0}_{-}(\xi; \bar{\xi}^{2}_{-}, \rm{m}_{-})$), $\ell^{+,0}_{+}$ (or $\ell^{-,0}_{-}$) and $\tilde{\mathcal{N}}^{\rm{I}}_{\pm}$, by
$\underline{\Upsilon}^{+,0}_{+}(\xi; \bar{\xi}^{2}_{+}, \underline{\rm{m}}_{+})$ (or $\eta=\underline{\Upsilon}^{-,0}_{-}(\xi; \bar{\xi}^{2}_{-}, \underline{\rm{m}}_{-})$), $\underline{\ell}^{+,0}_{+}$ (or $\underline{\ell}^{-,0}_{-}$) and $\underline{\tilde{\mathcal{N}}}^{\rm{I}}_{\pm}$.

\begin{figure}[ht]
\begin{center}
\begin{tikzpicture}[scale=1.2]

\draw [line width=0.04cm](-2.5,-2.0) --(3.5,-2.0);
\draw [line width=0.04cm](-2.5,0.8)--(3.5,0.8);
\draw [line width=0.04cm][red][dashed](-2.5,-0.5)--(3.5,-0.5);

\draw [thin](3.5,-2.0) --(3.5,0.8);
\draw [line width=0.02cm](-2.5,-2.0)--(-2.5,0.8);

\draw [line width=0.02cm](-2.5,-0.5)to(-0.5,0.8);
\draw [line width=0.02cm][blue](-2.5,0.8)to(1.75,-2.0);
\draw [line width=0.02cm](-2.5,-0.5)to(-0.5,-2.0);
\draw [line width=0.02cm][blue](-2.5,-2.0)to(1.3,0.8);

\draw [line width=0.02cm](-0.5,0.8)to(3.4,-2.0);

\draw [line width=0.02cm](-0.5,-2.0)to(2.9,0.8);

\node at (-2.1, 0.1) {$\underline{\tilde{\mathcal{N}}}^{\rm{I}}_{+}$};
\node at (-1.5, 0.5) {$\underline{\tilde{\mathcal{N}}}^{\rm{II}}_{+}$};
\node at (-1.5, -0.25) {$\underline{\tilde{\mathcal{N}}}^{\rm{III}}_{+}$};
%

\node at (-2.1, -1.3) {$\underline{\tilde{\mathcal{N}}}^{\rm{I}}_{-}$};
\node at (-1.5, -1.7) {$\underline{\tilde{\mathcal{N}}}^{\rm{II}}_{-}$};
\node at (-1.4, -0.85) {$\underline{\tilde{\mathcal{N}}}^{\rm{III}}_{-}$};


\node at (3.9,0.8) {$\underline{\tilde{\Gamma}}_{+}$};
\node at (3.9,-2.0) {$\underline{\tilde{\Gamma}}_{-}$};
\node at (3.9,-0.5) {$\underline{\tilde{\Gamma}}_{\textrm{cd}}$};
\node at (-2.8, 0.2) {$\underline{\tilde{\Gamma}}^{+}_{\textrm{in}}$};
\node at (-2.8, -1.3) {$\underline{\tilde{\Gamma}}^{-}_{\textrm{in}}$};
\node at (-2.8,-0.5) {$O$};
\end{tikzpicture}
\end{center}
\caption{$C^{0}$-priori estimates for $z^{\pm,\flat}_{+}$ and $z^{\pm,\flat}_{-}$}\label{fig5.1}
\end{figure}

Then, we divide the proof of Step 2 into three substeps.

$\mathbf{Step\ 2.1}$. \underline{\emph{$C^0$-estimates for $z^{\pm,\flat}_{+}$ and $z^{\pm,\flat}_{-}$ in $\underline{\tilde{\mathcal{N}}}^{\rm{I}}_{+}\cup \underline{\tilde{\mathcal{N}}}^{\rm{I}}_{-}$.}} Without loss of the generality, we only consider the estimate of $z^{\pm,\flat}_{+}$ in $\underline{\tilde{\mathcal{N}}}^{\rm{I}}_{+}$ since the argument for $z^{\pm,\flat}_{-}$ in $\underline{\tilde{\mathcal{N}}}^{\rm{I}}_{-}$ can be done in the same way. For any $(\bar{\xi},\bar{\eta})\in \underline{\tilde{\mathcal{N}}}^{\rm{I}}_{+}$, let $\eta=\underline{\Upsilon}_{+}^{\pm}(\tau;\bar{\xi},\bar{\eta})$ be a characteristic line corresponding to $\underline{\lambda}^{\pm}_{+}$ issuing from the point $(\bar{\xi},\bar{\eta})$. Then along it, we have
\begin{eqnarray}\label{eq-global-unique-5.6}
\begin{cases}
z^{+,\flat}_{+}(\bar{\xi},\bar{\eta})
=\int_0^{\bar{\xi}}\frac{\partial^{+,\flat}_{+}\Lambda_{+}}{2\Lambda_{+}}\big(z^{+,\flat}_{+}-z^{-,\flat}_{+}\big)(\tau,\underline{\Upsilon}^{+}_{+}(\tau;\bar{\xi},\bar{\eta}))\mathrm{d} \tau\\
\qquad\qquad-\frac{1}{2}\int_0^{\bar{\xi}}(\lambda^{+}_{+}-\underline{\lambda}^{+}_{+})\Big[(\partial_{\eta}z^{+,\flat}_{+}+\partial_{\eta}z^{-,\flat}_{+})+\Lambda_{+}\partial_{\eta}\Big(\frac{z^{+,\flat}_{+}-z^{-,\flat}_{+}}{\Lambda_{+}}\Big)\Big](\tau,\underline{\Upsilon}^{+}_{+}(\tau;\bar{\xi},\bar{\eta}))\mathrm{d} \tau,\\
z^{-,\flat}_{+}(\bar{\xi},\bar{\eta})
=-\int_0^{\bar{\xi}}\frac{\partial^{-,\flat}_{+}\Lambda_{+}}{2\Lambda_{+}}\big(z^{+,\flat}_{+}-z^{-,\flat}_{+}\big)(\tau,\underline{\Upsilon}^{+}_{+}(\tau;\bar{\xi},\bar{\eta}))\mathrm{d} \tau\\
\qquad\qquad-\frac{1}{2}\int^{\bar{\xi}}_{0}(\lambda^{-}_{+}-\underline{\lambda}^{-}_{+})\Big[(\partial_{\eta}z^{+,\flat}_{+}+\partial_{\eta}z^{-,\flat}_{+})-\Lambda_{+}\partial_{\eta}\Big(\frac{z^{+,\flat}_{+}-z^{-,\flat}_{+}}{\Lambda_{+}}\Big)\Big](\tau,\underline{\Upsilon}^{-}_{+}(\tau;\bar{\xi},\bar{\eta}))\mathrm{d} \tau.
\end{cases}
\end{eqnarray}

Let us introduce
\begin{eqnarray*}
\underline{R}_{+}^{\rm{I}}(\kappa)=\big\{(\xi,\eta)| 0\leq\xi\leq\kappa,\ \underline{\Upsilon}^{+,0}_{+}(\xi;\bar{\xi}^{2}_{+},\textrm{\underline{m}}_{+})\leq\eta\leq\underline{\Upsilon}_{+}^{-,1}(\xi;\bar{\xi}^{2}_{\rm{cd}},0)\big\},
\ (0\leq\kappa\leq \bar{\xi}_{+}^{1}),
\end{eqnarray*}
and define the norm
\begin{eqnarray*}
z_{\rm{I}^{+}}^{\pm,\flat}(\kappa)\triangleq\sup\limits_{(\bar{\xi},\bar{\eta})\in \underline{R}_{+}^{\rm{I}}(\kappa)}|z_{+}^{\pm,\flat}(\bar{\xi},\bar{\eta})|.
\end{eqnarray*}

Obviously, we have $\|z_{+}^{\pm,\flat}\|_{0,0; \underline{\tilde{\mathcal{N}}}^{\rm{I}}_{+}}=z_{\rm{I}^{+}}^{\pm,\flat}(\bar{\xi}^{1}_{+}).$
So, by \eqref{eq-global-unique-5.6} and Lemma \ref{lem:unique}, we can get that
\begin{align}\label{eq-global-unique-5.7}
\begin{split}
&|z_{+}^{+,\flat}(\bar{\xi},\bar{\eta})|\\
&\leq \Big\|\frac{\partial^{+,\flat}_{+}\Lambda_{+}}{2\Lambda_{+}}\Big\|_{0,0; \underline{\tilde{\mathcal{N}}}_{+}}\int_0^{\bar{\xi}}\Big(|z^{+,\flat}_{+}|+|z^{-,\flat}_{+}|\Big)(\tau,\underline{\Upsilon}^{+}_{+}(\tau;\bar{\xi},\bar{\eta}))\mathrm{d} \tau\\
& +\frac{1}{2}\Big\|(\partial_{\eta}z^{+,\flat}_{+}+\partial_{\eta}z^{-,\flat}_{+})+\Lambda_{+}\partial_{\eta}\Big(\frac{z^{+,\flat}_{+}-z^{-,\flat}_{+}}{\Lambda_{+}}\Big)\Big\|_{0,0; \underline{\tilde{\mathcal{N}}}_{+}}\int_0^{\bar{\xi}}|\lambda^{+}_{+}-\underline{\lambda}^{+}_{+}|(\tau,\underline{\Upsilon}^{+}_{+}(\tau;\bar{\xi},\bar{\eta}))\mathrm{d} \tau\\
&\leq \mathcal{O}\big(\|U_{+}\|_{1,0; \underline{\tilde{\mathcal{N}}}_{+}}\big)\cdot
\int_0^{\bar{\xi}}(|z^{+,\flat}_{+}|+|z^{-,\flat}_{+}|)(\tau,\underline{\Upsilon}^{+}_{+}(\tau;\bar{\xi},\bar{\eta}))\mathrm{d} \tau.
\end{split}
\end{align}
where $\mathcal{O}\big(\|U_{+}\|_{1,0; \underline{\tilde{\mathcal{N}}}_{+}}\big)$ represents its bounds depending only on $\|U_{+}\|_{1,0; \underline{\tilde{\mathcal{N}}}_{+}}$.

Similarly, we can also get
\begin{eqnarray}\label{eq-global-unique-5.8}
|z_{+}^{-,\flat}(\bar{\xi},\bar{\eta})|\leq \mathcal{O}\big(\|U_{+}\|_{1,0; \underline{\tilde{\mathcal{N}}}_{+}}\big)\cdot \int_0^{\bar{\xi}}\big(|z^{+,\flat}_{+}|+|z^{-,\flat}_{+}|\big)(\tau,\underline{\Upsilon}^{+}_{+}(\tau;\bar{\xi},\bar{\eta}))\dd \tau.
\end{eqnarray}

Combining the \eqref{eq-global-unique-5.7}-\eqref{eq-global-unique-5.8}, for any $\kappa\in [0,\bar{\xi}^{1}_{+}]$, we arrive at
\begin{align*}
\begin{split}
&z_{\rm{I}^{+}}^{+,\flat}(\kappa)+z_{\rm{I}^{+}}^{-,\flat}(\kappa)\leq \mathcal{O}\big(\|U_{+}\|_{1,0; \underline{\tilde{\mathcal{N}}}_{+}}\big)\cdot \int_0^{\kappa}\big(z_{\rm{I}^{+}}^{+,\flat}+z_{\rm{I}^{+}}^{-,\flat}\big)(\varsigma)\dd \varsigma.
\end{split}
\end{align*}

Then, by the Gronwall inequality and taking $\kappa=\bar{\xi}^{1}_{+}$, it follows that
\begin{align}\label{eq-global-unique-5.9}
\begin{split}
\|z_{+}^{+,\flat}\|_{0,0; \underline{\tilde{\mathcal{N}}}^{\rm{I}}_{+}}=\|z_{+}^{-,\flat}\|_{0,0; \underline{\tilde{\mathcal{N}}}^{\rm{I}}_{+}}=0.
\end{split}
\end{align}

In the same way, we also have
\begin{align}\label{eq-global-unique-5.10}
\begin{split}
\|z_{-}^{+,\flat}\|_{0,0; \underline{\tilde{\mathcal{N}}}^{\rm{I}}_{-}}=\|z_{-}^{-,\flat}\|_{0,0; \underline{\tilde{\mathcal{N}}}^{\rm{I}}_{-}}=0.
\end{split}
\end{align}

$\mathbf{Step\ 2.2}$. \underline{\emph{$C^0$-estimates for $z^{\pm,\flat}_{+}$ and $z^{\pm,\flat}_{-}$ in $\underline{\tilde{\mathcal{N}}}^{\rm{II}}_{+}\cup\underline{\tilde{\mathcal{N}}}^{\rm{II}}_{-}$.}}
For any $(\bar{\xi},\bar{\eta})\in \underline{\tilde{\mathcal{N}}}_{+}\cap\underline{\tilde{\mathcal{N}}}^{\rm{II}}_{+}$, we can define a backward characteristic line $\eta=\underline{\Upsilon}^{-, \rm{b}}_{+}(\xi;\bar{\xi},\bar{\eta})$ corresponding to $\underline{\lambda}^{-}_{+}$, issuing from the point $(\bar{\xi},\bar{\eta})$ and hitting
the upper nozzle wall $\underline{\tilde{\Gamma}}_{+}\cap \overline{\underline{\tilde{\mathcal{N}}^{\rm{II}}_{+}}}$ at the point $(\underline{\zeta}^{-,\rm{b}}_{+}(\bar{\xi},\bar{\eta}),\underline{\textrm{m}}_{+})$.
Then, along it, we have by equation $\eqref{gloabl-unique-problem-5.4}_1$ and the conditions $\eqref{gloabl-unique-problem-5.4}_3$ and $\eqref{gloabl-unique-problem-5.4}_5$ that
\begin{eqnarray*}
\begin{split}
&z^{-,\flat}_{+}(\bar{\xi},\bar{\eta})\\
&=-z^{+,\flat}_{+}(\underline{\zeta}^{-,\rm{b}}_{+}(\bar{\xi},\bar{\eta}),\underline{\textrm{m}}_{+})
-\int_{\underline{\zeta}^{-,\rm{b}}_{+}(\bar{\xi},\bar{\eta})}^{\bar{\xi}}\frac{\partial^{-,\flat}_{+}\Lambda_{+}}{2\Lambda_{+}}\big(z^{+,\flat}_{+}-z^{-,\flat}_{+}\big)
(\tau,\underline{\Upsilon}^{-,\rm{b}}_{+}(\tau;\bar{\xi},\bar{\eta}))\mathrm{d} \tau\\
&\quad -\frac{1}{2}\int_{\underline{\zeta}^{-,\rm{b}}_{+}(\bar{\xi},\bar{\eta})}^{\bar{\xi}}(\lambda^{-}_{+}-\underline{\lambda}^{-}_{+})\Big[(\partial_{\eta}z^{+,\flat}_{+}+\partial_{\eta}z^{-,\flat}_{+})-\Lambda_{+}\partial_{\eta}\Big(\frac{z^{+,\flat}_{+}-z^{-,\flat}_{+}}{\Lambda_{+}}\Big)\Big](\tau,\underline{\Upsilon}^{-, \rm{b}}_{+}(\tau;\bar{\xi},\bar{\eta}))\mathrm{d} \tau.
\end{split}
\end{eqnarray*}

Along the characteristic line $\eta=\underline{\Upsilon}^{+}_{+}(\xi;\underline{\zeta}^{-,\rm{b}}_{+}(\bar{\xi},\bar{\eta}),\underline{\textrm{m}}_{+})$, we have 
\begin{align*}
\begin{split}
&z_{+}^{+,\flat}(\underline{\zeta}^{-,\rm{b}}_{+}(\bar{\xi},\bar{\eta}),\underline{\textrm{m}}_{+})\\
& =\int_0^{\underline{\zeta}^{-,\rm{b}}_{+}(\bar{\xi},\bar{\eta})}\frac{\partial^{+,\flat}_{+}\Lambda_{+}}{2\Lambda_{+}}\big(z^{+,\flat}_{+}-z^{-,\flat}_{+}\big)(\tau,\underline{\Upsilon}^{+}_{+}(\tau;\underline{\zeta}^{-,\rm{b}}_{+}(\bar{\xi},\bar{\eta}),\underline{\textrm{m}}_{+}))\dd \tau\\
&\quad-\frac{1}{2}\int^{\underline{\zeta}^{-,\rm{b}}_{+}(\bar{\xi},\bar{\eta})}_{0}(\lambda^{+}_{+}-\underline{\lambda}^{+}_{+})\Big[(\partial_{\eta}z^{+,\flat}_{+}+\partial_{\eta}z^{-,\flat}_{+})+\Lambda_{+}\partial_{\eta}\Big(\frac{z^{+,\flat}_{+}-z^{-,\flat}_{+}}{\Lambda_{+}}\Big)\Big](\tau,\underline{\Upsilon}^{+}_{+}(\tau;\underline{\zeta}^{-,\rm{b}}_{+}(\bar{\xi},\bar{\eta}),\underline{\textrm{m}}_{+}))\dd \tau.
\end{split}
\end{align*}

Then, by Lemma \ref{lem:unique}, we obtain
\begin{align}\label{eq-global-unique-5.11}
\begin{split}
|z^{-,\flat}_{+}(\bar{\xi},\bar{\eta})|&\leq |z^{+,\flat}_{+}(\underline{\zeta}^{-,\rm{b}}_{+}(\bar{\xi},\bar{\eta}),\underline{\textrm{m}}_{+})|
\\
&\quad +\Big\|\frac{\partial^{-,\flat}_{+}\Lambda_{+}}{2\Lambda_{+}}\Big\|_{0,0; \underline{\tilde{\mathcal{N}}}_{+}}\cdot\int_{\underline{\zeta}^{-,\rm{b}}_{+}(\bar{\xi},\bar{\eta})}^{\bar{\xi}}\big(|z^{+,\flat}_{+}|+|z^{-,\flat}_{+}|\big)
(\tau,\underline{\Upsilon}^{-,\rm{b}}_{+}(\tau;\bar{\xi},\bar{\eta}))\mathrm{d} \tau\\
&\quad\
+\mathcal{O}\big(\|U_{+}\|_{1,0; \underline{\tilde{\mathcal{N}}}_{+}}\big)\cdot\int_{\underline{\zeta}^{-,\rm{b}}_{+}(\bar{\xi},\bar{\eta})}^{\bar{\xi}}|\lambda^{-}_{+}-\underline{\lambda}^{-}_{+}|(\tau,\underline{\Upsilon}^{-, \rm{b}}_{+}(\tau;\bar{\xi},\bar{\eta}))\mathrm{d} \tau\\
&\leq |z^{+,\flat}_{+}(\underline{\zeta}^{-,\rm{b}}_{+}(\bar{\xi},\bar{\eta}),\underline{\textrm{m}}_{+})|\\
&\quad +\mathcal{O}\big(\|U_{+}\|_{1,0; \underline{\tilde{\mathcal{N}}}_{+}}\big)
\cdot \int_{\underline{\zeta}^{-,\rm{b}}_{+}(\bar{\xi},\bar{\eta})}^{\bar{\xi}}\big(|z^{+,\flat}_{+}|+|z^{-,\flat}_{+}|\big)
(\tau,\underline{\Upsilon}^{-,\rm{b}}_{+}(\tau;\bar{\xi},\bar{\eta}))\mathrm{d} \tau.
\end{split}
\end{align}

For the term $|z^{+,\flat}_{+}(\underline{\zeta}^{-,\rm{b}}_{+}(\bar{\xi},\bar{\eta}),\underline{\textrm{m}}_{+})|$, we have
\begin{align}\label{eq-global-unique-5.12}
\begin{split}
&|z_{+}^{+,\flat}(\underline{\zeta}^{-,\rm{b}}_{+}(\bar{\xi},\bar{\eta}),\underline{\textrm{m}}_{+})|\\
& \leq\Big\|\frac{\partial^{+,\flat}_{+}\Lambda_{+}}{2\Lambda_{+}}\Big\|_{0,0; \underline{\tilde{\mathcal{N}}}_{+}}\int_0^{\underline{\zeta}^{-,\rm{b}}_{+}(\bar{\xi},\bar{\eta})}\big(|z^{+,\flat}_{+}|+|z^{-,\flat}_{+}|\big)(\tau,\underline{\Upsilon}^{+}_{+}(\tau;\zeta^{-,\rm{b}}_{+}(\bar{\xi},\bar{\eta}),\underline{\textrm{m}}_{+}))\dd \tau\\
&\quad+\frac{1}{2}\Big\|(\partial_{\eta}z^{+,\flat}_{+}+\partial_{\eta}z^{-,\flat}_{+})+\Lambda_{+}\partial_{\eta}\Big(\frac{z^{+,\flat}_{+}-z^{-,\flat}_{+}}{\Lambda_{+}}\Big)\Big\|_{0,0; \underline{\tilde{\mathcal{N}}}_{+}}\\
&\qquad\quad \times\int^{\underline{\zeta}^{-,\rm{b}}_{+}(\bar{\xi},\bar{\eta})}_{0}|\lambda^{+}_{+}-\underline{\lambda}^{+}_{+}|(\tau,\underline{\Upsilon}^{+}_{+}(\tau;\underline{\zeta}^{-,\rm{b}}_{+}(\bar{\xi},\bar{\eta}),\underline{\textrm{m}}_{+}))\dd \tau\\
& \leq\mathcal{O}\big(\|U_{+}\|_{1,0; \underline{\tilde{\mathcal{N}}}_{+}}\big)\cdot\int_0^{\underline{\zeta}^{-,\rm{b}}_{+}(\bar{\xi},\bar{\eta})}\big(|z^{+,\flat}_{+}|+|z^{-,\flat}_{+}|\big)(\tau,\underline{\Upsilon}^{+}_{+}(\tau;\underline{\zeta}^{-,\rm{b}}_{+}(\bar{\xi},\bar{\eta}),\underline{\textrm{m}}_{+}))\dd \tau.
\end{split}
\end{align}
Define
\begin{eqnarray*}
\underline{R}^{\rm{II}}_{+}(\kappa)=\{(\xi,\eta)| 0\leq\xi\leq\kappa,\ \underline{\Upsilon}^{+,0}_{+}(\xi;\bar{\xi}_{+}^{2},\underline{\textrm{m}}_{+})\leq\eta\leq\underline{\textrm{m}}_{+}\},\  \kappa\in [0,\bar{\xi}^{2}_{+}],
\end{eqnarray*}
and
\begin{eqnarray*}
z^{\pm,\flat}_{\rm{II}^{+}}(\kappa)\triangleq\sup\limits_{(\bar{\xi},\bar{\eta})\in \underline{R}^{\rm{II}}_{+}(\kappa)}|z^{\pm,\flat}_{+}(\bar{\xi},\bar{\eta})|.
\end{eqnarray*}
Obviously, we have
\begin{eqnarray*}
\|z_{+}^{\pm,\flat}\|_{0,0; \underline{\tilde{\mathcal{N}}}^{\rm{I}}_{+}\cup\underline{\tilde{\mathcal{N}}}^{\rm{II}}_{+}}=z_{\rm{II}^{+}}^{\pm,\flat}(\bar{\xi}^{2}_{+})\quad\mbox{and}\quad z^{\pm,\flat}_{\rm{II}^{+}}(\kappa)\geq z^{\pm,\flat}_{\rm{I}^{+}}(\kappa),\quad \kappa\in [0,\bar{\xi}^{1}_{+}].
\end{eqnarray*}

So, for any $\kappa\in [0,\bar{\xi}^{2}_{+}]$, one can get
\begin{align}\label{eq-global-unique-5.13}
\begin{split}
&z_{\rm{II}^{+}}^{-,\flat}(\kappa)\leq \mathcal{O}\big(\|U_{+}\|_{1,0; \underline{\tilde{\mathcal{N}}}_{+}}\big)\cdot
\int_0^{\kappa}\big(z_{\rm{II}^{+}}^{+,\flat}+z_{\rm{II}^{+}}^{-,\flat}\big)(\varsigma)\mathrm{d} \varsigma.
\end{split}
\end{align}
Following the argument in Step 2.1, we can also deduce that
\begin{align}\label{eq-global-unique-5.14}
\begin{split}
&z_{\rm{II}^{+}}^{+,\flat}(\kappa)\leq \mathcal{O}\big(\|U_{+}\|_{1,0; \underline{\tilde{\mathcal{N}}}_{+}}\big)\cdot
\int_0^{\kappa}\big(z_{\rm{II}^{+}}^{+,\flat}+z_{\rm{II}^{+}}^{-,\flat}\big)(\varsigma)\mathrm{d} \varsigma.
\end{split}
\end{align}
Adding \eqref{eq-global-unique-5.13} into \eqref{eq-global-unique-5.14}, and applying the Gronwall's inequality, we have
\begin{eqnarray*}
z_{\rm{II}^{+}}^{+,\flat}(\kappa)=z_{\rm{II}^{+}}^{+,\flat}(\kappa)\equiv 0,
\end{eqnarray*}
which implies by taking $\kappa=\bar{\xi}^{2}_{+}$ that
\begin{align}\label{eq-global-unique-5.15}
\begin{split}
\|z_{+}^{+,\flat}\|_{0,0; \underline{\tilde{\mathcal{N}}}^{\rm{I}}_{+}\cup\underline{\tilde{\mathcal{N}}}^{\rm{II}}_{+}}=\|z_{+}^{-,\flat}\|_{0,0; \underline{\tilde{\mathcal{N}}}^{\rm{I}}_{+}\cup\underline{\tilde{\mathcal{N}}}^{\rm{II}}_{+}}=0.
\end{split}
\end{align}

Similarly, we can also get
\begin{align}\label{eq-global-unique-5.16}
\begin{split}
\|z_{+}^{-,\flat}\|_{0,0; \underline{\tilde{\mathcal{N}}}^{\rm{I}}_{-}\cup\underline{\tilde{\mathcal{N}}}^{\rm{II}}_{-}}=\|z_{-}^{-,\flat}\|_{0,0; \underline{\tilde{\mathcal{N}}}^{\rm{I}}_{-}\cup\underline{\tilde{\mathcal{N}}}^{\rm{II}}_{-}}=0.
\end{split}
\end{align}

\emph{$\mathbf{Step\ 2.3}$. \underline{$C^0$-estimates for $z^{\pm,\flat}_{+}$ and $z^{\pm,\flat}_{-}$ in $\underline{\tilde{\mathcal{N}}}^{\rm{III}}_{+}\cup\underline{\tilde{\mathcal{N}}}^{\rm{III}}_{-}$.}}
For any $(\bar{\xi}, \bar{\eta})\in \underline{\tilde{\mathcal{N}}}^{\rm{III}}_{+}$, we still denote by $\eta=\underline{\Upsilon}^{+}_{+}(\xi;\bar{\xi},\bar{\eta})$ the backward characteristic line corresponding to $\underline{\lambda}^{+}_{+}$ 
and passing through the point $(\bar{\xi},\bar{\eta})$, and
denote by $\eta=\underline{\Upsilon}^{+,\rm{cd}}_{+}(\xi;\bar{\xi}, \bar{\eta})$ the backward characteristic curve corresponding to $\underline{\lambda}^{+}_{+}$,
passing through $(\bar{\xi}, \bar{\eta})$ and intersecting with $\underline{\tilde{\Gamma}}_{\rm{cd}}\cap \overline{\underline{\tilde{\mathcal{N}}}^{\rm{III}}_{+}}$ at $(\underline{\zeta}^{+,\rm{cd}}_{+}(\bar{\xi},\bar{\eta}),0)$.
Similarly, for any $(\bar{\xi}, \bar{\eta})\in \underline{\tilde{\mathcal{N}}}^{\rm{III}}_{-}$, let $\eta=\underline{\Upsilon}^{+}_{-}(\xi;\bar{\xi},\bar{\eta})$ be the backward characteristic curve corresponding to $\underline{\lambda}^{+}_{-}$
and passing through the point $(\bar{\xi},\bar{\eta})$, and let $\eta=\underline{\Upsilon}^{-,\rm{cd}}_{-}(\xi;\bar{\xi}, \bar{\eta})$ be the backward characteristic curve corresponding to $\underline{\lambda}^{-}_{-}$,
passing through $(\bar{\xi}, \bar{\eta})$ and intersecting with $\underline{\tilde{\Gamma}}_{\rm{cd}}\cap \overline{\underline{\tilde{\mathcal{N}}}^{\rm{III}}_{+}}$ at $(\underline{\zeta}^{-,\rm{cd}}_{-}(\bar{\xi},\bar{\eta}),0)$.
Then, by $\eqref{gloabl-unique-problem-5.4}_1-\eqref{gloabl-unique-problem-5.4}_2$ and $\eqref{gloabl-unique-problem-5.4}_3-\eqref{gloabl-unique-problem-5.4}_4$, we have
\begin{eqnarray}\label{eq-global-unique-5.17}
\begin{cases}
z^{+,\flat}_{+}(\bar{\xi},\bar{\eta})=z^{+,\flat}_{+}(\underline{\zeta}^{+,\rm{cd}}_{+}(\bar{\xi},\bar{\eta}),0)+\int_{\underline{\zeta}^{+,\rm{cd}}_{+}(\bar{\xi},\bar{\eta})}^{\bar{\xi}}
\frac{\partial^{+,\flat}_{+}\Lambda_{+}}{2\Lambda_{+}}\big(z^{+,\flat}_{+}-z^{-,\flat}_{+}\big)(\tau,\underline{\Upsilon}^{+}_{+}(\tau;\bar{\xi},\bar{\eta}))\mathrm{d} \tau\\
\qquad\qquad-\frac{1}{2}\int_{\underline{\zeta}^{+,\rm{cd}}_{+}(\bar{\xi},\bar{\eta})}^{\bar{\xi}}(\lambda^{+}_{+}-\underline{\lambda}^{+}_{+})\Big[(\partial_{\eta}z^{+,\flat}_{+}+\partial_{\eta}z^{-,\flat}_{+})
+\Lambda_{+}\partial_{\eta}\Big(\frac{z^{+,\flat}_{+}-z^{-,\flat}_{+}}{\Lambda_{+}}\Big)\Big](\tau,\underline{\Upsilon}^{+}_{+}(\tau;\bar{\xi},\bar{\eta}))\mathrm{d} \tau,\\
z^{-,\flat}_{+}(\bar{\xi},\bar{\eta})
=-\int_0^{\bar{\xi}}\frac{\partial^{-,\flat}_{+}\Lambda_{+}}{2\Lambda_{+}}\big(z^{+,\flat}_{+}-z^{-,\flat}_{+}\big)(\tau,\underline{\Upsilon}^{-}_{+}(\tau;\bar{\xi},\bar{\eta}))\mathrm{d} \tau\\
\qquad\qquad-\frac{1}{2}\int^{\bar{\xi}}_{0}(\lambda^{-}_{+}-\underline{\lambda}^{-}_{+})\Big[(\partial_{\eta}z^{+,\flat}_{+}+\partial_{\eta}z^{-,\flat}_{+})-\Lambda_{+}\partial_{\eta}\Big(\frac{z^{+,\flat}_{+}-z^{-,\flat}_{+}}{\Lambda_{+}}\Big)\Big](\tau,\underline{\Upsilon}^{-}_{+}(\tau;\bar{\xi},\bar{\eta}))\mathrm{d} \tau,
\end{cases}
\end{eqnarray}
and
\begin{eqnarray}\label{eq-global-unique-5.18}
\begin{cases}
z^{+,\flat}_{-}(\bar{\xi},\bar{\eta})=\int_0^{\bar{\xi}}\frac{\partial^{+,\flat}_{-}\Lambda_{-}}{2\Lambda_{-}}\big(z^{+,\flat}_{-}-z^{-,\flat}_{-}\big)(\tau,\underline{\Upsilon}^{+}_{-}(\tau;\bar{\xi},\bar{\eta}))\mathrm{d} \tau\\
\qquad\qquad-\frac{1}{2}\int_0^{\bar{\xi}}(\lambda^{+}_{-}-\underline{\lambda}^{+}_{-})\Big[(\partial_{\eta}z^{+,\flat}_{-}+\partial_{\eta}z^{-,\flat}_{-})+\Lambda_{-}\partial_{\eta}
\Big(\frac{z^{+,\flat}_{-}-z^{-,\flat}_{-}}{\Lambda_{-}}\Big)\Big](\tau,\underline{\Upsilon}^{+}_{-}(\tau;\bar{\xi},\bar{\eta}))\mathrm{d} \tau,\\
z^{-,\flat}_{-}(\bar{\xi},\bar{\eta})
=z^{-,\flat}_{-}(\underline{\zeta}^{-,\rm{cd}}_{-}(\bar{\xi},\bar{\eta}),0)-\int_{\underline{\zeta}^{-,\rm{cd}}_{-}(\bar{\xi},\bar{\eta})}^{\bar{\xi}}\frac{\partial^{-,\flat}_{-}\Lambda_{-}}{2\Lambda_{-}}
\big(z^{+,\flat}_{-}-z^{-,\flat}_{-}\big)(\tau,\underline{\Upsilon}^{-}_{-}(\tau;\bar{\xi},\bar{\eta}))\mathrm{d} \tau\\
\qquad\qquad-\frac{1}{2}\int^{\bar{\xi}}_{\underline{\zeta}^{-,\rm{cd}}_{-}(\bar{\xi},\bar{\eta})}(\lambda^{-}_{-}-\underline{\lambda}^{-}_{-})
\Big[(\partial_{\eta}z^{+,\flat}_{-}+\partial_{\eta}z^{-,\flat}_{-})-\Lambda_{-}\partial_{\eta}\Big(\frac{z^{+,\flat}_{-}-z^{-,\flat}_{-}}{\Lambda_{-}}\Big)\Big]
(\tau,\underline{\Upsilon}^{-}_{-}(\tau;\bar{\xi},\bar{\eta}))\mathrm{d} \tau.
\end{cases}
\end{eqnarray}

To obtain the estimates of $z^{\pm,\flat}_{+}$ and $z^{\pm,\flat}_{-}$, we introduce
\begin{eqnarray*}
\underline{R}^{\rm{III}}_{+}(\kappa)=\{(\xi,\eta)| 0\leq\xi\leq\kappa,\ 0\leq\eta\leq\underline{\Upsilon}^{-,1}_{+}(\xi; \bar{\xi}^{2}_{\rm{cd}},0)\}, \ \kappa\in[0,\bar{\xi}^{2}_{\rm{cd}}],
\end{eqnarray*}
and
\begin{eqnarray*}
\underline{R}^{\rm{III}}_{-}(\kappa)=\{(\xi,\eta)|0\leq\xi\leq\kappa,\ \underline{\Upsilon}^{+,1}_{-}(\xi;\underline{\xi}^{2}_{\rm{cd}},0)\leq\eta\leq0\},\ \kappa\in[0, \xi^{2}_{\rm{cd}}].
\end{eqnarray*}

Denote
\begin{eqnarray*}
z^{\pm,\flat}_{\rm{III}^{+}}(\kappa)\triangleq\sup\limits_{(\xi,\eta)\in \underline{R}^{\rm{III}}_{+}(\kappa)}|z^{\pm,\flat}_{+}(\xi,\eta)|,
\quad z^{\pm,\flat}_{\rm{III}^{-}}(\kappa)\triangleq\sup\limits_{(\xi,\eta)\in \underline{R}^{\rm{III}}_{-}(\kappa)}|z^{\pm,\flat}_{-}(\xi,\eta)|.
\end{eqnarray*}

Obviously, we have
\begin{eqnarray*}
\|z_{+}^{\pm,\flat}\|_{0,0; \underline{\tilde{\mathcal{N}}}^{\rm{I}}_{+}\cup\underline{\tilde{\mathcal{N}}}^{\rm{III}}_{+}}=z_{\rm{III}^{+}}^{\pm,\flat}(\bar{\xi}^{2}_{\rm{cd}}),\quad
\|z_{-}^{\pm,\flat}\|_{0,0; \underline{\tilde{\mathcal{N}}}^{\rm{I}}_{-}\cup\underline{\tilde{\mathcal{N}}}^{\rm{III}}_{-}}=z_{\rm{III}^{-}}^{\pm,\flat}(\bar{\xi}^{2}_{\rm{cd}}).
\end{eqnarray*}

Then
\begin{eqnarray}\label{eq-global-unique-5.19}
\begin{cases}
|z^{+,\flat}_{+}(\bar{\xi},\bar{\eta})|\leq |z^{+,\flat}_{+}(\underline{\zeta}^{+,\rm{cd}}_{+}(\bar{\xi},\bar{\eta}),0)|+\mathcal{O}\big(\|U_{+}\|_{1,0; \underline{\tilde{\mathcal{N}}}_{+}}\big)\cdot\int_{\underline{\zeta}^{+,\rm{cd}}_{+}(\bar{\xi},\bar{\eta})}^{\bar{\xi}}\big(z^{+,\flat}_{\rm{III}^{+}}+z^{-,\flat}_{\rm{III}^{+}}\big)(\varsigma)\dd\varsigma,\\
|z^{-,\flat}_{+}(\bar{\xi},\bar{\eta})|
\leq\mathcal{O}\big(\|U_{+}\|_{1,0; \underline{\tilde{\mathcal{N}}}_{+}}\big)\cdot\int_0^{\bar{\xi}}\big(z^{+,\flat}_{\rm{III}^{+}}+z^{-,\flat}_{\rm{III}^{+}}\big)(\varsigma)\dd\varsigma.
\end{cases}
\end{eqnarray}

Thus, in order to get the estimate on $|z^{+,\flat}_{+}(\bar{\xi},\bar{\eta})|$, one needs to consider the term $z^{+,\flat}_{+}(\underline{\zeta}^{+,\rm{cd}}_{+}(\bar{\xi},\bar{\eta}),0)$ in \eqref{eq-global-unique-5.17}.
To this end, by the boundary condition $\eqref{gloabl-unique-problem-5.4}_{6}$, we know that
\begin{eqnarray*}
\begin{split}
z^{+,\flat}_{+}(\underline{\zeta}^{+,\rm{cd}}_{+}(\bar{\xi},\bar{\eta}),0)=\Big(\frac{\Lambda_{-}-\Lambda_{+}}{\Lambda_{+}+\Lambda_{-}}z^{-,\flat}_{+}\Big)(\underline{\zeta}^{+,\rm{cd}}_{+}(\bar{\xi},\bar{\eta}),0)
+\Big(\frac{2\Lambda_{+}}{\Lambda_{+}+\Lambda_{-}}z^{+,\flat}_{-}\Big)(\underline{\zeta}^{+,\rm{cd}}_{+}(\bar{\xi},\bar{\eta}),0).
\end{split}
\end{eqnarray*}

For the terms $z^{-,\flat}_{+}(\underline{\zeta}^{+,\rm{cd}}_{+}(\bar{\xi},\bar{\eta}),0)$ and $z^{+,\flat}_{-}(\underline{\zeta}^{+,\rm{cd}}_{+}(\bar{\xi},\bar{\eta}),0)$,
we can get by \eqref{eq-global-unique-5.17} and \eqref{eq-global-unique-5.18} that
\begin{eqnarray*}
\begin{split}
&z^{-,\flat}_{+}(\underline{\zeta}^{+,\rm{cd}}_{+}(\bar{\xi},\bar{\eta}),0)\\
&=-\int_{0}^{\underline{\zeta}^{+,\rm{cd}}_{+}(\bar{\xi},\bar{\eta})}\frac{\partial^{-,\flat}_{+}\Lambda_{+}}{2\Lambda_{+}}\big(z^{+,\flat}_{+}-z^{-,\flat}_{+}\big)
(\tau,\underline{\Upsilon}^{-}_{+}(\tau;\underline{\zeta}^{+,\rm{cd}}_{+}(\bar{\xi},\bar{\eta}),0))\mathrm{d} \tau\\
&\quad\ -\frac{1}{2}\int_{0}^{\underline{\zeta}^{+,\rm{cd}}_{+}(\bar{\xi},\bar{\eta})}(\lambda^{-}_{+}-\underline{\lambda}^{-}_{+})\Big[(\partial_{\eta}z^{+,\flat}_{+}
+\partial_{\eta}z^{-,\flat}_{+})-\Lambda_{+}\partial_{\eta}\Big(\frac{z^{+,\flat}_{+}-z^{-,\flat}_{+}}{\Lambda_{+}}\Big)\Big](\tau,\underline{\Upsilon}^{-}_{+}(\tau;\underline{\zeta}^{+,\rm{cd}}_{+}(\bar{\xi},\bar{\eta}),0))\mathrm{d} \tau,
\end{split}
\end{eqnarray*}
and
\begin{eqnarray*}
\begin{split}
&z^{+,\flat}_{-}(\underline{\zeta}^{+,\rm{cd}}_{+}(\bar{\xi},\bar{\eta}),0)\\
&=\int_0^{\underline{\zeta}^{+,\rm{cd}}_{+}(\bar{\xi},\bar{\eta})}\frac{\partial^{+,\flat}_{-}\Lambda_{-}}{2\Lambda_{-}}
\big(z^{+,\flat}_{-}-z^{-,\flat}_{-}\big)(\tau,\underline{\Upsilon}^{+}_{-}(\tau;\underline{\zeta}^{+,\rm{cd}}_{+}(\bar{\xi},\bar{\eta}),0))\mathrm{d} \tau\\
&\quad\ -\frac{1}{2}\int_0^{\underline{\zeta}^{+,\rm{cd}}_{+}(\bar{\xi},\bar{\eta})}(\lambda^{+}_{-}-\underline{\lambda}^{+}_{-})
\Big[(\partial_{\eta}z^{+,\flat}_{-}+\partial_{\eta}z^{-,\flat}_{-})+\Lambda_{-}\partial_{\eta}\Big(\frac{z^{+,\flat}_{-}-z^{-,\flat}_{-}}{\Lambda_{-}}\Big)\Big]
(\tau,\underline{\Upsilon}^{+}_{-}(\tau;\underline{\zeta}^{+,\rm{cd}}_{+}(\bar{\xi},\bar{\eta}),0))\mathrm{d} \tau.
\end{split}
\end{eqnarray*}

Therefore, we have
\begin{align}\label{eq-global-unique-5.20}
\begin{split}
&|z^{+,\flat}_{+}(\underline{\zeta}^{+,\rm{cd}}_{+}(\bar{\xi},\bar{\eta}),0)|\\
&\leq\Big\|\frac{\Lambda_{-}-\Lambda_{+}}{\Lambda_{+}+\Lambda_{-}}\Big\|_{0,0; \underline{\tilde{\Gamma}}_{\rm{cd}}}|z^{-,\flat}_{+}(\underline{\zeta}^{+,\rm{cd}}_{+}(\bar{\xi},\bar{\eta}),0)|
+\Big\|\frac{2\Lambda_{+}}{\Lambda_{+}+\Lambda_{-}}\Big\|_{0,0; \underline{\tilde{\Gamma}}_{\rm{cd}}}|z^{+,\flat}_{-}(\underline{\zeta}^{+,\rm{cd}}_{+}(\bar{\xi},\bar{\eta}),0)|\\
&\leq \mathcal{O}\big(\|U_{\pm}\|_{1,0; \underline{\tilde{\mathcal{N}}}_{\pm}}\big)\cdot\int_{0}^{\underline{\zeta}^{+,\rm{cd}}_{+}(\bar{\xi},\bar{\eta})}
\big(z^{+,\flat}_{\rm{III}^{+}}+z^{-,\flat}_{\rm{III}^{+}}+z^{+,\flat}_{\rm{III}^{-}}+z^{-,\flat}_{\rm{III}^{-}}\big)(\varsigma)\dd\varsigma
\end{split}
\end{align}

Substituting \eqref{eq-global-unique-5.20} into \eqref{eq-global-unique-5.19}, one can derive
\begin{eqnarray}\label{eq-global-unique-5.21}
\begin{split}
z^{+,\flat}_{\rm{III}^{+}}(\kappa)+z^{-,\flat}_{\rm{III}^{+}}(\kappa)
\leq \mathcal{O}\big(\|U_{\pm}\|_{1,0; \underline{\tilde{\mathcal{N}}}_{\pm}}\big)\cdot\int_{0}^{\kappa}\big(z^{+,\flat}_{\rm{III}^{+}}+z^{-,\flat}_{\rm{III}^{+}}
+z^{+,\flat}_{\rm{III}^{-}}+z^{-,\flat}_{\rm{III}^{-}}\big)(\varsigma)\dd\varsigma.
\end{split}
\end{eqnarray}

In the same way, we can also obtain
\begin{eqnarray}\label{eq-global-unique-5.22}
\begin{split}
z^{+,\flat}_{\rm{III}^{-}}(\kappa)+z^{-,\flat}_{\rm{III}^{-}}(\kappa)
\leq \mathcal{O}\big(\|U_{\pm}\|_{1,0; \underline{\tilde{\mathcal{N}}}_{\pm}}\big)\cdot\int_{0}^{\kappa}\big(z^{+,\flat}_{\rm{III}^{+}}+z^{-,\flat}_{\rm{III}^{+}}+z^{+,\flat}_{\rm{III}^{-}}+z^{-,\flat}_{\rm{III}^{-}}\big)(\varsigma)\dd\varsigma.
\end{split}
\end{eqnarray}

Combining \eqref{eq-global-unique-5.21} with \eqref{eq-global-unique-5.22}, and applying the Gronwall's inequality for $\kappa=\bar{\xi}^{2}_{\rm{cd}}$, we have
\begin{eqnarray}\label{eq-global-unique-5.23}
\begin{split}
\|z_{+}^{\pm,\flat}\|_{0,0; \underline{\tilde{\mathcal{N}}}^{\rm{I}}_{+}\cup\underline{\tilde{\mathcal{N}}}^{\rm{III}}_{+}}=\|z_{-}^{\pm,\flat}\|_{0,0; \underline{\tilde{\mathcal{N}}}^{\rm{I}}_{-}\cup\underline{\tilde{\mathcal{N}}}^{\rm{III}}_{-}}= 0.
\end{split}
\end{eqnarray}

Finally, let $\underline{\xi}^{*}_1=\min\{\bar{\xi}^{1}_{+}, \bar{\xi}^{1}_{-}\}$. Then by \eqref{eq-global-unique-5.9}-\eqref{eq-global-unique-5.10} in Step 2.1,
\eqref{eq-global-unique-5.15}-\eqref{eq-global-unique-5.16} in Step 2.2 and \eqref{eq-global-unique-5.23} in Step 2.3, we have $z_{+}^{\pm,\flat}=z_{+}^{\pm,\flat}\equiv0$ in $\underline{\tilde{\mathcal{N}}}\cap\{0<\xi\leq\underline{\xi}^{*}_1\}$. Repeating this process sequentially in $\underline{\tilde{\mathcal{N}}}^{j}\triangleq\underline{\tilde{\mathcal{N}}}\cap\{(j-1)\underline{\xi}^{*}_1<\xi\leq j\underline{\xi}^{*}_1\} (j=1,2,3,\cdots, \underline{N})$ for $\underline{N}>0$ determined by
$L$ and $\underline{U}$, we can get $\|z_{+}^{\pm,\flat}\|_{0,0; \underline{\tilde{\mathcal{N}}}_{+}}=\|z_{-}^{\pm,\flat}\|_{0,0; \underline{\tilde{\mathcal{N}}}_{-}}= 0$.

\emph{$\mathbf{Step\ 3}$.\ \underline{Completion of the proof of the proposition.}}
By Step 2, we know that $z_{+}^{\pm,\flat}=z_{-}^{\pm,\flat}\equiv0$ in $\underline{\tilde{\mathcal{N}}}_{\pm}$.
Then, by \eqref{eq-global-unique-5.4}, we have
\begin{eqnarray*}
\delta \omega^{\flat}_{\pm}=\frac{z^{+,\flat}_{\pm}+z^{-,\flat}_{\pm}}{2}=0,\quad \delta p^{\flat}_{\pm}=\frac{z^{+,\flat}_{\pm}-z^{-,\flat}_{\pm}}{2\Lambda_{\pm}}=0, \quad \mbox{in}\quad \underline{\tilde{\mathcal{N}}}_{\pm},
\end{eqnarray*}
which implies that
\begin{eqnarray*}
\omega_{\pm}\equiv0,\quad  p_{\pm}\equiv \underline{p}^{\pm}, \quad \mbox{in}\quad \underline{\tilde{\mathcal{N}}}_{\pm}.
\end{eqnarray*}
\end{proof}

\begin{proof}[Proof of Theorem \ref{thm-global uniqueness}]
Since $(u_{\pm}, v_{\pm},\rho_{\pm})\in C^{1,\alpha}(\underline{\tilde{\mathcal{N}}}_{\pm})$ and $(\omega_{\pm}, p_{\pm}, B_{\pm}, S_{\pm}, Y_{\pm})\in C^{1,\alpha}(\underline{\tilde{\mathcal{N}}}_{\pm})$,
by Propositions \ref{prop-unique-5.1}-\ref{prop-unique-5.2} and the relation \eqref{eq:1.4}, we first have
\begin{eqnarray}\label{eq-global-unique-5.24}
\begin{split}
\rho_{\pm}=\big(A(S_{\pm})\big)^{-\frac{1}{\gamma}}(p_{\pm})^{\frac{1}{\gamma}}=\big(A(\underline{S}^{\pm})\big)^{-\frac{1}{\gamma}}(\underline{p}^{\pm})^{\frac{1}{\gamma}}=\underline{\rho}^{\pm},
\qquad \mbox{\rm{in}}\quad \underline{\tilde{\mathcal{N}}}_{\pm}.
\end{split}
\end{eqnarray}

Then, by \eqref{eq:2.23}, \eqref{eq:2.25} and the assumptions \eqref{eq-global-unique-5.1}, we have
\begin{eqnarray}\label{eq-global-unique-5.25}
\begin{split}
v_{\pm}=u_{\pm}\cdot\omega_{\pm}\equiv 0, \qquad \mbox{\rm{in}}\quad \underline{\tilde{\mathcal{N}}}_{\pm}.
\end{split}
\end{eqnarray}

Moreover, by Propositions \ref{prop-unique-5.1}-\ref{prop-unique-5.2} and the relation \eqref{eq:2.36}, we can further derive in $\tilde{\mathcal{N}}_{\pm}$ that
\begin{align}\label{eq-global-unique-5.26}
\begin{split}
u_{\pm}=\sqrt{\frac{2\Big[(\gamma-1)B_{\pm}-\gamma A^{\frac{1}{\gamma}}(S_{\pm})p^{1-\frac{1}{\gamma}}_{\pm}\Big]}
{(\gamma-1)\big(1+\omega^{2}_{\pm}\big)}}
=\sqrt{\frac{2\Big[(\gamma-1)\underline{B}^{\pm}-\gamma A^{\frac{1}{\gamma}}(\underline{S}^{\pm})(\underline{p}^{\pm})^{1-\frac{1}{\gamma}}\Big]}
{\gamma-1}}
=\underline{u}^{\pm}.
\end{split}
\end{align}

Therefore, under the assumptions in Theorem \ref{thm-global uniqueness} and by the inverse Lagrangian transformation $\mathcal{L}^{-1}$ and Proposition \ref{prop-unique-5.1},
we have $U_{\pm}=\underline{U}_{\pm}$ in $\tilde{\mathcal{N}}_{\pm}$. Finally, by the conditions $g'_{\rm cd}(x)=\frac{v_{\pm}}{u_{\pm}}$ and $g_{\rm cd}(0)=0$, we obtain $g_{\rm cd}(x)\equiv 0$ on $\underline{\Gamma}_{\rm{cd}}$ and thus complete the proof of Theorem \ref{thm-global uniqueness}.
\end{proof}

\bigskip

\appendix

\section{Properties of Characteristic Curves and Eigenvalues in Sections \ref{Sec3} and \ref{Sec6}.}\label{A}
In this Appendix, we give some properties on characteristics curves which are used in section 3 and some properties on eigenvalues which will are used in section 6. We establish the following lemmas.
\begin{lemma}\label{lem:A1}
For any point $(\bar{\xi},\bar{\eta})\in \tilde{\mathcal{N}}^{\rm{I}}_{k}$ for $k=\pm$, let $\eta=\Upsilon^{\pm}_{k}(\xi;\bar{\xi},\bar{\eta})$ be the backward characteristic curves that defined by
\begin{eqnarray}\label{eq:lem-A1-1}
\begin{cases}
\frac{\dd \Upsilon^{\pm}_{k}(\xi;\bar{\xi},\bar{\eta})}{\dd \xi}=\tilde{\lambda}^{\pm}_{k}(\xi,\Upsilon^{\pm}_{k}(\xi;\bar{\xi},\bar{\eta})),\quad \xi \in[0,\bar{\xi}],\\
\Upsilon^{\pm}_{k}(\bar{\xi};\bar{\xi},\bar{\eta})=\bar{\eta},
\end{cases}
\end{eqnarray}
where $\tilde{\lambda}^{\pm}_{k}(\xi,\Upsilon^{\pm}_{k}(\xi;\bar{\xi},\bar{\eta}))=\lambda^{\pm}(\mathcal{U}_{k}(\xi,\Upsilon^{\pm}_{k}(\xi;\bar{\xi},\bar{\eta})))$
and $\mathcal{U}_{k}=\underline{\mathcal{U}}_{k}+\delta\mathcal{U}_{k}$ with $\delta\mathcal{U}_{k}\in \mathcal{X}_{\sigma}$ for $k=\pm$.
Then,  if $\sigma>0$ is sufficiently small, there holds
\begin{eqnarray}\label{eq:lem-A1-2}
\sup_{\xi\in[0,\bar{\xi}^{1}_k]}\|\Upsilon^{\pm}_{k}(\xi)\|_{1,\alpha; \tilde{\mathcal{N}}^{\rm{I}}_{k}}<\tilde{C}_{\rm{in}},
\end{eqnarray}
for $k=\pm$, where the constant $\tilde{C}_{\rm{in}}>0$ depends only $\underline{\mathcal{U}}$, $\alpha$ and $L$.
\end{lemma}

\begin{proof}
Without loss of generality, we only consider the case $k=+$ that is $(\bar{\xi},\bar{\eta})\in\tilde{\mathcal{N}}^{\rm{I}}_{+}$ since the case $k=-$ can be done in the same method.
First, by \eqref{eq:lem-A1-1}, we have
\begin{eqnarray*}
\Upsilon^{\pm}_{k}(\xi;\bar{\xi},\bar{\eta})=\bar{\eta}+\int^{\xi}_{\bar{\xi}}\tilde{\lambda}^{\pm}_{k}(\tau,\Upsilon^{\pm}_{+}(\tau;\bar{\xi},\bar{\eta}))\dd\tau,
\end{eqnarray*}
which implies that
\begin{eqnarray*}
\sup_{\xi\in[0,\bar{\xi}^{1}_{+}]}\|\Upsilon^{\pm}_{+}(\xi)\|_{0,0;\tilde{\mathcal{N}}^{\rm{I}}_{+}}\leq |\textrm{m}_{+}|+L\cdot\sup_{\tau\in[0,L)}\|\tilde{\lambda}^{\pm}_{+}(\tau,\Upsilon^{\pm}_{+}(\tau;\bar{\xi},\bar{\eta}))\|_{0,0;\tilde{\mathcal{N}}^{\rm{I}}_{+}}<+\infty.
\end{eqnarray*}

Next, we turn to $C^{0}$-estimates for the gradient of $\Upsilon^{\pm}_{+}(\xi;\bar{\xi},\bar{\eta})$ with respect to $\bar{\xi}$ and $\bar{\eta}$. To this end,
by \eqref{eq:lem-A1-1} and straightforward calculations, we obtain
\begin{eqnarray}\label{eq:prop-A1-3}
\p_{\bar{\xi}}\Upsilon^{\pm}_{+}(\xi;\bar{\xi},\bar{\eta})=-\tilde{\lambda}^{\pm}_{+}(\bar{\xi},\bar{\eta})\rm{e}^{\int^{\xi}_{\bar{\xi}}\p_{\Upsilon^{\pm}_{+}}\tilde{\lambda}^{\pm}_{+}(\tau, \Upsilon^{\pm}_{+}(\tau;\bar{\xi},\bar{\eta}))\dd \tau},\
\p_{\bar{\eta}}\Upsilon^{\pm}_{+}(\xi;\bar{\xi},\bar{\eta})=\rm{e}^{\int^{\xi}_{\bar{\xi}}\p_{\Upsilon^{\pm}_{+}}\tilde{\lambda}^{\pm}_{+}(\tau, \Upsilon^{\pm}_{+}(\tau;\bar{\xi},\bar{\eta}))\dd \tau}.
\end{eqnarray}

Since, for $\sigma>0$ sufficiently small, we get
\begin{eqnarray}\label{eq:prop-A1-4}
\begin{split}
\sup_{\tau\in[0,L)}\|\p_{\Upsilon^{\pm}_{+}}\tilde{\lambda}^{\pm}_{+}(\tau)\|_{0,0;\tilde{\mathcal{N}}^{\rm{I}}_{+}}
&\leq \big\|\nabla_{\mathcal{U}_{+}}\lambda^{\pm}_{+}\big\|_{0,0;\tilde{\mathcal{N}}^{\rm{I}}_{+}}
\sup_{\tau\in[0,L)}\big\|\p_{\Upsilon^{\pm}_{+}}\delta\mathcal{U}_{+}(\tau)\big\|_{0,0;\tilde{\mathcal{N}}^{\rm{I}}_{+}}\\
&\leq\big\|\nabla_{\mathcal{U}_{+}}\lambda^{\pm}_{+}\big\|_{0,0;\tilde{\mathcal{N}}^{\rm{I}}_{+}}\|\delta\mathcal{U}_{+}\|_{1,0;\tilde{\mathcal{N}}^{\rm{I}}_{+}}<+\infty.
\end{split}
\end{eqnarray}

Then, it follows from \eqref{eq:prop-A1-4} that
\begin{eqnarray}\label{eq:prop-A1-5}
\begin{split}
\sup_{\xi\in[0,\bar{\xi}^{1}_{+}]}\|\nabla_{(\bar{\xi},\bar{\eta})}\Upsilon^{\pm}_{+}(\xi)\|_{0,0;\tilde{\mathcal{N}}^{\rm{I}}_{+}}
\leq \big(1+\|\tilde{\lambda}^{\pm}_{+}\|_{0,0;\tilde{\mathcal{N}}^{\rm{I}}_{+}}\big)
\textrm{e}^{L\cdot\sup_{\tau\in[0,L)}\|\p_{\Upsilon^{\pm}_{+}}\tilde{\lambda}^{\pm}_{+}(\tau)\|_{0,0;\tilde{\mathcal{N}}^{\rm{I}}_{+}}}
<+\infty.
\end{split}
\end{eqnarray}

Finally, we will consider the H\"{o}lder estimates for $\nabla_{(\bar{\xi},\bar{\eta})}\Upsilon^{\pm}_{+}$. To complete it, giving two points $(\bar{\xi}_1, \bar{\eta}_{1}), (\bar{\xi}_2, \bar{\eta}_{2})\in\tilde{\mathcal{N}}^{\textrm{I}}_{+}$ with $\bar{\xi}_{1}\neq \bar{\xi}_{2}$ or $\bar{\eta}_{1}\neq \bar{\eta}_{2}$. Without loss of generality, we also assume $\bar{\xi}_{1}<\bar{\xi}_{2}$.
Then, by using \eqref{eq:prop-A1-3}, we consider
\begin{eqnarray}\label{eq:prop-A1-6}
\begin{split}
&|\p_{\bar{\xi}}\Upsilon^{\pm}_{+}(\xi;\bar{\xi}_1,\bar{\eta}_1)-\p_{\bar{\xi}}\Upsilon^{\pm}_{+}(\xi;\bar{\xi}_2,\bar{\eta}_2)|\\
&\quad\ \leq \big|\tilde{\lambda}^{\pm}_{+}(\bar{\xi}_{1},\bar{\eta}_{1})-\tilde{\lambda}^{\pm}_{+}(\bar{\xi}_{2},\bar{\eta}_{2})\big|
\textrm{e}^{\int^{\bar{\xi}_{2}}_{\xi}|\p_{\Upsilon^{\pm}_{+}}\tilde{\lambda}^{\pm}_{+}(\tau, \Upsilon^{\pm}_{+}(\tau;\bar{\xi}_2,\bar{\eta}_2))|\dd \tau}\\
&\qquad\ +|\tilde{\lambda}^{\pm}_{+}(\bar{\xi}_{1},\bar{\eta}_{1})|
\Big|\textrm{e}^{\int^{\xi}_{\bar{\xi}_{2}}\p_{\Upsilon^{\pm}_{+}}\tilde{\lambda}^{\pm}_{+}(\tau, \Upsilon^{\pm}_{+}(\tau;\bar{\xi}_2,\bar{\eta}_2))\dd \tau}-\textrm{e}^{\int^{\xi}_{\bar{\xi}_{1}}\p_{\Upsilon^{\pm}_{+}}\tilde{\lambda}^{\pm}_{+}(\tau, \Upsilon^{\pm}_{+}(\tau;\bar{\xi}_1,\bar{\eta}_1))\dd \tau}\Big|.
\end{split}
\end{eqnarray}

Note that
\begin{eqnarray}\label{eq:prop-A1-7}
\begin{split}
&\big|\tilde{\lambda}^{\pm}_{+}(\bar{\xi}_{1},\bar{\eta}_{1})-\tilde{\lambda}^{\pm}_{+}(\bar{\xi}_{2},\bar{\eta}_{2})\big|\\
&\leq |\nabla_{\mathcal{U}_{+}}\lambda^{\pm}_{+}||\nabla \delta\mathcal{U}_{+}|\big(|\bar{\xi}_{1}-\bar{\xi}_{2}|+|\bar{\eta}_{1}-\bar{\eta}_{2}|\big)\\
&\leq \big(L+|\textrm{m}_{+}|\big)^{1-\alpha}\|\nabla_{\mathcal{U}_{+}}\lambda^{\pm}_{+}\|_{0,0;\tilde{\mathcal{N}}^{\rm{I}}_{+}}
\|\delta\mathcal{U}_{+}\|_{1,0;\tilde{\mathcal{N}}^{\rm{I}}_{+}}|(\bar{\xi}_{1},\bar{\eta}_{1})-(\bar{\xi}_{2},\bar{\eta}_{2})|^{\alpha}\\
&\leq \mathcal{O}(1)|(\bar{\xi}_{1},\bar{\eta}_{1})-(\bar{\xi}_{2},\bar{\eta}_{2})|^{\alpha},
\end{split}
\end{eqnarray}
and
\begin{align}\label{eq:prop-A1-8}
\begin{split}
&\Big|\textrm{e}^{\int^{\xi}_{\bar{\xi}_{2}}\p_{\Upsilon^{\pm}_{+}}\tilde{\lambda}^{\pm}_{+}(\tau, \Upsilon^{\pm}_{+}(\tau;\bar{\xi}_2,\bar{\eta}_2))\dd \tau}-\textrm{e}^{\int^{\xi}_{\bar{\xi}_{1}}\p_{\Upsilon^{\pm}_{+}}\tilde{\lambda}^{\pm}_{+}(\tau, \Upsilon^{\pm}_{+}(\tau;\bar{\xi}_1,\bar{\eta}_1))\dd \tau}\Big|\\
&\leq \bigg|\int^{1}_{0}\textrm{e}^{\theta\int^{\xi}_{\bar{\xi}_{2}}\p_{\Upsilon^{\pm}_{+}}\tilde{\lambda}^{\pm}_{+}(\tau, \Upsilon^{\pm}_{+}(\tau;\bar{\xi}_2,\bar{\eta}_2))\dd \tau+(1-\theta)\int^{\xi}_{\bar{\xi}_{1}}\p_{\Upsilon^{\pm}_{+}}\tilde{\lambda}^{\pm}_{+}(\tau, \Upsilon^{\pm}_{+}(\tau;\bar{\xi}_1,\bar{\eta}_1))\dd \tau}\bigg|\\
&\quad\ \times \bigg|\int^{\xi}_{\bar{\xi}_{2}}\p_{\Upsilon^{\pm}_{+}}\tilde{\lambda}^{\pm}_{+}(\tau, \Upsilon^{\pm}_{+}(\tau;\bar{\xi}_2,\bar{\eta}_2))\dd \tau-
\int^{\xi}_{\bar{\xi}_{1}}\p_{\Upsilon^{\pm}_{+}}\tilde{\lambda}^{\pm}_{+}(\tau, \Upsilon^{\pm}_{+}(\tau;\bar{\xi}_1,\bar{\eta}_1))\dd \tau\bigg|\\
&\leq \mathcal{O}(1)\bigg[\int^{\bar{\xi}_{2}}_{\bar{\xi}_{1}}|\p_{\Upsilon^{\pm}_{+}}\tilde{\lambda}^{\pm}_{+}(\tau, \Upsilon^{\pm}_{+}(\tau;\bar{\xi}_2,\bar{\eta}_2))|\dd \tau\\
&\qquad\qquad +\int^{\bar{\xi}_{1}}_{\xi}\big|\p_{\Upsilon^{\pm}_{+}}\tilde{\lambda}^{\pm}_{+}(\tau, \Upsilon^{\pm}_{+}(\tau;\bar{\xi}_1,\bar{\eta}_1))-\p_{\Upsilon^{\pm}_{+}}\tilde{\lambda}^{\pm}_{+}(\tau, \Upsilon^{\pm}_{+}(\tau;\bar{\xi}_2,\bar{\eta}_2))\big|\dd \tau\bigg]\\
&\leq \mathcal{O}(1)\Big[\|\nabla_{\mathcal{U}_{+}}\lambda^{\pm}_{+}\|_{0,0;\tilde{\mathcal{N}_{+}}}
\|\nabla\delta\mathcal{U}_{+}\|_{0,\alpha;\tilde{\mathcal{N}}^{\rm{I}}_{+}}\cdot|\bar{\xi}_1-\bar{\xi}_{2}|\\
&\qquad\qquad +\|\nabla_{\mathcal{U}_{+}}\lambda^{\pm}_{+}\|_{0,\alpha;\tilde{\mathcal{N}_{+}}}
\|\nabla\delta\mathcal{U}_{+}\|_{0,\alpha;\tilde{\mathcal{N}}^{\rm{I}}_{+}}|(\bar{\xi}_{1},\bar{\eta}_{1})-(\bar{\xi}_{2},\bar{\eta}_{2})|^{\alpha}\Big]\\
&\leq \mathcal{O}(1)|(\bar{\xi}_{1},\bar{\eta}_{1})-(\bar{\xi}_{2},\bar{\eta}_{2})|^{\alpha},
\end{split}
\end{align}
for $\sigma>0$ sufficiently small, where the constant $\mathcal{O}(1)$ depends only on $\underline{\mathcal{U}}_{+}$, $\alpha$ and $L$.

Then, substituting the estimates \eqref{eq:prop-A1-7}-\eqref{eq:prop-A1-8} into \eqref{eq:prop-A1-6}, we obtain
\begin{eqnarray}\label{eq:prop-A1-9}
\begin{split}
|\p_{\bar{\xi}}\Upsilon^{\pm}_{+}(\xi;\bar{\xi}_1,\bar{\eta}_1)-\p_{\bar{\xi}}\Upsilon^{\pm}_{+}(\xi;\bar{\xi}_2,\bar{\eta}_2)|\leq \mathcal{O}(1)|(\bar{\xi}_{1},\bar{\eta}_{1})-(\bar{\xi}_{2},\bar{\eta}_{2})|^{\alpha}.
\end{split}
\end{eqnarray}

In the same way, we can also get
\begin{eqnarray}\label{eq:prop-A1-10}
\begin{split}
|\p_{\bar{\eta}}\Upsilon^{\pm}_{+}(\xi;\bar{\xi}_1,\bar{\eta}_1)-\p_{\bar{\eta}}\Upsilon^{\pm}_{+}(\xi;\bar{\xi}_2,\bar{\eta}_2)|\leq \mathcal{O}(1)|(\bar{\xi}_{1},\bar{\eta}_{1})-(\bar{\xi}_{2},\bar{\eta}_{2})|^{\alpha}.
\end{split}
\end{eqnarray}

Combining \eqref{eq:prop-A1-9}-\eqref{eq:prop-A1-10} altogether, we can choose a constant $C_{\rm{in}}>0$ depending only on $\underline{\mathcal{U}}_{+}$, $\alpha$ and $L$ such that
\begin{eqnarray*}
\sup_{\xi\in[0,\bar{\xi}^{1}_+]}[\nabla_{(\bar{\xi},\bar{\eta})}\Upsilon^{\pm}_{+}(\xi)]_{0,\alpha;\tilde{\mathcal{N}}^{\rm{I}}_{+}}<C_{\rm{in}},
\end{eqnarray*}
which gives the estimate \eqref{eq:lem-A1-2} for $k=+$.

\end{proof}

\begin{lemma}\label{lem:A2}
For any $(\bar{\xi}, \bar{\eta})\in \tilde{\mathcal{N}}^{\rm{II}}_{k}$, let $\eta=\Upsilon^{\mp,\rm{b}}_{k}(\xi; \bar{\xi}, \bar{\eta})$ be
characteristic curves that defined by
\begin{eqnarray}\label{eq:lem-A2-1}
\begin{cases}
\frac{\dd \Upsilon^{\mp, \rm{b}}_{k}(\xi;\bar{\xi},\bar{\eta})}{\dd \xi}=\tilde{\lambda}^{\pm}_{k}(\xi,\Upsilon^{\pm}_{k}(\xi;\bar{\xi},\bar{\eta})),\quad \xi \in[0,\bar{\xi}],\\
\Upsilon^{\mp,\rm{b}}_{k}(\bar{\xi};\bar{\xi},\bar{\eta})=\bar{\eta},
\end{cases}
\end{eqnarray}
where $\tilde{\lambda}^{\pm}_{k}(\xi,\Upsilon^{\pm}_{k}(\xi;\bar{\xi},\bar{\eta}))
=\lambda^{\pm}(\mathcal{U}_{k}(\xi,\Upsilon^{\pm}_{k}(\xi;\bar{\xi},\bar{\eta})))$
and $\mathcal{U}_{k}=\underline{\mathcal{U}}_{k}+\delta\mathcal{U}_{k}$ with $\delta\mathcal{U}_{k}\in \mathcal{X}_{\sigma}$ for $k=\pm$.
Then,  if $\sigma>0$ is sufficiently small, there holds
\begin{eqnarray}\label{eq:lem-A2-2}
\sup_{\xi\in[0,\bar{\xi}^{2}_+]}\|\Upsilon^{-, \rm{b}}_{+}(\xi)\|_{1,\alpha; \tilde{\mathcal{N}}^{\rm{II}}_{+}}+\sup_{\xi\in[0,\bar{\xi}^{2}_-]}\|\Upsilon^{+, \rm{b}}_{-}(\xi)\|_{1,\alpha; \tilde{\mathcal{N}}^{\rm{II}}_{-}}<\tilde{C}_{\rm{b}},
\end{eqnarray}
where the constant $\tilde{C}_{\rm{b}}>0$ depends only $\underline{\mathcal{U}}$, $\alpha$ and $L$.
Specially, if $\zeta^{-,\rm{b}}_{+}(\bar{\xi}, \bar{\eta})$ and $\zeta^{+,\rm{b}}_{-}(\bar{\xi}, \bar{\eta})$ be the intersection points of the characteristics $\eta=\Upsilon^{-,\rm{b}}_{+}(\xi; \bar{\xi}, \bar{\eta})$ with the upper wall $\tilde{\Gamma}_{+}$ and $\eta=\Upsilon^{+,\rm{b}}_{-}(\xi; \bar{\xi}, \bar{\eta})$ with the lower wall $\tilde{\Gamma}_{-}$, respectively. Then, there holds
\begin{eqnarray}\label{eq:lem-A2-3}
\|\zeta^{-,\rm{b}}_{+}\|_{1,\alpha; \tilde{\mathcal{N}}^{\rm{II}}_{+}}+\|\zeta^{+,\rm{b}}_{-}\|_{1,\alpha; \tilde{\mathcal{N}}^{\rm{II}}_{-}}<\tilde{C}'_{\rm{b}},
\end{eqnarray}
where the constant $\tilde{C}'_{\rm{b}}>0$ depends only on $\underline{\mathcal{U}}$, $\alpha$ and $L$.
\end{lemma}

\begin{proof}
Since the proof of the estimates \eqref{eq:lem-A2-2} are the same as the proof of \eqref{eq:lem-A1-2} in Lemma \ref{lem:A1}.
So, we need to consider the estimate \eqref{eq:lem-A2-3}. To this end, let
\begin{eqnarray*}
\mathscr{F}(\xi;\bar{\xi},\bar{\eta})=\rm{m}_{+}-\Upsilon^{-, \rm{b}}_{+}(\xi;\bar{\xi},\bar{\eta}).
\end{eqnarray*}
Then, $\p_{\bar{\xi}}\mathscr{F}, \p_{\bar{\eta}}\mathscr{F}$ are continuous, and
\begin{eqnarray*}
\mathscr{F}(\zeta^{-,\rm{b}}_{+},\bar{\xi},\bar{\eta})=0,\quad \p_{\bar{\xi}}\mathscr{F}(\zeta^{-,\rm{b}}_{+},\bar{\xi},\bar{\eta})=-\tilde{\lambda}^{-}_{+}(\zeta^{-,\rm{b}}_{+}, \rm{m}_{+})\neq0,
\end{eqnarray*}
provided that $\sigma>0$ is sufficiently small. Therefore, by implicit function theorem, we know that $\zeta^{-,\rm{b}}_{+}$ can be solved as a $C^1$-function of $(\bar{\xi},\bar{\eta})$ satisfying
\begin{eqnarray}\label{eq:lem-A2-4}
\p_{\bar{\xi}}\zeta^{-,\rm{b}}_{+}=-\frac{\partial_{\bar{\xi}}\mathscr{F}(\zeta^{-,\rm{b}}_{+}(\bar{\xi},\bar{\eta});\bar{\xi},\bar{\eta})}
{\p_{\xi}\mathscr{F}(\zeta^{-,\rm{b}}_{+}(\bar{\xi},\bar{\eta});\bar{\xi},\bar{\eta})}
=\frac{\p_{\bar{\xi}}\Upsilon^{-,\rm{b}}_{+}(\zeta^{-,\rm{b}}_{+}(\bar{\xi},\bar{\eta});\bar{\xi},\bar{\eta})}
{\tilde{\lambda}^{-}_{+}(\zeta^{-,\rm{b}}_{+}(\bar{\xi},\bar{\eta}),\rm{m}_{+})},
\end{eqnarray}
and
\begin{eqnarray}\label{eq:lem-A2-5}
\p_{\bar{\eta}}\zeta^{-,\rm{b}}_{+}=-\frac{\p_{\bar{\eta}}\mathscr{F}(\zeta^{-,\rm{b}}_{+}(\bar{\xi},\bar{\eta});\bar{\xi},\bar{\eta})}
{\p_{\xi}\mathscr{F}(\zeta^{-,\rm{b}}_{+}(\bar{\xi},\bar{\eta});\bar{\xi},\bar{\eta})}
=\frac{\p_{\bar{\eta}}\Upsilon^{-,\rm{b}}_{+}(\zeta^{-,\rm{b}}_{+}(\bar{\xi},\bar{\eta});\bar{\xi},\bar{\eta})}
{\tilde{\lambda}^{\pm}_{k}(\zeta^{-,\rm{b}}_{+}(\bar{\xi},\bar{\eta}),\rm{m}_{+})}.
\end{eqnarray}

Next, we turn to H\"{o}lder estimates for $\nabla_{(\bar{\xi},\bar{\eta})}\zeta^{-,\rm{b}}_{+}$. Taking two points $(\bar{\xi}_1,\bar{\eta}_1), (\bar{\xi}_2,\bar{\eta}_2)\in \tilde{\mathcal{N}}^{\rm{II}}_{+}$ with $(\bar{\xi}_1,\bar{\eta}_1)\neq(\bar{\xi}_2,\bar{\eta}_2)$.
and substituting $(\bar{\xi}_1,\bar{\eta}_1)$ and $(\bar{\xi}_2,\bar{\eta}_2)$ into \eqref{eq:lem-A2-4}-\eqref{eq:lem-A2-4}, and taking difference between them, we have
\begin{eqnarray}\label{eq:lem-A2-6}
\begin{split}
&\frac{|\p_{\bar{\xi}}\zeta^{-,\rm{b}}_{+}(\bar{\xi}_1,\bar{\eta}_1)-\p_{\bar{\xi}}\zeta^{-,\rm{b}}_{+}(\bar{\xi}_2,\bar{\eta}_2)|}
{|(\bar{\xi}_{1},\bar{\eta}_{1})-(\bar{\xi}_{2},\bar{\eta}_{2})|^{\alpha}}\\
&=\frac{\Big|\frac{\p_{\bar{\xi}}\Upsilon^{-,\rm{b}}_{+}(\zeta^{-,\rm{b}}_{+}(\bar{\xi}_1,\bar{\eta}_1);\bar{\xi}_1,\bar{\eta}_1)}
{\tilde{\lambda}^{-}_{+}(\zeta^{-,\rm{b}}_{+}(\bar{\xi}_1,\bar{\eta}_1),\rm{m}_{+})}
-\frac{\p_{\bar{\xi}}\Upsilon^{-,\rm{b}}_{+}(\zeta^{-,\rm{b}}_{+}(\bar{\xi}_2,\bar{\eta}_2);\bar{\xi}_2,\bar{\eta}_2)}
{\tilde{\lambda}^{-}_{+}(\zeta^{-,\rm{b}}_{+}(\bar{\xi}_2,\bar{\eta}_2),\rm{m}_{+})}\Big|}{|(\bar{\xi}_{1},\bar{\eta}_{1})-(\bar{\xi}_{2},\bar{\eta}_{2})|^{\alpha}}\\
&\leq\frac{\Big|\frac{\p_{\bar{\xi}}\Upsilon^{-,\rm{b}}_{+}(\zeta^{-,\rm{b}}_{+}(\bar{\xi}_1,\bar{\eta}_1);\bar{\xi}_1,\bar{\eta}_1)}
{\tilde{\lambda}^{-}_{+}(\zeta^{-,\rm{b}}_{+}(\bar{\xi}_1,\bar{\eta}_1),\rm{m}_{+})}
-\frac{\p_{\bar{\xi}}\Upsilon^{-,\rm{b}}_{+}(\zeta^{-,\rm{b}}_{+}(\bar{\xi}_2,\bar{\eta}_2);\bar{\xi}_1,\bar{\eta}_1)}
{\tilde{\lambda}^{-}_{+}(\zeta^{-,\rm{b}}_{+}(\bar{\xi}_2,\bar{\eta}_2),\rm{m}_{+})}\Big|}
{|\zeta^{-,\rm{b}}_{+}(\bar{\xi}_{1},\bar{\eta}_{1})-\zeta^{-,\rm{b}}_{+}(\bar{\xi}_{2},\bar{\eta}_{2})|^{\alpha}}
\cdot\frac{|\zeta^{-,\rm{b}}_{+}(\bar{\xi}_{1},\bar{\eta}_{1})-\zeta^{-,\rm{b}}_{+}(\bar{\xi}_{2},\bar{\eta}_{2})|^{\alpha}}
{|(\bar{\xi}_{1},\bar{\eta}_{1})-(\bar{\xi}_{2},\bar{\eta}_{2})|^{\alpha}}\\
&\quad+\frac{1}{|\tilde{\lambda}^{-}_{+}(\zeta^{-,\rm{b}}_{+}(\bar{\xi}_2,\bar{\eta}_2),\rm{m}_+)|}
\frac{\Big|\p_{\bar{\xi}}\Upsilon^{-,\rm{b}}_{+}(\zeta^{-,\rm{b}}_{+}(\bar{\xi}_2,\bar{\eta}_2);\bar{\xi}_1,\bar{\eta}_1)
-\p_{\bar{\xi}}\Upsilon^{-,\rm{b}}_{+}(\zeta^{-,\rm{b}}_{+}(\bar{\xi}_2,\bar{\eta}_2);\bar{\xi}_2,\bar{\eta}_2)\Big|}
{|(\bar{\xi}_{1},\bar{\eta}_{1})-(\bar{\xi}_{2},\bar{\eta}_{2})|^{\alpha}}\\
&\leq\frac{\|\zeta^{-,\rm{b}}_{+}\|^{\alpha}_{1,0;\mathcal{N}^{\rm{II}}_{+}}}
{|\tilde{\lambda}^{-}_{+}(\zeta^{-,\rm{b}}_{+}(\bar{\xi}_1,\bar{\eta}_1),\rm{m}_{+})|}
\frac{\Big|\p_{\bar{\xi}}\Upsilon^{-,\rm{b}}_{+}(\zeta^{-,\rm{b}}_{+}(\bar{\xi}_1,\bar{\eta}_1);\bar{\xi}_1,\bar{\eta}_1)
-\p_{\bar{\xi}}\Upsilon^{-,\rm{b}}_{+}(\zeta^{-,\rm{b}}_{+}(\bar{\xi}_2,\bar{\eta}_2);\bar{\xi}_1,\bar{\eta}_1)\Big|}
{|\zeta^{-,\rm{b}}_{+}(\bar{\xi}_{1},\bar{\eta}_{1})-\zeta^{-,\rm{b}}_{+}(\bar{\xi}_{2},\bar{\eta}_{2})|^{\alpha}}\\
&\quad+\frac{\|\zeta^{-,\rm{b}}_{+}\|^{\alpha}_{1,0;\mathcal{N}^{\rm{II}}_{+}}
\cdot|\p_{\bar{\xi}}\Upsilon^{-,\rm{b}}_{+}(\zeta^{-,\rm{b}}_{+}(\bar{\xi}_2,\bar{\eta}_2);\bar{\xi}_1,\bar{\eta}_1)|}
{|\tilde{\lambda}^{-}_{+}(\zeta^{-,\rm{b}}_{+}(\bar{\xi}_1,\bar{\eta}_1),\rm{m}_{+})
\tilde{\lambda}^{-}_{+}(\zeta^{-,\rm{b}}_{+}(\bar{\xi}_2,\bar{\eta}_2),\rm{m}_{+})|}
\frac{\Big|\tilde{\lambda}^{-}_{+}(\zeta^{-,\rm{b}}_{+}(\bar{\xi}_1,\bar{\eta}_1),\rm{m}_{+})
-\tilde{\lambda}^{-}_{+}(\zeta^{-,\rm{b}}_{+}(\bar{\xi}_2,\bar{\eta}_2),\rm{m}_{+})\Big|}
{|\zeta^{-,\rm{b}}_{+}(\bar{\xi}_{1},\bar{\eta}_{1})
-\zeta^{-,\rm{b}}_{+}(\bar{\xi}_{2},\bar{\eta}_{2})|^{\alpha}}\\
&\quad+\frac{1}{\tilde{\lambda}^{-}_{+}(\zeta^{-,\rm{b}}_{+}(\bar{\xi}_2,\bar{\eta}_2),\rm{m}_{+})}
\frac{\Big|\p_{\bar{\xi}}\Upsilon^{-,\rm{b}}_{+}(\zeta^{-,\rm{b}}_{+}(\bar{\xi}_2,\bar{\eta}_2);\bar{\xi}_1,\bar{\eta}_1)
-\p_{\bar{\xi}}\Upsilon^{-,\rm{b}}_{+}(\zeta^{-,\rm{b}}_{+}(\bar{\xi}_2,\bar{\eta}_2);\bar{\xi}_2,\bar{\eta}_2)\Big|}
{|(\bar{\xi}_{1},\bar{\eta}_{1})-(\bar{\xi}_{2},\bar{\eta}_{2})|^{\alpha}}.
\end{split}
\end{eqnarray}

By $\eqref{eq:prop-A1-3}$, we note that
\begin{eqnarray}\label{eq:lem-A2-7}
\begin{split}
&\big|\p_{\bar{\xi}}\Upsilon^{-,\rm{b}}_{+}(\zeta^{-,\rm{b}}_{+}(\bar{\xi}_1,\bar{\eta}_1);\bar{\xi}_1,\bar{\eta}_1)
-\p_{\bar{\xi}}\Upsilon^{-,\rm{b}}_{+}(\zeta^{-,\rm{b}}_{+}(\bar{\xi}_2,\bar{\eta}_2);\bar{\xi}_1,\bar{\eta}_1)\big|\\
&=\big|\tilde{\lambda}^{-}_{+}(\bar{\xi}_1,\bar{\eta}_1)\big|\cdot
\bigg|\textrm{e}^{\int^{\zeta^{-,\rm{b}}_{+}(\bar{\xi}_2,\bar{\eta}_2)}_{\bar{\xi}_1}\p_{\Upsilon^{-,\rm{b}}_{+}}\tilde{\lambda}^{-}_{+}(\tau, \Upsilon^{-,\rm{b}}_{+}(\tau;\bar{\xi}_1,\bar{\eta}_1))\dd \tau}-\textrm{e}^{\int^{\zeta^{-,\rm{b}}_{+}(\bar{\xi}_1,\bar{\eta}_1)}_{\bar{\xi}_1}\p_{\Upsilon^{-,\rm{b}}_{+}}\tilde{\lambda}^{-}_{+}(\tau, \Upsilon^{-,\rm{b}}_{+}(\tau;\bar{\xi}_1,\bar{\eta}_1))\dd \tau}\bigg|\\
&\leq \bigg|\int^{1}_{0}\textrm{e}^{\theta\int^{\zeta^{-,\rm{b}}_{+}(\bar{\xi}_2,\bar{\eta}_2)}_{\bar{\xi}_1}\p_{\Upsilon^{-,\rm{b}}_{+}}
\tilde{\lambda}^{-}_{+}(\tau, \Upsilon^{-,\rm{b}}_{+}(\tau;\bar{\xi}_1,\bar{\eta}_1))\dd \tau+(1-\theta)\int^{\zeta^{-,\rm{b}}_{+}(\bar{\xi}_1,\bar{\eta}_1)}_{\bar{\xi}_1}\p_{\Upsilon^{-,\rm{b}}_{+}}\tilde{\lambda}^{-}_{+}(\tau, \Upsilon^{-,\rm{b}}_{+}(\tau;\bar{\xi}_1,\bar{\eta}_1))\dd \tau}\bigg|\\
&\quad \cdot\bigg|\int^{\zeta^{-,\rm{b}}_{+}(\bar{\xi}_2,\bar{\eta}_2)}_{\zeta^{-,\rm{b}}_{+}(\bar{\xi}_1,\bar{\eta}_1)}
\p_{\Upsilon^{-,\rm{b}}_{+}}\tilde{\lambda}^{-}_{+}(\tau, \Upsilon^{-,\rm{b}}_{+}(\tau;\bar{\xi}_1,\bar{\eta}_1))\dd \tau\bigg|\\
&\leq \mathcal{O}(1)\big|\zeta^{-,\rm{b}}_{+}(\bar{\xi}_1,\bar{\eta}_1)-\zeta^{-,\rm{b}}_{+}(\bar{\xi}_2,\bar{\eta}_2)\big|,
\end{split}
\end{eqnarray}
and by the estimate \eqref{eq:lem-A2-2},
\begin{eqnarray}\label{eq:lem-A2-8}
\begin{split}
\Big|\p_{\bar{\xi}}\Upsilon^{-,\rm{b}}_{+}(\zeta^{-,\rm{b}}_{+}(\bar{\xi}_2,\bar{\eta}_2);\bar{\xi}_1,\bar{\eta}_1)
-\p_{\bar{\xi}}\Upsilon^{-,\rm{b}}_{+}(\zeta^{-,\rm{b}}_{+}(\bar{\xi}_2,\bar{\eta}_2);\bar{\xi}_2,\bar{\eta}_2)\Big|
\leq \mathcal{O}(1)\big|(\bar{\xi}_1,\bar{\eta}_1)-(\bar{\xi}_2,\bar{\eta}_2)\big|^{\alpha},
\end{split}
\end{eqnarray}
where the constants $\mathcal{O}(1)$ depending only on $\underline{U}$, $\alpha$ and $L$.

Then, substituting the estimates \eqref{eq:lem-A2-7} and \eqref{eq:lem-A2-8} into \eqref{eq:lem-A2-6}, we obtain
\begin{eqnarray}\label{eq:lem-A2-9}
\begin{split}
|\p_{\bar{\xi}}\zeta^{-,\rm{b}}_{+}(\bar{\xi}_1,\bar{\eta}_1)-\p_{\bar{\xi}}\zeta^{-,\rm{b}}_{+}(\bar{\xi}_2,\bar{\eta}_2)|\leq \mathcal{O}(1)\big|(\bar{\xi}_{1},\bar{\eta}_{1})-(\bar{\xi}_{2},\bar{\eta}_{2})\big|^{\alpha},
\end{split}
\end{eqnarray}
where the constants $\mathcal{O}(1)$ depending only on $\underline{U}$, $\alpha$ and $L$.

Similarly, we can also get
\begin{eqnarray}\label{eq:lem-A2-10}
\begin{split}
|\p_{\bar{\eta}}\zeta^{-,\rm{b}}_{+}(\bar{\xi}_1,\bar{\eta}_1)-\p_{\bar{\eta}}\zeta^{-,\rm{b}}_{+}(\bar{\xi}_2,\bar{\eta}_2)|\leq \mathcal{O}(1)|(\bar{\xi}_{1},\bar{\eta}_{1})-(\bar{\xi}_{2},\bar{\eta}_{2})|^{\alpha},
\end{split}
\end{eqnarray}
where the constants $\mathcal{O}(1)$ depending only on $\underline{U}$, $\alpha$ and $L$.

Combining \eqref{eq:lem-A2-9}-\eqref{eq:lem-A2-10} altogether, we can choose a constant $\tilde{C}'_{\rm{b}}>0$ depending only on $\underline{\mathcal{U}}_{+}$, $\alpha$ and $L$ such that the estimate \eqref{eq:lem-A2-3} holds.
Finally, in the same way, we can also get the estimate \eqref{eq:lem-A2-3} for $k=-$.

\end{proof}

Similarly, we also have the estimates for $\eta=\Upsilon^{\pm,\rm{cd}}_{k}(\xi; \bar{\xi},\bar{\eta})$ on $\tilde{\mathcal{N}}^{\rm{III}}_{k}$ for $k=\pm$.
\begin{lemma}\label{lem:A3}
For any $(\bar{\xi}, \bar{\eta})\in \tilde{\mathcal{N}}^{\rm{III}}_{k}$, let $\eta=\Upsilon^{\pm,\rm{cd}}_{k}(\xi; \bar{\xi}, \bar{\eta})$ be
characteristic curves that defined by
\begin{eqnarray}\label{eq:lem-A3-1}
\begin{cases}
\frac{\dd \Upsilon^{\pm, \rm{cd}}_{k}(\xi;\bar{\xi},\bar{\eta})}{\dd \xi}=\tilde{\lambda}^{\pm}_{k}(\xi,\Upsilon^{\pm}_{k}(\xi;\bar{\xi},\bar{\eta})),\quad \xi \in[0,\bar{\xi}],\\
\Upsilon^{\pm,\rm{cd}}_{k}(\bar{\xi};\bar{\xi},\bar{\eta})=\bar{\eta},
\end{cases}
\end{eqnarray}
where $\tilde{\lambda}^{\pm}_{k}(\xi,\Upsilon^{\pm}_{k}(\xi;\bar{\xi},\bar{\eta}))
=\lambda^{\pm}(\mathcal{U}_{k}(\xi,\Upsilon^{\pm}_{k}(\xi;\bar{\xi},\bar{\eta})))$
and $\mathcal{U}_{k}=\underline{\mathcal{U}}_{k}+\delta\mathcal{U}_{k}$ with $\delta\mathcal{U}_{k}\in \mathcal{X}_{\sigma}$ for $k=\pm$.
Then,  if $\sigma>0$ is sufficiently small, there holds
\begin{eqnarray}\label{eq:lem-A3-2}
\sup_{\xi\in[0,\bar{\xi}^{2}_{\rm{cd}}]}\|\Upsilon^{-, \rm{cd}}_{+}(\xi)\|_{1,\alpha; \tilde{\mathcal{N}}^{\rm{III}}_{+}}+\sup_{\xi\in[0,\bar{\xi}^{2}_{\rm{cd}}]}\|\Upsilon^{+, \rm{cd}}_{-}(\xi)\|_{1,\alpha; \tilde{\mathcal{N}}^{\rm{III}}_{-}}<\tilde{C}_{\rm{cd}},
\end{eqnarray}
where the constant $\tilde{C}_{\rm{cd}}>0$ depends only $\underline{\mathcal{U}}$, $\alpha$ and $L$.

Specially, if $\zeta^{+,\rm{cd}}_{+}(\bar{\xi}, \bar{\eta})$ and $\zeta^{-,\rm{cd}}_{-}(\bar{\xi}, \bar{\eta})$ be the intersection points of the characteristics $\eta=\Upsilon^{+,\rm{cd}}_{+}(\xi; \bar{\xi}, \bar{\eta})$ with the contact discontinuity $\tilde{\Gamma}_{\rm{cd}}$ from above and $\eta=\Upsilon^{-,\rm{cd}}_{-}(\xi; \bar{\xi}, \bar{\eta})$ with the contact discontinuity $\tilde{\Gamma}_{\rm{cd}}$ from below, respectively. Then, there holds
\begin{eqnarray}\label{eq:lem-A3-3}
\|\zeta^{+,\rm{cd}}_{+}\|_{1,\alpha; \tilde{\mathcal{N}}^{\rm{III}}_{+}}+\|\zeta^{-,\rm{cd}}_{-}\|_{1,\alpha; \tilde{\mathcal{N}}^{\rm{III}}_{-}}<\tilde{C}'_{\rm{cd}},
\end{eqnarray}
where the constant $\tilde{C}'_{\rm{cd}}>0$ depends only on $\underline{\mathcal{U}}$, $\alpha$ and $L$.
\end{lemma}

Finally, we give some properties on eigenvalues by the following lemma.

\begin{lemma}\label{lem:unique}
Let $\lambda^{\pm}_{\pm}$ and $\underline{\lambda}^{\pm}_{\pm}$ be defined by \eqref{eq:2.18} and \eqref{eq-global-unique-5.5}, respectively. Then, under the assumptions of Proposition \ref{prop-unique-5.1}, there hold the estimates
\begin{eqnarray}\label{eq-unique-A1}
\begin{split}
|\lambda^{\pm}_{+}-\underline{\lambda}^{\pm}_{+}|\leq C_{\lambda^{\pm}_{+}}\big(|z^{+,\flat}_{+}|+|z^{-,\flat}_{+}|\big),\quad
|\lambda^{\pm}_{-}-\underline{\lambda}^{\pm}_{-}|\leq C_{\lambda^{\pm}_{-}}\big(|z^{+,\flat}_{-}|+|z^{-,\flat}_{-}|\big),
\end{split}
\end{eqnarray}
where $C_{\lambda_{\pm}^{\pm}}$ are positive constants depending only on $\|U_{\pm}\|_{1,0;\underline{\tilde{\mathcal {N}}}_{\pm}}$.
\end{lemma}

\begin{proof}
We first consider the $\lambda_{+}^{+}-\underline{\lambda}^{+}_{+}$. By \eqref{eq:1.4}, \eqref{eq:1.8} and Proposition \ref{prop-unique-5.1}, we have
\begin{eqnarray*}
\begin{split}
\lambda_{+}^{+}&=\frac{\rho_{+} c_{+}^{2} u_{+}}
{u_{+}^{2}-c_{+}^{2}}\bigg(\frac{v_{+}}{u_{+}}
+ \sqrt{\frac{u_{+}^{2}+v_{+}^{2}}{c_{+}^2}-1}\bigg)\\
&=\frac{\gamma p_{+} \sqrt{2(\gamma-1)(1+\omega_{+}^2)\big[(\gamma-1)\underline{B}^{+}-\gamma A(\underline{S}^{+})^{\frac{1}{\gamma}} p_{+}^{1-\frac{1}{\gamma}}\big]}}
{2\big[(\gamma-1)\underline{B}^{+}-\gamma A(\underline{S}^{+})^{\frac{1}{\gamma}}p_{+}^{1-\frac{1}{\gamma}}\big]-\gamma (\gamma-1)A(\underline{S}^{+})^{\frac{1}{\gamma}}(1+\omega_{+}^2)p_{+}^{1-\frac{1}{\gamma}}}\\
&\quad\ \times \bigg(\omega_{+}
+ \sqrt{2\gamma^{-1}A(\underline{S}^{+})^{-\frac{1}{\gamma}} p_{+}^{\frac{1}{\gamma}-1}-(\gamma-1)^{-1}(\gamma+1)}\bigg)\\
&\triangleq \hat{\lambda}^{+}(\omega_{+}, p_{+}; \underline{B}^{+}, \underline{S}^{+}).
\end{split}
\end{eqnarray*}

Note that $\hat{\lambda}^{+}(0, \underline{p}^{+}; \underline{B}^{+}, \underline{S}^{+})=\underline{\lambda}^{+}_{+}$.
Then, for any $U_{+}\in C^{1,\alpha}(\underline{\tilde{\mathcal {N}}}_{+})$, there exists a $\theta\in (0,1)$ such that
\begin{align}\label{eq-unique-A2}
\begin{split}
&\hat{\lambda}^{+}(\omega_{+}, p_{+}; \underline{B}^{+}, \underline{S}^{+})-\hat{\lambda}^{+}(0, \underline{p}^{+}; \underline{B}^{+}, \underline{S}^{+})\\
&\quad =\nabla_{(\omega_{+},p_{+})}\hat{\lambda}^{+}_{+}(\theta\delta\omega^{\flat}_{+}, \underline{p}^{+}+\theta\delta p^{\flat}_{+}; \underline{B}^{+}, \underline{S}^{+})\cdot\big(\delta\omega^{\flat}_{+},\delta p^{\flat}_{+}\big).
\end{split}
\end{align}

Moreover, by \eqref{eq-global-unique-5.4}, the equality \eqref{eq-unique-A2} can be further written as
\begin{eqnarray*}
\begin{split}
&\hat{\lambda}^{+}(\omega_{+}, p_{+}; \underline{B}^{+}, \underline{S}^{+})-\hat{\lambda}^{+}(0, \underline{p}^{+}; \underline{B}^{+}, \underline{S}^{+})\\
&=\frac{\Lambda_{+}\partial_{\omega_{+}}\hat{\lambda}^{+}_{+}(\theta\delta\omega^{\flat}_{+}, \underline{p}^{+}+\theta\delta p^{\flat}_{+}; \underline{B}^{+}, \underline{S}^{+})+\partial_{p_{+}}\hat{\lambda}^{+}_{+}(\theta\delta\omega^{\flat}_{+}, \underline{p}^{+}+\theta\delta p^{\flat}_{+}; \underline{B}^{+}, \underline{S}^{+})}{2\Lambda_{+}}\cdot z^{+,\flat}_{+}\\
&\quad+\frac{\Lambda_{+}\partial_{\omega_{+}}\hat{\lambda}^{+}_{+}(\theta\delta\omega^{\flat}_{+}, \underline{p}^{+}+\theta\delta p^{\flat}_{+}; \underline{B}^{+}, \underline{S}^{+})-\partial_{p_{+}}\hat{\lambda}^{+}_{+}(\theta\delta\omega^{\flat}_{+}, \underline{p}^{+}+\theta\delta p^{\flat}_{+}; \underline{B}^{+}, \underline{S}^{+})}{2\Lambda_{+}}\cdot z^{-,\flat}_{+},
\end{split}
\end{eqnarray*}
which leads to the first estimate in \eqref{eq-unique-A1}
by taking $C_{\lambda^{+}_{+}}=\frac{\big(\|\Lambda_{+}\|_{0,0;\underline{\tilde{\mathcal {N}}}_{+}}+1\big)\cdot\Big\|\frac{\nabla_{(\omega_{+},p_{+})}\lambda^{+}_{+}}{\Lambda_{+}}\Big\|_{0,0;\underline{\tilde{\mathcal {N}}}_{+}}}{2}$.

Similarly,  we also have
\begin{eqnarray*}
|\lambda^{-}_{+}-\underline{\lambda}^{-}_{+}|\leq\frac{\big(\|\Lambda_{+}\|_{0,0;\underline{\tilde{\mathcal {N}}}_{+}}+1\big)\cdot\Big\|\frac{\nabla_{(\omega_{+},p_{+})}\lambda^{-}_{+}}{\Lambda_{+}}\Big\|_{0,0;\underline{\tilde{\mathcal {N}}}_{+}}}{2}\cdot \big(|z^{+,\flat}_{+}|+|z^{-,\flat}_{+}|\big)=C_{\lambda^{-}_{+}}\big(|z^{+,\flat}_{+}|+|z^{-,\flat}_{+}|\big),
\end{eqnarray*}
and
\begin{eqnarray*}
|\lambda^{\pm}_{-}-\underline{\lambda}^{\pm}_{-}|\leq\frac{\big(\|\Lambda_{-}\|_{0,0;\tilde{\underline{\mathcal {N}}}_{-}}+1\big)\cdot\Big\|\frac{D\lambda^{\pm}_{-}}{\Lambda_{-}}\Big\|_{0,0;\underline{\tilde{\mathcal {N}}}_{-}}}{2}\cdot \big(|z^{+,\flat}_{-}|+|z^{-,\flat}_{-}|\big)=C_{\lambda^{\pm}_{-}}\big(|z^{+,\flat}_{-}|+|z^{-,\flat}_{-}|\big),
\end{eqnarray*}
where $C_{\lambda^{-}_{+}}=\frac{\big(\|\Lambda_{+}\|_{0,0;\underline{\tilde{\mathcal {N}}}_{+}}+1\big)\cdot\Big\|\frac{\nabla_{(\omega_{+},p_{+})}\lambda^{-}_{+}}{\Lambda_{+}}\Big\|_{0,0;\underline{\tilde{\mathcal {N}}}_{+}}}{2}$ and $C_{\lambda^{\pm}_{-}}=\frac{\big(\|\Lambda_{-}\|_{0,\alpha;\underline{\tilde{\mathcal {N}}}_{-}}+1\big)\cdot\Big\|\frac{\nabla_{(\omega_{-},p_{-})}\lambda^{\pm}_{-}}{\Lambda_{-}}\Big\|_{0,0;\underline{\tilde{\mathcal {N}}}_{-}}}{2}$.
\end{proof}

\bigskip
\section*{Acknowledgements}
Junlei Gao {\color{black}{was}} supported in part by the NSFC Project under Grant No.12201568 and Key Laboratory of Mathematics and Engineering Applications, Ministry of Education.
Feimin Huang was partially supported by National Center for Mathematics and Interdisciplinary Sciences, AMSS, CAS and NSFC Grant No.11371349 and No.11688101.
Jie Kuang was supported in part by the NSFC Project under Grant No.11801549, No.11971024, No.12271507 and the Multidisciplinary Interdisciplinary Cultivation Project No.S21S6401 from Innovation Academy for Precision Measurement Science and Technology.
Dehua Wang was supported in part by NSF under grants DMS-1907519 and DMS-2219384.
Wei Xiang was supported in part by the Research Grants Council of the HKSAR, China (Project No. CityU
11304820, CityU 11300021, CityU 11311722, and CityU 11305523).


\bigskip

\begin{thebibliography}{99}

\bibitem{bae} M.~Bae,
{Stability of contact discontinuity for steady Euler system in the infinite duct}.
Z. Angew Math. Phys., 64~(2013), 917-936.


\bibitem{bae-park1} M.~Bae, H.~Park,
{Contact discontinuity for 2-D inviscid compresible flows in infinitely long nozzles}.
SIAM J. Math. Anal., 51~(2019), 1730-1760.


\bibitem{bae-park2} M.~Bae, H.~Park,
{Contact discontinuity for 3-D axisymmetric inviscid compresible flows in infinitely long cylinders}.
 J. Differential Equations, 267~(2019), 2824-2873.


\bibitem{chen-huang-wang-xiang} G.-Q.~G.~Chen, F.~Huang, T.~Wang, W.~Xiang,
{Steady Euler flows with large vorticity and characteristic discontinuities in arbitrary infinitely long nozzles}.
Adv. Math., 346~(2019), 946-1008.


\bibitem{chen-kuang-zhang} G.-Q.~G.~Chen, J.~Kuang, Y.~Zhang,
{Two-dimensional steady supersonic exothermically reacting Euler flow past Lipschitz bending walls}.
SIAM J. Math. Anal., 49~(2017), 818-873.


\bibitem{chen-kukreja} G.-Q.~G.~Chen, V.~Kukreja,
{$L^1$-stability of vortex sheets and entropy waves in steady supersonic Euler flows over Lipschitz walls}.
Discrete Contin. Dyn. Syst., 43~(2023), 1239-1268.


\bibitem{chen-kukreja-yuan} G.-Q.~G.~Chen, V.~Kukreja, H.~Yuan,
{Stability of transonic characteristic discontinuities in two-dimensional steady compressible Euler flows}.
Z. Angew. Math. Phys., 64~(2013), 1711-1727.


\bibitem{chen-wagner} G.-Q.~G.~Chen, D.~Wagner,
{Global entropy solutions to exothermically reacting, compressible Euler equations}.
J. Differential Equations, 191~(2003), 277-322.


\bibitem{chen-zhang-zhu} G.-Q.~G.~Chen, Y.~Zhang, D.~Zhu,
{Stability of compressible vortex sheets in steady supersonic Euler flows over Lipschitz walls}.
SIAM J. Math. Anal., 38~(2007), 1660-1693.


\bibitem{chen-xin-zang} J.~Chen, Z.~Xin, A.~Zang,
{Subsonic flows past a profile with a vortex line at the trailing edge}.
SIAM J. Math. Anal., 54~(2022), 912-939.


\bibitem{chensx1} S.-X.~Chen,
{Stability of a Mach configuration}.
Comm. Pure Appl. Math., 59~(2006), 1-35.


\bibitem{chensx2} S.-X.~Chen,
{E-H type Mach reflection and its stability}.
Commun. Math. Phys., 315~(2012), 563-602.


\bibitem{chen-hu-fang} S.-X.~Chen, D.~Hu, B.~Fang,
{Stability of the E-H type regular shock refraction}.
J. Differential Equations, 254~(2013), 3146-3199.


\bibitem{chen-qu} S.-X.~Chen, A.~Qu,
{Interaction of rarefaction waves and vacuum in a convex duct}.
Arch. Ration. Mech. Anal., 213~(2014), 423-446.


\bibitem{cheng-du} J. Cheng, L. Du,
{Compressible subsonic impinging flows}.
Arch. Ration. Mech. Anal., 230~(2018), 427-458.


\bibitem{cheng-du-wang} J. Cheng, L. Du, Y. Wang,
{The existence of steady compressible subsonic impinging jet flows}.
Arch. Ration. Mech. Anal., 229~(2018), 953-1014.


\bibitem{cheng-du-wang-2} J. Cheng, L. Du, Y. Wang,
{The uniqueness of the asymmetric jet flow}.
J. Differential Equations, 269~(2020), 3794-3815.


\bibitem{cheng-du-xiang} J. Cheng, L. Du, W. Xiang,
{Compressible subsonic jet flows issuing from a nozzle of arbitrary cross-section}.
J. Differential Equations, 266~(2019), 5318-5359.


\bibitem{cheng-du-zhang} J. Cheng, L. Du, Q. Zhang,
{Existence and uniqueness of axially symmetric compressible subsonic jet impinging on an infinite wall}.
Interfaces Free Bound. 23~(2021), 1-58.


\bibitem{courant-friedrichs} R.~Courant, K.~O.~Friedrichs,
{Supersonic Flow and Shock Waves}.Interscience Publishers Inc., New York, 1948.


\bibitem{du-wang} L. Du, X. Wang,
{Steady compressible subsonic impinging flows with non-zero vorticity}.
J. Differential Equations, 268~(2020), 2587-2621.


\bibitem{fang-wang-yuan} B.~Fang, Y.~Wang, H.~Yuan,
{Reflection and refraction of shocks on an interface with a reflected rarefaction wave}.
J. Math. Phys., 52~(2011), 14 pp.


\bibitem{gilbarg-trudinger} D.~Gilbarg, N.~S.~Trudinger,
{Elliptic Partial Differential Equations of Second Order}. Classics in Mathematics. Springer, Berlin, 2001.


\bibitem{gao-yuan} J.~Gao, H.~Yuan,
{Stability of stationary subsonic compressible Euler flows with mass-additions in two-dimensional straight ducts}.
J. Differential Equations, 334~(2022), 87-156.


\bibitem{gao-liu-yuan} J.~Gao, L.~Liu, H.~Yuan,
{On stability of transonic shocks for stationary Rayleigh flows in two-dimensional ducts}.
SIAM J. Math. Anal., 52~(2020), 5287-5337.


\bibitem{huang-kuang-wang-xiang1} F.~Huang, J.~Kuang, D.~Wang, W.~Xiang,
{Stability of supersonic contact discontinuity for 2-D steady compressible Euler flows in a finitely long nozzle}.
J. Differential Equations, 266~(2019), 4337-4376.


\bibitem{huang-kuang-wang-xiang2} F.~Huang, J.~Kuang, D.~Wang, W.~Xiang,
{Stability of transonic contact discontinuity for two-dimensional steady compressible Euler flows in a finitely long nozzle}.
Ann. PDE, 7~(2021), Paper No. 23, 96 pp.


\bibitem{lai} G.~Lai,
{Global nonisentropic rotational supersonic flows in a semi-infinite divergent duct}.
SIAM J. Math. Anal., 52~(2020), 5121-5154.


\bibitem{li-xu-yu-lv-cheng} R.~Li, J.~Xu, K.~Yu, Z.~Lv, K.~Cheng,
{Design and analysis of the scramjet nozzle with contact discontinuity}.
Aerospace Science and Technology, 113~(2021), 10669


\bibitem{li-yu} T.~Li, W.~Yu,
{Boundary Value Problems for Quasilinear Hyperbolic Systems}. Duke University Mathematics Series, vol.5, 1985.


\bibitem{pei-xiang} Y. Pei, W. Xiang,
{Global subsonic jet with strong transonic shock over a convex cornered wedge for the two-dimensional steady full Euler equations}.
SIAM J. Math. Anal., 54~(2022), 4407-4451.


\bibitem{qu-xiang} A.~Qu, W.~Xiang,
{Three-dimensional steady supersonic Euler flow past a concave cornered wedge with lower pressure at the downstream}.
Arch. Rational Mech. Anal., 228~(2018), 431-476.


\bibitem{shi-tang-xie} W. Shi, L. Tang, C. Xie,
{Variational structure and two dimensional jet flows for compressible Euler system with non-zero vorticity}.
preprint at arXiv:2006.05672, 2020.


\bibitem{menshov}I. Menshov,
{Quasi-one-dimensional approximation in two-dimensional problems of gas dynamics}.
Fluid Dyn., 24~(1989), 277-284.


\bibitem{wang-yuan} Y.~Wang, H.~Yuan,
{Weak stability of transonic contact discontinuities in three dimensional steady non-isentropic compressible Euler flows}.
Z. Angew. Math. Phys., 66~(2015), 341-388.


\bibitem{wang-yu} Y.~Wang, F.~Yu,
{Structural stabiliy of supersonic contact discontinuities in three-dimensional compressible steady flows}.
SIAM J. Math. Anal., 47~(2015), 1291-1239.

{\color{black}
\bibitem{weng-zhang-1} S.~Weng, Z.~Zhang,
{Subsonic flows with a contact discontinuity in a two-dimensional finitely long curved nozzle}.
preprint at	arXiv:2303.15096, 2023.


\bibitem{weng-zhang-2} S.~Weng, Z.~Zhang,
{Supersonic flows with a contact discontinuity to the two-dimensional steady rotating Euler system}.
preprint at arXiv:2307.00199, 2023.


\bibitem{weng-zhang-3} S.~Weng, Z.~Zhang,
{Subsonic flows with a contact discontinuity in a finitely long axisymmetric cylinder}.
preprint at arXiv:2308.02758, 2023.
}

\bibitem{williams} F.A.~Williams, {Combustion Theory}. 2nd ed., The Benjamin-Cummings Publishing Company, 1985.


\bibitem{xiang-zhang-zhao} W.~Xiang, Y.~Zhang, Q.~Zhao,
{Two-dimensional steady supersonic exothermically reacting Euler flows with strong contact discontinuity over a Lipschitz wall}.
Interfaces Free Bound., 20~(2018), 437-481.


\bibitem{xu-yin2} G.~Xu, H.~Yin,
{3-D global supersonic Euler flows in the infinite long divergent nozzles}.
SIAM J. Math. Anal., 53~(2021), 133-180.


\bibitem{yuan-zhao} H.~Yuan, Q.~Zhao,
{Stabilization effect of frictions for transonic shocks in steady compressible Euler flows passing three-dimensional ducts}.
Acta Math. Sci. Ser. B (Engl. Ed.), 40~(2020), 470-502.
\end{thebibliography}
\end{document}